\documentclass[reqno]{amsbook}
\usepackage[utf8]{inputenc} 
\usepackage{enumitem}
\usepackage{amssymb}
\usepackage[initials]{amsrefs} 
\usepackage{array}    
\usepackage{graphicx} 
\usepackage[counterclockwise]{rotating} 
\usepackage{verbatim} 
\usepackage{tikz}
\usepackage{tikz-cd}  
\usepackage[font=small]{caption}

\usepackage{hyperref} 

\newcommand{\A}{\mathcal A}

\newcommand{\silenceable}[1]
  {#1} 

\usetikzlibrary{trees,math}

\tikzset{
 every picture/.style={
 inner sep=1}
}

\graphicspath{{Figures_Gindikin/}}


\numberwithin{section}{part}
\numberwithin{subsection}{section}
\numberwithin{equation}{section}
\numberwithin{table}   {section}
\numberwithin{figure}  {section}

\theoremstyle{plain}
 \newtheorem {theorem}    {Theorem}[section]
 \newtheorem {proposition}[theorem]{Proposition}
 \newtheorem {lemma}      [theorem]{Lemma}
 \newtheorem {corollary}  [theorem]{Corollary}
\theoremstyle{definition}
 \newtheorem {definition} [theorem]{Definition}

 \newtheorem {example}    [theorem]{Example}
\theoremstyle{remark}
 \newtheorem {remark}     [theorem]{Remark}

\newcommand{\abs}    [1]{\left\lvert#1\right\rvert}

\newcommand{\mb}   {\mathbf}
\newcommand{\boldh}{\mathbf{h}}
\newcommand{\boldphi}{\bs{\phi}}
\newcommand{\bs}{\boldsymbol}

\newcommand{\mathI}{\mathbb{I}}
\newcommand{\mathC}{\mathbb{C}}
\newcommand{\mathN}{\mathbb{N}}
\newcommand{\mathZ}{\mathbb{Z}}
\newcommand{\mathR}{\mathbb{R}}
\newcommand{\mathF}{\mathbb{F}}
\newcommand{\mathQ}{\mathbb{Q}}
\newcommand{\mathT}{\mathbb{T}}
\newcommand{\CE}   {C^E} 
\newcommand{\CV}   {C^V} 
\newcommand{\calE} {\mathcal{E}}	
\newcommand{\calH} {\mathcal{H}}
\newcommand{\calF} {\mathcal{F}}    
\newcommand{\calS} {\mathcal{S}}	
\newcommand{\mathfrakE} {\mathfrak{E}}	
\newcommand{\calU} {\mathcal{U}}   
\newcommand{\Radial}  {\mathcal{R}_\#} 
\newcommand{\HorV} {\mathcal{H}_V} 
\newcommand{\HorE} {\mathcal{H}_E} 
\newcommand{\sHorV} {\mathcal{SH}_V} 
\newcommand{\sHorE} {\mathcal{SH}_E} 
\newcommand{\HorF} {\mathcal{H}_F} 
\newcommand{\VRad} {\mathcal{R}_V} 
\newcommand{\ERad} {\mathcal{R}_E} 
\newcommand{\FRad} {\mathcal{R}_F} 
\newcommand{\Rad}  {\mathcal{R}}   
\DeclareMathOperator{\Poisstr}{\mathcal K}    
\newcommand{\Poiss}{K}      				  
\newcommand{\CircV}{\mathcal{C}_V} 

\newcommand{\spherFour}{\mathcal{G}}       
\newcommand{\radspherFour}{\mathcal{F}}    
\newcommand{\lround}{(}
\newcommand{\rround}{)}

\newcommand{\graphE}{E}

\DeclareMathOperator{\Aut} {Aut}
\DeclareMathOperator{\dist}{dist}

\DeclareMathOperator{\interior} {Int}
\DeclareMathOperator{\xiomega} {\xi_\omega}  

\newcommand{\mathfrakF} {\mathfrak{F}} 

\DeclareMathOperator{\Real} {{\rm Re}}
\DeclareMathOperator{\Imag }{{\rm Im}}
\DeclareMathOperator {\Conv} {{\rm Cv}}
\DeclareMathOperator{\spectrum}{sp}

\DeclareMathOperator{\Ind} {Ind}

\newcommand{\spaceU} {\mathfrak{U}} 
\newcommand{\spaceW} {\mathfrak{W}} 
\newcommand{\spaceF} {\mathfrak{F}} 
\newcommand{\spacef} {\mathfrak{f}} 

\begin{document}
\frontmatter

\title
[Radon transforms on homogeneous trees] 
{Integral geometry and harmonic analysis on homogeneous trees}
\author
[E.     Casadio~Tarabusi]
{Enrico Casadio~Tarabusi}
\address
{Dipartimento di Matematica ``G. Castelnuovo''\\
 Sapienza Universit\`a di Roma\\
 Piazzale A. Moro 5\\00185~Roma\\Italy}
\email{enrico.casadio\_tarabusi@mat.uniroma1.it}
\author
[S.       ~G. Gindikin]
{Sem\"{e}n~G. Gindikin}
\address
{Department of Mathematics\\
 Rutgers University\\
 110 Frelinghuysen Rd.\\Piscataway, NJ~08854-8019\\USA}
\email{gindikin@math.rutgers.edu}
\author
[M.     ~A. Picardello]
{Massimo~A. Picardello}
\address
{Dipartimento di Matematica\\
 Universit\`a di Roma ``Tor Vergata''\\
 Via della Ricerca Scientifica\\00133~Roma\\Italy}
\email{picard@mat.uniroma2.it}
%
%
\begin{abstract}
This monograph studies integral geometry on homogeneous trees, and more precisely the different horospherical Radon transforms that arise by regarding the tree as a simplicial complex whose simplices are vertices, edges or flags (a flag is a pair  consisting of an edge and one of its end-points). The ends 
(infinite geodesic rays starting at a reference vertex) provide a boundary $\Omega$ for the tree and are the set of tangency points of horospheres. Then the horospheres form a trivial principal fiber bundle with base $\Omega$ and fiber $\mathZ$. There are three such fiber bundles, consisting of horospheres of vertices, edges or flags, but they are isomorphic: however, no isomorphism between these fiber bundles maps special sections to special sections (a special section consists of the set of horospheres through a given vertex, edge or flag).

The groups of automorphisms of the fiber bundles contain a subgroup $A$ of parallel shifts, analogous to the Cartan subgroup of a semisimple group. The normalized eigenfunctions of the Laplace operator on the homogeneous tree, called spherical functions, are boundary integrals of complex powers of the Poisson kernel, that is characters of $A$, and are matrix coefficients of representations induced from $A$ in the sense of Mackey, the so-called \emph{spherical representations}. 
 
 The vertex-horospherical Radon transform consists of summation over the vertices in each vertex-horosphere, and similarly for edges or flags.
 By direct computation, we prove inversion formulas for all these Radon transforms, 
 and give applications to harmonic analysis and the Plancherel measure on trees. We also present the theory of spherical functions and show that the spherical representations for vertices and edges are equivalent. Also, we define the Radon back-projections and find the inversion operator of each Radon transform by composing it  with its back-projection. This gives rise to a convolution operator with respect to the group of automorphisms of the tree, whose symbol is obtained via the spherical Fourier transform, and its inverse gives the Radon inversion formula.

Homogeneous and semi-homogeneous trees are the lowest rank affine Bruhat--Tits buildings. The present monograph focuses attention onto homogeneous trees, because in the semi-homogeneous case the Poisson kernels are not characters of $A$. Affine buildings are simplicial complexes and homogeneous spaces for higher rank simple $p$-adic groups. Our Radon transforms extend immediately to this more complicated environment. Therefore the present work could be a preliminary step towards a revisitation of the theory of spherical functions on groups of $p$-adic type in terms of integral geometry.
\end{abstract}
\maketitle

\newpage
\tableofcontents

\mainmatter

\newpage
\part{Horospheres on trees}
%
\section{Motivation}
\label{Sec:intro}
%
%
Harmonic analysis on a tree $T$ and representation theory of groups acting on it have been studied in many articles and books: see, for instance, \cites{Cartier, Figa-Talamanca&Picardello-JFA,Figa-Talamanca&Picardello, Faraut&Picardello, Figa-Talamanca&Steger,  Figa-Talamanca&Nebbia, Casadio_Tarabusi&Picardello-algebras_generated_by_Laplacians} and their bibliographies. In all these references, harmonic analysis was studied by looking at algebras of functions on the set $V$ of vertices of a homogeneous tree (or semi-homogeneous, in the last reference) and at the action on $V$ of subgroups of the automorphism group of $T$. The present book's aim is to give a different approach to harmonic analysis on trees, based on integral geometry and inspired by \cite{Gelfand&Graev&Vilenkin}, starting not from the set of vertices, but from horospheres on $T$ and their fiber bundle $\calH$, whose fibers, isomorphic to $\mathZ$,  are given by the horospherical index and whose base is the boundary $\Omega$ ot $T$ (the set of geodesic rays in $T$ starting at a reference element, now ideally regarded as boundary points of horospheres). The reason of our approach is that $\Aut \calH$ is much larger than $\Aut T$, and in particular has a non-trivial center which contains a group $A\approx \mathZ$ of parallel shifts along fibers that can be used, in harmonic analysis on trees, much in the same way as the Cartan subgroup of the Iwasawa decomposition for semi-simple groups and their symmetric spaces.

Here is an outline of our approach. We start with the fiber bundle $\calH$ of horospheres in a homogeneous 
tree $T$. Then we consider some  
space of functions $\spaceU$ on $\calH$ and decompose it as direct integral over $\widehat{A}\approx\mathT$ of 
subspaces $\spaceU_\sigma$ invariant under the action of the parallel shift group $A$, where $A$ acts via its characters $\sigma$. Since $A$ commutes with $\Aut \calH$, the spaces $\spaceU_\sigma$ are invariant under $\Aut \calH$, and the action of $\Aut \calH$ (or its subgroups) gives rise to representations $\pi_\sigma$ of $\Aut \calH$ and its subgroups, called \emph{spherical representations}. Indeed, for each $\sigma$, $\pi_\sigma$ has a unique normalized radial coefficient, the \emph{spherical function} at the corresponding eigenvalue. These representations are induced representations from $A$ to $\Aut \calH$ in the sense of Mackey, and can be naturally realized on the boundary $\Omega\approx \calH/A$. For this goal it is more appropriate to give a different model for $\Omega$, by realizing it as a   \emph{special section} of $\calH$, defined as the set $\Sigma_v$ of all horospheres that pass through a given vertex $v$. Although this definition is obviously based upon the set of vertices, these special sections could be also identified in an abstract fiber bundle $\calH$ equipped with an appropriate definition of join of two elements. The choice of a special section induces a global chart on $\calH$, that is, it gives rise to a specific choice of integer coordinates on the fibers: the special section is at level 0. This allows to write explicitly the representations $\pi_\sigma$ in these coordinates. 

Next, we observe that a special section $\Sigma_{v}$ is endowed with a unique normalized measure $\nu_{v}$ invariant under the action of the stability subgroup $K_{v}\subset \Aut T$ of $v$. When we choose a different special section $\Sigma_{v_0}$, the invariant probability measure $\nu_v$ turns out to be absolutely continuous with respect to $\nu_{v_0}$. The Radon--Nikodym derivative is homogeneous along the fibers, and all homogeneous functions along the fibers are related to it: in the fiber coordinate $n$ their values are $\sigma(n)$, where $\sigma$ are the characters of $A$. Therefore, for each character $\sigma$, we obtain a Poisson transform from functions on $\Sigma_{v_0}$ to functions on $V$. The image of the Poisson transform consists of eigenfunctions (with eigenvalues depending on $\sigma$) of the generator of the algebra (under the convolution defined by $\Aut T$) of functions radial around $v_0$, that in the literature is usually called the \emph{Laplace operator} $\mu_{v_0}$. It is known that this map is surjective from finitely additive measures on $\Sigma_{v_0}$ (sometimes called \emph{boundary distributions} in the literature) and eigenfunctions of $\mu_{v_0}$. Some eigenvalues are positive-definite with respect to $\Aut T$, and for their eigenvalues the representations $\pi_\sigma$ are unitary.

In the direct integral decomposition of the function $f\in\spaceU$ into the spaces $\spaceU_\sigma$, the component $f_\sigma$ in $\spaceU_\sigma$ defines the value of the spherical Fourier transform $\spherFour^V_{v_0} f(\sigma,\omega)$ at the character $\sigma$ and the fiber $\omega$. If the points of the horospherical fiber bundle $\calH$ are realized as sets of vertices (horospheres in $V$), this allows to introduce a horospherical transform (usually called Radon transform in the literature) from functions on $V$ to functions on $\calH$, given by summation over each horosphere. The interplay between the spherical transform, the horospherical transform and the Fourier series expansion in $\widehat A \approx \mathT$ leads to the Fourier slice theorem, that can be used to derive a Plancherel formula for the spherical transform over $\mathT$. The integral weight that appears in the Plancherel formula is related to a function that has an important analogue in harmonic analysis on semisimple groups, the Harish-Chandra $c-$function, and it was originally computed on trees in \cite{Figa-Talamanca&Picardello}; see also \cites{Kuhn&Soardi, Faraut&Picardello} for the analogous set-up of polygonal graphs, that correspond to free products of cyclic groups.

By introducing a suitable space of test functions $\calS$ on $V$, small enough that its dual $\calS'$ (space of distributions on $V$) contains all eigenfunctions of the Laplace operators corresponding to unitary representations, we extend the horospherical transform to distributions. By the same token we also prove a Paley--Wiener theorem for the spherical Fourier transform.

All of this will also be done for edges. However, the vertex-horospherical fiber bundle is isomorphic to the edge-horospherical fiber bundle via a canonical isomorphism, so several results are similar, notably those on spherical functions and representation theory, although the eigenvalue map in the case of edges has a real shift and the $\ell^1$ spectrum of the corresponding Laplacian is no longer centered at the origin. The spherical representations for vertices and for edges have the same expression: they act on some $A$-isotypic space of functions on the respective horospherical bundles and are based on the choice of a special section therein. Because the canonical isomorphism does not map special sections to special sections, these spherical representations for functions on vertices and on edges are not obviously equivalent. However, we shall prove that there is an isomorphism that maps any given special section of $\HorV$ to any special section of $\HorE$. This shows that the spherical representations for vertices and edges that correspond to the same character of $A$ are equivalent, and that the theory of spherical functions,  Plancherel formula and Radon back-projection for vertices and for edges are related: we shall present in Chapter \ref{Chap:spherical_functions} these results for vertices only.

Indeed, we define a back-projection operator $R^*$ \emph{(dual Radon transform}) and show that $R^*R$ is a convolution operator that maps $\calS$ to $\calS'$, and compute its symbol. By inverting the symbol we obtain another proof, not combinatorial but analytic, of the inversion formula for the Radon transform (in the case of the Radon transform on vertices this approach goes back to \cite{Betori&Faraut&Pagliacci}).

More precisely, we consider three different horospherical Radon transforms on $T$: one, $\VRad$, for functions on vertices, another, $\ERad$, for functions on edges, and the third, $\FRad$, for functions on flags. The first transform has been introduced in~\ocite{Cartier} and extensively studied afterwards. For homogeneous trees the inversion of $\VRad$ appears in~\ocite{Betori&Pagliacci-2}, where a radial inversion formula is obtained by a combinatorial approach similar to the one carried out in the present monograph for the same purpose, then again in~\ocite{Betori&Faraut&Pagliacci}, where the same result is obtained by using the dual Radon transform, and in~\ocite{Casadio_Tarabusi&Cohen&Colonna}. A more ambitious breakthrough is in~\ocite{Casadio_Tarabusi&Cohen&Picardello}, where the inversion is carried out recursively for semi-homogeneous or even general non-homogeneous trees, and explicit inversion formulas are derived for both the homogeneous and semi-homogeneous environments in a non-geometric way, namely by an ingenious splitting of the space of horospheres.
This geometric inversion, that is based on the  Radon back-projection, was achieved in \ocite{Casadio_Tarabusi&Cohen&Picardello}  for functions on vertices of homogeneous trees, a result that had been already obtained via a different approach in \ocite{Betori&Faraut&Pagliacci}.

Instead, the horospherical Radon transform $\ERad$ for functions on the edges $E$ of a homogeneous or semi-homogeneous tree, and $\FRad$ for functions on flags,
were never considered so far in the literature. Here we give direct proofs of inversion formulas for all these  Radon transforms.

The horospherical transform, or Radon transform, introduced in \cite{Radon} for Euclidean spaces (and shortly before on spheres \cite{Funk}), was later extended to complex classical groups 
\cites{Gelfand&Graev, Gelfand&Graev&Vilenkin} and symmetric spaces (see \cite{Helgason-GGA} for references). More recently, it has been 
introduced on homogeneous trees in \cite{Cartier} and studied in \cites{Betori&Faraut&Pagliacci, Betori&Pagliacci-2}; later on, it was extended to $V$ for non-homogeneous trees $T$ in \cites{Casadio_Tarabusi&Cohen&Picardello, Casadio_Tarabusi&Cohen&Colonna}.

Our motivation to study horospherical transforms and harmonic analysis based on horospheres for functions defined not only on vertices of a tree, but also on edges (and flags, i.e. oriented edges), is the following. A tree is a simplicial complex whose simplices are 
vertices and edges. Then homogeneous and semi-homogeneous trees are the lowest rank case of the simplicial complexes defined by the Bruhat--Tits buildings (in their combinatorial definition). Therefore the present book could be a preliminary step towards a revision in terms of integral geometry of the fundamental work \ocite{Macdonald} on the spherical Fourier transform on $p$-adic groups. 

However, the Poisson kernel on semi-homogeneous trees is not a character of the Cartan subgroup $A$, hence it does not lift to a function on the horospherical fiber bundle and has no interest in integral geometry: its computation relies on methods of probability theory (the Markov chain generated by the isotropic Laplacian) or difference equations \cite{Casadio&Picardello-semihomogeneous_spherical_functions}. Therefore, in this monograph, we limit attention to homogeneous trees.

Other types of \emph{integral} (i.e., summation)  transforms on trees have been considered in the literature. One is the X-ray transform, that is, the  operator acting on functions on the vertices $V$ of a (homogeneous or non-homogeneous) tree $T$ by summation over geodesics; it was introduced and studied in full detail, only for functions on vertices, in~\ocite{Berenstein&Casadio_Tarabusi&Cohen&Picardello}, where inversion formulas are proved for general non-homogeneous trees. This transform does not appear to be related to spherical harmonic analysis. Another intriguing transform is the circle transform (summation over all circles), an over-determined transform introduced and studied in~\ocite{Casadio_Tarabusi&Gindikin&Picardello}, again for vertices only. The circle transforms (both for vertices and for edges) could have a significant role in the present context, as tools to describe eigenfunctions of the Laplace operators.

\section{Trees, geodesic rays and automorphisms}
%
%
\subsection{Vertices, edges, flags}
\label{SubS:Vertices_edges_flags}
A tree is a connected, infinite, locally finite graph without loops. We restrict attention to homogeneous or semi-homogeneous trees. A tree is homogeneous if each vertex has the same number of neighbors $q+1$; the number $q$ is called homogeneity degree of the vertices. Instead, $T$ is semi-homogeneous if each vertex has one of two possible homogeneity degrees $q_+\neq q_-$ and neighbors have different homogeneities.

We denote by $V=V(T)$ the set of vertices of the tree $T$, and by $E=E(T)$ the set of edges. The notion of graph assigns to every vertex its adjacent vertices, and to every edge its adjacent edges: in other words, it defines, for each $v\in V$ the set of vertices at distance 1, usually called the star of $v$ in graph theory, and it also defies  the set of edges at distance 1 fron an edge $e$. By iteration, this  leads to the natural notion of distance between the vertices in any graph, obtained by the length of a minimal path between two vertices, or between edges, or between vertices and edges. Here is the definition:
\begin{definition}[Distance]\label{Def:Distance}
The distance $\dist(v,v')$ between two vertices $v,v'$ is the number of edges in the shortest path (or chain) that connects them. Vertices at distance $1$ are called \textit{adjacent}. Similarly, the distance between two edges is the number of vertices in the shortest path connecting them. In particular, two edges $e,e'$ are adjacent if they share exactly one vertex. We clarify the notion of shortest path as follows: the \textit{chain} of edges from $e$ to $e'$ is the minimal finite sequence $e=e_0,e_1,\dotsc,e_n=e'$ in $E$ such that $e_{j-1},e_j$ are adjacent for every $j=1,\dotsc,n$. So, the \textit{distance} of $e$ and $e'$ is the non-negative integer $n$. 

Similarly, the distance between a vertex $v$ and an edge $e$ is  
\[
\dist(v,e):=\frac12 + \min\{\dist(v,v'):v'\in e\} = \frac12 + \min\{\dist(e',e):e'\ni v\} := \dist (e,v).
\]
Note that the distance between vertices and edges has half-integer values: in particular, the distance from an edge to one of its endpoint vertices is $1/2$.
\end{definition}

Once a reference vertex $v_0$ and a reference edge $e_0$ are chosen, the distance $\dist(v,v_0)$ is called the \emph{length} of $v$ and denoted by $|v|$, and similarly we set $\dist(e,e_0)=|e|$.

We denote by $F=F(T)$ the set of all flags in $T$. 
A flag is a pair consisting of an edge and one of its vertices (see Definition \ref{def:flags} below). 
Each edge in $E$ belongs to exactly two flags, depending on the choice of endpoint. An appropriate distance between flags will be defined later in \eqref{eq:distance_on_flags}, as well as a unified notion of distance between flags, edges and vertices (Definition       \ref{def:unified_distance}). All the contents of the present Section will be extended to flags in Section \ref{Sect:Flags}.

\subsection{Automorphisms and their orbits}
\label{SubS:automorphisms}
The automorphisms of $T$ are the bijections that preserve adjacency, hence they preserve the distance introduced above. If $T$ is homogeneous, then its group of automorphisms $\Aut T$ is transitive on $V$, $E$ and $F$. Indeed, it is doubly transitive: it maps any pair of vertices at a given distance to any other pair of vertices at the same distance, and does the same to edges and flags.

\begin{proposition} \label{prop:Aut_T_for_T_homogeneous}
On a homogeneous tree, $\Aut T$ is doubly transitive (two-point homogeneous) on $V$: every pair of vertex $v_0, v_1$ at any given distance $n$ is mapped by some automorphism to any other pair of vertices $v'_0, v'_1$ at the same distance $n$. The same holds for the action of $\Aut T$ on edges and flags.
\end{proposition}
\begin{proof} For simplicity, we limit attention to the action on vertices. It is enough to  prove that $\Aut T$ is transitive on $V$ and that, for each distance $d>0$.  the stability subgroup at each vertex $v$ is transitive on the circle of all vertices at  distance $d$ from $v$. Transitivity on the circle is easy: if $\dist(v_1,v)=\dist(v_2,v)$, consider the last vertex in common between the geodesic paths $[v,v_1]$ and $[v,v_2]$ from $v$ to $v_1$, $v_2$ respectively. In the sequel, this last vertex will be called the \emph{join} $j(v_1,v_2,v)$. If the join has length $k>d$, define the automorphism $\lambda$ as the identity on the disc of radius $k$ around $v$. Let now $w$ be any vertex at distance $k$ from $v$. Then define $\lambda$ on vertices of radius $k+1$ that are neighbors of $w$
as a permutation of these vertices, chosen in any way such that the $(k+1)-$th vertex of $[v,v_1]$ is mapped to the $(k+1)-$th vertex of $[v,v_1]$. Iterate the same construction on vertices at distance $k+1, k+2, \dots, $ from $v_1$ to obtain an automorphism $\lambda$ that fixes $v$ and maps $v_1$ to $v_2$.

Now let us prove transitivity on $V$: for all $v_0, v_1\in V$ let us produce a $\lambda\in \Aut T$ such that $\lambda (v_0) = v_1$.
As observed in the proof of Proposition \ref{prop:unique_geodesic_ray_to_omega_starting_at_v}, for all $v_0, v_1\in V$ there is a unique geodesic path from $v_0$ to $v_1$. Therefore it is enough to show that there exists an automorphism $\lambda$ that maps $v_0$ to $v_1$ when $v_0$ and $v_1$ are neighbors. This is proved ain a way similar to the previous argument, s follows. 

Define $v_1=\lambda (v_0)$. On all neighbors $w\neq v_1$ of $v_0$ let $\lambda(w)$ act as a permutation, and then, inductively, let $\lambda$ act as a permutation on the descendants of each vertex $u$, that is,  the neighbors of $u$ at the other side of $v_0$. Now we have built $\lambda$ with the property that $\dist (u, v_0) = \dist (\lambda(u), v_1)$ for every $u\in V$. We need to show that $\dist (u, v) = \dist (\lambda(u), \lambda(v))$ for each $u,v\in V$. Let $j\in V$ be the join $j(u,v,v_0)$, and let $n=\dist(u,j)$, $m=\dist(v,j)$. Then $\dist(u,v)=n+m$. Similarly, $\dist(\lambda(u),\lambda(v))=\dist(\lambda(u),\lambda(j))+\dist(\lambda(v),\lambda(j))$, because $\lambda$ maps joins to joins, by construction. Let $[u_0=j, u_1,\dots,u_n=u]$ be the geodesic path from $j$ to $u$. Then, again by the way $\lambda$ is constructed, the geodesic path from $\lambda(j)$ to $\lambda(u)$ is $[\lambda(u_0), \lambda(u_1),\dots,\lambda(u_n)]$. Therefore, $\dist(\lambda(u),\lambda(j))=n$. Similarly, $\dist(\lambda(u),\lambda(j))=m$, hence $\dist(\lambda(u),\lambda(v))=n+m=\dist(u,v)$, and the proof is finished.

Similarly, 
$\Aut T$ is doubly transitive on $E$: it maps any pair of edges $e_1, e_2$ to any other pair $e'_1, e'_2$ the same distance apart.

For the same reason, $\Aut T$ is doubly transitive on the space of flags $F$: any pair of flags can be moved to any other pair at the same distance. 
\end{proof}

\subsection{Boundary of a tree and automorphisms}
\label{SubS:Boundary}
A geodesic ray on $T$ is an infinite set of contiguous vertices without repetitions, starting at some vertex $v$. On the set of all geodesic rays we introduce the equivalence relation such that two rays are equivalent if they merge after a finite number of steps and coincide afterwards. Each equivalence class $\omega$ is called a \emph{boundary point} of $T$. We obtain the same boundary
by regarding geodesic rays as sequences of adjoining edges. 
If we would like to fix a reference vertex $v_0$, in each equivalence class we could choose the representative given by the geodesic ray that starts at $v_0$ (or at a reference edge $e_0$).
The boundary of $T$ is denoted by $\Omega$.

Equivalently, $\Omega$ can be constructed in the same way by means of geodesic rays consisting of consecutively adjacent edges. Note that the boundary $\Omega$ is the same for $V$ and $E$, because each path of vertices corresponds to a unique path of edges.

\begin{proposition} \label{prop:unique_geodesic_ray_to_omega_starting_at_v}
For each vertex $v$,
the equivalence class that defines a boundary point $\omega$ contains exactly one geodesic ray that starts at $v$. A similar statement holds for equivalence classes of geodesic rays consisting of edges.
\end{proposition}
\begin{proof}
Let $\{u_0,u_1,u_2,\dots\}$ be a geodesic ray that belongs to the equivalence class $\Omega$, and denote it by $\bs{u}$. Given a vertex $v$ we now illustrate the merging procedure of a ray starting at $v$ and the ray $\bs{u}$ as follows. Let $u_k$ be the vertex in $\bs{u}$ at minimal distance, let us say $n$, from $v$ (a trivial case arises if $v$ belongs to $\bs{u}$, in which case $n=0$). Since $T$ is connected, there is a geodesic path of vertices from $v$ to $u_k$: denote these vertices by $[v_0=v, v_1, v_2, \dots,v_n=u_k]$. The vertex $u_k\in \bs{u}$ at minimal distance from $v$ is unique: indeed, if there would be another vertex $u_m\in \bs{u}$ at the same distance from $v$,  then the vertices $v$, $u_k, u_m$ would form a loop with non-empty interior, and the tree has no such loops.
Notice that $u_{k+1}\neq v_{n-1}$, because otherwise $u_{k+1}$ would be at a distance from $v$ less than $n$. Then the path $\{v=v_0, v_1,\dots,v_n=u_k, u_{k+1}, u_{k+2},\dots\}$ is a geodesic ray that merges with $\bs{u}$ at $u_k$ and coincides with it afterwards, hence it belongs to the equivalence class $\omega$. 

Now let $\bs{w}=\{w_0,w_1,\dots\}$ be another geodesic ray in the same class of equivalence $\omega$, and let $w_j$ be its (unique) vertex at minimal distance from $v$. Consider the tree vertices $v, u_k, w_j$. They cannot form a non-trivial loop, hence either $u_k$ belongs to the (unique) geodesic path $[v,w_j]$ that connects $v$ and $w_j$, or $w_j\in[v,u_k]$ or $v\in[u_k,w_j]$. In the first two cases one of the two geodesic rays $\bs{u}$ and $\bs{w}$ contains the other, and the rays starting at $v_0$ and definitely merging with these two rays coincide. In the latter case, 
$v$ belongs to the largest  of the two geodesic rays $\bs{u}$ or $\bs{w}$, and again the merging into these rays is unique. The argument for edges is the same.
\end{proof}

\begin{remark}\label {rem:boundary_arcs}
$\Omega$ is equipped with a natural totally disconnected compact topology: a boundary point $\omega=\{v_0, v_1, \dots, v_n, \dots\}$ has a local base of neighborhoods that consists of the \emph{boundary arcs} $\Omega_n=\Omega(v_0,v_n)$ given by all the boundary points $\omega'$ that have representative geodesic paths that start at $v_0$ and contain $v_n$ (hence also all of the $n$ vertices $v_1,\dots,v_{n-1}$). It is easy to see that this topology is independent of the choice of reference vertex. Similarly, the vertices $w$ such that the geodesic path from $v_0$ to $w$ contains a given vertex $v$ are called the \emph{sector} $S(v,v_0)\subset V$. The collection of sets $S(v,v_0)\cup \Omega(v,v_0)\subset V\cup\Omega$ is a base for a compact topology of $V\cup\Omega$ that is independent of the choice of $v_0$. We can build a larger family of sets by letting $v_0$ vary, that is, by considering the boundary arcs $\Omega(u,v)$ whenever $u, v\in V, u\neq v$, consisting of all boundary points $\omega$ whose equivalence class of geodesic rays contains the path from $u$ to $v$. The family of such sets is the same as the family of sets $\Omega(u,v)$ with $\sim v$. Note that, if $u\sim v$, then $\Omega(u,v)$ coincides with the set of boundary points $\omega$ whose class of equivalence of geodesic rays made of oriented adjoining edges contains the oriented edge $e=[u,v]$: we denote this set by $\Omega(f)$, where $f$ is the flag $[u,v]$. 
This family of  sets generates on $\Omega$ the same topology as the one defined by all vertex-arcs $\Omega(v_0,v)$.

Similarly, for $u\sim v$ we can consider vertex-sectors $S(u,v)=\{w\in V: \dist(v,w)<\dist(u,w)\}$, and edge-sectors 
\[
S(f)=S([u,v])=\{e\in E: e \text{ is closer to $v$ than to $u$}\}.
\]
 This family generates a topology on $V\cup \Omega$ that makes it compact and totally disconnected, and it generates an analogous topology on $E\cup\Omega$.
\end{remark}

The action of $\Aut T$ preserves adjacency, hence it maps geodesics to geodesics. Hence $\Aut T$ acts on $\Omega$.
On a a homogeneous tree, $\Aut T$ acts transitively on $\Omega$, and so does the stability subgroup $K_v$ of any vertex $v$  or $K_e$ of an edge $e$. 

\begin{remark}[A primer on action of automorphisms on vertices, edges and flags]\label{rem:automorphisms}
To help the readers to develop the appropriate intuition, we begin here a primer on automorphisma of homogeneous trees. We have defined automorphisms as bijections that preserve adjacency, hence they map geodesics to geodesics. We shall show in the next Proposition how to construct all automorphisms $\lambda_0$ that fix a reference vertex $v_0$: we give an outline here. These automorphisms form a subgroup $\Aut_0 T$ whose elements must permute the neighbors of $v_0$. Call $v_1$ one such neighbor. Then $\lambda_0$ must permute also the neighbors of $v_1$, but since it fixes the predecessor $v_1^-=v_0$, then it must permute the forward neighbors $v_2\sim v_1$, $v_2\neq v_0$. By iterating this argument, we see that the general automorphism that fixes $v_0$ is obtained by permuting in all possible ways all forward neighbors of each vertex. Note that $\Aut_0 T$ acts transitively on each circle of vertices at a fixed distance from $v_0$.
\\
Then any automorphism $\lambda$ of $T$ that maps $v_0$ to any other vertex $v=\lambda(v_0)$ is obtained by taking all automorphisms $\lambda_0$ that fix $v_0$ and composing them with a fixed automorphism that maps $v_0$ to $v$. 
\end{remark}

\begin{example}[Some simply transitive subgroups of $\Aut T$\,: free groups and free products]\label{example:trees_as_Cayley_graphs_of_free_products}
If $T$ is homogeneous, subgroups of $\Aut T$ are connected to representation theory of free groups, free products and some $p$-adic semisimple groups like $PGL_2(\mathQ_p)$\cites{Cartier, Figa-Talamanca&Picardello-JFA, Figa-Talamanca&Picardello}. This is because a homogeneous tree is a Cayley graph of a free group or free product, that is, we can label its vertices so that each label is a word whose letters are the generators of the group or their inverses. Then the translation action of the group upon itself gives rise to an action on the Cayley graph, and so the group embeds into the automorphism group of the graph. However, the image of this embedding is a small subgroup of $\Aut T$ (but co-compact). Indeed, there is a huge flexibility in constructing automorphisms of $T$, as we outline here.

We continue our primer by providing more details on some well known facts on the action of free groups and free products on homogeneous trees (see \cites{Furstenberg,Figa-Talamanca&Picardello} and references therein).
\\
We start by considering a tree $T_q$ of odd homogeneity degree $q$, that is, with an even number $q+1=2r$ of neighbors. To avoid trivialities, assume $r>1$. Choose a set of generators $a_1,\dots,a_r$ of the free group $\mathF_r$. Then $T_q$ can be labeled in such a way as to become the Cayley graph of $\mathF_r$ with respect to this choice of generators. Such Cayley graph is defined to be a homogeneous space for $\mathF_r$ consisting of a graph whose vertices are in one to one correspondence with $\mathF_r$
and where the neighbors of a vertex $v\in\mathF_r$ correspond to the right translates of $v$ by $\{a_1,\dots,a_r,a_1^{-1},\dots, a_r^{-1}\}$.
\\
For simplicity, we restrict attention to the free group $\mathF_2$ with two generators $a$ and $b$. We start the construction of the Cayley graph with respect to the generators $a$ and $b$ by associating a vertex with the identity element 1 of the group. Then its neighbors consist of the generators and their inverses. Now, the neighbors of the vertex $a$ are $a^2,ab,ab^{-1}$ and $aa^{-1}=1$. Continuing in this way, we see that the neighbors of any vertex regarded as an element of $\mathF_2$, that is a reduced word $w$ in the alphabet $a,b,a^{-1},b^{-1}$, consists of the four words obtained by adding a letter of the alphabet at the end of $w$ (and performing the necessary reduction when the added letter is the inverse of the last letter of $w$). As in the free group  no nontrivial reduced word equals 1, the Cayley graphs that we have just introduced has no loops, that is, is a tree.
In this way we build a tree of homogeneity 4, drawn in Figure \ref{Fig:T_3_as_Cayley_graph_of_F_2}. 

\begin{figure}[h!]
\silenceable{\begin{tikzpicture}[grow cyclic,level/.style={
 sibling  angle=90  /(2-1/7)^(#1-1),
 level distance=15mm/(1+2/5)^(#1-1)}]
\begin{scope}
\tikzstyle{every node}=[circle,fill,inner sep=.7pt]
\path[rotate=-135]node(startA){}
 child foreach\x in{a,b,c,d}{node{}
   child foreach\y in{a,b,c} {node{}
     child foreach\z in{a,b,c} {node{}
       child foreach\w in{a,b,c} {node{}
       }
     }
   }
 };
\end{scope}
 \node[label=-135:$1     $]at(startA){};
 \node[label=- 45:$b     $]at(startA-1){};
 \node[label=  45:$a^{-1}$]at(startA-2){};
 \node[label=  45:$b^{-1}$]at(startA-3){};
 \node[label= 100:$a     $]at(startA-4){};
 \node[label=-180:$ab^{-1}$]at(startA-4-1){};
 \node[label=  90:$a^2     $]at(startA-4-2){};
 \node[label= 150:$ab     $]at(startA-4-3){};
 \node[label=- 90:$ba     $]at(startA-1-1){};
 \node[label=   0:$b^2     $]at(startA-1-2){};
 \node[label=- 93:$ba^{-1}     $]at(startA-1-3){};
 \node[label= 180:$ab^2     $]at(startA-4-3-3){};
 \node[label=- 90:$ba^2     $]at(startA-1-1-1){};
\end{tikzpicture}}
\caption{Beginning of the labeling of the homogeneous tree $T_2$ as Cayley graph of the free group $\mathF_2$ in two generators $a,b$. Note that the left multiplication by the element, say, $a$ maps the neighbors of the identity 1 into the neighbors of $a$, and is an isometry in the natural distance, hence an automorphism. $\mathF_2$ acts transitively. The stabilizers are trivial, hence the action is simply transitive. Moreover, the action is transitive on circles around 1, so it is also doubly transitive.}
\label{Fig:T_3_as_Cayley_graph_of_F_2}
\end{figure}
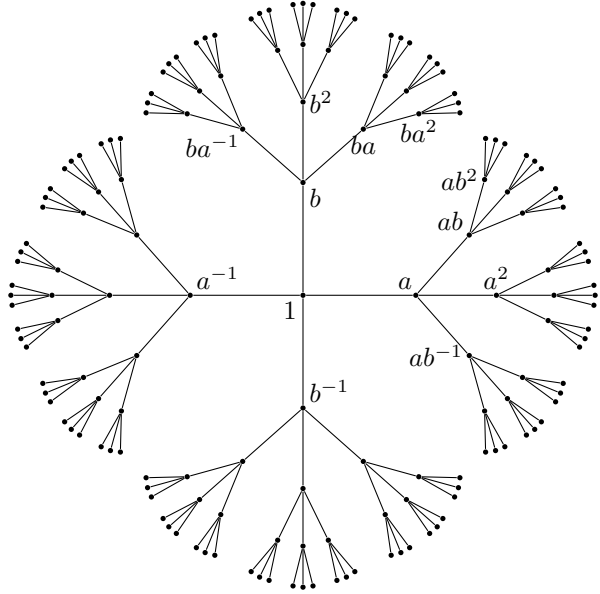
The free group acts on this Cayley graph by left multiplication, that is, left juxtaposition of words (and corresponding reductions when necessary). This action is an isometry in the natural distance between vertices, because it is given by left multiplication on words, while the neighbors of a word are given by right multiplication by the generators or their inverses.
Obviously, the natural distance of a word $w$ from 1 in the Cayley graph coincides with the length $|w|$ of $w$, and the distance between $w$ and $u$ is $|w^{-1}u|$ (that of course is the same as $|u^{-1}w|$).  \\
Again because every element of the free group can be represented in a unique way as a reduced word, the stabilizer in $\mathF_r$ of any vertex in $T_{2r}$ is trivial. On the other hand, the action of $\mathF_r$ on $T_{2r}$ is transitive, because so is left translation on words, and in particular, for all words $w,u$ at the same distance from 1, there is an element  $\lambda\in\mathF_r$ such that $\lambda w=u$. This means that $\mathF_r$ is simply transitive on the vertices of $T$. Instead, as already seen in Remark \ref{rem:automorphisms}, 
the full group $\Aut T$ of automorphisms is
doubly transitive, that is, for all pairs $(w_1, u_1)$ and $(w_2,u_2)$ with $|w_1^{-1}u_1|=|w_2^{-1}u_2|$, there exists 
$\lambda\in\mathF_r$ such that $\lambda w_1=u_1$ and $\lambda w_2=u_2$.
\\
By the universal property of free groups, a change of generators extends uniquely to an isomorphism of $\mathF_r$, but in general this map is not an isometry in distance defined by the original set of generators, hence it is not an automorphism of the tree built by the original generators.

Of course, we need an even number of neighbors in order to regard the tree as the Cayley graph of a free group. But if there is an odd number of neighbors, say $k$, then we can repeat the same construction in terms of words in the alphabet $a_1,\dots,a_k$ with the relations $a_1^2=\dots=a_k^2=1$. By this mean, the tree $T_k$ is the Cayley graph, in the given choice of generators, of the free product $\mathZ_2*\dots*\mathZ_2$ ($k$ times) (see Figure \ref{Fig:T_2_as_Cayley_graph_of_Z_2*Z_2*Z_2}). This is true for all $k$, even or odd (for $k=2$ the graph consists of the integers), and, more generally, for all $k\geqslant 1$ a discrete group acting simply transitively and doubly transitively on $T_{k}$ is the free product $\mathF_r * (*_{j=1}^m\mathZ_2)$, with $2r+m=k$.

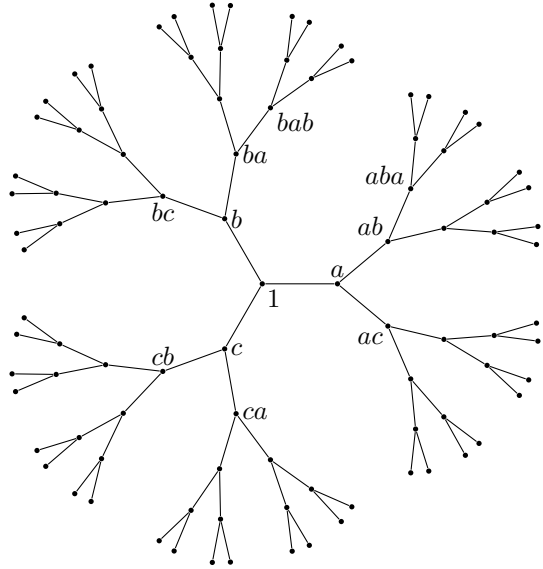
\begin{figure}[h!]
\silenceable{\begin{tikzpicture}[grow cyclic,level/.style={
 sibling  angle=120 /(1+1/2)^(#1-1),
 level distance=10mm/(1+1/7)^(#1-1)}]
\begin{scope}
\tikzstyle{every node}=[circle,fill,inner sep=.7pt]
\path[rotate=-120]node(startA){}
 child foreach\x in{,,}{node{}
   child foreach\x in{,} {node{}
     child foreach\x in{,} {node{}
       child foreach\x in{,} {node{}
         child foreach\x in{,} {node{}
         }
       }
     }
   }
 };
\end{scope}
 \node[label=- 60:$1     $]at(startA){};
 \node[label=   0:$b     $]at(startA-1){};
 \node[label=   0:$c$]at(startA-2){};
 \node[label=  90:$a$]at(startA-3){};
 \node[label=-120:$ac$]at(startA-3-1){};
 \node[label= 120:$ab     $]at(startA-3-2){};
 \node[label=   0:$ba     $]at(startA-1-1){};
 \node[label=- 90:$bc     $]at(startA-1-2){};
 \node[label=  90:$cb$]at(startA-2-1){};
 \node[label=   0:$ca$]at(startA-2-2){};
 \node[label= 150:$aba     $]at(startA-3-2-2){};
 \node[label=- 30:$bab     $]at(startA-1-1-1){};
\end{tikzpicture}}
\caption{Beginning of the labeling of the homogeneous tree $T_2$ as Cayley graph of the free product $\mathZ_2*\mathZ_2*\mathZ_2$ with generators $a,b,c$ of order 2, that is, $a^2=b^2=c^2=1$.}
\label{Fig:T_2_as_Cayley_graph_of_Z_2*Z_2*Z_2}
\end{figure}
\end{example}

\begin{remark}[$E(T)$ is not the Cayley graph of a free product]
Observe that this action of $\mathF_r$ or $\mathZ_2 * \dots * \mathZ_2 \subset \Aut T$ on vertices, given by left multiplication on words, yields of course an action also on the set $E(T)$ of edges, but this is not simply transitive on $E(T)$, and not even transitive. For instance, restrict attention to the homogeneous tree of degree 4 and the free group $\mathF_2$ with generators $a,b$ and inverses $a^{-1}, b^{-1}$. Consider the reference edge $e_0=[1,a]$. Write $e_1=[1,b]$. Then there is no element $x\in\mathF_r$ such that $x\cdot e_0=e_1$ under the action given by left multiplication. Indeed, if $x\cdot[1,a]=[1,b]$ then either $x\cdot 1 = 1$ or $x\cdot 1=b$. In the first case, $x=1$ because the action on vertices has trivial stabilizers, but then $x\cdot a=a\neq b$ and $x\cdot e_0\neq e_1$. In the second case, for the same reason, $x=b$, but then $x\cdot a=ba\neq b$, and again $x\cdot e_0\neq e_1$. Moreover, the stability subgroup of $e_0$ is trivial in this instance, but if we consider instead the tree of degree 3 and the subgroup $\mathZ_2 * \mathZ_2 *\mathZ_2\subset \Aut T$ with generators $a,b,c$ of order 2, then $a\cdot e_0=a\cdot[1,a]=[a,1]=e_0$, hence the stability subgroup of $e_0$ is non-trivial (this argument fails if we consider the action on oriented edges instead of geometric edges). In particular, $E(T)$ cannot be labeled by the elements of a free group or free product under the natural action of these groups on words.
\end{remark}

\section {Horospheres and the horospherical fiber bundle}
\subsection{An informal introduction to horospheres}
In this introductive Subsection we want to give a sketch of the notion of horosphere in a tree in a constructive, easily understandable way. For the sake of simplicity, here we shall only consider horospheres of vertices. although everything holds also for edges. All details will be rephresed in a more precise and accurate way in the following Subections.

Consider a vertex $v_0$ and its unique geodesic ray $[v,\omega)=\{v_0, v_1, v_2, \dots\}$ belonging to the equivalence class of a boundary point $\omega$, defined in Proposition \ref{prop:unique_geodesic_ray_to_omega_starting_at_v}. We construct the \emph{vertex-horosphere}, here simply called horosphere, that contains $v_0$ and is tangent at $\omega$, that we denote by $\bs{h}(v,\omega)$. First we include $v_0$ in $\bs{h}(v,\omega)$. Then we move one step towards $\omega$, that is we move to $v_1$, but instead of including $v_1$ in the horosphere we move one step sideways (that is, to neighbors different from tha backward $v_0$ and the forward neighbor $v_2$) and include all these sideway neighbors in the horosphere. Then we move a further step forward to $v_2$ (at distance 2 from $v_0$ along $\omega$) and include in the horosphere the 2-step sideway neighbors of $v_2$ (to reach them we move to one of the neighbors of $v_2$ different from $v_1$ or $v_3$ and continue one step further away from $\omega$). Then we proceed iteratively in the same way. The set obtained from this recursion is $\bs{h}(v_0,\omega)$, and is called the horosphere through $v_0$ tangent at $\omega$. Observe that the distance from any two vertices of the horosphere is even. Notice that the horospheres tangent at $\omega_1$ form a family disjoint from the family of horospheres tangent at $\omega_2$ for any $\omega_2\neq\omega_1$, and that the only accumulation point of $\bs{h}(v_0,\omega)$ in the topology of $T\cup\Omega$ is $\omega$
  
Observe that, if $v$ and $w$ belong to the same horosphere, then the constructive procedure outlined above is symmetric: $\bs{h}(v,\omega)=\bs{h}(w,\omega)$.

\begin{proposition} Every vertex-horosphere is the union of a nested family of arcs of circles (i.e., subsets of circles) in $V$ (whose radii grow as the length of the center while this center moves along the geodesic ray to the tangency point at the boundary).\label{prop:horosphere=unions_of_circles}
\end{proposition}
\begin{proof}
Choose a vertex $v_0$ and the horosphere $\bs{h}$ tangent at $\omega\in\Omega$ that contains $v_0$. 
Consider the $j-$th vertex $v_j$ in the ray $[v_0,\omega)$  from $v_0$ to $\omega$ (that is, belonging to the class of equivalence of $\omega$, in the sense of Proposition \ref{prop:unique_geodesic_ray_to_omega_starting_at_v}). Consider the set $C(v_j)$ of the vertices $w\in\bs{h}$ such that the geodesic ray from  $w$ to $\omega$ merges with $[v_0,\omega)$ at the vertex $v_j$ (this set could be understood as the set of descendants pf generation $j$ with respect to the mythical ancestor $\omega$). Then, by the way we defined the horospheres, all these vertices belong to the same
 horosphere $\bs{h}$.
Therefore their distance from $v_j$ is $j$. This means that the set $C(v_j)$ belongs to the circle of radius $j$ around $v_j$. Now consider the vertex $v_{j+1}$ and its set $C(v_{j+1})$ of descendants of generation $j+1$ with respect to $\omega$. It is clear the $C(v_{j+1})\supset C(v_{j})$. Hence every vertex in $\bs{h}$ belongs to these circles of radius $j$ for $j$ large enough.  Therefore $\bs{h}=\cup_j C(v_j)$. For the last part of the statement, it is enough to observe that the center of the circle $C(v_j)$ is the vertex $v_j$ lying in the geodesic ray to $\omega$ and the radius is $j$.
\end{proof}

Now let us repeat the construction above, starting not at $v_0$ but at $v_1$. Then we obtain another horosphere, $\bs{h}(v_1,\omega)$, tangent at $\omega$, that passes through $v_1$. It is clear that every vertex of $\bs{h}(v_1,\omega)$ is one step forward towards $\omega$, that is, is obtained by
moving one step towards $\omega$ from vertices in $\bs{h}(v_0,\omega)$. This new horosphere is said to be \emph{parallel} to the previous one. We can also repeat the construction one step backward: start with a neighbor $u$ of $v_0$ one step backwards from the boundary point $\omega$ (that is, different from $v_0$). Observe that, if $w$ is another such backward neighbor, then $w\in \bs{h}(u,\omega)$, and vice-versa. Therefore there is only one parallel horosphere one step backward. Thus every boundary point $\omega$ generates a one-parameter family of parallel horospheres.  Therefore, not only each boundary point determines a family of parallel horospheres, but also, the family of parallel horospheres determined by the boundary point $\omega$ depends only on $\omega$ and not on an individual representative ray in its class. In other words, families of parallel horospheres are in one to one correspondence with boundary points.

Since we know that $\bs{h}(w,\omega)=\bs{h}(v,\omega)$ for all $w\in\bs{h}(v,\omega)$, we see that two parallel horospheres either coincide or are disjoint.

For every $\omega\in\Omega$, we now have a group of actions on the parallel family of horospheres at $\omega$, given by parallel translations. Let us look at the one-step forward translation: it is implemented by a suitable automorphism of $T$. Clearly this automporphism is not given by a map that moves each vertex of a horosphere one step forward towards $\omega$, because this map is not injective. Instead, it is obtained as follows. Choose a vertex $v_0$ in the horosphere $\bs{h}(v_0,\omega)$, and consider the set $C(v_j)$ of all its vertices $w$ that are starting vertices of geodesic rays $[w,\omega)$ towards $\omega$ that merge with the ray $[v_0,\omega)=\{v_0, v_1, v_2, \dots\}$ at $v_j$. Then the one-step forward shift of $\bs{h}(v_0,\omega)$ to the next parallel horosphere $\bs{h}(v_1,\omega)$ tangent at $\omega$ maps bijectively  $C(v_j)$ to $C(v_{j+1})$. Similarly, 
$\bs{h}(v_k,\omega)$ is mapped into $\bs{h}(v_{k+1},\omega)$. So, the forward shift along $\omega$ gives rise to an action on the family of parallel horospheres at $\omega$. The group of such actions at $\omega$ is a subgroup of $\Aut T$ isomorphic to $\mathZ$, and it fixes the boundary point $\omega$ but not the other boundary points (for its action on the boundary see Subsection \ref{SubS:isotropy_group_of_horospheres}).
\\
However, we shall consider the subgroup of the group of all actions on the set of horospheres that is generated
 by the one-step forward shift at each $\omega$ simultaneously. This group is again isomorphic to $\mathZ$, and clerly fixes every boundary point, but is not contained in $\Aut T$. We shall denote this group by $A$ and call it the Cartan subgroup of the group of actions on horospheres. 

In particular, the set of horospheres forms a fiber bundle $\HorV$, with base $\Omega$ and fibers given by the integers. The one-step forward shift yields a one step shift on the sections of this bundle (Definition \ref{def:sections}), and $\Omega\approx\HorV/A$. Of course, $\Aut T$ acts on $V$ and preserves distances, so it acts on $\HorV$ and commutes with the action of $A$ (and in particular it acts on $\Omega$).

It is interesting to observe that the vertices that belong to  $\bs{h}(v_0,\omega)$ are the equivalent classes of the following equivalence relation induced by $\omega$ on $V$: two vertices $v,w$ are $\omega-$equivalent if the geodesic rays $[v,\omega)$ and $[w,\omega)$ starting at $v,w$ respectively merge after the same number of steps. The merging vertex is called their join with respect to $\omega$. To clarify this idea, we explain joins in more details in the next Subsection.

\subsection{Joins}
\label{SubS:Joins}
Denote by $[v,\omega)$  the infinite geodesic ray from a vertex $v$ to the boundary point $\omega$, and by $[v_0,v_1]$ the finite path from the vertex $v_0$ to $v_1$.\\
Given three different vertices $v_0, v_1, v_2$ we define their \emph{join} $j(v_0, v_1, v_2)$ as the last common vertex in the geodesic paths from $v_0$ to $v_1$ and from $v_0$ to $v_2$, i.e., the unique common vertex in the paths $[v_0, v_1], [v_1, v_2]$ and $[v_0,v_2]$. In particular,
\[
\dist(v_1,v_2)=\dist(v_1,v_0)+\dist(v_2,v_0)-2\dist(j(v_0,v_1,v_2),v_0).
\]
For future reference, we rewrite the last identity with the terminology where $v_0$ is regarded as the reference vertex, the distance of a vertex $w$ from $v_0$ is written $|w|$ and instead of $j(v_0,v_1,v_2)$ we write $v_1\wedge v_2$::
\begin{equation}\label{eq:distance,join_and_horospherical_number}
\dist(v_1,v_2)=|v_1|+|v_2|-2|v_1\wedge v_2|.
\end{equation}
Note that $j(v_k,v_n,v_m)=j(v_0,v_1,v_2)$ for every permutation $(k,m,n)$ of $(1,2,3)$. By letting, say, $v_2$ move to boundary point $\omega$ we define the join $j_\omega(v_0,v_1)$ of $v_0$ and $v_1$ with respect to $\omega$. We recall that the set of all vertices $w$  such that $v$ belongs to $[v_0,w]$ was denoted by $S(v,v_0)$ in Subsection \ref{SubS:Boundary}  and called the {sector} subtended by $v$ with respect to $v_0$. 
Observe that $\dist(v_0,w)-\dist(v,w)$ is constant when $w$ varies in the sector $S(v,v_0)$ subtended by $v$, hence there exists $\lim_{w\to\omega} (\dist (v_0,w)-\dist(v,w))$ when  $w$ is a variable vertex that converges to $\omega$ in the compact topology of $T\cup\Omega$: for these $w$, the limit is $\dist (v_0,j(v_0,v,w))-\dist(v,j(v_0,v,w))$. 
This limit can also be written as
$2N(v,v_0,\omega) - \dist(v_0,v)$, where $N(v,v_0,\omega)=\dist(v_0,j_\omega(v_0,v))$.

\begin{definition}\label{Def:sectors} We are now familiar with the vertex-sector $S(v,v_0)$ subtended by a vertex $v$ with respect to $v_0$. In the sequel we shall also need sectors in $E$: such a sector is defined, for every two edges $e,e'$, as the set $S(e,e_0)$ of all edges $e'$ such that $e$ belongs to the geodesic chain $[e_0,e']$.
\end{definition}

\subsection{Horospheres}
\label{SubS:Horospheres}
This Subsection gives a more detailed presentation of the geometry of horospheres on trees. Its contents hold for both homogeneous and semi-homogeneous trees.

\begin{definition}[Horospherical index]
\label{def:horospherical_index}
Given three vertices $v_0, v,w$, we define the horospherical index of $v$ and $w$ with respect to $v_0$ as
\[
h(v,v_0,w)= \dist(v_0,j(v_0,v,w))-\dist(v,j(v_0,v,w)).
\]
Note that
\begin{equation}\label{eq:vertex_horospherical_index}
\begin{split}
h(v,v_0,\omega) &=\dist(v_0,j_\omega(v,v_0))-\dist(v,j_\omega(v,v_0)) =  2N(v,v_0,\omega)-\dist(v,v_0)   \\[.1cm]
& = \lim_{w\to\omega} (\dist (v_0,w)-\dist(v,w))\,,
\end{split}
\end{equation}
where $w$ is a variable vertex that converges to $\omega$ in the compact topology of $V\cup\Omega$, and
$N(v,v_0,\omega)$ is the number of vertices shared by the infinite geodesic ray $[v_0,\omega)$ and by the finite geodesic ray  $[v_0,v]$.
\end{definition}

\begin{remark}\label{rem:horospherical number_of_a_neighbor}
Observe that the horospherical index increases by $1$ when $v$ moves one step towards $\omega$ along $\lambda_\omega$, but decreases by $1$ when $v$ moves one step sideways. Also notice that $h(v,v,\omega)=0$ for every $v$ and $\omega$.
\end{remark}
From now on, when we do not need to change the reference vertex, we simply write the horospherical index $h(v,\omega)$ instead of $h(v,v_0,\omega)$.

We can  extend the definition of horospherical index to express it for  vertices in terms of reference edges, as
\begin{equation}\label{eq:mixed_vertex_edge_horospherical_index}
h(v,e,\omega)= \frac{h(v,v_+,\omega)+ h(v,v_-,\omega)}2\;,
\end{equation}
where $v_+$ and $v_-$ are the end-vertices of the edge $e$. When expressed in this way, the horospherical index has the following geometrical interpretation:
$ h(v,e,\omega)= \lim_{w\to\omega} (\dist(e,w)-\dist(v,w))$. Moreover, 
\begin{equation}\label{eq:mixed_vertex_edge_horospherical_index_as_related_to_vertex_horospherical_index}
h(v,e,\omega)=\frac12 + h(v,v_+,\omega)
\end{equation} 
where $v_+$ is the end-vertex of $e$ at the side of $\omega$. The mixed edge-vertex horospherical index $h(e,v,\omega)$ is defined similarly:
\begin{equation}\label{eq:mixed_edge-vertex_horospherical_index_as_related_to_edge_horospherical_index}
h(e,v,\omega)=-h(v,e,\omega)=-\frac12 - h(v,v_+,\omega).
\end{equation} 

Similarly, if we fix a reference edge $e_0$, the \emph{edge-horospherical index}  
is defined as follows. For every fixed $\omega\in\Omega$, let us associate to each edge $e=[v_-,v_+]$ 
 its vertex $v_+=v(e,\omega)$ closer to $\omega$, that is, on the side of $\omega$, and associate $v(e_0,\omega)$ to $e_0$ in the same way. (Inversely, we shall  associate to each vertex $v$ its edge 
 $e_+ \equiv e(v,\omega)  \sim v$ closer to $\omega$: see later  Proposition \ref{prop:horosphere_correspondence} for a deeper study of this bijection).
 Then set
\begin{equation*}
   h(e,e_0,\omega)
=  h(v(e,\omega),\,v(e_0,\omega),\, \omega).
\end{equation*}
 
 \begin{remark}\label{rem:subtle_difference_between_horospherical_indices_for_vertices_and_for_edges}
The definition of  horospherical index for edges cannot be elegantly written in terms of joins exactly as it was done in the case of vertices. Indeed, the join of three edges is a vertex, not an edge: there is no natural definition of joining edge of three edges. But once $\omega$ is fixed, then we could introduce a merging edge of the edge-chains to $\omega$ as follows. Let $[e_0,\omega)$ be the infinite edge-chain 
from $e_0$ to $\omega$ and $[e,\omega)$ be the infinite edge-chain
from $e$ to $\omega$. 
Then we can consider the first common edge  of these two chains: let us call it
 $\text{merge}_\omega(e,e_0)$. For coherence of terminology, note that  $\text{merge}_\omega(e,e_0)$ is the edge whose vertex farther from $\omega$ is the join $j_\omega(v_0,v_1)$ where $v_0$, respectively $v_1$, are the vertices belonging to $e_0$, respectively $e_1$,  farther from $\omega$ (see Subsection \ref{SubS:Joins}).
Then $h(e,e_0,\omega) 
$ is the number of \emph{edges} from $e_0$ to $\text{merge}_\omega(e,e_0)$  minus the number of  \emph{edges} from $\text{merge}_\omega(e,e_0)$ to $e$, but 
does not always coincide with twice the distance $N(e,e_0,\omega)$ from $e_0$ to $\text{merge}_\omega(e,e_0)$ minus the distance from $e_0$ to $e$: indeed, it may differ by 1. 
For instance, consider an edge $e_0$. For a given $\omega_1$, let $e_1$ be the adjoining edge of $e_0$ along the chain from $e_0$ to $\omega_1$ and $e_2$ the following edge in the chain $\omega_1$, and let $e$ be another edge adjoining $e_1$ but different from $e_0$ and $e_2$. Then $h(e,e_0,\omega_1)=0$. But if $\omega_2$ is a boundary point consisting of a chain that starts at $e_0$ and continues with an adjoining edge different from $e$ and $e_1$, then $h(e,e_0,\omega)=1$, while, if $\omega_3$ is a boundary point consisting of a chain that starts at $e$ and continues with an adjoining edge different from $e_0$ and $e_1$, then $h(e,e_0,\omega)=-1$. We stress the following consequence: by \eqref{eq:vertex_horospherical_index}, for every $\omega$ the number $h(v,v_0,\omega)$ has the same parity of $\dist(v,v_0)$, whereas $h(e,e_0,\omega)$ does not have the same parity of $\dist(e,e_0)$ for every $\omega$.
\end{remark}

\begin{definition}[Horospheres]\label{def:horospheres}
The \emph{vertex-horospheres} $\boldh^V$ (respectively, \emph{edge-horospheres} $\boldh^E$) are the set of vertices $v$ (respectively, edges $e$) with the same vertex-horospherical index. \\
When we wish to make this index explicit, we use the following notation:
the vertex-horosphere $\boldh_{n}(\omega,v_0)=\boldh_{n}^V(\omega,v_0)$ (respectively, edge-horosphere $\boldh_{n}(\omega,e_0)=\boldh_{n}^E(\omega,e_0)$) tangent at infinity at the boundary point $\omega$ is the set of vertices $v$ (respectively, edges $e$) with the same vertex-horospherical index $n=h(v,v_0,\omega)$ (respectively, edge-horospherical index $n=h(e,e_0,\omega)$).
\end{definition}
It is immediate to check that horospheres do not depend on the choice of the reference vertex (respectively, reference edge), but of course horospherical indices do.

\begin{remark}[Review of horospheres and shrinking to the boundary]
\label{rem:horospherical_indices_and_side_steps}
By its definition, the vertex-horospherical index is related to the number of edges in the path from $v_0$ to $v$ in the following way: with the orientation induced by the boundary point $\omega$, $h(v,v_0,\omega)$ is the number of positively oriented edges minus the number of negatively oriented edges. 

The same relation holds for the edge-horospherical index: $h(e,e_0,\omega)$ is the number of positively oriented minus the number of negatively oriented edges in the path from $e_0$ to $e$, with the orientation induced by $\omega$.
Therefore, when the horospherical indices grow to infinity, the corresponding horospheres are contained in smaller and smaller sectors around the tangency point on the boundary. We have already observed that if the horospherical index $n$ (with respect to $\omega$) tends to $+\infty$ then all vertices of the horosphere of index $n$ tend to the boundary point $\omega$ in the topology of $V\cup\Omega$.
Notice  that the choice of a reference edge $e_0$ splits the tree into two connected components: we shall call \emph{positive} with respect to $\omega\in\Omega$ the connected component whose boundary contains $\omega$. Similarly, the choice of a reference vertex $v_0$ splits the tree into $q$ connected components: we  call \emph{positive} with respect to $\omega\in\Omega$ the connected component whose boundary contains $\omega$ (and the union of the others is the \emph{negative} part; we adopt the same terminology for the edge-horospheres).

The horospheres tangent at $\omega$ contained in the positive component are those with positive horospherical index (the \emph{small} horospheres, that shrink to $\omega$ when the horospherical index tends to $+\infty$). Indeed, if the sign of $h(v,v_0,\omega)$ and $h(e,e_0,\omega)$ is positive then the vertex $v$, respectively the edge $e$, is in the positive connected component of $T\setminus \{v_0\}$ (respectively,  $T\setminus \{e_0\}$)
induced by $\omega$ (and so all of the horosphere lies in this component).
\\
Finally, observe
that $|h(v,v_0,\omega)|=\min\{\dist(v,v_0): v\in \boldh_{h(v,v_0,\omega)}(\omega,v_0)\}$; similarly,
$|h(e,e_0,\omega)|=\min\{\dist(e,e_0): e\in \boldh_{h(e,e_0,\omega)}(\omega,e_0)\}$.
\end{remark}

\begin{definition}[Large and small horospheres]\label{def:large_and_small_horospheres}
Observe that every vertex-Horosphere $\boldh^V$ splits the tree into connected components, one and only one of which has a unique boundary point (the boundary point of the horosphere). This connected component will be called the interior of the horosphere, $\interior(\boldh^V)$.
We shall say that $\boldh^V$ is \emph{small horosphere} (with respect to a reference vertex $v_0$) if $v_0\in \interior(\boldh^V) \cup \boldh^V$, and a \emph{large horosphere} otherwise. The same definition holds for edge-horospheres and flag-horospheres, based on their respective reference elements.
\end{definition}

The results on horospheres and horospherical spaces and bundles that we present from now on in this and the next Subsections hold in a similar way for vertices and for edges. For simplicity, we limit all statements and proofs to the case of vertices.

\begin{remark}[Horospherical indices as cocycles]\label{rem:cycle_relation}
The vertex-horospherical index satisfies the cycle relation
\[
h(v_0,v_1,\omega)+h(v_1,v_2,\omega)=h(v_0,v_2,\omega)
\]

Indeed, let $j_\omega(v_0,v_1)$ be the join introduced in Subsection \ref{SubS:Joins}, that coincides with $v_0$ or with $v_1$ if and only if one of the two vertices is in the geodesic ray from the other to $\omega$. Then, for every vertex $v$,
\begin{equation*}
 h(v_0, v_2,\omega)
= h(v_0, v_1,\omega)
+\dist(v_1,j_\omega)
                    -\dist(v_2,j_\omega)
= h(v_0, v_1,\omega)
- h(v_2, v_1,\omega).
\end{equation*}
Another way to state the cocycle identity is the following. If $\boldh(\omega,n;v_0)$ is the horosphere in the horosphere tangent at $\omega$ with horospherical index $n$ with respect to a reference vertex $v_0$, and we change the reference vertex to $v_1$, then 
\begin{equation}\label{eq:change_of_horospherical_index_under_change_of_reference_vertex}
\boldh(\omega,m;v_1) = \boldh(\omega,n;v_0) \quad \text{if and only if } \; m-n=h(v_0,v_1,\omega).
\end{equation}
\end{remark}
In particular, since we know that $h(v_0,v_0,\omega)=0$ one has $h(v_0, v_1,\omega)=-h(v_1,v_0,\omega)$. Note also that the identity
$h(v_0,v_0,\omega)=0$ follows independently by the previous Remark, by choosing $v_0=v_1=v_2$.

It is appropriate to regard a fixed boundary point $\omega$ as an infinitely remote ancestor in a family tree (a tree in which each vertex $v$ is a family member, and the edges that originate at $v$ on the side opposite to $\omega$ are its siblings: see Figure \ref{Fig:family_tree}). 

\begin{figure}[h!]
\silenceable{\begin{tikzpicture}[level/.style={grow via three points={%
one child at (0,-1)
and two children at (-{5*3^(-#1)},-1) and ({5*3^(-#1)},-1)}}]
\begin{scope}
\tikzstyle{every node}=[circle,fill,inner sep=.7pt]
\path node(start){}
 child foreach\x in{,,}{node{}
   child foreach\x in{,,}{node{}
     child foreach\x in{,,}{node{}
       child foreach\x in{,,}{node{}
       }
     }
   }
 };
\end{scope}
\draw[dashed](0,0)--(0,1.5)node[anchor=east]{$\omega$};
\end{tikzpicture}}
\caption{Family tree: the horospheres are aligned horizontally \label{Fig:family_tree}}
\end{figure}
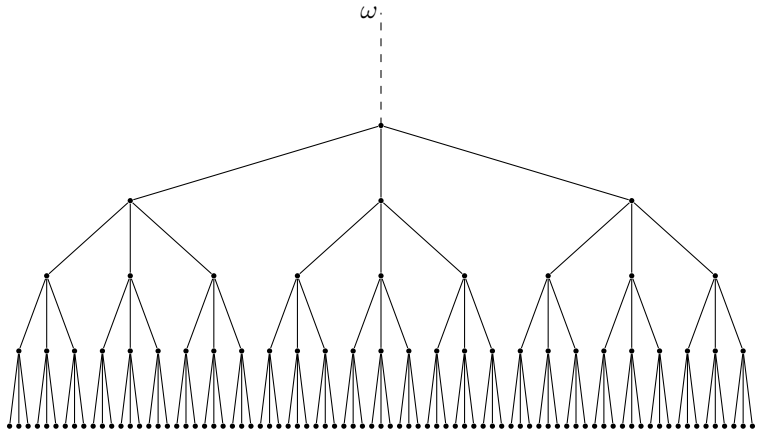
\begin{figure}[h!]
\silenceable{\begin{tikzpicture}[grow cyclic,level/.style={
 sibling  angle=90  /(2-1/7)^(#1-1),
 level distance=15mm/(1+2/5)^(#1-1)}]
\begin{scope}
\tikzstyle{every node}=[circle,fill,inner sep=.7pt]
\path[rotate=-135]node(startA){}
 child foreach\x in{,,,}{node{}
   child foreach\x in{,,} {node{}
     child foreach\x in{,,} {node{}
       child foreach\x in{,,} {node{}
       }
     }
   }
 };
\end{scope}
\foreach\x in{1,  3}{
\foreach\y in{1,2,3}{
 \node[draw,circle,     inner sep=1.5pt]at(startA){};
 \node[draw,circle,fill,inner sep=1.5pt]at(startA-\y){};
 \node[draw,circle,     inner sep=1.5pt]at(startA-4-\x){};
 \node[draw,circle,fill,inner sep=1.5pt]at(startA-4-\x-\y){};
 \node[draw,circle,     inner sep=1.5pt]at(startA-4-2-\x-\y){};
};
};
\node[label=- 60:$1     $]at(startA){};
\draw[dashed](4,0)--(6,0)node[label=- 60:$\omega$]{};
\end{tikzpicture}}
\caption{The vertex-horospheres tangent at $\omega$ of index 0 (hollow dots) and $-1$ (solid dots) near the origin in the homogeneous tree $T_3$
}\label{Fig:horospheres_on_T3}
\end{figure}

\begin{figure}[h!]

\silenceable{\begin{tikzpicture}[grow cyclic,level/.style={
 sibling  angle=120 /(1+1/2)^(#1-1),
 level distance=10mm/(1+1/7)^(#1-1)}]
\begin{scope}
\tikzstyle{every node}=[circle,fill,inner sep=.7pt]
\path[rotate=-120]node(startA){}
 child foreach\x in{,,}{node{}
   child foreach\x in{,} {node{}
     child foreach\x in{,} {node{}
       child foreach\x in{,} {node{}
         child foreach\x in{,} {node{}
         }
       }
     }
   }
 };
\end{scope}
\foreach\x in{1,2}{
\foreach\y in{1,2} {
 \node[draw,circle,     inner sep=1.5pt]at(startA){};
 \node[draw,circle,fill,inner sep=1.5pt]at(startA-\y){};
 \node[draw,circle,     inner sep=1.5pt]at(startA-3-1){};
 \node[draw,circle,fill,inner sep=1.5pt]at(startA-3-1-\y){};
 \node[draw,circle,     inner sep=1.5pt]at(startA-3-2-2-\x){};
 \node[draw,circle,fill,inner sep=1.5pt]at(startA-3-2-2-\x-\y){};
};
};
\node[label=- 60:$1     $]at(startA){};;
\draw[dashed](3.8,.8)--(6,1.2)node[label=- 60:$\omega$]{};
\end{tikzpicture}}
\caption{The vertex-horospheres tangent at $\omega$ of index 0 (hollow dots) and $-1$ (solid dots) near the origin in the homogeneous tree $T_2$
\label{Fig:horospheres_on_T2}}
\end{figure}

Let us give a geometric picture of the vertex-horospheres tangent at $\omega$. Choose a vertex $v$ in a horosphere tangent at $\omega$. The other vertices with the same horospherical index  (Definition \ref{def:horospheres}) are  reached by moving up some steps towards $\omega$ and then moving down the same number of steps. Moving up and down one step we reach the siblings of $v$. Moving a given number of steps up and then down we reach all cousins of the same generation. Therefore the horosphere consists of all cousins of a given generation (displayed horizontally in this arrangement of the family tree). Every tree can be realized in this way, by drawing at each vertex $v$ the edge issuing from $v$ towards $\omega$ as an upward edge, and all others as downward. This holds for all connected trees, not only homogeneous or semi-homogeneous.

\subsection{The stability subgroups of a horosphere and of its boundary point}\label{SubS:isotropy_group_of_horospheres}
A careful treatment if the subjects of this Subsection is in \cite{Figa-Talamanca&Nebbia} and references therein. Here we  give only an outline.

In Figure \ref{Fig:family_tree} a (part of a) homogeneous tree is shown as a genealogical tree of descendents of a mythical ancestor, denoted by $\omega$. Then $\omega$ is a boundary point, and all the other boundary points are represented by the geodesic rays that go down to the bottom. The set of vertices
at any given level form a horosphere tangent at $\omega$: the level is the horospherical number. Of course, the difference of any two levels is reference-free, but to assign to each level a numerical value we must choose where to locate level zero: that is, to assign a numerical value to the horospherical index of any of the parallal horospheres at $\omega$ we need to choose the horosphere of index zero, or equivalently a reference vertex $v_0$ belonging to it. 
\\
So, let us choose as representative geodesic  going to $\omega$ the central vertical line of the Figure, and denote its vertices by $\{\dots, v_2, v_1, v_0, v_1, v_2, \dots\}$. The rays starting at any of the $v_j$ and going up vertically are representative rays in the class of equivalence of $\omega$. Any other geodesic ray in this class starts at some vertex $w$  at one of the two sides and merges with the vertical line when the index is large enough: let us call this merging index $n=n(w)$ in the choice of index given by the chosen labeling $\{\dots, v_2, v_1, v_0, v_1, v_2, \dots\}$. In partucular, $w$ belongs to the horosphere $\bs{h}=\bs{h}(n(w),\omega)$.
\\
It is now clear from the Figure that the vertex $w$ belongs to the circle of vertices at distance $n(w)$ from $v_{n(w)}$: actually, these vertices, in the figure, are positioned as a fan $C(v_{n})$ of height $n$ opening down from $v_{n}$. Therefore the subgroup $G_{v_n}(\omega)$ of the isometry group at $v_{n}$ that fixes every vertex up, that is the  ray $[ v_{n}. v_{n+1}. v_{n+2},\dots)$ is a subgroup of the stability group of the horosphere $\bs{h}$ that contains $w$: let us refer to this subgroup as the group of isotropy of the fan $C(v_n)$. Since this ray is a representative of the class of $\omega$, we now see that the union of all the fan isotropy subgroups, $N_\omega:=\cup_j G_{v_j}(\omega)$, fixes $\omega$ and preserves each of the parallel horospheres tangent at $\omega$, acting transitively on each. It is easy to see  by Proposition \ref{prop:horosphere=unions_of_circles} that this group $N_\omega$ is the full stability group of each horosphere tangent at $\omega$.

However, $N_\omega$ is not the full group of isotropy $(\Aut T)_\omega$ of $\omega$ in $\Aut T$, because $(\Aut T)_\omega$ preserves the family of parallel horospheres tangent at $\omega$ but not their level, that is their horospherical number $h(n,\omega)$. Indeed, let us look at the action of automorphisms $\lambda$ that fix $\omega$. Let us look at a vertex $w$ and the unique geodesic ray $[w, w_1, w_2,\dots)=[w,\omega)$  that starts at $w$ (Proposition \ref{prop:unique_geodesic_ray_to_omega_starting_at_v}) and goes to $\omega$.
\\
Since $\lambda$ preserves the class of equivalence of $\omega$, it must map $[w,\omega)$ to a geodesic ray $[\lambda(w),\omega)$ in the class of $\omega$. Therefore $[\lambda(w),\omega)$ must merge, at some index $k$, with the geodesic $\{\dots, v_2, v_1, v_0, v_1, v_2, \dots\}$ chosen before, and agree with it for all larger indices. Since $\lambda$ preserves the distance, it must map bijectively all vertices in the fan $C(v_{n(w)})$ to vertices in the fan 
$C(v_{n(\lambda(w))})=C(v_k)$.
But then $lambda$ maps the horosphere $\bs{h}(n(w),\omega)$ to the horosphere $\bs{h}(k,\omega)$, and therefore it induces a shift in the horospherical indices of the family of horospheres tangent at $\omega$. It is now easy to see that $\lambda$ is obtained by an automorphism in $N_\omega$ followed by a shift along the geodesic  $\{\dots, v_2, v_1, v_0, v_1, v_2, \dots\}$. Choose a generator for the group odf such shifts, call it $\lambda_+^\omega$, and denote by $A_\omega$ the subgroup of $\Aut T$, isomorphic to $\mathZ$, that it generates: then the isotropy group atv $\omega$ in $\Aut T$ is $A_\omega N_\omega$. Moreover, it is easily seen from this construction that $A_\omega$ normalizes $N_\omega$: therefore $(\Aut T)_\omega$ is the semidirect product $A_\omega N_\omega$. See \cite{Figa-Talamanca&Nebbia} for more details of this proof.

We have already made the obvious observation that $N_\omega$ fixes the boundary point $\Omega$ (the point at infinity that is in the upward direction in Figure \ref{Fig:family_tree}.
Let us now see how it acts on boundary points in $\Omega\setminus\{\omega\}$. These boundary points $\omega'\neq \omega$ correspond to the downward points at infinity in Figure \ref{Fig:family_tree}. Consider $\omega', \omega'' \neq \omega$ ant the join $v=j(\omega',\omega'',\omega)$. The vertex $v$ is the merge point of the geodesics $(\omega',\omega)$ and $(\omega'',\omega)$. We have seen above that $N_\omega$ contains automorphisms that fix $v$, move any child vertex of $v$ to any of its brothers, moves any child of each of these children to any of it brothers, and so on. In other words, $N_\omega$ contains automorphisms that map $\omega'$ to $\omega''$. We state this fact as a separate statement.

\begin{corollary}\label{cor:transitivity_of_the_unipotent_group_N_omega}
The stability subgroup $N_\omega \subset (\Aut T)_\omega$ that fixes one, hence all horospheres  in the fiber $\omega\in\Omega
\approx \HorV/A$, is transitive on $\Omega\setminus\{\omega\}$.
\end{corollary}

\subsection{Principal fiber bundle isomorphisms, sections and special sections}\label{SubS:Sections_and_special_sections}
We consider principal trivial fiber bundles $\calH$. The base is a topological space $\Omega$ and the fiber a locally compact group $Z$, that coincides with the structure group. A global chart is a bijection between $\calH$ and $\Omega\times Z$. Each chart lifts to  $\calH$ the product topology of $\Omega\times Z$: we limit attention to a subcollection of compatible global charts that induce mutually equivalent topologies, introduced in the following Definition. 
Here, for simplicity, we limit attention to the topological space $\Omega$ that is the boundary of the homogeneous tree $T$, and the structure group given by the integers, $\mathZ$. 
    \begin{definition}[Special sections]\label{def:special_sections} Given a reference vertex $v_0$, we call \emph{special section} $\Sigma_{v_0}$ the map that to each fiber $\omega$ associates the unique horosphere tangent at $\omega$ that contains $v_0$. An analogous definition can be given for special sections of edge-horospheres.
    \end{definition}
     So, every special section corresponds to a map on $\Omega$ and we now use it to define a global chart. Indeed, to each fiber $\omega$ there now corresponds a horosphere $\bs h(\omega)$, and we shall choose its horospherical index $h$ as 0: that is, the horospheres in the special sections fix a privileged horospherical index for each fiber. Every other horosphere in the same fiber now acquires in a unique way a horospherical index (given by the opposite of the integer needed to shift it to the privileged horoshere).
So we have a global chart on $\Omega$, for each $v_0$. It is easy to verify that all such charts, when $v_0$ varies in $V$, induce the same topology on $\calH$, as required.  
 It is also clear that the map $\Sigma_{v_0}:\Omega\to\calH$ is continuous. More generally:
  \begin{definition}[Sections]\label{def:sections}
The continuous maps $\Sigma:\Omega\to\calH$ such that $\pi \circ \Sigma = \mathI_\Omega$ are called  \emph{sections}.
  \end{definition}

  Every  section is a map on $\Omega$ that associates to each fiber $\omega$  a horosphere, hence, in a fixed chart, an element of the structure group $\mathZ$. Let us denote by $\bs n(\omega)$ this integer. 
   By the fact that the section is a continuous functions and by the definition of the topology on $\Omega$, it is clear that the function $\Sigma(\omega)=\bs n(\omega)$   is a section if and only if  it is locally constant on $\Omega$. It is also clear that he canonical projection  $\pi$  onto the base $\Omega$ is continuous in the topology of $\calH$.
   \\
Each section could be equivalently used to provide the global chart used to specify the horospherical index.    Conversely, each compatible global chart gives rise to a section. The definition of horosphere makes sense in this general framework.

\begin{definition}[Fiber bundle morphisms]\label{def:bundle_homomorphisms}
Given two fiber bundles $\calH_1$ and $\calH_2$ with  bases $\Omega_1$ and $\Omega_2$, with the same structure group $Z$ and
 canonical projections $\pi_i:\calH_1\to \Omega_i$, a continuous map $\alpha:\calH_1 \to \calH_2$ is a bundle morphism if there exists a continuous map 
 $\beta:\Omega_1\to\Omega_2$ such that $\pi_2 \circ \alpha = \beta \circ \pi_1$ 
and  the action of $\alpha$ on the fibers is equivariant under $Z$, that is, $z \alpha (\boldh) = \alpha (z\boldh)$ for every $\boldh\in\calH_1$ and $z\in Z$. 
\end{definition}
So, $\alpha$ maps the fiber  $\omega$ of $\calH_1$ to the fiber $\beta(\omega)$ of $\calH_2$, hence it defines a map of $\calH_1/\pi_1\approx \Omega_1$ to $\calH_2/\pi_2\approx \Omega_2$, and this map is precisely $\beta$: with abuse of notation, we could say that $\alpha$ acts on $\Omega_1$ and $\beta=\alpha\left|_{\Omega_1} \right. $. In this notation the commutative diagram of the morphism would be written as $\pi_2 \circ \alpha = \alpha \circ \pi_1$.
 \\
A fiber bundle morphism is an isomorphism if it is a bijection and its inverse is also a bundle morphism. 

Let us rephrase all this for our horospherical bundles:  bundle morphisms map horospheres to horospheres and fibers to fibers, and we now show that the fact that they commute with the action of the structure group yields  an invariant  in each fiber, that we call \emph{difference}.
Let us explain the notion of difference in a fiber. Our structure group $\mathZ$ acts simply transitively on horospheres tangent at $\omega$. 
That is, if $\boldh$ and $\boldh'$ are horospheres tangent at $\omega$, there exists a unique $n\in\mathZ$ such that $n \boldh = \boldh'$. This difference does not depend on the choice of  charts, because on each fiber $\omega$, a chart expresses the horosphere $\boldh(n,\omega)$ by a horospherical index $n$ that changes by an additive constant if we change the chart: so, the difference is the same in every chart.
 
\begin{corollary}\label{coro:bundle_isomorphisms_preserve_differences_along_fibers}
A fiber bundle isomorphism maps sections to sections. The difference of two horospheres in the same fiber is preserved by  bundle morphisms: that is, for every fiber bundle morphism $\alpha$, one has $\alpha(\boldh')-\alpha(\boldh)=\boldh'-\boldh$ whenever $\boldh'$ and $\boldh$ belong to the same fiber.
\end{corollary}
\begin{proof}
Let $\calH_1$, $\calH_2$ be fiber bundles and $\Sigma$ a section in $\calH_1$.
If $\alpha:\calH_1 \to \calH_2$ is a bundle isomorphism, then, by associativity,  $\pi_2 \circ (\alpha \circ \Sigma \circ \beta^{-1}) =( \pi_2 \circ \alpha) \circ \Sigma\circ \beta^{-1} = (\beta \circ \pi_1) \circ \Sigma \circ \beta^{-1} = \beta \circ (\pi_1 \circ \Sigma )\circ \beta^{-1} = \beta \circ \beta^{-1} = \mathI_{\Omega}$, and so $\alpha\circ \Sigma \circ \beta^{-1}$ is a section in $\calH_2$. Or, with abuse of notation as before, $\alpha\circ \Sigma \circ \alpha^{-1}$ is a section in $\calH_2$ for every section $\Sigma$ in $\calH_1$.

It is clear that a horosphere $\boldh\in\calH_1$ in the fiber $\omega$ is mapped by a bundle morphism $\alpha$ to a horosphere $\alpha(\boldh)$ in the fiber $\beta(\omega)$. For any two horospheres $\boldh$, $\boldh'$ in the fiber $\omega$, we have defined their difference as the integer that, regarded as an element of the structure group $\mathZ$, maps $\boldh$ to $\boldh'$. Hence $\boldh'-\boldh$ if and only if $n\boldh=\boldh'$, so $\alpha(n\bold h)=\alpha(\boldh')$. But since is $\alpha$ is equivariant with respect to the action of $\mathZ$, we have $\alpha(n\boldh)=n\alpha(\boldh)$, and so
 $\alpha(\boldh') - \alpha(\boldh) = \boldh'-\boldh$.
\end{proof}

\subsection{The horospherical fiber bundles}
\label{SubS:Fiber bundles}
The horospherical space $\HorV$
has a natural structure of a trivial principal fiber bundle with base $\Omega$ and fiber $\mathZ$. The choice of a reference vertex $v_0$ 
determines a global coordinate chart that identifies a pair $(\omega,n)\in\Omega\times\mathZ$ with the horosphere $\boldh_{n}(\omega,v_0)$. 
It makes sense to express the canonical projection $\pi$ of $\HorV$ on its base $\Omega$ in terms of coordinates $(\omega,n)$ independent of the reference vertex: $\pi(\omega,n)=\omega$. 

The following is a direct consequence of the cycle relation of Remark \ref{rem:cycle_relation}:
\begin{corollary}[Change of reference in the horospherical fiber bundle]
\label{cor:change_of_reference_for_horospheres}
The element of the structure group $\mathZ$ that acts on the fiber coordinate to give the chart change in the fiber on $\omega$ from the reference vertex $v_0$ to the reference vertex $v_1$ is $h(v_0,v_1,\omega)$; in other words, the horosphere that corresponds to the pair $(\omega,n)$ with respect to $v_0$ corresponds also to the pair $(\omega,n+h(v_0,v_1,\omega))$ with respect to $v_1$. 
\end{corollary}
%


\subsection{Equivariance of the horospherical fiber bundle under automorphisms}
\label{SubS:Automorphisms_on_horospheres}
By~Definition~\ref{def:horospheres}, every automorphism of a homogeneous (or semi-homogeneous) tree maps horospheres to horospheres: more precisely, the horospheres tangent at $\omega\in\Omega$ are mapped by the automorphism $\lambda$ to horospheres tangent at $\lambda\omega$. 
The horospherical indices satisfy
\begin{align}\label{eq:Aut_acts_equivariantly_on_the_horospherical_index}
  h(          v,              v_0,             \omega)
&=h( \lambda   v,   \lambda    v_0,   \lambda  \omega).  \end{align}
This is so because of the geometrical construction of the horospherical indices as the number of positive oriented edges from the reference vertex   
minus the negatively oriented ones in the path to the target vertex $v$:  
these numbers are clearly preserved by automorphisms in the way specified above. In particular:
\begin{corollary}[Equivariance under automorphisms of $T$]\label{cor:equivariance} Automorphisms of the tree map horospheres to horospheres, and preserve the horospherical index in the sense shown in \eqref{eq:Aut_acts_equivariantly_on_the_horospherical_index}.
\end{corollary}
So, a tree automorphism $\lambda$ induces a bundle automorphism, namely 
\begin{equation}\label{eq:equivariance_of_horospheres_under_automorphisms}
\lambda^{-1}\boldh_{n}(\omega,v_0)=\boldh_{n}(\lambda\omega,\lambda v_0),
\end{equation}
 and $\lambda$ commutes with the canonical boundary projection $\pi$.

\subsection{Action of tree automorphisms on the chart specified by a reference element}
It is clear from Corollary~\ref{cor:change_of_reference_for_horospheres} that if we fix the reference vertex $v_0$ %
and compute the horospherical indices before and after the action of the automorphism $\lambda$ in terms of this fixed vertex,  
then the automorphism simply induces a shift in the horospherical indices, as follows.

\begin{lemma}[Action of automorphisms on the horospherical fiber bundle]
\label{lemma:automorphism_shift_for_horospheres}
For every automorphism $\lambda$ of $T$, for every $v\in V$ and $\omega\in\Omega$, one has
\begin{align*}
  h( \lambda v,  v_0, \lambda  \omega)
&= h(v,  v_0, \omega)
 + h( \lambda v_0,  v_0, \lambda  \omega).
\end{align*}
\end{lemma}
\begin{proof}
By \eqref{eq:Aut_acts_equivariantly_on_the_horospherical_index}
 and the
cocycle identity of Remark \ref{rem:cycle_relation},
\begin{align*}
 h( \lambda v,  v_0, \lambda  \omega)
&=  h( v,  \lambda^{-1} v_0, \omega) =  h(v,  v_0, \omega) +  h(v_0,  \lambda^{-1}  v_0, \omega)
\\[.1cm]
&=  h(v,  v_0, \omega) +     h( \lambda v_0,  v_0, \lambda  \omega).
\end{align*}
\end{proof}

From this lemma, or more precisely from~Corollary~\ref{cor:change_of_reference_for_horospheres}, one has:

\begin{corollary}
\label{cor:change_of_reference_for_horospheres-2}
Every automorphism $\lambda$ of $T$ (hence of $\Omega$) induces an automorphism of the fiber bundle $\HorV$ as follows. Let us parameterize $\HorV$ with the coordinate system in $\HorV$ induced by the choice of a reference vertex $v_0$. Then

\begin{equation*}
                      \lambda(\omega,n)
=                    (\lambda \omega,n
 +  h( \lambda v_0,  v_0, \lambda  \omega)).
\end{equation*}
\end{corollary}

Another consequence of the cocycle identity of Remark \ref{rem:cycle_relation} will be useful later on to prove the Plancherel formula for the Radon transform. 
\begin{lemma}[The convolution on the fibers is reference-free] \label{lemma:convolution-on_fibers_is_independent_of_chart}
Let $\psi$ be a function on $\HorV$. Once a reference vertex $v_0$ is chosen, the values of $\psi$ can be written as $\psi(\omega,n)=\psi(\omega,n; v_0)$. In particular, the labeling of these values depends on the choice of the chart, that is of the special section. Let $\{d_n\}$ be a sequence on the structure group $\mathZ$, and, for each $\omega$ in the special section $\Sigma_{v_0}$, let us consider the convolution on $\mathZ$ given by
\[
\psi(\cdot\,,\,\omega; v_0) * d (n)= \sum_{k=-\infty}^\infty \psi(\omega, n-k;v_0) d_k
\]
(whenever the series converges). Then the convolution does not depend on the choice of the special section: for every vertex $v_1$,
\[
\psi(\cdot\,,\,\omega;v_0) * d = \psi(\cdot\,,\,\omega;v_1) * d.
\]
A similar statement holds for convolutions on the fibers of $\HorE$.
\end{lemma}
\begin{proof}
We prove the statemrnt for $\HorV$: the argument for $\HorE$ is similar.

By
\eqref{eq:change_of_horospherical_index_under_change_of_reference_vertex}, if $m$ denotes the value of the along the fiber $\omega$ in the chart given by the special section $\Sigma_{v_0}$ (that is, the horospherical index $n=h(\boldh,\omega; v_0))$, then the horospherical index of the same horosphere in the chart given by $\Sigma_{v_1}$ is $m=n-h(v_0,v_1,\omega)$. Therefore
\begin{align*}
(\psi(\cdot\,,\,\omega;v_1) * d) (m)&= 
\sum_{k=-\infty}^\infty \psi(\omega, k;v_1) d_{m-k} = \sum_{k=-\infty}^\infty \psi(\omega, k+ h(v_0,v_1,\omega;v_0) d_{m-k}\\[.2cm]
&=
 \sum_{k=-\infty}^\infty \psi(\omega, k+ h(v_0,v_1,\omega) d_{n-h(v_0,v_1,\omega)--k} \\[.2cm]
 &= (\psi(\cdot\,,\,\omega;v_0) * d) (m).
\end{align*}
\end{proof}

\subsection{Orbits of tree automorphisms on the fiber bundles}\label{SubS:Orbits_of_Aut_T_on_fiber_bundles}

Let us recapitulate the properties of the stabilizer $(\Aut T)_\omega$ of a boundary point $\omega$ that we have seen in the previous Subsections, notably in Subsection \ref{SubS:isotropy_group_of_horospheres}.
\\
Automorphisms map geodesic paths to geodesic paths, hence their action extends to the boundary $\Omega$. We have seen in \eqref{eq:Aut_acts_equivariantly_on_the_horospherical_index} that
the diagonal action of $\Aut T$ acts in an equivariant way on the horospherical indices. Therefore an automorphism $\lambda$ maps horospheres tangent at $\omega$ to horospheres tangent at $\lambda\omega$. So the stabilizer $(\Aut T)_\omega$ of $\omega$ preserves the set of horospheres tangent at $\omega$, but does not fix each such horosphere (because $ h(v, v_0,  \omega)=h( \lambda v,  \lambda v_0, \lambda  \omega)$ and in general the latter horospherical index is different from $ h( \lambda v,  v_0, \lambda  \omega)$). Let us give a geometric description of this action in reference to the pictures of the vertex-horospheres tangent at  $\omega$ given in Figure \ref{Fig:family_tree}.

The stabilizer at $\omega$, $(\Aut T)_\omega$, acts on the set of horospheres tangent at $\omega$. $(\Aut T)_\omega$ has a subgroup $N_\omega$ that preserves each such horosphere: it acts by permuting the cousins of each generation compatibly with their parenthood. There are elements of $(\Aut T)_\omega$ that do not preserve generations. An interesting one-parameter subgroup $S_\omega$ consists of the automorphisms that, choosen a two-sided geodesic ray to $\omega$ and given our picture of the hanging tree, slide the tree vertically along this ray. Observe that, by this splitting, the group $(\Aut T)_\omega$ of a homogeneous tree is transitive on both $E$ and $V$.
 Clearly, $N_\omega$ has index one in $(\Aut T)_\omega$, the quotient is isomorphic to $S_\omega\approx\mathZ$ and $(\Aut T)_\omega$ is isomorphic to the semidirect product  $S_\omega \ltimes N_\omega$. This is analogous to the $AN$ splitting of the $KAN$ decomposition of $SL_2(\mathR)$ regarded as a group of automorphisms of the hyperbolic disc.

\begin{corollary}
\label{cor:transitivity_on_horospheres}
Let $T$ be a homogeneous tree. Then
\begin{enumerate}
\item[$(i)$]
 for every fiber $\omega\in\Omega$,
the stability subgroup $(\Aut T)_\omega$ in $\Aut T$ acts transitively on the fiber $\pi^{-1}\omega$, both in $\HorV$ and in $\HorE$; 
%
\item[$(ii)$]
the stability subgroup $(\Aut T)_\omega$  of every boundary point $\omega$ acting on the flag horospherical bundle $\HorF$ has two orbits on $\HorF$; 
\item[$(iii)$] 
the actions of $\Aut T$ are transitive on $\HorV$ and $\HorE$. 
\end{enumerate}
\end{corollary} 

\begin{proof}
For part $(i)$, it only remains to prove that the action of $(\Aut T)_\omega$ is transitive on vertices of the geodesic ray $\omega$ of the same homogeneity degree. This is true because every vertex is moved to any other vertex of the same degree by an automorphism in $S_\omega$, that maps $\omega$ onto itself. 

For part $(ii)$, it is enough to remember from Subsection \ref{SubS:automorphisms} that the full automorphism group $\Aut T$ cannot swap 
the two boundary arcs subtended by a flag, but $(\Aut T)_\omega$ is transitive on $E$,  hence $(\Aut T)_\omega$ has two orbits on $\HorF$, consisting of all horospheres tangent at $\omega$ whose flags $f=(e,v)$ have flag-vertex $v$ at the same side of $\omega$ with respect to $e$ or at the opposite side, respectively.

Part $(iii)$ follows easily from part $(i)$ and the fact that the action of $\Aut T$ is transitive on $\Omega$. 
\end{proof}

\subsection{The parallel shift subgroup of $\Aut H$}
\label{SubS:Cartan}
Let $\calH$ be either $\HorV$ or $\HorE$. We have defined in Definition \ref{def:sections} the sections $\Sigma:\Omega\to\calH$ as the continuous right inverses of the canonical projection on the base $\Omega$ introduced at the beginning of Subsection \ref{SubS:Fiber bundles}: that is, $\pi\circ \Sigma=\mathbb I$. 
By abuse of notation, we shall identify a section with its own image, that is, we shall regard sections as subsets of $\calH$.
 We have also introduced in Definition \ref{def:special_sections} special sections in $\HorV$: those sections whose image consists of all horospheres that contain a vertex $v$. The special section through $v$ is denoted by $\Sigma_v$. An analogous definition holds for sections in $\HorE$ or $\HorF$.

We shall now introduce two important abelian subgroups of $\Aut\HorV$ (and of $\Aut\HorE$), consisting of shifts along fibers.

The choice of any section $\Sigma$ in $\HorV$ endows every fiber $\omega$ in the base space $\Omega\approx \HorV/A$ with an explicit choice of integer coordinate: the horosphere $\boldh$ tangent at $\omega$ (that is, belonging to the fiber $\omega$) is parameterized by the integer $n=n(\Sigma, \omega)$ given by the element of the fiber group $\mathZ$ that carries to $\boldh$ its parallel horosphere $\boldh_0\in\Sigma$. In other words, the choice of $\Sigma$ generates a global chart on $\HorV$, and indeed is equivalent to choosing a global chart. Of course the same is true for $\HorE$ and $\HorF$.

Interesting sections of $\HorV$ are $\Sigma_n=\{\boldh_{n}(\omega,v_0):\omega\in\Omega\}$: these are the \textit{circular} sections, in the sense that, as $\omega$ varies in $\Omega$, the vertices in $\Sigma_n(\omega)$ at minimal distance from $v_0$ form a circle (the circle at distance $|n|$ from $v_0$). One can define circular sections in $\HorE$ similarly. For $n=0$ one obtains the special section $\Sigma_{v_0}$.

\begin{definition}[The shifts subgroup]\label{def:shift_subgroup}
We have parameterized the vertex-horospheres as a fiber bundle $\HorV$ with base space given by the boundary $\Omega$ of a tree, and fibers given by the integers: in the chart obtained by fixing a reference vertex $v_0$, the integer on the fiber at $\omega$ is the horospherical number $h(v,v_0,\omega)$. The group $\widetilde A$ is defined as the (abelian) subgroup of $\Aut\HorV$ whose elements $\tilde a$  leave every fiber $\omega$ in itself but change the horospherical index $n(\omega)$ to $n(\omega)+k(\omega)$, where $k$ is a function on $\Omega$ necessarily locally constant (this condition is necessary to insure that $\widetilde a$ is a continuous map in the topology of $\HorV$, as remarked after Definition \ref{def:sections}) The same group can also be embedded in $\Aut\HorE$. Note that $\widetilde A$ preserves each fiber, and is exactly the subgroup of $\Aut\HorV$ and of $\Aut\HorE$ that preserves all the fibers, that is, acts trivially on $\Omega$.
\end{definition}

\begin{remark}\label{rem:subgroups_of_AutH_that_fix_the_fibers}
Observe that $\Aut T$ and $\widetilde A$ are both subgroups of $\Aut\HorV$ and of $\Aut\HorE$, non-overlapping except at the identity.
We shall see in Proposition \ref{prop:horosphere_correspondence} that $\HorV$ and $\HorE$ are canonically isomorphic, hance $\Aut\HorV\approx\Aut\HorE$, and that the embeddings of $\Aut T$ into $\Aut\HorV$ and into $\Aut\HorE$ are compatible with the canonical isomorphism. 
\end{remark}

By~Corollary~\ref{cor:change_of_reference_for_horospheres}, every automorphism $\lambda\in \Aut T$ maps a horosphere tangent at $\omega$ to a horosphere tangent at $\lambda\omega$: thus $\Aut T$ commutes with the canonical projection $\pi$, and acts upon the space of sections of $\HorV$ and $\HorE$ (of course, the same is true, more generally, for every fiber bundle automorphism). 
All circular sections centered at $v_0$ are fixed points of the action of the stabilizer $(\Aut T)_{v_0}$, because automorphisms preserve distances. There is a one-to-one correspondence that is equivariant under automorphisms between $V$ and the set of circular sections of radius 0.

\begin{definition}[The parallel shift subgroup, or Cartan subgroup]\label{def:paralel_shift_subgroup}
The parallel shift subgroup  $A\subset\Aut \HorV$ is the subgroup of $\widetilde{A}$ consisting of automorphisms that preserve the fibers and operate on each fiber $\omega$ by the same shift $k(\omega)$: that is, $k$ is constant.
Clearly, $A$ is isomorphic to $\mathZ$, and is generated by the the elementary automorhism that increases by one the index in each fiber. In other words, the group $A$ applies a parallel shift to the the sections, that is, it maps circular sections to circular sections. 

We shall call $A$ the \emph{parallel shift subgroup of $\Aut\HorV$}, or \emph{Cartan subgroup}.  The same subgroup can be embedded into $\Aut\HorE$.
\end{definition}

\begin{proposition}\label{prop:A_commutes_with_Aut(HorV)}
The subgroup of parallel shifts $A\subset \Aut\HorV$ is in the center of $\Aut\HorV\approx\Aut\HorE$. Instead, $\widetilde A$ is not contained in the center, but is a normal subgroup of $\Aut\HorV$, and $\Aut\HorV/\widetilde A$ is the group   $\Aut \Omega$ of homeomorphisms of $\Omega$ to $\Omega$.

Moreover, if  $(\Aut\HorV)^\sharp$ is the subgroup of all $\tau \in \Aut\HorV$ whose action on $\Omega$ coincides with the action of some $\lambda\in\Aut T$, then $(\Aut\HorV)^\sharp$ is generated by $\Aut T \cdot \widetilde A$,  and $\Aut T \cdot \widetilde A$ is a semi-direct product:  $(\Aut\HorV)^\sharp= \Aut T \ltimes \widetilde A$. 
\end{proposition}
\begin{proof}
Let $\lambda\in\Aut\HorV$ and $\bs n\in A$. 
Call $n$ the integer associated to $\bs n$: that is, $\bs n$ shifts forward each horosphere  by $n$ steps in each fiber $\omega$   (that is, along a geodesic ray to the boundary point $\omega$).  Observe that $\bs n\lambda\boldh $ is the shift of $n$ steps of $\lambda \boldh$ towards $\lambda\omega$. On the other hand, $\lambda \bs n\boldh $ is the horosphere tangent at $\lambda\omega$ obtained by acting by $\lambda$ on the horosphere $\boldh$ shifted of $n$ steps toward $\omega$, and it coincides with  $\bs n\lambda\boldh $ by definition of automorphism of $\HorV$ (that is, the fact that an automorphisms preserves the index of horospheres inside fibers, i.e., with notion of difference $\alpha$ introduced in Subsection \ref{SubS:Sections_and_special_sections}, $\alpha(\bs n \boldh )-\alpha(\boldh )=n$ and  $\alpha(\lambda \bs n \boldh )-\alpha(\lambda \boldh )=n$, hence $\alpha(\lambda \bs n \boldh ) - \alpha(\bs n \lambda  \boldh )=0$.
Of course $\widetilde A$ is not contained in the center of $\Aut\HorV\approx\Aut\HorE$, because, clearly,  an element $\widetilde a\in \widetilde A$ does not commute with $\Aut T$ unless the associated shift function $k(\omega)$ is not constant. 

Let us take any $\gamma\in\Aut\HorV$ and restrict attention to its action on the fibers of $\HorV$: then we obtain an automorhism $\lambda\in\Aut \Omega$. Now $\gamma \lambda^{-1}$ preserves each fiber $\omega$, hence it is an automorphism in $\widetilde A$.

Moreover, $\widetilde A$ is a normal subgroup in $\Aut\HorV\approx\Aut\HorE$. Indeed, let $\widetilde a\in\widetilde A$ and $\gamma\in\Aut\HorV$, and consider any boundary point $\omega_0$. Let $\omega_1=\gamma\omega_0$ and $\bs h_{\omega_0}=\bs h_{\omega_1}$ two horospheres, belonging to the fibers $\omega_0,\,\omega_1$ respectively,
 such that $\gamma \bs h_{\omega_0}=\bs h_{\omega_1}$. Therefore $\widetilde a \gamma_{\omega_0} \bs h_{\omega_0}$ is the horosphere $\bs h'_{\omega_1}$ tangent at $\omega_1$ given by the shift $\bs k(\omega_1)  \bs h_{\omega_1})$. Now, $\gamma^{-1}  \widetilde a \gamma \bs h_{\omega_0} $ is a horosphere tangent at $\omega_0$, hence obtained by shifting by $\bs k(\omega_0)$ along the fiber $\omega_0$ the original horosphere  $\bs h_{\omega_0}$. This happens for each fiber: hence $\gamma^{-1}  \widetilde a \gamma$ is a shift along each fiber, hence it belongs to $\widetilde A$. In other words, $\widetilde A$ is a normal subgroup of $\Aut\HorV$.
 
 Now the last part of the statement is clear.
\end{proof}

\subsection{The canonical edge-vertex horospherical correspondence and its equivariance under tree automorphisms}\label{Subs:edge-vertex_horospherical_correspondence}
\begin{proposition}[Edge-vertex horospherical correspondence]
\label{prop:horosphere_correspondence}
Given a vertex $v$ and a boundary point $\omega$, we denote by $v_+$ the neighbor of $v$ that is closer to $\omega$ 
and by $e(v,\omega)$ the edge $[v,v_+]$.
Then, for each $\omega\in\Omega$, the map $\xiomega : v\mapsto e(v,\omega)$ is a injection of $V$ into $E$ (clearly, the inverse map associates to every edge its vertex  farther from $\omega$). This map induces a bijection from vertex-horospheres tangent at $\omega$ onto edge-horospheres tangent at $\omega$. When $\omega$ varies in $\Omega$, this yields a bijection  $\Xi\,:\HorV \to \HorE$. This map is equivariant under $\Aut T$, that is, it commutes with its action. Finally, the bijection  $\Xi$ is an automorphism.
\end{proposition}
\begin{proof}
The only parts that might require a proof are the fact that the injection $\xiomega\,:V \to E$ extends to a bijection $\Xi\,:\HorV \to \HorE$, and
 equivariance.  The fact that the injection $\xiomega :V \to E$ extends to an injection $\Xi\,:\HorV \to \HorE$ is obvious. On the other hand, on each fiber this map is clearly a bijection, because the vertices  that belong to an edge-horosphere in the boundary point $\omega$ and are at the opposite side of $\omega$ form a vertex-horosphere, by the definition of horosphere. So on the horospherical bundles the map becomes a bijection.

The fact that $\Aut T$ acts transitively on $\Omega$
is equivalent to saying that the action of $\Aut T$ is equivariant on geodesic rays, that is, every $\lambda\in\Aut T$ maps all geodesic rays $[v,\omega)$ to $[\lambda v,\lambda \omega)$. But then, $\Aut T$ maps the edge that contains $v$ and is closer to $\omega$ into the the edge that contains $\lambda v$ and is closer to $\lambda \omega$. In other words, $\lambda$ commutes with the canonical injection $\xiomega\,:V \to E$, hence also with the bijection $\Xi\,:\HorV \to \HorE$.
\end{proof}

\begin{remark}
We could embed the fiber bundles $\HorV$ and $\HorE$ into a larger (but actually isomorphic) fiber bundle $\calH_{V\cup E}$, where each element in each fiber is a vertex-horosphere and an edge-horosphere, alternately, the two being related so that the edge-horosphere is the image of the corresponding vertex-horosphere under the bijection of Proposition \ref{prop:horosphere_correspondence}. Then we would obtain a parallel shift group isomorphic to $\frac12\mathZ$, where the element $\frac12$ is the parallel shift from a vertex-horosphere to the edge-horosphere canonically associated to it by the bijection $\sigma$. In other words, 
this bijection
extends the action of the parallel shift group $A$ on the horospherical bundles.
\end{remark}

\begin{proposition} \label{prop:HorV=HorE_but_pairs_of_special_sections_do_not_correspond}
We have seen in Proposition \ref{prop:horosphere_correspondence}  that the horospherical fiber bundles $\HorV$ and $\HorE$ are isomorphic in the sense that the canonical correspondence  yields a bijection from one to the other that commutes with the action of $\Aut T$. On the other hand, this isomorphism does not map special sections to special sections. More generally, there does not exist any fiber bundle isomorphism between $\HorV$ and $\HorE$ that maps a pair of special sections to a pair of  special sections. However, for every special sections $\Sigma_v$ and $\Sigma_e$, there is an homomorphism from $\HorV$ to $\HorE$ that maps $\Sigma_v$ to $\Sigma_e$, and vice-versa.
\end{proposition}
\begin{proof} Suppose that there exists an isomorphism $\alpha:\HorV \to \HorE$ that preserves special sections. Choose $v_0, v\in V$ and let $\Sigma_{e_0}=\alpha \circ \Sigma_{v_0}$ and $\Sigma_{e}=\alpha \circ \Sigma_{v}$. Let $\omega$ be an arbitrary boundary point and $\boldh_{v_0}\in\Sigma_{v_0}$,  $\boldh_{v}\in\Sigma_{v}$ belong to the same fiber $\omega$ (and write $\boldh_{e_0}=\alpha(\boldh_{v_0})$, $\boldh_{e}=\alpha(\boldh_{v})$.
Then $\boldh_{e_0}\in\Sigma_{e_0}$,  $\boldh_{e}\in\Sigma_{e}$ and these two edge-horospheres belong to the same fiber $\alpha(\omega)$. Now,
the difference $\boldh_v - \boldh_{v_0}$ is the difference of the vertex-horospherical indices, $h(v,v_0,\omega)-h(v_0,v_0,\omega)=h(v,v_0,\omega)$, and
similarly $\boldh_e-\boldh_{e_0}=h(e,e_0,\alpha(\omega))$. 

We have proven in Corollary \ref{coro:bundle_isomorphisms_preserve_differences_along_fibers} that fiber bundle isomorphisms preserve differences along fibers. Therefore now we should have $h(e,e_0,\alpha(\omega)))=h(v,v_0,\omega)$ for every $\omega\in\Omega$.
On the other hand we have seen in Remark \ref{rem:subtle_difference_between_horospherical_indices_for_vertices_and_for_edges} that the horospherical indices $h(v,v_0,\omega)$
have the same parity of $\dist(v,v_0)$, constant when $\omega$ varies in $\Omega$, but the parity of $h(e,e_0,\alpha(\omega))$ is not constant over $\Omega$, a contradiction.

Finally, in order to prove the last part of the statement, it is enough to choose an edge $e$ adjacent to $v$, because we can reduce to this case via the action of an automorphism in $\Aut T$, that certainly maps special sections to special sections. Now, regard the canonical bijection $\Xi$ as an isomorphism from $\HorV$ to $\HorE$, and split $\Omega=\Omega_+ \coprod \Omega_-$, where $\Omega_+$ is the set of boundary points determined by boundary arcs on the side of $e$ with respect to $v$. Let $\bs h\in \Sigma_v$ with boundary point in $\Omega_+$: then we have seen that $\Xi\bs h\in \Sigma_e$. But if the boundary point $\omega_-$ of $\bs h\in \Sigma_v$ belongs to $\Omega_-$ then $\Xi\bs h\notin \Sigma_e$: instead, $\Xi\bs h$ has boundary point $\omega_-$ but horospherical index $+1$ in the chart determined by the special section $\Sigma_e$. But then, let $\beta\in\widetilde{A}$ be the fiber shift  that shifts by $-1$ the edge-horospheres in the fibers in $\Omega_+$ and leaves fixed the horospheres in the fibers $\Omega_-$: then $\beta\, \Xi \,\Sigma_v = \Sigma_e$.
\end{proof}

\section{Invariant measures for homogeneous trees}
\subsection{A measure on $\Omega$ invariant under the stability subgroup of a vertex or an edge in a homogeneous tree}\label{SubS:invariant_measure_on_boundary}
 
 The Borel $\sigma-$algebra of open sets in the topology on $\Omega$ defined in Subsection \ref{SubS:Boundary} is naturally associated with appropriate  normalized measures $\nu$. For simplicity, fix a reference vertex $v_0$ Then $\nu_{v_0}$ on $\Omega$ defined by the rule that $\nu(\Omega_n)$ is the reciprocal of the number of vertices at distance $n$ from $v_0$. In particular, on the boundary of the homogeneous tree $T_q$, one has
\begin{equation}\label{eq:homogeneous_invariant_vertex-measure}
\nu_{v_0}(\Omega(v_n;v_0))=\frac1{(q+1)q^{n-1}}\;.
\end{equation}

The measure $\nu_{v_0}$ is the only invariant probability measure under $K_{v_0}$, because any two arcs $\Omega(v)$ and $\Omega(w)$ have the same measure if (and only if) $\dist(v,v_0)=\dist(w,v_0)$: it will be called \emph{equidistributed boundary measure} with respect to $v_0$.
Note that there does not exist a finite measure on $\Omega$ invariant under the full group $\Aut T$, because every boundary arc is mapped to any other by some automorphisms, hence any invariant measure under $\Aut T$ should give the same mass to all arcs. 

Observe also that, by Definition \ref{def:special_sections}, the measure $\nu_{v_0}$ can be regarded as a measure on the special section $\Sigma_{v_0}$.

For $E\subset\Omega$ measurable, $\nu$  a measure on $\Omega$ and $\lambda\in\Aut T$, then let us write $\lambda^{-1}\circ \nu(E)=\nu(\lambda E)$. Then the following obvious equivariance relation holds: for every $\lambda\in\Aut T$ and $v\in V$, 
\begin{equation}\label{eq:covariance_of_measure_nu}
d\nu_{\lambda v}(\lambda \omega) = d\nu_v(\omega).
\end{equation}

These topology and measure have been deeply investigated in \cite{Cartier} and \cite{Figa-Talamanca&Picardello}. The measure depends on the choice of $v_0$, but its Radon--Nikodym equivalence class does not, and neither does the topology.

Similarly, if we regard $\Omega$ as the set of (equivalence classes of) sequences of adjoining edges, $\omega=\{e_0, e_1, \dots, e_n, \dots\}$, that is, geodesic rays of edges, then the topology obtained in this realization of $\Omega$ is the same, because the forward boundary arcs $\Omega(e_n)=\Omega^E_n\subset\Omega$ (given by all geodesic edge-rays starting at $e_0$ and containing $e_n=[v_{n-1}, v_n]$) verify $\Omega^E_n=\Omega(e_n)=\Omega(v_n)=\Omega^V_n$.

A normalized boundary measure $\nu_{e_0}$ centered at $e_0$ is given by the rule that $\nu_{e_0}(\Omega(e_n))$ is the reciprocal of the number of edges at distance $n$ from $e_0$. If $T$ is a homogeneous tree, this is the only normalized measure on $\Omega$ invariant under $K_{e_0}$. This measure will be called the \emph{equidistributed boundary measure} with respect to $e_0$. On semi-homogeneous not homogeneous trees, $K_{v_0}$ is still transitive on $\Omega$, but we have seen that $K_{e_0}$ is not, because it cannot reverse $e_0$. This fact gives rise to a larger set of  boundary measures invariant under $K_{e_0}$, consisting of
a one-parameter family (see Remark \ref{rem:orbits_over_F-horospheres}). As before, there is no non-zero measure on $\Omega$ invariant under $\Aut T$.

\begin{remark}[Semi-homogeneous invariant boundary measures]\label{rem:semihomogeneous_invariant_measures}
In a similar way we introduce the normalized invariant measure $\nu_{v_0}$  on the boundary of a semi-homogeneous tree with two alternating homogeneity degrees $q_+, q_-$, centered at a vertex $v_0\in V$ (invariant under the isotropy subgroup $(\Aut T)_{v_0}$). Let the homogeneity degree of $v_0$ be, say, $q_+$. Consider the boundary arcs $\Omega_{v_1}=\Omega_{v_0,v_1}$ subtended by a vertex $v_1$ at distance 1 from $v_0$, that is, represented by geodesic rays starting at $v_0$ and containing $v_1$: then, by isotropy around $v_0$, one must have 
$\nu_{v_0}(\Omega_{v_0, v_1} = 1/(q_++ 1)$. Let now $\Omega_{v_0,v_1,v_2}$ be the boundary arcs subtended by vertices $v_2$ at distance 2 from $v_0$, that is, whose elements are represented by geodesic rays starting with the finite path $[v_0, v_1, v_2]$: then, again by isotropy, $\nu_{v_0}(\Omega_{v_0, v_1, v_2} = 1/((q_++ 1)q_-)$. Iterating this remark, we note that  
\[
\nu_{v_0}(\Omega_{v_0, v_1, v_2, \dots, v_{2n}}) = \frac 1{(q_++ 1) (q_-q_+)^{n}},
\]
 and 
\[
\nu_{v_0}(\Omega_{v_0, v_1, v_2, \dots, v_{2n+1}} )= \frac 1{(q_++ 1) \,(q_-q_+)^{n}\, q_-},
\]
The normalized positive measures $\nu_{e_0}$ invariant under the isotropy subgroup $(\Aut T)_{e_0}$ is constructed in the same way, keeping into account that every edge has $q_+ + q_-$ adjoining edges. However, in this setup the isotropy subgroup of $e_0$ does not interchange neighboring edges of $e_0$ adjoining $e_0$ at the vertex $v_+$ of homogeneity $q_+$ with those at the vertex $v_-$ of homogeneity $q_-$. Therefore the boundary arcs $\Omega_+$  closer to $v_+$ and $\Omega_-$  closer to $v_-$, respectively, are not interchanged, and we can freely assign their masses $\nu_+=\nu_{e_0}(\Omega_+)$ and $\nu_-=\nu_{e_0}(\Omega_-)$. The only conditions are positivity and normalization, that is, $c_+ + c_- = 1$. This leaves one free parameter in the choice of the masses of arcs subtended by edges at distance 1 from $e_0$, that we call arcs of the first generation. Each subsequent generation yields an additional degree of freedom. This fact imposes severe constraints on radial harmonic analysis on the space of edges of a semi-homogeneous tree, as shown in  \cite{Casadio_Tarabusi&Picardello-algebras_generated_by_Laplacians}.
\end{remark}

\subsection{A measure invariant under automorphisms on the horospherical fiber bundle of a homogeneous tree}\label{SubS:invariant_measure_on_Hor}
Let $T=T_q$ be a homogeneous tree. We have seen in Subsection \ref{SubS:Boundary} that $\Omega$ does not admit any non-zero measure invariant under automorphisms. We shall now prove  that such a measure exists on the horospherical fiber bundles $\HorV$ and $\HorE$. We limit attention to $\HorV$, the argument for $\HorE$ being identical.

The reference vertex $v_0$ gives rise to a global chart on $\HorV$ that identifies $\HorV$ with $\Omega \times \mathZ$: the pair $(\omega,n)\in\Omega\times\mathZ$ is identified with the horosphere $\boldh_{n}(\omega,v_0)$. Choose an open base in $\HorV$ given by the sets
$I(v,v_0,n)=\{\boldh(\omega,v_0)=\boldh_n(\omega,v_0): \omega\in \Omega(v,v_0), \,h(v, v_0,\omega)=n\}$, where $\Omega(v,v_0)$ is the boundary arc subtended by $v$ defined at the beginning of Subsection \ref{SubS:Boundary}. 
With this notation:
\begin{definition}[Invariant measure on $\HorV$]\label {def:invariant_measure_on_Hor}
Let $T=T_q$ be a homogeneous tree. We define $\xi$ as the measure on $\HorV$ given by \[\xi(I(v,v_0,n))=q^n\nu_{v_0}(\Omega(v,v_0)).\] 
A similar definition holds for $\HorE$.
\end{definition}

\begin{lemma} \label{lemma:the_product_measure_does_not_depend_on_the_choice_of_reference_vertex}
The measure $\xi$ does not depend on the choice of the global coordinate chart, that is, on the choice of $v_0$, and is normalized so that the mass of every special section in $\HorV$ is 1. 
\end{lemma}
\begin{proof}
Let $v_1\neq v_0$ be another reference vertex. Then, by the cocycle relation of Remark \ref{rem:cycle_relation}, $h(v,v_0,n)=h(v,v_1,n)-h(v_0,v_1,n)$. On the other hand, $\nu_{v_0}(\Omega(v,v_0))=q^{1-\dist(v,v_0)}/(q+1)$ by \eqref{eq:homogeneous_invariant_vertex-measure}, and so 
\begin{equation}\label{eq:Poisson_kernel}
\frac{\nu_{v_1}(\Omega(v,v_0))}{\nu_{v_0}(\Omega(v,v_0))}=q^{\dist(v,v_0)-\dist(v,v_1)}=q^{h(v_1,v_0,\omega)}
\end{equation}
for every $v\in V$ and $\omega\in \Omega(v,v_0)$. The result follows easily.
\end{proof}

\begin{remark}
The argument of the previous proof shows that, given any non-negative sequence $a_n$, the measure on $\HorV$ or $\HorE$ given by $a_n\nu_{v_0}(\Omega(v,v_0))$ (with $n=h(v,v_0,n)$ for $\omega\in \Omega(v,v_0)$)
 does not depend on the choice of $v_0$ if and only if $a_n=c q^n$ for some constant $c$.
\end{remark}

\begin{proposition} \label{prop:uniqueness_of_invariant_measure_on_HorV_and_HorE} If $T$ is a homogeneous tree, the measure $\xi$ on $\HorV$ and its analogous measure on $\HorE$ are invariant under $\Aut T$. 
The invariant measure  is unique up to multiples.
\end{proposition}
\begin{proof}
Again, we limit attention to $\HorV$. Since the horospherical index depends only on the relative positions of $v, v_0$ and the boundary point, for every $\lambda\in\Aut T$ one has
$\lambda I(v,v_0,n)) = I(\lambda v, \lambda v_0, n)$. On the other hand, $\nu_{v_0}(\Omega(v,v_0))$ depends only on $\dist(v,v_0)$, by \eqref{eq:homogeneous_invariant_vertex-measure}. So, since the distance is invariant under automorphisms, $\nu_{v_0}(\Omega(v,v_0))=\nu_{\lambda v_0}(\Omega(\lambda v,\lambda v_0))$ for every $\lambda\in\Aut T$.
Hence
\begin{align}\label{eq:Poisson=Radon-Nikodym}
\xi(\lambda I(v,v_0,n))&=\xi(I(\lambda v, \lambda v_0, n))=q^n\nu_{\lambda v_0}(\Omega(\lambda v, \lambda v_0))\notag\\[.2cm]
&=
q^n\nu_{ v_0}(\Omega( v,  v_0)) =\xi(I(v,v_0,n)).
\end{align}
Uniqueness of the normalized invariant measure follows immediately from transitivity of the action of $\Aut T$ on $\Omega$.
\end{proof}

The function at the right hand side of \eqref{eq:Poisson_kernel} is called the \emph{Poisson kernel} (we shall give a different approach to the Poisson kernel, as a function on the horospherical bundle, in the next Chapter). 
\\
The Poisson kernel is denoted by $\Poiss(v,v_0,\omega)$. For $E\subset\Omega$ measurable, $\nu$  a measure on $\Omega$ and $\lambda\in\Aut T$, then let us write $\lambda^{-1}\circ \nu(E)=\nu(\lambda E)$. Then  \eqref{eq:Poisson=Radon-Nikodym} states that the Poisson kernel is a Radon--Nikodym derivative:
\begin{equation}\label{eq:Poisson_kernel_as_Radon-Nikodym_derivative}
\Poiss(\lambda v_0,v_0,\omega) = \frac {d(\lambda^{-1}\circ \nu_{v_0})}{d\nu_{v_0}}(\omega)=q^{h(\lambda v_0,v_0,\omega)} .
\end{equation}

This expression for the Poisson kernel on homogeneous trees was first obtained in \cite{Furstenberg,Figa-Talamanca&Picardello-JFA,Figa-Talamanca&Picardello}. An entirely analogous description is obtained by using edges instead than vertices. All this shows that the Poisson kernel is constant on horospheres, and therefore is a function on the horospherical bundles $\HorV$ and $\HorE$, and in particular on the fiber bundles $\sHorV$ and $\sHorE$. If we denote  points of the bundles by  $(\omega,n)$, the Poisson kernel becomes a function $\widetilde \Poiss(\omega,n)$ that varies multiplicatively on the fibers: $\widetilde K(\omega,0)=1$ for every $\omega$ (because $(\omega, 0)$ corresponds to the horospheres through $v_0$), and $\widetilde \Poiss(\omega,n+1) = q \,\widetilde K(\omega,n)$. All other multiplicative functions are given by its complex powers $\widetilde K(\omega,n)^z$.

It has been observed in \cite{Furstenberg} 
that the Poisson kernel is a reproducing kernel from  finitely additive measures (also called \emph{distributions}) 
 on $\Omega$ to harmonic functions on $V$ with respect to the one-step mean value property. This averaging operator, that will be denoted by $\mu_1$ in the sequel (see Definition \ref{def:convolutions&Laplacians}) is called the \emph{Laplace operator}.
By using the kernel $\widetilde K(\omega,n)^z$ where $z$ varies in $\mathC$, we obtain all the other eigenfunctions of the Laplace operator (see Proposition \ref{prop:Poisson_transform_of_distributions} below).
In \cite{Figa-Talamanca&Picardello-JFA,Figa-Talamanca&Picardello}, the $z-$Poisson transform was used to transport the action of $\Aut T$ from Banach spaces of measures on the boundary to eigenspaces of the Laplace operators,  obtaining therein a complex family of unitary or uniformly bounded representations of $\Aut T$ and its subgroups, where particular attention is devoted to simply transitive subgroups like the free group $\mathbb F_q$ or suitable free products. This theory of representation has been developed by starting with radial eigenfunctions on vertices of the tree, called \emph{spherical functions} or \emph{zonal functions}. 
In the sequel of this book (Chapter \ref{Chap:spherical_functions}) we shall present the theory of spherical functions for edges (not previously studied) and spherical unitary representations, but we shall not fully investigate the subject of representations of subgroups of $\Aut T$, that is beyond our purpose here.

\section{The invariant measure on $\Omega$ induced by the the stability group of a horosphere}
Let us consider the stability subgroup $N_\omega$  of each horosphere in the fiber $\omega$ and its action on $\Omega$ explained in Subsection \ref{SubS:isotropy_group_of_horospheres}

\section{Automorphisms and horospheres on semi-homogeneous trees}
A semi-homogeneous tree is a tree whose vertices have two different homogeneity degrees $q_-<q_+$ alternating along each geodesic.
Many of the previous results stated for homogeneous trees hold also in the semi-homogeneous environment. In this Section we collect those statements that are different in the semi-homogeneous case.

\begin{proposition} \label{prop:Aut_T_for_T_semi-homogeneous} On  a semi-homogeneous non-homogeneous tree, $\Aut T$ maps every pair of vertices at even distance to every other pair of vertices at the same distance, and the same is true flags, but it is doubly transitive on edges.
\end{proposition}
\begin{proof} Let $T$ be a semi-homogeneous not homogeneous tree. Then the same argument of the proof of Proposition \ref{prop:Aut_T_for_T_homogeneous} shows that $\Aut T$ maps any pair of vertices $v_1, v_2$ to any other pair $v'_1, v'_2$ at the same distance such that $v'_i$ has the same homogeneity degree of $v_i$, for $i=1,2$.

Instead, $\Aut T$ is
transitive on edege, because it maps every pair of adjacent vertices to any other such pair.

Finally, $\Aut T$  has two orbits on $F$. Indeed, it it cannot flip any edge, because the two end-points have different homogeneity degrees. Therefore $\Aut T$ preserves any given orientation of edges: since flags are oriented edges, the two equivariant choices of orientation form two orbits upon which $\Aut T$ acts transitively.
More precisely, any flag $(e, v)$ can be moved to any other flag $(e', v')$ such that $v'$ has the same parity of $v$; moreover, any pair of flags $\{(e_1, v_1), (e_2, v_2)\}$ can be moved to any other pair $\{(e'_1, v'_1), (e'_2, v'_2)\}$ the same distance apart, such that $v'_i$ has the same homogeneity degree of $v_i$, for $i=1,2$.
\end{proof}

So, on a semi-homogeneous not homogeneous tree, $\Aut T$ is transitive on $E$ but has two orbits on $V$ (the two subsets $V_\pm$ of constant homogeneity) and on $F$ (the two subsets of all flags whose vertices belong to $V_+$, $V_-$, respectively). We refer to this fact by saying that, on a semi-homogeneous not homogeneous tree, $\Aut T$ preserves the parity of vertices. Instead, if the tree is homogeneous, then $\Aut T$ contains elements that flip an edge, hence swap the parity of vertices. The subgroup of $\Aut T$ of a homogeneous tree that preserves the parity of vertices will be denoted by $\Aut^p T$: it has index two, because for every $\lambda_1, \lambda_2\in\Aut^p T$ the element $\lambda_2^{-1}\lambda_1$ preserves the parity.
The non-trivial coset has a representative that flips an edge, thereby exchanging the parity of its endpoints.

If $T$ is semi-homogeneous but not homogeneous, the isotropy subgroup $K_v\subset \Aut T$ of any vertex $v$ acts transitively upon the set of edges that join at $v$, hence, by iteration, acts transitively on $\Omega$. In particular, $K_v$ contains automorphisms that interchange any pair of boundary arcs $\Omega(x)$, $\Omega(y)$ where $x,y\in V$ are neighbors of $v$ (or, more generally, equidistant from $v$). Instead, The action of automorphisms upon edges is described similarly. If $T$ is semi-homogeneous but not homogeneous, then every automorphism that fixes an edge $e$ must preserve the subsets of adjoining edges at each side, because it must preserve the number of edges that join at a vertex.
Equivalently, the isotropy sugbroup $K_e$  of an edge has two orbits on $\Omega$, because it cannot interchange the two subtrees obtaining by removing $e$, and more generally any two boundary arcs that lie at opposite sides of $e$ (as obvious from Figure \ref{Fig:semihomogeneous_tree_T23_centered_at_n_edge}
).

\begin{figure}[h!]
\silenceable{\begin{tikzpicture}[grow cyclic,level/.style={
 sibling  angle=90  /(2-1/4)^(#1-1),
 level distance=15mm/(1+1/3)^(#1-1)}]
\begin{scope}
\tikzstyle{every node}=[circle,fill,inner sep=.7pt]
\path[rotate=135]node(start4){}
 child[level distance=20mm]
 child foreach\x in{,,}{node{}
   child foreach\x in{,} {node{}
     child foreach\x in{,,}{node{}
       child foreach\x in{,} {node{}
       }
     }
   }
 };
\path node at(start4-1){}
 child foreach\x in{,} {node{}
   child foreach\x in{,,}{node{}
     child foreach\x in{,} {node{}
       child foreach\x in{,,}{node{}
       }
     }
   }
 };
\end{scope}
\node at(1.05,0)[label=-90:$e_0$]{};
\end{tikzpicture}}
\caption{In a semi-homogeneous tree, opposite sides of any edge cannot be swapped by automorphisms.}
\label{Fig:semihomogeneous_tree_T23_centered_at_n_edge}
\end{figure}
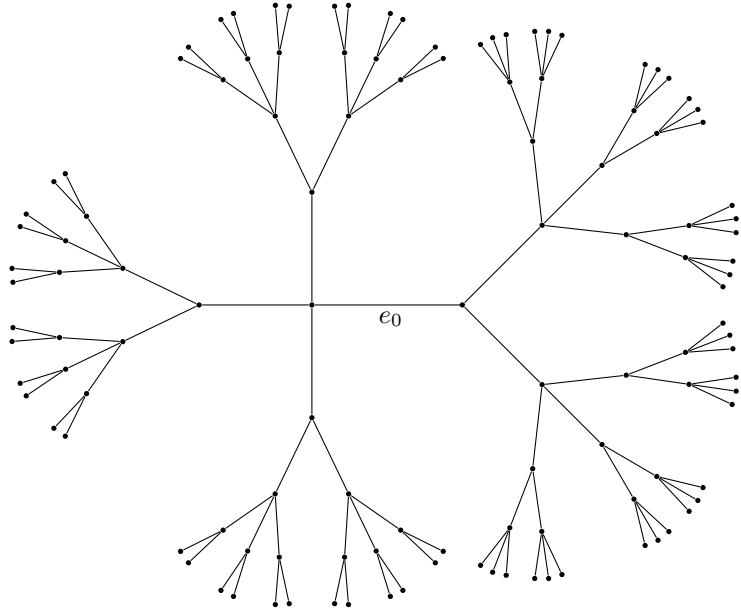

 The stabilizer of a flag in a 
 semi-homogeneous 
 tree has two orbits on $\Omega$, because if an automorphism fixes a flag it cannot reverse its orientation (the same holds in the homogeneous setting).

Since automorphisms of semi-homogeneous non-homogeneous trees must preserve the homogeneity degree of vertices, hence their parity, it follows immediately that the structure group of their horospherical fiber bundles is the subgroup $2\mathZ$ of the general structure group $\mathZ$.

Automorphisms map geodesic paths to geodesic paths, hence their action extends to the boundary $\Omega$. We have seen in \eqref{eq:Aut_acts_equivariantly_on_the_horospherical_index} that
the diagonal action of $\Aut T$ on $V\times V \times \Omega$ acts in an equivariant way on the horospherical indices. Therefore an automorphism $\lambda$ maps horospheres tangent at $\omega$ to horospheres tangent at $\lambda\omega$. So the stabilizer $(\Aut T)_\omega$ of $\omega$ preserves the set of horospheres tangent at $\omega$, but does not fix each such horosphere (because $ h(v, v_0,  \omega)=h( \lambda v,  \lambda v_0, \lambda  \omega)$ and in general the latter horospherical index is different from $ h( \lambda v,  v_0, \lambda  \omega)$). Let us give a geometric description of this action in reference to the pictures of the vertex-horospheres tangent at  $\omega$ given in Figure \ref{Fig:family_tree}.

We have seen in  
Subsection \ref{SubS:Orbits_of_Aut_T_on_fiber_bundles}
 that the stabilizer of a boundary point $\omega$, $(\Aut T)_\omega$,  is  transitive on $V$. 
For the same reason, $(\Aut T)_\omega$ is transitive on $E$.

Let us see how the action of $\Aut T$ on the horospherical fiber bundles, presented in Corollary \ref{cor:transitivity_on_horospheres} for homogeneous trees, changes in the semi-homogeneous case.

\begin{corollary}
\label{cor:transitivity_on_horospheres-semi-homogeneous}
Let $T$ be a semi-homogeneous but not homogeneous tree.
\begin{enumerate}
\item[$(i)$]
The action on the fiber at $\omega$ of $(\Aut T)_\omega$ has two orbits both on $\HorV$ and $\HorE$: it preserves the parity of the horospherical index. The fiber $\pi^{-1}\omega$ is isomorphic to $\mathZ$ and $(\Aut T)_\omega$ acts on the fibers as $2\mathZ$.

\item[$(ii)$]
$(\Aut T)_\omega$ has four orbits on $\HorF$.
\item[$(iii)$] 
The actions of $\Aut T$ on $\HorV$ and $\HorE$ have two orbits, given by the parity of the horospherical index. On $\HorF$, $\Aut T$ has 
four orbits. 
\end{enumerate}
\end{corollary} 

\begin{proof}
On a semi-homogeneous tree, the action of $(\Aut T)_\omega$ on vertex-horospheres has two orbits because automorphisms preserve homogeneity. The action on edge-horospheres splits into two orbits because an automorphism that fixes $\omega$ can only move edges along a geodesic ray ending at $\omega$ by an even number of steps, since it must preserve the homogeneities of their endpoints. This proves part $(i)$.
We have already observed in Subsection \ref{SubS:Orbits_of_Aut_T_on_fiber_bundles}
  that $(\Aut T)_\omega$ shifts edges along a geodesic ending at $\omega$ only by an even number of steps, and so 
each of the two  orbits in $\HorF$ of the homogeneous case corresponds to two parity-related orbits.

Part $(iii)$ follows easily from part $(i)$ and the fact that the action of $\Aut T$ is transitive on $\Omega$. Note that the parity of the horospherical index is one to one correspondence with the homogeneity of any vertex of a vertex-horosphere. Instead, in an edge-horosphere $\boldh_n(\omega, e_0)$, the parity of $n$ is in one to one correspondence with the homogeneity of the endpoint on the side of $\omega$ of any edge in the horosphere.
\end{proof}

\section*{Appendix: which fiber bundles are horospherical fiber bundles}
\label{Sec:App_fiber_bundles}

\subsection{Sub-bases for the topology of a fiber bundle}

Let us return to the general set-up of principal trivial fiber bundles $\calH$ introduced in Subsection \ref{SubS:Sections_and_special_sections}. The base is a topological space $\Omega$ and the fiber a locally compact group $Z$, that coincides with the structure group. A global chart is a bijection between $\calH$ and $\Omega\times Z$ and lifts to  $\calH$ the product topology of $\Omega\times Z$. The canonical projection $\pi: \calH \to \Omega$ is continuous. Therefore the inverse images under $\pi$ of open sets in $\Omega$ are open in $\calH$. We shall call \emph{tubes}, or \emph{tubular sets}, these inverse images: that is, given a global chart and an open set $A\subset \Omega$,  the tube $U(A)\subset \calH$ generated by $A$ is
\[
U(A)=\pi^{-1}(A) = \{ (\omega, z): z\in Z, \,\omega\in A\}
\]
The fibers are the subsets $\pi^{-1}(\omega)$: they correspond to $\{\omega\} \times Z$ via a global chart. Then a subset of $\calH$ is a tube if and only if it is a union of fibers.

Now we choose a sub-base $A=\{A_\alpha\}$ for the topology of $\Omega$ and a sub-base $\{Z_\beta\}$ for the topology of $Z$, and we lift $Z_\beta$ to $\calH$ by setting $\widetilde Z_\beta= \Omega\times Z_\beta$ in the  coordinates given by the global chart. Then the \emph{truncated tubes} $U(A_\alpha)\cap \widetilde Z_\beta$ form a sub-base for the topology of $\calH$.

We shall prove that $\calH$ is isomorphic to the horospherical fiber bundle determined by a tree if and only if its structure group $Z$ is isomorphic to $\mathZ$ and there is a collection $\calU$ of tubes that satisfies the following properties:

\begin{enumerate}
\item[$(i)$] $\calH\notin\calU$;
\item[$(ii)$] if $U\in\calU$ then $\overline U=\calH\setminus U\in\calU$;
\item[$(iii)$] if $U,U'\in\calU$ then one between $U$ or $\overline U$ is contained in (equivalently: disjoint from) one between $U'$ or $\overline{U'}$;
\item[$(iv)$] if $U,U'\in\calU$ are not disjoint then $U\cap U'$ is a finite disjoint union of tubes in $\calU$ (possibly, but not necessarily, a single tube);
\item[$(v)$] every fiber is the intersection of a (finite or infinite) decreasing sequence of tubes in $\calU$ (and is possibly a tube itself);
\item[$(vi)$] The intersection of a decreasing sequence of tubes in $\calU$ is non-empty.
\item[$(vii)$] No disjoint union of tubes in $\calU$ is a tube in $\calU$.
 \end{enumerate}

\subsection{Axiomatically  constructing trees starting from boundaries}
By applying the canonical projection $\pi:\calH\to\Omega$ to the family of tubes $\calU$ we obtain a family $\A\subset\Omega$ that we call arcs. The family $\A$ is a sub-base for the relative topology of $\Omega$. We now
show that, by axioms $(i)$ to $(vi)$, the space $\Omega$ is associated to a tree in the sense
presented in Subsection \ref{SubS:Boundary}: when regarded in this sense, $\Omega$ will be called a \emph{tree-boundary}.  
We start by transporting to $\A$ axioms $(i)$ to $(v)$ and illustrating their consequences.

\begin{definition}
A \textit{tree boundary} is a non-empty set $\Omega$ together with a family $\A$ of its subsets (called \textit{arcs}) with the following axioms:
\begin{enumerate}
\item[$(i)$] $\Omega\notin\A$;
\item[$(ii)$] if $A\in\A$ then $\overline A=\Omega\setminus A\in\A$;
\item[$(iii)$] if $A,A'\in\A$ then one between $A$ or $\overline A$ is contained in (equivalently: disjoint from) one between $A'$ or $\overline{A'}$;
\item[$(iv)$] if $A,A'\in\A$ are not disjoint then $A\cap A'$ is a finite disjoint union of arcs (possibly, but not necessarily, a single arc);
\item[$(v)$] every $\omega\in\Omega$ is the intersection of a (finite or infinite) decreasing sequence of arcs (and is possibly an arc itself
 if terminal vertices are allowed in a tree);
 \item[$(vi)$] the intersection of a decreasing sequence of arcs is non-empty.
 \end{enumerate}
\end{definition}
Observe that axiom $(iv)$ excludes from our construction edges of infinite valency, that is, trees not locally finite (and it also excludes $R$-trees, that is, trees with continuous bifurcations along edges: see Proposition \ref{prop:finite_decreasing_sequence_of_arcs} below).

The following property follows from the previous ones:
\begin{corollary}\label {cor:property_vi}
If $A,A'\in\A$ are not disjoint then $A\subseteq A'$ or $A'\subseteq A$ or $A\cup A'=\Omega$.
\end{corollary}

Given a tree boundary $\Omega$, we build a tree $T=(V,E)$ (where $V$ is the set of its vertices and $E$ the set of its edges) in the following way. For every pair of finite unordered partitions $P=(A_1,\dotsc,A_k),P'=(A'_1,\dotsc,A'_{k'})$ of $\Omega$ into arcs, we say that $P$ is \textit{coarser} than $P'$ if every $A'_{j'}$ is contained in some $A_j$; this gives a partial ordering of such partitions. The maximal (coarsest) such partitions are exactly $(A,\overline A)$ for any $A\in\A$ (thus $k=2$); each will be called an \textit{edge}. Besides edges, each sub-maximal partition $(A_1,\dotsc,A_k)$ divides $\Omega$ into $k>2$ arcs and is called a \textit{vertex}. Let us define the valency of a vertex as the number $k$ of arcs in its partition: hence the valency of a vertex is at least three. Our construction starting from the tree  boundary does not allow vertices of valency two; each such ``vertex'' $A, \overline A$ is regarded as an edge, that is the edge $(A, \overline A)$.
(Arcs that are points of $\Omega$
will be identified as terminal vertices, i.e., vertices of valency one.)

A vertex $(A_1,\dotsc,A_k)$ \textit{belongs} exactly to the edges $(A_j, \overline{A_j})$ for $j=1,\dotsc,k$. Two distinct vertices $(A_1,\dotsc,A_k)$ and $(B_1,\dots,B_n)$ are \textit{adjacent} if they belong to the same edge, that is, if and only if, up to relabeling, $\overline{A_1} =B_1$. Then their connecting edge is $(A_1,B_1)$.

\begin{proposition} \label {prop:uniqueness_of_partition} 
Every non-minimal arc $A$
admits a partition into finitely many proper sub-arcs; indeed it has a unique coarsest such partition.
\end{proposition}
\begin{proof}
If $A_0\subset A$ is a proper sub-arc, then by Property~$(iv)$ $A\setminus A_0=A\cap\overline{A_0}=\coprod_{j=1}^m A_j$, a disjoint union. Therefore $A=\coprod_{j=0}^m A_j$.

Let $A=\coprod_{j=0}^m A_j=\coprod_{j=0}^{m'} A'_j$ be two different such partitions. We may assume that $A_0$ and $A'_0$ intersect; by the Corollary, one of the two, say $A_0$, contains the other, $A'_0$, because their union is contained in $A\neq \Omega$. If $A_0=A'_0$ we can apply the same argument starting with $A_1$, and so on. Thus we can assume that $A_0$ contains $A'_0$ properly. By the same token, $A'_j$ is contained in $A_0$ whenever it intersects $A_0$, therefore $A_0$ is the disjoint union of all $A'_j$ that intersect $A_0$. This contradicts the minimality of the second partition.
\end{proof}

We now add another axiom:
\begin{enumerate}
\item[$(vii)$] No disjoint union of infinitely many arcs is  an arc.
\end{enumerate}
\begin{remark} \label {rem:Boundary_is_not_infinite_union_of_arcs}
It follows that $\Omega$ itself is not the union of infinitely many disjoint arcs. Indeed, otherwise, by removing one of these arcs, the union of the others would consist of its complement.
\end{remark}

\begin{proposition}\label {prop:finite_decreasing_sequence_of_arcs}
If $A_1, A_2, \dots$ is a decreasing nested sequence of distinct arcs and $A=\bigcap A_j$ is a non-minimal arc, then this sequence is finite.
\end{proposition}
\begin{proof}
This is an immediate consequence of axioms $(iv)$ and $(vii)$. Indeed, $A_1\setminus A_2$ is a finite union of arcs, hence $A_2$ is an element of a finite partition associated to $A_1$. 
In the same way, $A_2$ splits as a finite union of sub-arcs, one of which is $A_3\supset A$, and so on 
since the nested sequence is strictly decreasing. But then the arc $A_1$  is a disjoint union of arcs. This union must be finite by axiom $(vii)$, hence the sequence is finite.
\end{proof}

\begin{remark}[Adjoining edges]\label {rem:adjoining_edges}
Given two distinct edges $(A,\overline{A})$ and $(B,\overline{B})$, up to relabeling we can assume by axiom $(iii)$ that $A\subset B$. Such edges are \textit{adjacent} if $(A,\overline B)$ can be completed (via a partition of $\overline{A}\cap{B}$) to a sub-maximal partition of $\Omega$; in this case they share the vertex given by such partition, unique by~Proposition~\ref{prop:uniqueness_of_partition}. Therefore two distinct edges $(A,\overline{A})$ and $(B,\overline{B})$ are adjacent if and only if one of the two arcs of the partition associated to the first, say, $A$, is contained in one of the arcs of the partition of the second, say $B$ (and vice-versa for the other pair of arcs, of course), and $B\setminus A$ can be decomposed  (by axiom $(iv)$) as a finite union of arcs $C_1,\dots C_n$ such that $A, C_1,\dots C_n$ is a maximal partition of $B$.
\end{remark}

The graph structure of $T$ (consisting of vertices, edges and the incidence relation between them) is now completely described as follows.

\begin{proposition}
\begin{enumerate}
\item\label{TwoVertices}
Each edge contains exactly two vertices.
\item\label{NoLoops}
There are no loops in the graph, that is, the graph is a tree, that will be denoted by $T$ henceforth.
\item\label{Connected}
The tree is connected.
\end{enumerate}
\end{proposition}

\begin{proof}
Let us consider an edge $(A,\overline A)$. Then, by Proposition \ref{prop:uniqueness_of_partition}, $A\in\A$ (unless it consists of a single point of $\Omega$, which corresponds to a terminal vertex) has a single coarsest proper partition into arcs $A_1,\dotsc,A_k\neq A$ for $k>1$. Then $(A_1,\dotsc,A_k,\overline A)$ is a vertex, and it belongs to the edge. Exchanging the role of $A,\overline A$ one obtains the other vertex belonging to the edge $(A,\overline A)$. Since the vertices belonging to $(A,\overline A)$ must correspond to sub-maximal partitions that include either $A$ or $\overline A$, there are no other vertices contained in the edge. This proves \eqref{TwoVertices}.

Consider a chain of subsequently adjoining  edges $e_j=(A_j,\overline{A_j})$ , $j=1,\dots,n$. If the chain forms a loop, then $e_1=e_n$, that is, either $A_1=A_n$ or $A_1=\overline{A_n}$.
By Remark \ref{rem:adjoining_edges}
we can assume that $A_1\subsetneq A_2 \subsetneq \dots \subsetneq A_n$, therefore $A_1\neq A_n$ and $A_1\neq \overline{A_n}$. Therefore no such chain forms a loop. This proves \eqref{NoLoops}.

Let $e_1=(A_1,\overline{A_1})$ and $e_2=(A_2,\overline{A_2})$ be any two distinct edges. Then $A_1\neq A_2$ and $A_1\neq \overline{A_2}$.
By Corollary \ref{cor:property_vi}, either $A_1\subsetneq A_2$ or $A_2 \subsetneq \overline {A_1}$, that is, . By an appropriate labeling we may assume the latter. Let $A_1 = \coprod_{j=1}^m B_j$ be the unique coarsest partition of Proposition \ref {prop:uniqueness_of_partition}. By minimality, $A_2$ is contained in one of the sets $B_j$, say $B_1$. We know from Remark \ref{rem:adjoining_edges} that $e_1=(A_1,\overline{A_1})$ and $e_2=(B_1,\overline{B_1})$ are adjacent edges, and, among the edges $(B_j,\overline{B_j})$ adjacent to $e_1$, $e_2$ is the one whose first arc belongs to the partition of $e_1$ and also contains $A_n$ (we think of $e_2$ as the neighbor of $e_1$ one step closer to $e_n$). By continuing in this way we build a chain of edges $e_1, e_2, \dots$ such that $\{A_j\}$ is  a strictly decreasing sequence of arcs containing the arc $A_n$. But then this sequence is finite by Proposition \ref {prop:finite_decreasing_sequence_of_arcs}, hence  $e_1$ and  $e_2$ are connected by a finite path. This proves part
\eqref{Connected}.
\end{proof}

We equip $\Omega$ with the topology whose base is the family of all arcs.
Choose any edge $e=(A,\overline A)$, and consider the unique coarsest partitions associated to $A$ and to $\overline A$. The union of the elements of these two partitions is a finite family of arcs, by Proposition \ref{prop:uniqueness_of_partition} . We call the arcs in this family \emph{arcs of first generation} with respect to $e$. Iterating this process we define arcs of any generation $n\mathN$, and every generation leads to a finite family of arcs, each of which is contained in one arc of the previous generation.

Now consider an infinite family $\Theta$ of points in $\Omega$. By finiteness, for every $n$ at least one arc $A_n$ of generation $n$ must contain an infinite sub-family $\Theta_n$ of $\Theta$. Therefore there is a sub-arc $A_{n+1}\subset A_n$ of the next generation that contains an infinite sub-family of $\Theta_n$. This process of iterative bisections, together with axiom $(vi)$, shows that there is a subfamily of $\Theta$ that converges in the topology of $\Omega$, hence $\Omega$ is compact.

Given an edge $e=(A,\overline A)$, the set $e'=(B,\overline B)$ such that $B\subset A$ is called an \emph{edge-sector}$S_A$ subtended by $e$ (there is another such edge-sector, consisting of those edges that have an associated arc contained in $\overline B$). Then it follows as in Subsection \ref{SubS:Boundary}
that the family of closed edge-sectors $S_A\cup A \subset T\cup\Omega$ generates a topology that makes $T\cup\Omega$ compact: in other words, $\Omega=\partial T$ in this topology. Moreover, by axiom $(v)$, $\Omega$ is a Hausdorff space, and by axiom $(ii)$ it has a base of open and closed sets, hence it is totally disconnected.

Note that the tree $T$ is homogeneous if and only if all sub-minimal partitions of $\Omega$ into arcs have the same cardinality. Also note that, given a tree, arcs correspond to the boundary arcs subtended by edges in $T$, defined 
in Remark \ref{rem:boundary_arcs}. We observe that, if axiom $(vi)$ is not included, then $\Omega$ is not compact in the topology generated by the arcs. Indeed, let us consider the instance where $\Omega$ is the boundary  of a tree and we restrict attention to the subset $\Omega'$ obtained by removing a single point $\omega_0$ from $\Omega$. Define arcs on $\Omega'$ by removing $\omega_0$ from the arcs of $\Omega$. Then axioms $(i)$-$(v)$ and $(vii)$ are still valid, but $\Omega'$ is not compact and it is not the boundary of a tree.

\subsection{Horospherical bundle}
Now that we have built a homogeneous tree, we can construct a horospherical bundle starting from it following the approach of Subsection \ref{SubS:Sections_and_special_sections}. Let us briefly rephrase that approach in terms of tree-boundaries.

We have seen that an edge is defined as a minimal partition: it corresponds to two complementary arcs in $\Omega$. A vertex is a sub-minimal partition: it corresponds to as many boundary arcs as the valency of the vertex. If the tree is disconnected by removing a vertex, it splits into as many connected components as the valency, and each of these arcs is the boundary of one of the components. A general partition (necessarily finite by Remark \ref{rem:Boundary_is_not_infinite_union_of_arcs}) corresponds to a finite \emph{contour} of vertices (or edges), that is, a finite set that disconnects the tree into a finite family of subtrees, one for each of the arcs in the partition, plus one finite subtree having as boundary the contour; each of these arcs is the boundary of the corresponding infinite subtree.

We now build principal trivial fiber bundles $\HorV$ and $\HorE$ on $T$. For simplicity, let us unify the definition of horospheres as consisting of vertices or edges, and use just one symbol $\calH$. The base of $\calH$ is the tree-boundary $\Omega$ and the structure group is $\mathZ$, isomorphic to each fiber. 
The set of sections, introduced in Definition \ref{def:sections}, can be regarded as the set of locally constant functions on $\Omega$ to $\mathZ$. A section $\Sigma$ can be realized by assigning a partition of $\Omega$ into arcs $A_1,\dots,A_n$, and for each $j=1,\dots, n$ a choice of a vertex $w_j$ such that, for $\omega\in A_j$, $\Sigma(\omega)$ is the horosphere tangent at $\omega$ (i.e., associated to the fiber $\omega$) and containing $w_j$. Note that $w_j$ is not necessarily the vertex that subtends the boundary arc $A_j$. The same realization holds in the setup of edges.

The difference operation in $\mathZ$ induces an algebraic difference between any two sections $\Sigma_0$ and $\Sigma_1$. Now $\Sigma_1 - \Sigma_0$ is a locally constant function on $\mathZ$ whose level sets are unions of arcs $A_1,\dots,A_k$.
Therefore each difference of two sections is subordinated to a minimal partition $A_1,\dots,A_k$ such that the associated difference function is constant on each $A_j$. 

We now define a global chart on $\calH$, namely, a bijection of $\calH$ to $\Omega\times \mathZ$. Let us choose $\Sigma_0$ as a reference section, and assign to each section $\Sigma_1$ the locally constant function given by the difference function  $\Sigma_1-\Sigma_0$. Then $\Sigma_0$ is associated to the null $k$-tuple for every partition $A_1,\dots,A_k$, and if $\Sigma_1 - \Sigma_0$ is subordinated to the partition $A_1,\dots,A_k$ then
 $\Sigma_1$ is associated to a unique $k$-tuple $(n_1,\dots,n_k)$. This yields a global chart on $\calH$, that depends on the choice of reference section. 
More precisely, the coordinates of $h\in\calH$ are $\bigl(\pi(h), h-\Sigma_0(\pi(h))\bigr)$, where $\pi(h)$ is the canonical projection of $h$ to the base $\Omega$.

The geometric interpretation is as follows. A special section in $\HorV$ corresponds to a sub-minimal partition $\{A_1,\dots,A_k\}$, that is, is associated to a vertex $v_0$: choose it as $\Sigma_0$. For each $\omega$, $\Sigma_0(\omega)$ is the horosphere through $v_0$ tangent at $\omega$.
As before, the partition associated to $\Sigma_1-\Sigma_0$ corresponds to a set of vertices $v_1,\dots,v_k$ (respectively, of edges $e_1,\dots,e_k$).
If $n_j=0$ 
we say that, for $\omega\in A_j$,  $\Sigma_1(\omega)$ is the vertex-horosphere \emph{tangent at $\omega$ and containing $v_0$}. Instead, if $n_j\neq 0$, we say that the horosphere $\Sigma_1(\omega)$, for $\omega\in A_j$, has signed offset $n_j$ with respect to $v_0$. The number $n_j$ is called the horospherical index of $\Sigma_1(\omega)$ for $\omega\in A_j$ (with respect to the special section determined by $v_0$). 
If a horosphere $h$ is tangent at $\omega\in A_j$ and contains a vertex $v_j$ belonging to the geodesic ray from $v_0$ to $\omega$, then the index is the distance from $v_0$ to $v_j$; otherwise, $v_0$ belongs to the ray from some vertex $v_j\in h$ to $\omega$ and the horospherical index is $-\dist(v_j,v_0)$.

Similar considerations hold for edge-horospheres, except that we use minimal rather than sub-minimal partitions.

\part{Harmonic analysis on the horospherical fiber bundles of homogeneous trees} 

\section[Decomposition into isotypical components on $\HorV$]{Decomposition of function spaces on $\HorV$ into isotypical components under the parallel shift group}

From now on, when we mention the automorphism group $\Aut\HorV$, we restrict attention to its subgroup $\Aut T \times A$.
Consider any vector space $\spaceU$ of functions on $\HorV$ (respectively, any vector space $\spaceW$ of functions on $\HorE$) invariant under the automorphism group of $\HorV$. For any $\lambda\in\Aut \HorV$ let us write the action as a linear representation: $\pi(\lambda)f=\lambda\circ f$, where $\lambda\circ f(\boldh)=f(\lambda^{-1}\boldh)$. Of course
the reason for the inverse is to make sure that the action of $\Aut T$ is homomorphic, that is 
$\pi(\lambda_1)\,\pi(\lambda_2)f=\pi(\lambda_1\lambda_2)f$.

\section{The Poisson kernel as a homogeneous function on fibers of $\HorV$}\label{Sec:Poisson_kernel}
In order to keep notation easily understandable, and have at our disposal a suitable set of coordinates on the fiber bundle, we fix  a special section $\Sigma_{v_0}$ (see Definition \ref{def:special_sections}). Therefore every element of $\HorV$ is parameterized by coordinates $(\omega, n)$, where $\omega$ is its fiber (that is, as usual, its tangency point at the boundary) and $n$ is its level over the base section $\Sigma_{v_0}$ (that is, the signed distance along the fiber from its parallel horosphere that belongs to $\Sigma_{v_0}$). We shall write $n=n(\boldh,v_0)$. Then  $n(\boldh,v_0)$ is the same as the horospherical index $h(v,v_0,\omega)$ of Definition \ref{def:horospherical_index} for $\boldh $ tangent at $\omega\in\Omega$ and $v\in \boldh$.

For simplicity, we shall present our results in harmonic analysis mostly on $\HorV$ and not on $\HorE$. Of course  we have proved in Corollary \ref{prop:HorV=HorE_but_pairs_of_special_sections_do_not_correspond}
 that $\HorV \approx
  \HorE$, and so  on both  the full automorphism group $\Aut T$ acts transitively, and they have the same structure groups $\mathZ$ on each fiber and the same parallel shift group $A\approx\mathZ$ introduced in Subsection \ref{SubS:Cartan}. 
  However, in thei Chapter we develop an approach to representation theory of the automorphism group $\Aut T$ based on Poisson transforms, and these transforms depend not only on the horospherical fiber bundles but also on their special sections. But we have seen in   Proposition \ref{prop:HorV=HorE_but_pairs_of_special_sections_do_not_correspond} that, for any fixed special sections $\Sigma_v$ of  $\HorV$ and $\Sigma_e$ of  $\HorE$,  there is an isomorphism $\HorV \approx \HorE$ that maps $\Sigma_v$ to $\Sigma_e$. Therefore we  obtain equivalent representations by basing the same construction on $\HorE$.

For every character $\sigma \in \widehat{A}\approx\mathT$, let $\spaceU_\sigma\subset \spaceU$ be the $\sigma-$isotypical subspace, that is the subspace on which $A$ acts via the character $\sigma$: in other words, for every vector space $\spaceU$ of functions on $\HorV$,
\begin{equation}\label{eq:isotypical_subspaces_under_A}
\spaceU_\sigma=\{f\in\spaceU: a\circ f=\sigma(a^{-1})f \text{ for every } a\in A\},
\end{equation} 
where  $a^{-1}$ is the inverse of $a\in A$, that from now on we shall denote by $-a$. Again, the reason for the inverse is to have a homomorphic action of $A$, that is $(a_1a_2)f=a_1(a_2f)$, but here this way of writing is purely formal, since $A$ is commutative. 

If $\spaceW$ is a vector space of functions on $\HorE$ we define the isotypical component $\spaceW_\sigma$ analogously.

\begin{definition} [The Poisson transform]\label{def:Poisson_transform}
Let $\sigma$ be a character of the parallel shift group $A\subset \Aut\HorV$, and $\Sigma_{v_0}$ a special section in $\HorV$. 
Consider the shift that moves the horosphere $\boldh\in\Sigma_{v_0}$ to the parallel horosphere $\widetilde\boldh\in\Sigma_{v}$   (here, as usual, \emph{parallel} means in the same fiber, identified by the tangency point). Then this shift is $n(\widetilde\boldh,v_0) = h(v,v_0,\omega) = -h(v_0,v,\omega) = -n(\boldh,v)$.

The (vertex-)Poisson transform of weight $\sigma$ is the map from functions $f
\in\spaceU_\sigma$ to functions on $V$ defined by
\begin{equation}\label{eq:Poisson_transform}
\begin{split}
\Poisstr f(v) = \Poisstr_\sigma f(v) 
(v) &=  \int_{\Sigma_{v_0}}f
(\boldh)\,\sigma(-n(\boldh,v))\,d\xi_{v_0}(\boldh)\\[.2cm]
&=\int_{\Sigma_{v_0}} f
(\boldh)\,\sigma(-n(\boldh,v))\,d\nu_{v_0}(\pi (\boldh)).
\end{split}
\end{equation}
Here $\xi$ is the  measure on $\HorV$ invariant under $\Aut T$ and normalized on each special section that was introduced in Subsection \ref{SubS:invariant_measure_on_Hor},  
$\nu_{v_0}$ is the boundary measure invariant under the isotropy subgroup $(\Aut V)_{v_0}$, defined in \eqref{eq:homogeneous_invariant_vertex-measure}, and $\pi$ is the canonical projection of $\HorV$ onto its base $\Omega\approx \Sigma_{v_0}$.
\\
Since $\xi$ is invariant under $\Aut T$, it is easy to see that the right hand side is independent of $v_0$: for clarity, let us show this in detail. We know that the section $\Sigma_{v_0}$ is in one to one correspondence with the boundary $\Omega$ (as usual, this correspondence associates to each horosphere through $v_0$ its tangency point in $\Omega$). So we write $
\boldh= \boldh(v_0,\omega;0)$ for the horosphere in $\Sigma_{v_0}$, and $
\boldh= \boldh(v_0,\omega;n)$ for the horospheres in the parallel section of fiber shift $n=n(\boldh, v)$ along the fiber at $\omega$.
The integral over $\Sigma_{v_0}$ can be written as an integral over $\Omega$, and 
\begin{equation}\label{eq:Poisson_transform_expressed_on_Omega}
\Poisstr f
(v) = \int_{\Omega} f
(\boldh(v_0,\omega;0))\,\sigma(h(v,v_0,\omega))\,d\nu_{v_0}(\omega).
\end{equation}
Now, for $v_0'\neq v_0$,  $f
(\boldh(v_0,\omega;0)) = f
(\boldh(v_0',\omega;0))\,\sigma(-h(v_0',v_0,\omega))$ as $f
\in\spaceU_\sigma$, and, 
by the cocycle relation of Remark \ref{rem:cycle_relation},
$h(v,v_0,\omega) = h(v,v_0',\omega)+h(v_0',v_0,\omega)$. 
Since $\sigma$ is multiplicative, $f
(\boldh(v_0,\omega;0))\,\sigma(h(v,v_0,\omega)) = f
(\boldh(v_0',\omega;0))\,\sigma(h(v,v_0',\omega))$ for every $v\in T$. 
 Therefore the right hand side of \eqref{eq:Poisson_transform} does not depend on the choice of $v_0$.
\\
Then, by choosing $v=v_0$ in \eqref{eq:Poisson_transform}, we find  $n(\boldh,v)= n(\boldh,v_0)=0$ and $\sigma(-n(\boldh,v))=1$, and
the equality  becomes 
\begin{equation}\label{eq:Poisson_transform-elegant}
\Poisstr f
(v) = \int_{\Sigma_{v}} f
(\boldh)\,d\nu_{v}(\pi(\boldh)) = \int_{\Sigma_{v}} f
(\boldh)\,d\xi(\boldh),
\end{equation}

Summarizing, every multiplicative function on the fibers $\omega$ of $\HorV$ (each isomorphic to $A\approx \mathZ$), that is, every character of $A$ gives rise to the Poisson transform associated to that character. Equivalently, the complex powers of the Poisson kernels are identified with the functions on $\HorV$ induced by the characters of $A$. More precisely,
let us write the characters of $A\approx \mathZ$ as  $\sigma(n)=q^{nz}$ for $z\in\mathC$, and rewrite $\boldh\in\Sigma_{v_0}$ as $\boldh(v_0,\omega)$ (this is the horosphere through $v_0$ tangent at $\omega$). We write the Poisson kernel $\Poiss(v,v_0,\omega) $ of
 \eqref{eq:Poisson_kernel_as_Radon-Nikodym_derivative} in the following way: if $\boldh=\boldh(v,\omega)$ is the horosphere through $v$ tangent at $\omega$ and $v_0$ is a  fixed vertex, then the Poisson kernel normalized at $v_0$ is the function defined on $\HorV$, denoted by $\Poiss(\boldh,v_0)$, that, regarded as a function on $V\times V\times \Omega$  has value $q^{h(v,v_0,\omega)}$ at the vertex $v\in\boldh$. We shall call the function $\Poiss^z(\boldh,v_0)$ given by $q^{z h(v,v_0,\omega)}$ the \emph{generalized Poisson kernel}: then the usual Poisson kernel corresponds to the character $\sigma(n)=q^n$, that generates the group $\widehat A$.

 Then, for $f\in\spaceU$ ($\sigma\in\widehat A$),   formula \eqref{eq:Poisson_transform_expressed_on_Omega} becomes 
\begin{equation}\label{eq:Poisson_transform,explicit}
\begin{split}
\Poisstr f
(v) &= \int_{\Sigma_{v_0}} f
(\boldh) \Poiss^{-z} (\boldh,v_0)\,d\xi (\boldh) \\[.2cm]
&= \int_{\Sigma_{v_0}}f
(\boldh(v_0,\omega)) \Poiss^{-z} (\boldh,v_0) \,d\nu_{v_0}(\pi (\boldh)) \\[.2cm] 
&= 
\int_{\Omega} \widetilde f
(\omega) q^{-z h(v,v_0,\omega)}d\nu_{v_0}(\omega),
\end{split}
\end{equation}
where, in the last term, we denoted by $\widetilde f
$ the function on $\Omega$ associated to $f
$ via the correspondence between $\Omega$ and the special section $\Sigma_{v_0}$, for any fixed $v_0\in T$.

The edge-Poisson transform $\Poisstr_E$ is defined analogously for functions on edges, by making use of the edge-horospherical index.
\end{definition}

\section[Image of the Poisson transform: eigenvectors of Laplacian on $V$]{The image of the Poisson transform: eigenvectors of the Laplace operator on $V$}

We are going to express the image of the Poisson transform on $\spaceU_\sigma$ as a space of eigenfunctions on $V$ with respect to a suitable Laplace operator on functions on $V$.

\begin{definition}[Convolutions and Laplace operators]
\label{def:convolutions&Laplacians}
Let $T=T_q$ be the homogeneous tree of homogeneity $q$. Denote by $K_v$ and $J_e$ the stability subgroups in the group of automorphisms $\Aut T$ at $v\in V$, $e\in E$, respectively. These are compact subgroups of $\Aut T$. The quotient $\Aut(T)/K_v$ is discrete, in bijection with $V$\!, and the quotient $\Aut(T)/J_e$ is also discrete, in bijection with $E$. These bijections allow to lift summable functions on the discrete spaces $V$, $E$ to summable functions on $\Aut T$. Hence, the convolution product on $\Aut T$ gives rise to convolution products both on $\ell^1(V)$ and $\ell^1(E)$ (and in particular on finitely supported functions therein: we shall often restrict attention to these subspaces). The convolution product, however, depends on the choice of reference vertex and edge. When we need to compute convolutions explicitly, we shall choose a reference vertex $v_0$ and a reference edge $e_0$.

Let $\chi_1$ be the characteristic function of the vertex-circle $C(v_0,1)$ consisting of all vertices that are neighbors of $v_0$. The \emph{vertex-Laplacian} on $T$ based at $v$ is the convolution operator defined by the function $\mu_1=\frac 1{q+1}\chi_1$, that is, the averaging operator on neighbors. (We remark that, in order to be consistent with the corresponding terminology on Euclidean spaces, we should define the Laplace operator as $\mu_1-\delta_{v_0}$, where $\delta_v$ is the Dirac delta at $v_0$, that is, the characteristic function of $\{v_0\}$.) The \textit{edge-Laplacian} based at an edge $e$ is defined similarly, as the average over adjoining edges: that is, as the convolution operator on functions on edges by the characteristic function $\eta_1$ of the edges that are neighbors of $e_0$. Part $(i)$ of the following simple result was observed in \cite{Figa-Talamanca&Picardello-JFA}.
\end{definition}

\begin{proposition}\label{prop:image_of_boundary_functions_under_Poisson_transform} Let  $\sigma$ be a character of $A\approx\mathZ$, $\spaceU$ a space of functions on $\HorV$ and $\spaceW$ on $\HorE$.
\begin{enumerate}
\item[$(i)$]
For every $f\in\spaceU_\sigma$, $\Poisstr f$ is an eigenfunction of $\mu_1$. If $\sigma$ is given by $\sigma(n)=q^{zn}$, then
\[
\mu_1 \Poisstr f = \gamma^V(z) \Poisstr f \qquad \forall f\in \spaceU_\sigma
\]
where
\[
\gamma^V(z)=\frac {q^z + q^{1-z}}{q+1}\;.
\]
\item[$(ii)$]
For every $g\in\spaceW_\sigma$, $\Poisstr_E g$ is an eigenfunction of $\eta_1$. If $\sigma$ is the character of $A$ given by $\sigma(n)=q^{zn}$, then
\[
\eta_1 \Poisstr_E g = \gamma^E(z) \Poisstr g \qquad \forall g\in\spaceW_\sigma
\]
where 
\[
\gamma^E(z)=\frac 1{2q}\;q^z + \frac {q-1}{2q}\;+ \frac12 \; q^{-z}  = \frac  {q^z + (q-1) + q^{1-z}}{2q}\;.
\]
\end{enumerate}
\end{proposition}
\begin{proof}
Part $(i)$ is a simple computation that follows immediately from the explicit expression \eqref{eq:Poisson_transform,explicit} of the Poisson transform, and formula
\eqref{eq:Poisson_kernel_as_Radon-Nikodym_derivative} for the Poisson kernel. Indeed,
\[
\mu_1 \Poisstr f(v) = \frac 1{q+1} \int_{\Omega} f(\boldh(v_0,\omega)) \sum_{w\sim v} q^{-z h(w,v_0,\omega)}d\nu_{v_0}(\omega)
,
\]
and, as observed in Remark \ref{rem:horospherical number_of_a_neighbor}, $h(w,v_0,\omega)=h(v,v_0,\omega)+1$ if $w$ is the neighbor of $v$ at the side of $\omega$, whereas $h(w,v_0,\omega)=h(v,v_0,\omega)-1$
is one of the $q$  backward neighbors, that is the neighbors at the opposite side.

Part $(ii)$ is proved similarly. Let $\omega$ be represented by $[v_0, v_1, \dots, v_j,\dots)$ and consider an edge $e=[v_-,v_+]$ where $v_+$ is at the side of $\omega$ with respect to $v_-$ (that is, $v_-=v_{j-1}$ and $v_+=v_{j}$ for some $j$). Then  $e$ has only one adjacent edge $e_+$ in the direction of $\omega$, namely $[v_j, v_{j+1}]$, and $h(e,e_0,\omega)=h(e,e_0,\omega)+1$. Moreover, the other edges adjacent to $e$ split into two sets. There are $q-1$ edges adjacent to $e$ whose beginning vertex is $v_+$ and the end vertex is different from $v_{j+1}$: their horospherical number is $h(e,e_0,\omega)$. Finally, there are $q$ backward edges adjacent to $e$ whose beginning vertex is $v_-$: their horospherical number is $h(e,e_0,\omega)-1$. This yields the three terms at the numerator of the expression for $\gamma^E(z)$ in  part $(ii)$.
\end{proof}

Proposition \ref{prop:image_of_boundary_functions_under_Poisson_transform} shows that, up to normalization, the spectrum (that is, the eigenvalue map) of the edge-Laplacian is obtained from the corresponding spectrum of the vertex-Laplacian via a shift along the real axis. Why this is so is better understood from the following result.
\begin{proposition}\label{prop:intertwining_of_eta_and_mu_up_to_the_identity}
Let $f$ be a function on $V$ and $\Theta$ be the operator from functions on $V$ to functions on $E$ given by $\Theta f(e) =(f (v_-)+f(v_+))/2$ if $e=[v_-,v_+]$. Then 
\[
\eta_1 \Theta f      =   \frac{q+1}{2q} \mu_1 f +  \frac{q-1}{2q}  f.
\]
\end{proposition}
\begin{proof}
Denote by $e_j^\pm=[v_\pm , v_j^\pm]$ the edges that join $e$ at $v_\pm$, respectively. Then the statement follows easily from the identities
\[
\eta_1 \Theta f  (e) = \frac 1{2q}\sum_{j=1}^q \left( \Theta f(e_j^+) + \Theta f(e_j^-) \right) = \frac 1{4q} \sum_{j=1}^q \left( f(v_j^+) + f(v_j^-)\right)  + q (f(v_+) + f(v_-)),
\]
and
$ \Theta \mu_1 f  (v_+) =  \frac12 \left( f(v_-) + \sum_{j=1}^q f(v_j^+)\right)$, and similarly for $ \Theta \mu_1 f  (v_-)$.
\end{proof}

To which space of objects on the boundary (that is, to which objects on any special section of the horospherical fiber bundle) can we extend the Poisson transform? The following definition answers this question.
\begin{definition}[Distributions on $\Omega$ and on the horospherical fiber bundles]\label{def:distributions_on_the_boundary} A function $\tau$ on $\Omega\approx\Sigma_{v_0}$ is locally constant if, for every $\omega\in\Omega$, there is an open neighborhood of $\omega$ on which $\tau$ is constant. Of course we can choose this open neighborhood as one of the basis open sets given by the boundary arcs $\Omega(v,v_0)$.  Locally constant functions on $\Omega$ are called \emph{test functions}. Since $\Omega$ is compact, it follows that the boundary arcs $\Omega(v,v_0)$ where a test function $\tau$ is constant can be chosen so that $|v|$ is bounded: this maximum number is the \emph{order} of the test function.  In particular, every test function is integrable on $\Omega$ (with respect to any measure $\nu_v$), with finite integral. 
 \\
 The space of distributions $\zeta$ on $\Omega$ is the dual space of test functions. It is clear that distributions are nothing else than finitely additive measures on $\Omega$ \cite{Figa-Talamanca&Picardello, Mantero&Zappa}.  In particular, every test function $\tau$ is a distribution: it acts as a continuous linear functional on test functions $\xi$ by integration (the duality is given by $\langle \tau,\xi\rangle = \int_\Omega \overline {\tau}\,\xi\,d\nu$).
 
 In the same way we define test functions on the fiber bundles $\HorV$ and $\HorE$ (with respect to the product topology of $\mathZ\times\Omega$). Distributions on these fiber bundles are defined as dual spaces of test functions. We shall be interested in test functions and distributions on $\HorV$ and $\HorE$ that are $\sigma-$isotypic in the sense of \eqref{eq:isotypical_subspaces_under_A}.
  
 Observe that, by Definitions \ref{def:horospherical_index} and \ref{eq:Poisson_kernel_as_Radon-Nikodym_derivative}, the Poisson kernel $\Poiss(v,v_0,\omega)$ is a test function on $\Omega$ (of order $|v|=\dist(v,v_0)$), and similarly $K(\boldh,v_0)$ is a test function on $\HorV$,  Note that, for every fixed reference vertex $v_0$ (or edge),  every distribution $\tau$ on $\Omega\approx \Sigma_{v_0}$ extends uniquely to a  $\sigma-$isotypic distribution on $\HorV$ (or $\HorE$). By \eqref{eq:Poisson_transform,explicit}, the Poisson kernel $\Poiss(\boldh(v,v_0,\omega))$ is a test function of the variable $\omega$. Hence  the Poisson transform \eqref{eq:Poisson_transform,explicit} extends to distributions. We shall make  its dependance on $\sigma$ explicit by writing it as $\Poisstr(\sigma, \widetilde f
 )$ instead of $\Poisstr(\widetilde f
 )$.

\end{definition}

The following result was proved in several references. Its first proof is in \cite{Mantero&Zappa} (see also \cite{Figa-Talamanca&Picardello}) for a homogeneous tree (identified with a free group or a free power of the group $\mathZ_2$) and the isotropic Laplace operator $\mu_1$ (exactly the present environment). Another proof for this homogeneous isotropic setup, and more generally for buildings, is in \cite{Kato_1, Kato_2}.  In the case $z=1$ (that is, for the usual Poisson transform), a greatly more general statement, for not necessarily homogeneous or semi-homogeneous trees and a very large class of Laplacians (here understood as transient nearest neighbor transition operators) was proved in \cite{Picardello&Woess}, via the solution of the Dirichlet problem on a finite tree. The last article gives also an explicit expression of the Poisson transform and its inverse in terms of first hit probabilities of the random walks generated by these Laplacians.
The same result, but without explicit formulas for the inverse Poisson transform,  was reproved for all group-invariant nearest neighbor Laplacians on homogeneous trees (not necessarily isotropic)  in \cite{Figa-Talamanca&Steger}, again by by a process of exhaustion of the homogeneous tree with finite subtrees. Our proof below adapts to the isotropic setting the argument of \cite{Figa-Talamanca&Steger}, and gives a more detailed presentation.

\begin{proposition}[Eigenfunctions of the Laplacian]\label{prop:Poisson_transform_of_distributions}
If $z\neq k\pi i/\ln q$ ($k\in\mathZ$), then
the Poisson transform $\zeta\mapsto \Poisstr(\sigma_z,\zeta)$ is a bijection between the space of distributions on $\Omega$ and the eigenspace of $\mu_1$ with eigenvalue $\gamma^V(z)$. Equivalently, $\zeta\mapsto \Poisstr(\sigma_z,\zeta)$ is a bijection between the space of distributions on $\HorV$  and the eigenspace of $\mu_1$ with eigenvalue $\gamma^V(z)$.

Similarly, for the same values of $z$ the edge-Poisson transform $\Poisstr_E(\sigma_z,\eta)$ is a bijection between the space of distributions on $\Omega$ and the eigenspace of $\eta_1$ with eigenvalue $\gamma^E(z)$.
\end{proposition}

\begin{proof}
We write the proof for vertices: the same argument applies to edges.
We first establish the bijection of the statement for finite trees. Denote by $S$ a finite connected subtree $S$, and, without loss of generality, assume that $S$ contains the reference vertex $v_0$. Eigenfunctions of the Laplace operator $\mu_1$ on $S$ are the functions defined on $S$ that are eigenfunctions of $\mu_1$ at all its interior vertices (a vertex of $S$ is interior if all its neighbors are in $S$). The boundary $\partial S$ consists of all vertices of $S$ that are not interior. The (finite-dimensional) space of eigenfunctions of $\mu_1$
on the finite subtree $S$ is denoted by $H_\gamma(S)$. For simplicity, let us write $\gamma(z)=\gamma^V(z)$.
We have already remarked that, by
\eqref{eq:Poisson_kernel}, the generalized Poisson kernel is given by
$K^z(v,v_0,\omega)=q^{z h(v_1,v_0,\omega)}$. Therefore it follows from
\eqref{eq:vertex_horospherical_index}
that
the restriction $K^z(v,v_0,\omega)\left|_S\right.$ is 
\begin{equation}\label{eq:restriction_of_Poisson_kernel_to_finite_subtree}
K^z(v,w_j,\omega)=  q^{z\dist(v_0,w_j)}/q^{z\dist(v,w_j)},
\end{equation}
 where $w_j$ is the last vertex in $S$ in the geodesic ray $\omega$ (that is the path from $v_0$ to the boundary point $\omega$).  Of course
$w_j\in \partial S$. We shall write $\Poiss(v,w_j)$ instead of $\Poiss(v,v_0,\omega)\left|_S\right.$ (remember, however, that  $w_j$ varies with $\omega$: it is the vertex in $\partial S$ that subtends the boundary arc that contains $\omega$).

The first step is the following claim: if $z\neq k\pi i/\ln q$ ($k\in\mathZ$), the functions $\{ K^z(v,w_j):\, w_j\in\partial S\}$ are linearly independent. This is done by induction on the cardinality $n$ of $S$. If $S$ has no interior point then there is nothing to prove, and if the only interior point is $v_0$ and $S$ consists of $v_0$ and (all) its neighbors, then by \eqref{eq:restriction_of_Poisson_kernel_to_finite_subtree} $K^z_z(v,v_j)$ has value $q^z$ on $v_j$ and $q^{-z}$ on $\partial S\setminus\{v_j\}$, and it is clear that these functions are linearly independent if $z\neq k\pi i/\ln q$ (instead, if $z\in k\pi i/\ln q$, these values are all the same and the functions are linearly dependent). This proves the starting case of the induction process.
\\
Now add a vertex $y$ to $S$ to form a larger  connected subtree $S_+$. So $y$ must be neighbor of one of the previous boundary vertices $w_j$ (and of course only one, since these are vertices of a tree). Without loss of generality, let $w_1$ be this vertex in $\partial S$. 
We need to consider two cases.
\\
The first case is when $w_1$, that was a boundary vertex of $S$, becomes an interior vertex of $S_+$. By induction, if $\sum_{j=1}^n c_j K^z(v,w_j)=0$ for all $v\in S$, then $c_j=0$ for all $j=1,\dots,n$. Note that 
$K^z(v,y)=K^z(v,\omega)$ for $\omega\in \Omega(v_0,y)$, and this value is precisely $K^z(v,w_1)$.
Now
if $\partial S_+ = \{ y, w_2, w_3, \dots, w_n\}$, and so, if
$c_0 K^z(v,y) + \sum_{j=2}^n c_j K^z(v,w_j)=0$ for all $v\in S$, then 
$c_0 K^z(v,w_1) + \sum_{j=2}^n c_j K^z(v,w_j)=0$. As we have seen, this implies $c_j=0$ for all $j$, and the induction step is proved in this case.
\\
The other case is when $S_+=S\cup \{y\}$ but $\partial S_+=\{y,w_1,w_2,\dots,w_n\}$. Suppose 
\begin{equation}\label{eq:linear_independence_on_enlarged_finite_tree}
c_0 K^z(v,y) + \sum_{j=1}^n c_j K^z(v,w_j)=0\end{equation}
for all $v\in S_+$. We have already noted that $K^z(v,y)=K^z(v,w_1)$ for every such $v\in S$. Hence, if we now restrict attention to $v\in S$, we have
$(c_1+c_0) K^z(v,w_1)+  \sum_{j=2}^n c_j K^z(v,w_j)=0=0$. Again by the induction hypothesis, this means that $c_1=-c_0$ and $c_2=c_3=\dots = c_n=0$. Therefore, if we now return to consider $v\in S_+$, \eqref{eq:linear_independence_on_enlarged_finite_tree} becomes 
\begin{equation}\label{eq:linear_independence_on_enlarged_finite_tree-2}
c_0 K^z(v,y) + c_1 K^z(v,w_1) = 0
\end{equation}
for all $v\in S_+=S\cup \{y\}$. By \eqref{eq:restriction_of_Poisson_kernel_to_finite_subtree}, $K^z(y,y)=q^{z|y|}=q^{z(w_1|+1)}$, and $K^z(y,w_1)= q^{z|y|}/q^z=q^{z(|w_1|-1)}$.
So, if we choose $v=y$, we see that \eqref{eq:linear_independence_on_enlarged_finite_tree-2} becomes
$ q^{z|w_1|}(c_0 q^z + c_1 q^{-z}) = 0$, and since $c_1=-c_0$ this yields c$c_0=0=c_1$ if $q^z\neq q^{-z}$, that is, if $z\neq k\pi i/\ln q$ ($k\in\mathZ$). This completes the induction step and proves the claim.

This shows that the dimension of the space $H_\gamma$(S) of $\gamma-$eigenfunctions of the Laplacian on a finite connected tree $S$ is at least the cardinality $n= | \partial S|$. We claim that the equality holds:  $\dim H_\gamma(S) = n =| \partial S|$. The proof is again by induction on $n$.
This statement is clear if $S$ $S$ consists of $v_0$ and all its neighbors. 
Again, we enlarge $S$ to a connected finite tree $S_+$ by adding a vertex $y$ at distance 1 from a boundary vertex $w_1$ of $S$ and consider the two cases where $w_1\in\partial S_+$ or $w_1\notin \partial S_+$. In the first case, we extend every $\gamma-$eigenfunction on $S$  to a $\gamma-$eigenfunction on $S_+$ by assigning an arbitrary value at $y$, since $y$ is not a neighbor of any interior vertex of $S$ and the eigenfunction equation concerns the values of $h$ only on neighbors of interior vertices. So, in this case, $| \partial S_+ | = | \partial S + 1|$ and $\dim H_\gamma(S_+)=\dim H_\gamma(S) + 1$.
 In the second case, the eigenfunction $h$ on $S$ extends uniquely to an eigenfunction on $S_+$, since its value at $y$ is determined by the eigenfunction equation at its neighbor $w_1$: so $| \partial S_+ | = | \partial S|$  and $\dim H_\gamma(S_+)=\dim H_\gamma(S)$. In both cases, we have proved the induction step, hence the claim.

Now let $f\in H_\gamma(T)$. We have just shown that, for any vertex $v$ in a finite connected subtree $S$ and for suitable reproducing coefficients $\mu_w \in\mathC$ depending on $S$ (with $w\in\partial S$), 
\begin{equation}\label{eq:Poisson_representation_in_a_finite_ball}
f(v)=\sum_{w\in\partial S} \beta_w K^z(v,w)= \sum_{w\in\partial S} \beta_w q^{z h(v,v_0,w)}
\end{equation}
 (notation as at the beginning of Definition \ref{def:horospherical_index}). For a given $w\in\partial S$, let us set $\mu(\Omega(v_0,w))=\beta_w$. Now start with $S=\{v_0\}$ and iteratively extend $S$ to a larger connected finite tree $\widetilde S$ by choosing a vertex  $w\in\partial S$ and adding all its children, that is the vertices in $D=\{ x\sim w: \ |x|=|w|+1\}$. At each step of this iteration, let us denote by $\widetilde\beta_u$ the reproducing coefficients relative to vertices $u\in \partial\widetilde S$.

 For each interior vertex $v$ of $S$, the coefficients of the expansion \eqref{eq:Poisson_representation_in_a_finite_ball} on $\partial S$ are unique.  It follows from \eqref{eq:restriction_of_Poisson_kernel_to_finite_subtree} that $q^{z h(v,v_0,w)}= q^{z h(v,v_0,x)}$ for every $x\in D$, hence these factors in \eqref{eq:Poisson_representation_in_a_finite_ball} do not change  when we extend $S$ to  $\widetilde S$.
 Then, by comparing the coefficients of the linear expansion in  \eqref{eq:Poisson_representation_in_a_finite_ball} for $S$ with the analogous coefficients for $\widetilde S$, we see 
  that
 $\beta_w = \sum_{x\in D} \widetilde\beta_{x}$, and $\beta_u=\widetilde\beta_u $ for $u\in\partial S$, $u\neq w$. This shows that   
  $\mu$ defines a finitely additive measure on $\Omega$ (also called a distribution, in the terminology of Definition \ref{def:distributions_on_the_boundary}), and  for every $v\in V$ we have $h(v)=\int_\Omega q^{z h(v,v_0,\omega)}\,d\mu(\omega)$; moreover, this distribution (that is, finitely additiva measure) on the boundary is unique.
\end{proof}

\section{Zonal spherical functions on $\HorV$ and on $V$}  \label{Sec:Zonal_vertex-spherical_functions}

\begin{remark}[Radial eigenfunctions of the Laplace operator]\label{rem:radial_eigenfunctions_of_Laplacian}
We have seen in Subsection \ref{SubS:Automorphisms_on_horospheres} that $\Aut T$ acts equivariantly on $\HorV$. In particular, by \eqref{eq:Aut_acts_equivariantly_on_the_horospherical_index}, 
for $\lambda\in\Aut T$,
the parallel shift $n(\lambda \boldh, \lambda v)$  is the same as  $n(\boldh,v)$ introduced in Definition \ref{def:Poisson_transform}
 (see also \eqref{eq:Poisson_transform}), or equivalently the horospherical number $h(v,v_0,\omega)$ coincides with $h(\lambda v, v_0, \lambda \omega)$. This means that
the Poisson kernel is invariant, in the sense that $\Poiss^{zh(\lambda v, v_0,\lambda \omega)}=\Poiss^{zh(v,v_0,\omega)}$.

As a consequence, those eigenfunctions of the Laplace operator that are radial around $v_0$ are Poisson transforms of functions on $\Sigma_{v_0}\approx\Omega$ that are invariant under $(\Aut T)_{v_0}$, that is, constant on $\Omega$, and conversely. In particular, for each eigenvalue $\gamma^V(z)$ there is only one normalized radial eigenfunction of the Laplace operator. Note that, on the contrary, there are many non-radial eigenfunctions: according to Proposition \ref{prop:Poisson_transform_of_distributions}
, all the Poisson transforms $\Poisstr(\sigma_z,\zeta)$ of non-constant distributions $\zeta$ are eigenfunctions with eigenvalue $\gamma^V(z)$. In particular, taking $\zeta$ to be a Dirac measure, we see that all the generalized Poisson kernels $\sigma_z(n(\boldh,v))=
\Poiss^{z} (h(v,v_0,\omega))= q^{z h(v,v_0,\omega)}$ are eigenfunctions with this eigenvalue. Moreover, the eigenvalue map $z\mapsto\gamma^V(z)$ is two-to-one, because the expression of $\gamma^V(z)$ in Proposition \ref{prop:image_of_boundary_functions_under_Poisson_transform} shows that $\gamma^V(z)=\gamma^V(1-z)$. It follows that $K^z$ and $K^{1-z}$ are two different generalized Poisson kernels with the same eigenvalue. As just noticed, this cannot happen for radial eigenfunctions.
\end{remark}

\begin{definition}\label{def:zonal_spherical_functions}
The radial eigenfunctions $\phi^V_z$, normalized so that $\phi^V_z(v_0)=1$, will be called \emph{zonal (vertex-)spherical functions}:
%
\begin{equation}\label{eq:def_of_vertex-spherical}
\phi^V_z(v) 
=  \int_{\Sigma_{v_0}} \sigma_z(n(\boldh,v))\,d\nu_{v_0}(\pi(\boldh)) = 
 \int_{\Omega}  q^{z h(v,v_0,\omega)}d\nu_{v_0}(\omega),
\end{equation}
where, as at the beginning of Subsection \ref{SubS:Fiber bundles},
 $\pi:\HorV\to\Omega$ is the projection of $\HorV\sim \mathZ\times\Omega$ onto its second component.
 The zonal (edge-)spherical functions are defined similarly as integrals on sections $\Sigma_{e_0}$.
\end{definition}

\begin{corollary}[The Weyl shift]\label{cor:the_swap_z<->1-z_leaves_vertex-spherical_functions_invariant}
\[
\phi^V_z(v)=\phi^V_{1-z}(v). 
\]
The same statement holds for zonal edge-spherical functions.
\end{corollary}
\begin{proof}
This follows from the uniqueness of radial normalized eigenfunctions proved at the end of Remark \ref{rem:radial_eigenfunctions_of_Laplacian}, since the eigenvalues are equal: $\gamma^V(z)=\gamma^V(1-z)$ (we shall give a direct proof in Proposition \ref{prop:computation_of_vertex-spherical_functions}).
\end{proof}

\begin{proposition}\label{prop:initial_value_of_spherical_functions}
If $v\sim v_0$ then $\phi^V_z(v)=\gamma^V(z)$.
\end{proposition}
\begin{proof}
We know from Proposition \ref{prop:image_of_boundary_functions_under_Poisson_transform} that $\mu_1 \phi^V_z(v) = \gamma^V(z) \phi^V_z(v)$.  Moreover, $\phi^V_z(v_0)=1$. In particular, if $v$ is a neighbor of $v_0$, $\mu_1 \phi^V_z(v_0) = \gamma^V(z)$.
On the other hand, $\phi^V_z$ is radial around $v_0$. 
Therefore $\mu_1 \phi^V_z(v_0) = \phi^V_z(v)$ for every $v\sim v_0$. This proves the statement.
\end{proof}

We can now give a new proof, based on horospheres, of a known result on zonal vertex-spherical functions \cite{Figa-Talamanca&Picardello}*{Chapter 3, Theorem 2.2}.
\begin{proposition}[Vertex-zonal spherical functions on homogeneous trees]
\label{prop:computation_of_vertex-spherical_functions}
Let $z\neq\frac12+ik\pi/\ln q$ with $k\in\mathZ$,  that is, $q^{2z-1} \neq 1$, and
\begin{equation}\label {eq:c(z)}
c(z)=\frac 1{q+1}\; \frac{q^{1-z}-q^{z-1}}{q^{-z}-q^{z-1}}\;.
\end{equation}
Then, for these values of $z$, the radial exponential functions $q^{-z|v|}$ and $q^{(z-1)|v|}$ do not coincide, and 
\begin{equation}\label{eq:vertex_spherical_functions_as linear_combinations_of_exponentials}
\phi^V_z(v)=c(z)\, q^{-z|v|}+c(1-z)\,q^{(z-1)|v|}.
\end{equation}
Instead, if $q^{2z-1} = 1$, that is $z=\frac12+ik\pi/\ln q$ with $k\in\mathZ$, one has
\begin{equation} \label{eq:the_vertex_spherical_function_at_the_spectral_radius}
\phi^V_z(v)=\left(1+\frac {q-1}{q+1}|v|\right)q^{-z|v|}.
\end{equation}
Note that the sign in the last expression is $(-1)^{k|v|}$: for odd $k$, the sign is $-1$ on odd vertices and $+1$ on even vertices.
\\
The spherical functions $\phi^V_z$ are real valued if and only if either $z=x+k\pi i/\ln q$ for any real $x$ and integer $k$, or $\Real z=1/2$.
\end{proposition}
\begin{proof}
To compute $\phi^V_z(v)$, let $|v|=n$ and denote by $v_0 \to v_1 \to \dots v_n=v$ the geodesic path from $v_0$ to $v$.
Observe that $K_V^z(v,\,\omega)$ is constant when $\omega$ varies in the boundary arc $\Omega(v_0,v)=\Omega(v)=\Omega_v$ subtended by $v$, introduced in Subsection \ref{SubS:Boundary}, and also in the complements $D_j=\Omega(v_j)\setminus \Omega(v_{j+1})$: indeed, by Definition \ref{def:horospherical_index} of horospherical index, $h(v,v_0,\omega)=n$ if $\omega\in \Omega(v)$, and $h(v,v_0,\omega)=2j-n$
if $\omega\in D_j$ (in particular, $h(v,v_0,\omega)=-n$ if $\omega\in D_0$).

 Moreover, the boundary measure $\nu_{v_0}$, being invariant under the isotropy subgroup $\Aut T_{v_0}$, is an equidistributed probability measure around $v_0$, hence, as seen in Subsection \ref{SubS:Boundary},  $\nu_{v_0}(\Omega(v))$ is the reciprocal of the number of vertices whose distance from $v_0$  is the same as the  distance of $v$. Thus, as $|v|=n$,
\begin{subequations}\label{eq:measures_of_nested_arcs}
\begin{align*}
\nu_{v_0}(\Omega(v_0,v)))&=\frac1{(q+1)q^{n-1}},\\[.2cm]
\nu_{v_0}(D_j)&=\frac1{(q+1)q^{j-1}}\;-\;\frac1{(q+1)q^{j}}=\frac{q-1}{q+1}\;q^{-j} \quad\text{ for } j=1,\dots, n-1,\\[.2cm]
\nu_{v_0}(D_0)&=\frac q{q+1}\;.
\end{align*}
\end{subequations}
Therefore
\begin{align}\label{eq:vertex-spherical_function_via_direct_computation}
\phi_z^V(v) = \radspherFour \delta_v(z) &= \int_\Omega \Poiss^z(v,\,\omega)\,d\nu_{v_0}(\omega)\notag\\[.2cm]
& = 
\frac q{q+1}\;q^{-zn}+\left(\sum_{j=1}^{n-1} \frac{q-1}{(q+1)q^j}\;q^{z(2j-n)}\right)+\frac1{(q+1)q^{n-1}}\;q^{zn}
\notag\\[.2cm]
& = 
\frac q{q+1}\;(q^{-zn}+q^{(z-1)n}) + \frac{q-1}{q+1} \sum_{j=1}^{n-1} q^{z(2j-n)-j}.
\end{align}
Now the calculation follows  by computing the geometric sum. Note that the last identity gives a direct proof of the fact that $\phi^V_z(v)=\phi^V_{1-z}(v)$. The special cases $z=1/2$ and $z=1/2 + i\pi/\ln q$ are left to the reader.

It is obvious from Definition \ref{def:zonal_spherical_functions}
that the spherical functions $\phi^V_z$ are real if $z$ is real or, more generally,  $\Imag z=k\pi/\ln q$. It is easily seen that $c(1/2+it)=\overline{c(1/2-it)}$: therefore \eqref{eq:vertex_spherical_functions_as linear_combinations_of_exponentials} implies that $\phi^V_{1/2 +it}$ is real valued. For the other values of $z$, we leave to the reader to verify, from the explicit expression \eqref{eq:vertex_spherical_functions_as linear_combinations_of_exponentials}, that $\Imag\phi^V_z$ is not identically zero.
\end{proof}

\begin{corollary} [Parity property of spherical functions]\label{cor:parity_of_spherical_functions} For every $v\in V$, the real part of  the function $t\mapsto \phi_{x+it}(v)$ is even and the imaginary part is odd. The same is true for edge-spherical functions.
Note that the restriction to the half-period $[0,\pi/\ln q]$ of
 $t
 \mapsto \Real \phi_{x+it}$ is even under the reflection around the center point $\pi/2\ln q$ if $n$ is even, and odd if $n$ is odd, while the same restriction for
 $t
 \mapsto \Imag \phi_{x+it}$ is odd under the center point reflection  if $n$ is even, and even if $n$ is odd (of course the same happens in the opposite interval).
\end{corollary}

\begin{corollary} \label{cor:imaginary_part_of_gamma_under_real_displacement}
\begin{enumerate}
\item[$(i)$] The map $\gamma=\gamma^V$ is periodic  of period $2\pi i/\ln q$, 
and $\gamma \left(\frac12+it\right)=\gamma \left(\frac12-it\right)$. Therefore $\gamma $ maps $\frac12 + i\mathR$ onto the half-period $\left[0,\,\frac12+i\frac{\pi}{\ln q}\right]$.
\item[$(ii)$]
For $0<t<\frac{\pi}{2\ln q}\;$,
\begin{align*}
\Real \gamma \left(\frac12+it+\delta\right)&>0\quad \text{for every}\; \delta\in\mathR,\\
\Imag \gamma \left(\frac12+it+\delta\right)&>0\quad \text{if and only if}\;\; \delta >0,
\end{align*}
and, for $\frac{\pi}{2\ln q}<t<\frac{\pi}{\ln q}$, 
\begin{align*}
\Real \gamma \left(\frac12+it+\delta\right)&<0\quad \text{for every}\; \delta\in\mathR,\\
\Imag \gamma \left(\frac12+it+\delta\right)&<0\quad \text{if and only if}\;\; \delta >0.
\end{align*}
\end{enumerate}
\end{corollary}

\section{Zonal spherical functions on $\HorE$ and on $E$}\label{Sec:Zonal_edge-spherical_functions}
The edge-zonal spherical functions on $E$ are defined as in \eqref{eq:def_of_vertex-spherical} with respect to the edge-horospherical index and the boundary measure invariant under the isotropy group of $r_0$. The equivalent of Proposition \ref{prop:initial_value_of_spherical_functions} holds for edge-zonal spherical functions,  with the same proof.

A theory of zonal spherical functions on edges has been developed only very recently \ocite{Casadio_Tarabusi&Picardello-spherical_functions_on_edges} by means of harmonic analysis, but here we give an equivalent definition and computation via integral geometry:
\begin{proposition}
[Edge-zonal spherical functions on homogeneous trees] 
\label{prop:computation_of_edge-spherical_functions}
Let $z\neq\frac12+ik\pi/\ln q$ with $k\in\mathZ$,  that is, $q^{2z-1} \neq 1$, and
\begin{equation}\label{eq:d(z)}
d(z)=\frac12\;\frac{q-1+q^{1-z}-q^z}{q^{1-z}-q^z}\;.
\end{equation}
Then, for these values of $z$,
\begin{equation}\label{eq:edge_spherical_functions_as linear_combinations_of_exponentials}
\phi^E_z(e)=d(z)\, q^{-z|e|}+d(1-z)\,q^{(z-1)|e|}.
\end{equation}
\\
Instead, if $q^{2z-1} = 1$, that is, $z=\frac12+ik\pi/\ln q\,,$
then 
\begin{equation} \label{eq:the_edge_spherical_function_at_the_spectral_radius}
\phi^E_z(e)=
(- 1)^{k|e|}\left(1+\frac{q-1}{2q}\;|e|\right) q^{-|e|/2}.
\end{equation}
\end{proposition}
\begin{proof}
We leave to the reader the special case  $q^{2z-1} = 1$ and deal with the general case.
Let $|e|=n$ and $e_0 \to e_1 \to \dots e_n=e$ be the geodesic path from $e_0$ to $e$.
Observe that $K_E^z(e,\,\omega)$ is constant when $\omega$ varies in the forward (with respect to $e_0$) boundary arc $\Omega(e)=\Omega_e$ subtended by the edge $e$, introduced in Subsection \ref{SubS:Boundary}, and also in the complements $D_j=\Omega(e_j)\setminus \Omega(e_{j+1})$: indeed, by Definition \ref{def:horospherical_index} of horospherical index, $h(e,e_0,\omega)=n$ if $\omega\in \Omega(e)$, and $h(e,e_0,\omega)=2j-n$
if $\omega\in D_j$ (in particular, $h(e,e_0,\omega)=-n$ if $\omega\in D_0$).

 Moreover, the boundary measure $\nu_{e_0}$, being invariant under the isotropy subgroup $\Aut T_{e_0}$, is an equidistributed probability measure around $e_0$, hence, as seen in Subsection \ref{SubS:Boundary},  $\nu_{e_0}(\Omega(e))$ is the reciprocal of the number of edges whose distance from $e_0$ is the same as the distance of $e$. Thus, as $|e|=n$,
\begin{align}\label{eq:edge-spherical_function_via_direct_computation}
\nu_{e_0}(\Omega(e))&=\frac1{2q^{n}},\notag\\[.2cm]
\nu_{e_0}(D_j)&=\frac1{2q^{j}}\;-\;\frac1{2q^{j+1}}=\frac{q-1}{2q^{j+1}} \quad\text{ for } j=1,\dots, n-1,\notag\\[.2cm]
\nu_{e_0}(D_0)&=\frac {2q-1}{2q}\;.
\end{align}
Therefore
\begin{align*}
\phi_z^E(v) = \calF \delta_e(z) &= \int_\Omega K^z(e,\,\omega)\,d\nu_{e_0}(\omega)\\[.2cm]
& = 
\frac {2q-1}{2q}\;q^{-zn}+\left(\sum_{j=1}^{n-1} \frac{q-1}{2q^{j+1}}\;q^{z(2j-n)}\right)+\frac1{2q^{n}}\;q^{zn}.
\end{align*}
The remainder of the calculation is a straightforward verification based upon the geometric summation formula.
\end{proof}

\begin{remark} The way the coefficients  $d(z)$ in Proposition \ref{prop:computation_of_edge-spherical_functions} and 
$c(z)$ in Proposition \ref{prop:computation_of_vertex-spherical_functions} depend on $z$ is similar, except for an additional term in the
numerator of $d(z)$ that, however, does not depend on $z$. As a consequence, it is easy to verify that the parity properties of vertex-spherical functions proved in Corollary \ref{cor:parity_of_spherical_functions} also hold for edge-spherical functions.
\end{remark}

It follows from \eqref {eq:c(z)}, \eqref{eq:d(z)} that, for every $z$,
\begin{equation}\label{eq:conjugation_property_of_c_and_d}
\begin{split}
\overline{c(z)}=c(1-z),
\\
\overline{d(z)}=d(1-z).
\end{split}
\end{equation}
Then it follows from this and from
Propositions \ref{prop:computation_of_vertex-spherical_functions} and \ref{prop:computation_of_edge-spherical_functions} that
\begin{equation}\label{eq:formula_simmetrica_per_phi/|d|^2}
\begin{split}
\frac { \phi^V_{\frac12 +it}(n)}{\left|c\left(\frac12+it\right)\right|^2}   &= \frac{q^{-\left(\frac12 + it\right)n}}{
c\left(\frac12-it\right)} + \frac{q^{-\left(\frac12 - it\right)n}}{
c\left(\frac12+it\right)}\,,
\\[.2cm]
\frac { \phi^E_{\frac12 +it}(n)}{\left|c\left(\frac12+it\right)\right|^2}   &= \frac{q^{-\left(\frac12 + it\right)n}}{
d\left(\frac12-it\right)} + \frac{q^{-\left(\frac12 - it\right)n}}{
d\left(\frac12+it\right)}\,,
\end{split}
\end{equation}

\section{The zonal spherical Fourier transform}
In this Section we define the spherical Fourier transform on vertices and edges of a tree.  The definition of spherical Fourier transform is already known only for functions on vertices \cite{Cartier, Figa-Talamanca&Picardello}.

In view of \eqref{eq:Poisson_transform}, the Poisson representation given by Proposition \ref{prop:Poisson_transform_of_distributions} expresses all eigenfunctions of the Laplacian as spherical Fourier transforms, that is integrals of generalized Poisson kernels, regarded as in Definition \ref{def:Poisson_transform} as the functions given by the characters  $\sigma(n(\boldh))$ over any special section $\Sigma_{v_0}$, or equivalently, as integral of generalized Poisson kernels $\Poiss^{zh(v,v_0,\omega)}$ with respect to the boundary variable, that is over the fibers.

\begin{definition}[The zonal spherical Fourier transform]
\label{def:spherical_Fourier-Laplace_transform_on_homogeneous_trees}
Let $T$ be a homogeneous tree of homogeneity degree $q$, with reference vertex $v_0$. For every boundary point $\omega$, 
and for every finitely supported function $f$ on $V$, define the \emph{vertex-spherical Fourier transform at $\omega$} of $f$ (with respect to $v_0$) as
\begin{equation}\label{eq:vertex-horospherical_Zeta_transform}
\spherFour^{\omega}_{v_0}f (z)=\spherFour^{\omega,V}_{v_0}f (z)=\sum_{v\in V} f(v)\,q^{zh(v,v_0,\omega)}\,.
\end{equation}
Or else, we can express the zonal spherical Fourier transform as an operator on functions on the horospherical fiber bundle as follows:
\begin{equation}\label{eq:zonal_vertex-horospherical_Fourier_transform} 
\spherFour f (z, \boldh)= \langle \VRad f (n+\boldh), \,\overline{\sigma}(n) \rangle_{\ell^2(\mathZ)}\,,
\end{equation}
where $\boldh\in\HorV$ and $\sigma$ is the character of the structure group $\mathZ$ given by $\sigma(n)=q^{zn}$. Since $\sigma$ is a character, one finds
\[
\spherFour f (z, \boldh)= \spherFour^{\omega}_{v_0}f (z) \;\overline{\sigma}(n(\boldh,\omega)).
\]
The last formula does not depend on the global chart induced by the choice of $v_0$.
%
The {zonal vertex-spherical Fourier transform} of $f$  is the radialization of $\spherFour^{\omega}_{v_0}$ around $v_0$:  
\begin{equation} \label{eq:vertex-horospherical_Fourier_transform}
\begin{split}
\radspherFour^V_{v_0} f (z)&=\int_\Omega \spherFour^{\omega}_{v_0}f (z) \, d\nu_{v_0}(\omega) = \sum_{v\in V} f(v) \int_\Omega q^{zh(v,v_0,\omega)}\, d\nu_{v_0}(\omega)  \\[.2cm]&=\sum_{v\in V} f(v) \phi^V_z(v)=\langle f,\phi^V_z\rangle_V
\,, 
\end{split}
\end{equation}
where $\langle\cdot,\cdot\rangle_V$ denotes the inner product in $\ell^2(V)$ and $\phi^V_z$ is the zonal spherical function introduced in \eqref{eq:def_of_vertex-spherical}.
The equivalent formulation in terms of the horospherical fiber bundle is
\begin{equation}   \label{eq:vertex-horospherical_Fourier_transform_in_invariant_formulation}
\radspherFour^V_{v_0} f (z) = \int_{\Sigma_{v_0}} \left\langle \VRad f (n+\boldh), \,\overline{\sigma}(n) \right\rangle_{\ell^2(\mathZ)}   \; d\xi (\boldh)
\end{equation}
where $\xi$ is the invariant measure on $\HorV$ of Definition \ref{def:invariant_measure_on_Hor}.

For compatibility with the terminology of \cite{Cartier, Figa-Talamanca&Picardello}, we shall usually call $\radspherFour^V_{v_0}$ the \emph{spherical Fourier transform}, rather than using the more appropriate name \emph{zonal vertex-spherical Fourier transform}, and shall denote it by $\radspherFour$.
Similar definitions hold for functions on edges: for every $f\in \ell^1(E)$ we set
\begin{align} \label{eq:edge-horospherical_Fourier_transform}
\spherFour^{\omega,E}_{e_0}f (z)&=\sum_{e\in E} f(e)\,q^{zh(e,e_0,\omega)}\,,\notag\\[.2cm]
\radspherFour^E_{e_0} f (z)&=\int_\Omega \spherFour^{\omega,E}_{e_0}f (z) \, d\nu_{e_0}(\omega) = \sum_{e\in E} f(e) \int_\Omega q^{zh(e,e_0,\omega)}\, d\nu_{e_0}(\omega)=\langle f,\phi^E_z\rangle_E\,,
\end{align}
that is,
\[
\radspherFour^E_{e_0} f (z) = \int_{\Sigma_{e_0}} \langle \ERad f (n+\boldh), \,\overline{\sigma}(n) \rangle_{\ell^2(\mathZ)}   \; d\xi(\boldh)
\]
(now $\xi$ is the invariant measure on $\HorE$).
\end{definition}

%
\begin{definition}[Radialization; see also \eqref{eq:radialization_on_HorV}]\label{def:vertex-radialization}
Denote again by $\mathfrakE_{v_0}$ is the radialization operator around $v_0$ introduced in Remark \ref{rem:horospherical_notation_for_Radon_transforms_of_radial_functions}
 on functions $g$ on $\HorV$, that can be written as 
\[
\mathfrakE_{v_0} g (\boldh') = \int_\Omega g(\boldh(n(\boldh',v_0),v_
0,\omega))\,d\nu_{v_0}(\omega) = \int_{\Sigma_{v_0}} g(n(\boldh',v_0)+\boldh'') \,d\xi(\boldh'')\,.
\]
Of course $\mathfrakE_{v_0} g (n)$ depends only on $n$. It can be thought of as a function on horospheres that is radial, that is invariant under the subgroup of automorphisms of $T$ that fix $\Sigma_{v_0}$ (or equivalently, that fix $v_0$). The edge-radialization is defined in a completely analogous way.
\end{definition} 

\begin{remark}\label{rem:radialization}
If $f$ is a radial function on $V$ and $h$ is a radial function on $E$, it is obvious that, for all functions $f'$ on $\HorV$ and $h'$ on $\HorE$, one has
 $\langle g, g'\rangle_{\HorV}=\langle g, \mathfrakE_{v_0} g'\rangle_{\HorV}$ and $\langle h, h'\rangle_{\HorE}=\langle h, \mathfrakE _{e_0}h'\rangle_{\HorE}$. 
\end{remark}
\begin{remark}\label{rem:spherical_vertex-Fourier_transform_of_a_delta}
Observe that, for every vertex $v$,
\begin{subequations}\label{eq:spherical_and_zonal_spherical_V-transforms}
\begin{align}
\spherFour^V_{v_0} \delta_v (z,\omega)&=q^{zh(v,v_0,\omega)}\,,\label{subeq:spherical}\\[.2cm]
   \radspherFour^V_{v_0} \delta_v (z)&=\int_\Omega q^{z h(v,v_0,\omega)} \, d\nu_{v_0}(\omega)\,.\label{subeq:zonal}
\end{align}
\end{subequations}
\end{remark}

\begin{remark}[Radiality and parity properties of spherical Fourier transforms] \label{rem:spherical_transforms_are_radial_and_real_part_even_if_f>0}
We know that the horospherical index $h(v,v_0,\omega)$ depends only on the relative positions of $v$, $v_0$ and $\omega$, that is, on the branchings between the rays that they represent in $T$. Therefore, for every sequence $a_n$, 
the distribution of values of $v\mapsto a_{h(v,v_0,\omega)}$ when $v$ varies in the circle $C(v_0,n)=\{ u:\,|u|=n\}$ 
is the same as for the function $\omega\mapsto a_{h(v,v_0,\omega)}$ 
when $\omega$ varies in $\Omega$, 
then
\[
\frac 1{|C(v_0,n)|} \sum_{|v|=n} a_{h(v,v_0,\omega)} = \int_\Omega a_{h(v,v_0,\omega)}\;d\nu_{v_0}(\omega).
\]
This fact was first observed in 
\cite{Figa-Talamanca&Picardello-JFA}, and turned out to be an important step in developing uniformly bounded representation theory of groups 
of automorphisms of homogeneous trees via Poisson transforms (and in particolar a new understanding of the theory of spherical functions on homogeneous trees, that we are now going to develop via integral geometry, see also \cite{Figa-Talamanca&Picardello}).

As a consequence, \eqref{subeq:zonal} shows that $v\mapsto \radspherFour^V_{v_0} \delta_v(z)$ is a radial function of $v$ with respect to $v_0$ (that is, it depends only on $\dist(v,v_0)$). Moreover, since $q^{\overline z} = \overline{q^z}$, we see that,
for real functions $f$, $\Real \radspherFour^V_{v_0} f (\overline z) = \Real \radspherFour^V_{v_0} f ( z)$  and $\Imag \radspherFour^V_{v_0} f (\overline z) = -\Imag \radspherFour^V_{v_0} f ( z)$. Equivalently, $\radspherFour f(\overline z) = \overline{\radspherFour f(z)}$: the zonal spherical Fourier transform commutes with conjugation (the same proof shows that this is also true for the spherical Fourier transform $\spherFour f$).

For the same argument, if $f$ has finite support and $\mathfrakE_{v_0}f$ is its radialization  around $v_0$, then $\spherFour^V_{v_0} \mathfrakE_{v_0} f$ is radial, in the sense that does not depend on $\omega$: indeed, 
\begin{equation}\label{eq:radialization_of_zonal_spherical_Fourier_transform}
\spherFour^V_{v_0} \mathfrakE_{v_0}f  (z,\omega) = \int_\Omega \spherFour^V_{v_0} f (z,\omega)\,d\nu_{v_0}(\omega) =  \radspherFour^V_{v_0} f  (z).
\end{equation}
In particular, we obtain the radiality property
\begin{equation}\label{eq:radiality_of_spherical_Fourier_transform}
 \radspherFour^V_{v_0} \mathfrakE_{v_0} f  =  \radspherFour^V_{v_0} f .
 \end{equation}
and an analogous property holds for $ \radspherFour^E_{e_0}$. 
\end{remark}

\begin{remark}[Smoothness of spherical Fourier transforms]\label{rem:smoothness_of_spherical_transforms}
By their definition,  $\radspherFour^V_{v_0} f(z)$ and   $\radspherFour^E_{e_0} f(z)$ are holomorphic functions whenever $f$ is finitely supported, and more generally for every $f$ on which they are defined, that is, such that the series $\sum_v f(v) \int_\Omega q^{z h(v,v_0,\omega)}\,d\nu_{v_0}(\omega) = \sum_v f(v)\,\phi_z(v)$ converges. In particular, by Corollary \ref{cor:the_spherical_function_extends_to_L1_in_a_strip}, $\radspherFour^V_{v_0} f(z)$ and   $\radspherFour^E_{e_0} f(z)$ are holomorphic in the strip $0\leqslant \Real z \leqslant 1$ and continuous on its bounding lines $\{z=it, t\in\mathR\}$ and $\{z=1+it,\ t\in\mathR\}$ for every $f\in\ell^1$. The same is true for $\spherFour^V_{v_0} f(z)$ and   $\spherFour^E_{e_0} f(z)$.
\end{remark}

\section[Spherical functions via integral geometry]{Zonal spherical Fourier transform on vertices and on edges of homogeneous trees and spherical functions via integral geometry}
\label{Sec:spherical_Fourier_transform_and_spherical_functions}

\begin{remark}[Poisson kernels and spherical functions]\label{rem:spherical_functions_and_Poisson_kernels}
We have denoted by \emph{vertex-Poisson kernel} the function $\Poiss(v,v_0\, \omega)=K_V(v,v_0,\omega)=q^{h(v,v_0,\omega)}$. With this notation, the spherical Fourier transform and the zonal spherical Fourier transform can be rewritten in terms of the bilinear form $\lround f|g\rround=\lround f|g\rround_V=\sum_{v\in V} f(v)\,g(v)$ as
\begin{align*}
\spherFour^{\omega}_{v_0}f (z) &= \lround f| K_V^{z}(\,\cdot\,,v_0,\,\omega)\rround_V\,,\\
\radspherFour_{v_0} f (z) &=  \left\lround f \biggm| \int_\Omega K_V^{z}(\,\cdot\,,v_0,\,\omega)\,d\nu_{v_0}(\omega)\right\rround_V\,.
\end{align*}
This yields a nice presentation of the zonal spherical functions of Section \ref{Sec:Zonal_vertex-spherical_functions} (and \ref{Sec:Zonal_edge-spherical_functions}) in the present context. Indeed, for $z\in\mathC$ and $v\in V$, the vertex-spherical function $\phi_z^V(v)$ is now
\[
\phi_z^V(v) = \radspherFour^V_{v_0} \delta_v(z).
\]
Analogously, the  \emph{edge-Poisson kernel} is $K_E(e,e_0,\omega)=q^{h(e,e_0,\omega)}$, hence, if $f\in \ell^1(E)$ and $\lround f|g\rround_E$ denotes the bilinear form $\sum_{v\in V} f(v)\,g(v)$, one has again
\begin{align*}
\spherFour^{\omega,E}_{e_0}f (z) &= \lround f| K_E^{z}(\,\cdot\,,e_0,\,\omega)\rround_E,\\
\radspherFour^E_{e_0} f (z) &=  \left\lround f\biggm| \int_\Omega K_E^{z}(\,\cdot\,,e_0,\,\omega)\,d\nu_{e_0}(\omega)\right\rround_E\,,
\end{align*}
and the edge-spherical function is
\[
\phi_z^E(e) = \radspherFour^E_{e_0} \delta_e(z).
\]
\end{remark}

\begin{remark}\label{rem:obvious_properties_of_spherical_functions}
The spherical functions $\phi_z^V$ and $\phi_z^E$ are radial (respectively around $v_0$ and $e_0$), and, for $v|=|e|=n$, $\phi_z^V(v)=\lround \phi_z^V| \mu_n\rround_V$ and $\phi_z^E(e)=\lround \phi_z^E| \eta_n\rround_E$ (notation as in Corollary \ref{cor:radial_convolution_vertex-recurrence_relations} and Lemma \ref{lemma:radial_convolution_recurrence_relations_for_edges}).
\end{remark}

\begin{remark}\label{rem:spherical_transform_in_terms_of_spherical_functions}
The spherical Fourier transform $\widehat{h}(z)$ of suitable functions $f$ on $V$, at $z\in\mathC$, has been defined in \cite{Figa-Talamanca&Picardello-JFA} as
\begin{equation}\label{eq:the_spherical_transform_of_FT-P}
\widehat{f}(z) = \lround f|\phi^V_z\rround = \sum_{v\in V} f(v)\,\phi^V_z(v).
\end{equation}
Hence we see that
\begin{equation}\label{eq:spherical_transform_in_terms_of_spherical_functions}
\widehat{f}(z) = 
\sum_{v\in V} f(v)\,\radspherFour^V_{v_0}\delta_v(z) =\radspherFour^V_{v_0} f(z),
\end{equation}
and so the definition of spherical Fourier transform given in  \cite{Figa-Talamanca&Picardello-JFA} coincides with our Definition \ref{def:spherical_Fourier-Laplace_transform_on_homogeneous_trees} of $\radspherFour^V_{v_0}$. 

The spherical Fourier transform of functions in $\ell^1(E)$ can be rewritten analogously.
\end{remark}

\begin{corollary}\label{cor:the_spherical_function_extends_to_L1_in_a_strip}
 The spherical Fourier transform $\radspherFour^V_{v_0} f$, defined on finitely supported functions $f$, extends to 
functions in $\ell^\infty(V)$  if and only if $0 \leqslant \Real z \leqslant 1$.
\\
The same is true for $\radspherFour^E_{e_0}$.
\end{corollary}
\begin{proof} As a consequence of \eqref{eq:spherical_transform_in_terms_of_spherical_functions},
this is equivalent to show that
 $\phi_z^V\in \ell^\infty$ if and only if $0 \leqslant \Real z \leqslant 1$ (and similarly for $\radspherFour^E_{e_0}$ and $\phi_z^E$). This fact follows directly from the explicit expressions of the spherical functions given in Propositions \ref{prop:computation_of_vertex-spherical_functions} and \ref{prop:computation_of_edge-spherical_functions} (for more details, see Proposition \ref{prop:spectral_theory_of_spherical_functions}\,$(i)$ below).
 \end{proof}

\begin{corollary} [Parity properties od the zonal spherical Fourier transform]\label{cor:parity_of_spherical_Fourier_transform} 
For every function $f$ on $V$ and for each complex $z$,
$\spherFour^{\omega,V}_{v_0} f(z) = \spherFour^{\omega,V}_{v_0} f(1-z)$ for every $\omega$ and
$\radspherFour^V_{v_0} f(z) = \radspherFour^V_{v_0} f(1-z)$. In particular, $\spherFour^{\omega,V}_{v_0} f(1/2+it)$ is real for every real valued $f$.
\\ 
In particular, $\radspherFour^V_{v_0} f\left(\frac 12 +it\right) = \radspherFour_{v_0}^V f \left(\frac12 - it\right)$ for every $t\in\mathR$.
\\
Moreover, if $f$ is real valued, then $t\mapsto \spherFour^{\omega,V}_{v_0} f(x+it)$ has even real part and odd imaginary part for every $\omega$, and the same is true for $t\mapsto \radspherFour^{V}_{v_0} f(x+it)$.

The same properties hold for $\spherFour^{\omega,E}_{e_0}$ and $\radspherFour^E_{e_0}$, respectively.
\end{corollary}
\begin{proof}
It is enough to give the proof for vetrex-spherical functions. We have seen in Corollary \ref{cor:the_swap_z<->1-z_leaves_vertex-spherical_functions_invariant} that $\phi^V_z(v)=\phi^V_{1-z}(v)$. Then the invariance of the spherical Fourier transforms under the map $z\mapsto 1-z$ follows from
 \eqref{eq:the_spherical_transform_of_FT-P} and \eqref{eq:spherical_transform_in_terms_of_spherical_functions}.
The parity properties of its real and imaginary parts when $f$ is real follow from Corollary \ref{cor:parity_of_spherical_functions}. The fact that the spherical transform of a real valued function is real on the axis $\Real z=1/2$ follows from this and the identity $\radspherFour^V_{v_0} f\left(\frac 12 +it\right) = \radspherFour_{v_0}^V f \left(\frac12 - it\right)$, or also directly from Proposition \ref{prop:computation_of_vertex-spherical_functions}, that shows that the sherical functions are real on this axis.
\end{proof}

\section{The spherical representations on the horospherical bundles}
\label{Sec:spherical_representations_on_HorV}

\subsection{Spherical representations of $\Aut \HorV$}
Let us denote by $\pi$ the regular representation of $\Aut\HorV\approx\Aut\HorE$: 
$\pi(\lambda) f(\boldh)=\pi(\lambda^{-1} \boldh)$. The same expression holds for
the regular representation of
$\Aut T$ acting on functions on $\HorV$ (or $\HorE$).
\\
The space $\spaceU_\sigma$ is clearly $A-$invariant, and by Lemma \ref{lemma:automorphism_shift_for_horospheres} $A$ commutes with $\Aut\HorV$.
Hence $\spaceU_\sigma$ is also invariant under $\Aut\HorV$, and in particular under $\Aut T$. We denote by $\pi_\sigma$ the representation of $\Aut \HorV$ (and of $\Aut T$) obtained by restricting $\pi$ to $\spaceU_\sigma$. 
Clearly, $\pi_\sigma$ is the representation of $\Aut \HorV$ induced by $\sigma\in\widehat A$ in the sense of Mackey: $\pi_\sigma = \Ind_A^{\Aut \HorV} \sigma$. 

We now derive an explicit formula for $\pi_\sigma$. 
We have remarked several times that $A$ is canonically isomorphic to $\mathZ$, and so now we write its elements as $n$. The integer $n$ represents the level 
of the  section $n
+ \Sigma_{v}$ with respect to $\Sigma_v$.
Observe that, by the same notation and argument of the proof of Proposition \ref{prop:A_commutes_with_Aut(HorV)}, for every $\lambda\in\Aut\HorV$ and $\boldh\in\HorV$ and $\bs n\in A$, one has $\lambda(\boldh +\bs n)=\lambda\bs h + \bs n$. Therefore,
for every $f\in\spaceU_{\sigma_z}$ and $n\in A$,
\[
\pi_z(\lambda)f(\boldh)=f(\lambda^{-1}\boldh)=\sigma_z(\bs n) f(\lambda^{-1} (\boldh - \bs n)).
\]
Recall, from Subsection \ref{SubS:Sections_and_special_sections}, that a section $\Sigma\subset\HorV$ is a map $\Sigma:\Omega\to\HorV$ such that $\pi\circ\Sigma=\mathI$ (here $\pi$ is the canonical projection on the base of the fiber bundle). So $\Sigma(\omega)$ is a canonical element in the fiber $\omega$, and $\lambda^{-1}\Sigma(\omega)-\Sigma(\lambda^{-1}(\omega)$ is the difference between the image under $\lambda^{-1}$ of this canonical horosphere and the canonical element in the fiber $\lambda^{-1}(\omega)$.  Then the previous identity becomes
\[
\pi_z(\lambda)f(\boldh) = \sigma_z\bigl(\lambda^{-1}\Sigma(\omega)-\Sigma(\lambda^{-1}\omega) \bigr)\,f(\Sigma(\lambda^{-1}\omega)).
\]
Now let us lift functions $f$ on $\HorV$ to $F$ on $\Omega$ by the natural lifting $F(\omega)=f(\Sigma(\omega))$. Denoting by 
$\bs n$  the element in $A$ that shifts by the difference considered above, 
we rewrite the last identity as
\[
\pi_z(\lambda)F(\omega)=\pi_z(\lambda)f(\Sigma(\omega))=f(\lambda^{-1} \Sigma(\omega)) = \sigma_z(\bs n) F(\lambda^{-1}\omega).
\]
For the sake of simplicity, we rewrite this argument in the special case of $\lambda\in\Aut T$, using explicit coordinates in the fibers instead of differences. This amounts to choosing a special section $\Sigma_v$.
Every $a\in A$ is identified with the integer $n$ given by the level over $\Sigma_{v_0}$. 
We write each $\boldh\in\HorV$ as $\boldh(\omega_{\boldh},n_{\boldh})=\boldh(\omega_{\boldh},n_{\boldh};v_0)$ and we note that, 
if $\boldh_0\in\Sigma_{v_0}$ 
is a  horosphere parallel to $\boldh$, that is 
$\boldh=a\cdot \boldh_0$ for some $a\in A$, then $a$ is identified with the integer $n=n_{\boldh}=h(v,v_0,\omega)
$ for any $v\in\boldh$. Here $h(v,v_0,\omega)$ is the usual horospherical index of Definition \ref{def:horospherical_index}. 
Let us now see how the parameterization changes under automorphisms. With respect to ${v_0}$, 
$\boldh $ has horospherical index $n=h(v,v_0,\omega)
$ for all its vertices $v$. Hence, for $\lambda\in\Aut\HorV$, by \eqref{eq:Aut_acts_equivariantly_on_the_horospherical_index} $\lambda^{-1}\boldh $ has horospherical index, with respect to the special section $\lambda^{-1}\Sigma_{v_0}=\Sigma_{\lambda^{-1} v_0}$ equal to $n=h(\lambda^{-1} v,\lambda^{-1} v_0,\lambda^{-1}\omega)$. But with respect to the special section $\Sigma_{v_0}$, that is the reference section for  our coordinate system on the fibers, the index changes to $h(\lambda^{-1} v,v_0,\lambda^{-1} \omega)=h(v,\lambda v_0,\omega)$. By the cycle relation of Remark \ref{rem:cycle_relation}, this index is equal to 
\begin{equation}\label{change_of_fiber_coordinate_under_automorphism}
h(\lambda^{-1} v,v_0,\lambda^{-1} \omega)=h(v,\lambda v_0,\omega)=h(v,v_0,\omega)+h(v_0,\lambda v_0,\omega)=
n-h(\lambda v_0,v_0,\omega),
\end{equation}
where $v$ is any vertex in $\boldh=\boldh(\omega,n)=\boldh(\omega,n;v_0)$.
Since the characters of $\mathZ$ are all the complex exponentials,
 we write $\sigma(a)=\sigma(n)=\sigma_z(n)=q^{zn}$ for some $z\in\mathC$, and $\pi_z$ instead of $\pi_{\sigma_z}$. 
The action of $\Aut T$ on $\HorV$ induces a representation of $\Aut T$ on $\spaceU$: 
\[
f(\boldh(\lambda^{-1}\omega,k))
\]
where $k$ is the level of the horosphere $\lambda^{-1}\boldh(\omega,n)$ with respect to the section $\Sigma_{v_0}$. For each $v\in\boldh$, $\lambda^{-1}\boldh(\omega,n)$ is the horosphere that contains $\lambda^{-1}v_0$, is tangent at $\lambda^{-1}\omega$ (that is, it belongs to 
 the fiber  $\lambda^{-1}\omega$), 
 and by \eqref{change_of_fiber_coordinate_under_automorphism} has index with respect to $\Sigma_{v_0}$ given by $k=n-h(\lambda v_0,v_0,\omega)$.
So for $f\in\spaceU_{\sigma_z}$,  by \eqref{eq:isotypical_subspaces_under_A} one has
 \begin{align}\label{eq:explicit_expression_pf_spherical_representations+on_HorV}
 \pi_z(\lambda) f(\boldh(\omega, n))&= f(\boldh(\lambda^{-1}\omega,k)) = \sigma_z(h(\lambda v_0,v_0,\omega)) f(\boldh(\lambda^{-1}\omega,n)). 
 \end{align}

We see that the parameter $n$ (the level in the fiber) is left unchanged by $\pi_z$. Therefore $\pi_z$ is an infinite multiple of a representation on the space $\spaceU_\sigma$ of functions on $\HorV$ restricted to the base $\Sigma_{v_0}\approx\Omega$. Let us write $F(\omega)=f(\omega,0)$ the restriction of $f$ to the base of the fiber. This gives the explicit expression of the representation on the base space $\Sigma_{v_0}\approx\Omega$, that we write as a separate statement:
\begin{proposition}\label{prop:spherical_representation_realized_on_Omega}
The (vertex-)spherical representations 
of $\Aut T$ on the $\sigma-$isotypic subspace $\spaceU_z$ of function spaces $\spaceU$ on $\HorV$ (that is, the subspace on which the parallel fiber group $A$ acts as its character $\sigma_z$),  can be realized on the space of restrictions of functions in $\spaceU_{\sigma_z}$
 to the base space $\Sigma_{v_0}\approx\Omega$  as
\begin{equation} \label{eq:spherical_representation_realized_on_Omega}
\pi_z(\lambda) F (\omega) = \sigma_{z}(h(\lambda v_0,v_0,\omega)) F(\lambda^{-1}\omega) 
\end{equation}
for every $F\in\spaceU_{\sigma_z}$. Equivalently, by  \eqref{eq:Poisson_kernel_as_Radon-Nikodym_derivative},
\begin{equation}
\label{eq:the_usual_form_of_the_spherical_representations}
\pi_z(\lambda) F (\omega) = K^z(\lambda v_0,v_0,\omega) F(\lambda^{-1}\omega) .
\end{equation}
These representations are the restrictions to $\Aut T$ of representations of $\Aut\HorV\approx\Aut\HorE$ that can be written without explicit reference to the special section $\Sigma_{v_0}$, hence do not depend on the choice of the vertex $v_0$.
\end{proposition}
%
%

A completely analogous argument describes the action of $\Aut T$ on the $\sigma-$isotypic subspace $\spaceW_\sigma$ of functions on $\HorE$, and it leads
to the edge-spherical representation 
\[
\pi^E_z(\lambda) F (\omega) = K^z(\lambda e_0,e_0,\omega) F(\lambda^{-1}\omega)
\]
(here $F\in\spaceW_\sigma$). Note that these spherical representations do not depend on the choice of $v_0$ or $e_0$. Moreover, their definitions rely solely on the choice of an $A$-isotypic component in a space of functions on $\HorV$ and $\HorE$ 
and on the choice of a special section, that yields a global chart providing the value of the horospherical index. Since $\HorV$ and $\HorE$ are isomorphic and for any special section of $\HorV$ and of $\HorE$ there is an isomorphism that maps the first to the second (Proposition \ref{prop:HorV=HorE_but_pairs_of_special_sections_do_not_correspond}),
the edge-spherical representation $\pi^E_z$  is equivalent to the vertex-spherical representation $\pi^V_z$.

The realization \eqref{eq:the_usual_form_of_the_spherical_representations} of the spherical representations was introduced in \cite{Figa-Talamanca&Picardello-JFA}, where it was obtained via a completely different approach, based on the Gelfand pair $(\Aut T,K_{v_0})$ and its multiplicative functionals (zonal spherical functions on $V$) instead of on horospheres and the horospherical fiber bundle. It was shown in \cite{Figa-Talamanca&Picardello-JFA} that the spherical representations are topologically irreducible for $z\neq \frac12 + ik\pi/\ln q, \;k\in\mathZ$; irreducibility of the spherical representations $\pi_z^E$ was proved along the same lines in \cite{Iozzi&Picardello-Springer} by considering the action of a simply transitive group of isometries on the graph associated to $E$ (see Subsection \ref{SubS:convolution} for detaBut upon restriction to the free group or free product acting as a simply transitive group of isometries, the spherical representations $\pi_z$ of $\Aut T$ that we have just constructed restricts to the spherical representation of \cite{Figa-Talamanca&Picardello-JFA}. Therefore:

\begin{corollary} \label{cor:equivalence_of_vertex_and_edge-spherical_representations} The spherical representations of $\Aut \HorV$ restrict to the already known spherical representations of $\Aut T$ and of its simply transitive subgroups.The spherical representation $\pi_z^V$ of $\Aut \HorV$ is topologically irreducible $z\neq \frac12 + ik\pi/\ln q, \;k\in\mathZ$. Two spherical representations of $\Aut \HorV$ are equivalent if and only if the eigenvalues of their spherical functions are equal: that is, $\pi^V_z\sim\pi^V_w$ if and only if $\gamma^V(z)=\gamma^V(w)$. The spherical representation $\pi_z^E$ of $\Aut \HorE$ is equivalent to $\pi^V_z$, and $\pi^E_z\sim\pi^E_w$ if and only if $\gamma^E(z)=\gamma^E(w)$. 
\end{corollary}

\section{Unitary spherical representations}\label{Sec:Unitary_spherical_representations}

The special fiber $\Sigma_{v_0}\approx\Omega$ is equipped with the normalized positive Borel measure $\nu_{v_0}$, invariant under the isotropy group $\Aut T_{v_0}$, introduced in Subsection \ref{SubS:invariant_measure_on_boundary}. In the norm $\| \cdot\|_\Omega$ on functions on $\Omega$ defined by this measure, which of the spherical representations are unitary?

Since $K(\lambda v_0,v_0,\omega)= d\nu_{\lambda v_0}/d\nu_{v_0}(\omega)$,
\begin{align*}
\| \pi_z(\lambda) F (\omega)\|^2_\Omega &=  \int_\Omega |\pi_z(\lambda) F (\omega)|^2 \,d\nu_{v_0}(\omega) =  \int_\Omega K^{\Real z}(\lambda v_0,v_0,\omega\,|F (\lambda^{-1} \omega)|^2 \,d\nu_{v_0}(\omega)
\\[.2cm]
&= 
 \int_\Omega K^{\Real z- \frac12}(\lambda v_0,v_0,\omega)\,|F (\omega)|^2 \,d\nu_{v_0}(\omega)
\end{align*}
is equal to $\|  F \|^2_\Omega $ for every $F$ if and only if $\Real z = 1/2$. The series of spherical representations $\pi_{\frac12 +it}$ $(t\in\mathR)$ is the celebrated principal series of unitary representations introduced for the free group (a doubly transitive subgroup of $\Aut T$) in \cite{Figa-Talamanca&Picardello-JFA}.

The principal series, however, does not cover all cases of unitary spherical representations, but only all spherical representations that are unitary in the natural norm on $\Omega\approx\Sigma_{v_0}$ defined by $\nu_{v_0}$. In order to find out which other spherical representations become unitary if the norm is redefined in a suitable way we just need to find out which zonal spherical functions are positive definite. For this, see \cite{Figa-Talamanca&Picardello}*{Chapter III}, or Theorem \ref{theo:ell^2-spectra} below.

In this Section we prove  inversion formulas for the spherical Fourier transform on vertices and edges of a tree via integral geometry (that is, based on the horospherical bundles). An inversion formula on $\ell^1(V)$ was proved in \cite{Figa-Talamanca&Picardello}*{Chapter 3, Section IV} by a direct computation of the Plancherel measure for vertices: in a related environment, a Plancherel formula for vertex-functions was later obtained in \cite{Faraut&Picardello} by making use of Carleman's formula (see also \cite{Kuhn&Soardi}). We refer the reader to Section \ref{Sect:hom_Plancherel} for more details. Our inversion formulas extend to edge-spherical transforms.

For any finitely supported
sequence $g=\{g_n\}$ of complex numbers, we denote as follows the (complexified) Fourier series that it generates:
\begin{equation}\label{eq:Fourier_series}
\calF^{\mathZ}\, g(t)=\sum_n g_{n}\,q^{nz}\,.
\end{equation}
The operator $\calF^{\mathZ}$ is defined for all finitely supported functions, but it extends to the sequences such that
$\{g_n\}$such that $\{g_n q^{n\Real z}\}$ is summable: in other words,
in the complex strip $|\Real z|\leqslant x$, it extends to the space of sequences such that $\sum_{n=-\infty}^\infty |g_n| q^{nx}<\infty$.
Observe that the inverse operator is given by taking Fourier coefficients: if $u=\calF^{\mathZ}\, g$, then
\begin{equation*}
g_n = \frac {\ln q}{2\pi} \int_{-\pi/\ln q}^{\pi/\ln q} u(it)\,q^{-int}\,dt,
\end{equation*}
or, more generally, for every $x\in\mathR$, 
\begin{equation}\label{eq:Fourier_coefficients}
g_n = \frac {\ln q}{2\pi} \int_{-\pi/\ln q}^{\pi/\ln q} u(x+it)\,q^{n(x-it)}\,dt.
\end{equation}
(remember that $u$ is holomorphic function of $z$, hence its values on each line $t\mapsto =x+it$ are determined uniquely by the values on any other such line).

The definition \eqref{eq:Fourier_series}
 leads  immediately to the following factorization:
\begin{proposition}[Fourier slice theorem]\label{prop:Fourier_slice_theorem}
If $\calF^{\mathZ}$ is the Laplace transform on $\mathZ$, i.e. the complex Fourier series operator, 
then
\begin{align*}
\spherFour^V_{v_0} &=\calF^{\mathZ} \, \VRad ,\\[.2cm]
\spherFour^E_{e_0} &=\calF^{\mathZ} \, \ERad ,
\end{align*}
in the sense that, for $f$ on $V$ such that the series $\spherFour^V_{v_0} f$ converges,
$\spherFour^V_{v_0} f(z,\omega) = \calF^{\mathZ} (\{g_n(\omega)\})$, with $g_n(\omega)=\VRad f(n,\omega)$, with respect to a choice of reference vertex $v_0$, that is, a special section. In more intrinsic geometric notation, that does not require a reference special section, this becomes
\[
\spherFour^V_{v_0} f(z,\boldh) = (\calF^{\mathZ}_n [\VRad f(n+\boldh)]) (z).
\]
The formula for $\spherFour^E$ is analogous.
\end{proposition}

\begin{corollary}\label{cor:consequence_of_Fourier_slice_theorem}
With $\Rad^\omega_V$, $\Rad^\omega_E$ as in Definition \ref{def:Radon},   for all $0\leqslant \Real z \leqslant 1$, $f\in\ell^1(V)$ and $g\in\ell^1(E)$ one has
\begin{align*}
\radspherFour^V_{v_0} f(z)&=\int_\Omega \spherFour^V_\omega f (z) \,d\nu_{v_0}(\omega) =\calF^{\mathZ} \int_\Omega \Rad^\omega_V f\, d\nu_{v_0}(\omega) =\calF^{\mathZ} \,\calE_{v_0} \,\Rad^\omega_V f 
,\\[.2cm]
\radspherFour^E_{e_0} g(z)&=\int_\Omega \spherFour^E_\omega g (z) \,d\nu_{v_0}(\omega) =   \calF^{\mathZ} \int_\Omega \Rad^\omega_E g\, d\nu_{e_0}(\omega) =\calF^{\mathZ} \,\calE_{e_0} \,\Rad^\omega_E g 
.
\end{align*}

By equivariance (Corollary \ref{cor:equivariance}), when the  identities for $\radspherFour^V_{v_0,\omega}$ and $\radspherFour^E_{e_0,\omega}$ are applied to functions $f$ radial around $v_0$ (respectively, $e_0$), then $\Rad^\omega_V f$ and $\Rad^\omega_E f$ do not depend on $\omega$ and, with abuse of notation, these identities could be written for radial $f$ as
\begin{subequations}
\begin{align}
\radspherFour^V_{v_0} f&= \calF^{\mathZ} \, \Rad_V f , \label{eq:global_Fourier_slice_theorem_for_VRad}
\\[.2cm]
\radspherFour^E_{e_0} f&= \calF^{\mathZ} \, \Rad_E f . \label{eq:global_Fourier_slice_theorem_for_ERad}
\end{align}
\end{subequations}
\end{corollary}
The composition of these operators is applied to functions $f$ of $v\in V$ or $e\in E$. Then $\Rad^\omega_V  f$ is a function on $\HorV$, with parameters that depend on the choice of the reference vertex $v_0$: namely, $\omega\in\Omega$ (fixed) and the horospherical index $n=h(v,\omega; v_0):=h(v,v_0,\omega)$. The radialization $\calE$ on the space of vertex-horospheres (integration with respect to the boundary measure $\nu_{v_0}$ on the special section $\Sigma_{v_0}$) gives rise to a sequence of coefficients $\calE \,\Rad^\omega_V f$ that depends only on the horospherical index with respect to $v_0$, i.e., is a sequence of complex numbers. 
Moreover, $\calF^{\mathZ} \int_\Omega \Rad^\omega_V f$ is the Fourier series obtained from the previous sequence of coefficients, and so it is a function of a parameter 
$t$ in the circle $[0,2\pi/\ln q)$. Similarly, $\calF^{\mathZ} \, \Rad^\omega_V f$ depends on $\omega$ and $t$.
An entirely analogous description of the variables of the corresponding functions hold in the case of edges.

It is appropriate to write down the compatibility condition associated to a change of reference vertex, from $v_0$ to, say, $v_1$, or from $e_0$ to $e_1$. 

\begin{lemma}[Equivariance of Fourier transforms under change of reference vertex]\label{lemma:change_of_reference_for_spherical_Fourier_transforms}
For any reference vertex $v$, let us adopt the notation $\spherFour^V_\omega f(t,\omega;v)$ and $\Rad_V^\omega f(n;v)$, where we have made explicit the dependence on $v$ (of course, usually $v$ is fixed and the spherical Fourier transform is a function of the variable $t\in\mathR$). Then the following equivariance rules hold:
\begin{align*}
\spherFour^V f(z,\omega;\,v_1)&=\spherFour^V f(z,\omega;\,v_0)\,q^{zh(v_1,v_0,\omega)}\\[.2cm]
\VRad f(n,\omega;\,v_1)&=\VRad f(n+h(v_1,v_0,\omega),\omega;\,v_0).
\end{align*}

Completely analogous identities hold for spherical Fourier transform and the horospherical Radon transform on edges.
\end{lemma}
\begin{proof}
By \eqref{subeq:spherical} and the formula of change of base  of Corollary \ref{cor:change_of_reference_for_horospheres},
\[
\spherFour^V f(z,\omega;\,v_1)=\sum_{v\in V} f(v)\,q^{z h(v,v_1,\omega)}=\sum_{v\in V} f(v)\,q^{z\left(h(v,v_0,\omega)+h(v_0,v_1,\omega)\right)}.
\]
The second identity is already known thanks to  the equivariance under reference change of the horospherical index in Corollary \ref{cor:change_of_reference_for_horospheres}:
\begin{align*}
\VRad f(n,\omega;\,v_1)&= \sum_{h(w,v_1,\omega)=n} f(w) = \sum_{h(w,v_0,\omega)-h(v_1,v_0,\omega)=n} f(w)\\[.2cm]
&=
\VRad f(n+h(v_1,v_0,\omega),\omega;\,v_0).
\end{align*}
The proof for $\spherFour^E_{\omega} $ and $\Rad_E^\omega$ is identical.
\end{proof}

\section[Zonal spherical Fourier inversion  via the Fourier slice theorem]{Inversion of the zonal spherical Fourier transform via the Fourier slice theorem}\label{sec:Inversion_of_radspherFour_via_Radon_slice}
We know how to invert the Fourier series operator on $\mathZ$. Indeed, given the Fourier $f$ series generated by the sequence $g=\{g_n\}$, we know how to recover the sequence $g$: it is just enough to compute the Fourier coefficients of the function $f$. Therefore, inverting the horospherical Radon transform at $\omega$ is equivalent to inverting the vertex-horospherical Fourier transform:
\begin{equation}\label{eq:non-radial_inversion_via_the_Fourier_slice_theorem}
\left(\spherFour^V_{v_0}\right)^{-1}  =  \Rad_V^{-1}    (\calF^{\mathZ})^{-1}.    
\end{equation}
If we restrict attention to functions on $V$ that are \emph{radial} around $v_0$, integrating this identity with respect to $\nu_{v_0}$ we obtain, with the same abuse of notation of \eqref{eq:global_Fourier_slice_theorem_for_VRad}, 
the inversion formula
\begin{align}\label{eq:radial_vertex-Fourier_inversion_via_Radon}
\left(\radspherFour^V_{v_0}\right)^{-1}  &=  \Rad_V^{-1}    (\calF^{\mathZ})^{-1}
\end{align}
where $\Rad_V^{-1}$ denotes the left inverse of $\Rad_V$. Here, for every $z=x+it$, $\calF^{\mathZ}$ maps a sequence of coefficients to a trigonometric series in the variable $q^{it}$, that is, a function of $t\in[-\pi/\ln q,\pi/\ln q]$, and we choose $\left(\radspherFour^V_{v_0}\right)^{-1}$ so as it maps this function to a \emph{radial} functions on $V$ (here this is the natural choice compatible with the radiality property \eqref{eq:radiality_of_spherical_Fourier_transform}). The inversion formula states that the horospherical Radon transform of this radial function is the original function. 
\noindent
Similarly,
\[
\left(\spherFour^E_{e_0}\right)^{-1}  =  \Rad_E^{-1}    (\calF^{\mathZ})^{-1}.   
\]
and, if we restrict attention to functions radial around $e_0$,
\begin{align}\label{eq:edge-Fourier_inversion_via_Radon}
\left(\radspherFour^E_{e_0}\right)^{-1}  &=  \Rad_E^{-1}.   (\calF^{\mathZ})^{-1}
\end{align}

Usually the focus of attention is the first identity, that allows to compute the inverse horospherical Radon transform via Fourier transforms. But here we have already computed the inverse horospherical Radon transform in Theorem
\ref{theo:radial_homogeneous_inversion_for_VRad}, hence we could now focus attention on the second identity to produce the spherical Fourier inversion formula via integral geometry on trees. We omit the computation, because two more elegant derivations of the Radon inversione formula will be obtained  in Chapter \ref{Sec:hom_V-Radon_and_E-Radon} via direct computation, and in Chapter \ref{Chap:spherical_functions} via the Radon back-projection.

\part{Horospherical transforms on homogeneous trees}
\section[$\HorV$ and $\HorE$: their range and inversion]{Horospherical Radon transform
for functions on vertices and on edges, their range and inversion }
\label{Sec:hom_V-Radon_and_E-Radon}


In this chapter we prove inversion formulas for the Radon transforms $\VRad$ and $\ERad$ on a homogeneous tree, by a direct approach. An alternative method based on the back-projection operator will be dealt with in later chapters. We also study the range of these transforms.

\begin{definition}[Horospherical transforms]
\label{def:Radon}
Let $T$ be a homogeneous tree. The \textit{vertex-horospherical Radon transform} $\VRad$ is the operator on (finitely supported, or more generally $\ell^1$) functions on $V$ to functions on the vertex-horospherical fiber bundle defined by
\begin{equation*}
 \VRad                 f(\boldh)
=\sum_{v\in\boldh}f(v),
\end{equation*}
that is,
\begin{equation*}
 \VRad                 f(\omega,n)
=\sum_{v\in\boldh_{n} (\omega,v_0)}f(v).
\end{equation*}
The \textit{edge-horospherical Radon transform} $\ERad$ is the operator on functions on $E$ to functions on the edge-horospherical fiber bundle defined by
\begin{equation*}
 \ERad                 f(\boldh)
=\sum_{v\in\boldh}f(e),
\end{equation*}
that is,
\begin{equation*}
 \ERad                 f(\omega,n)
=\sum_{e\in\boldh_{n} (\omega,e_0)}f(e).
\end{equation*}
When appropriate, we shall write 
$\Rad^\omega_V f(n)$ and $\Rad^\omega_E f(n)$
instead of $\VRad f(\omega,n)$ and $\ERad f(\omega,n)$.
\end{definition}

\begin{remark}[Equivariance of $\VRad$ and $\ERad$]
$\VRad$ is immediately seen to be equivariant with respect to $\Aut T$. Indeed, each $\lambda\in \Aut T$ acts on vertices, hence on vertex-horospheres, hence on functions thereon, and $\VRad (f\circ \lambda)= (\VRad f)\circ\lambda$, and the same holds for edges. That is, 
\[
\VRad                 (f\circ\lambda)(\omega,h(v,v_0,\omega))=\VRad  f(\lambda\omega,h(\lambda v, \lambda v_0, \lambda \omega)) 
\]
for every  $\lambda\in \Aut T$,
and similarly for $\ERad$. We shall often denote the action of $\Aut T$ on functions on $V$ or $E$ by $\lambda f$ instead of $f\circ\lambda$, and similarly we shall write $\lambda \VRad  f (\boldh) =  \VRad  f (\lambda \boldh) $. With this notation, the equivariance becomes: for every $\lambda\in\Aut T$ and $f:V\to \mathC$,
\begin{equation}\label{eq:equivariance_of_Radon}
\VRad  \lambda f =
  \lambda \VRad  f,
\end{equation}
and similarly for $f:E\to \mathC$.
\end{remark}

%
\subsection{The support theorem}\label{SubS:support_theorem}

\begin{definition} A subset $C$ of $V$ or $E$ is said to be convex if, for every two pairs of elements in $C$, all the elements in the geodesic path (consisting of a chain of adjoining vertices in the case of $V$, or of adjoining edges in the case of $E$) also belong to $C$.
\end{definition}
\begin{lemma} \label{Lemma:complement_of_convex_splits_as_union_of_sectors} The complement of every convex set $C$ in $V$ (respectively, in $E$) is a disjoint union of finitely many vertex-sectors (respectively, edge-sectors).
\end{lemma}
\begin{proof} We explain the case of vertices: the argument for edges is the same. 
Denote by $\partial C$ the boundary of $C$, that is, the set of vertices $v$ of $C$ that have a neighbor $u_v$ outside of $C$. The sectors $S(v,u_v)$ and $S(v',u'_v)$ are disjoint if  $v\neq v'\in \partial C$. Moreover, if we choose another neighbor $u'_v\notin C$, then $S(v,u_v)$ is disjoint from $S(v,u'_v)$. So, all these sectors subtended by boundary vertices and their neighbors are mutually disjoint. If a vertex $w$ does not belong to $C$, by convexity it must belong to a unique sector $S(v,u_v)$ with $v\in\partial C$ and $u_v\in[v,w]$. Therefore the complement of $C$ is the union of all these sectors.
\end{proof}

\begin{theorem}[Support theorem]\label{Theo:support_theorem}
Let  $C$ be a convex subset of $V$. If $\VRad f=0$ on all horospheres that do not intersect $C$, then $f=0$ in the complement of $C$. Obviously, the converse also holds: for any $C\subset V$ (not necessarily convex), $\VRad f=0$ on all horospheres that do not intersect $C$.
\\
The same statement holds for $\ERad$.
\end{theorem}
\begin{proof}
We give the proof for vertices. First note that every sector $S(v,u)$ is a disjoint union of horospheres. This was proved in \ocite{Casadio_Tarabusi&Cohen&Picardello}*{Proposition 2.6}, so we limit ourselves to sketch the idea of this result. By induction on $\dist(u,v)$, it is enough to limit attention to the case where $u$ and $v$ are neighbors. For simplicity, let us choose $v$ as reference vertex. Then we claim that $S(v,u)$ is the set of horospheres with positive horospherical index tangent at boundary points in $\Omega(v,u)$ (the boundary arc generated by $S(v,u)$). Indeed, $u$ belongs to any horosphere in $S(v,u)$ tangent at any $\omega\in \Omega(v,u)$ and of index 1. Every neighbor $w$ of $u$ except $v$ belongs to a horosphere of index 2 tangent at a boundary point in $\Omega(u,w)$ and lying in $S(u,w)$: these horospheres are disjoint for different $w$. Note that $\dist(w,v)=\dist(u,v)+1$. By iterating this argument, we see by induction on the distance that all vertices of of $S(v,u)$ belong to a horosphere as above, and all such horospheres can be chosen disjoint. This proves the claim.
\\
Next, obviously, for every vertex $v$, one has 
\begin{equation}\label{eq:V_as_disjoint_union_of_sectors_and_one_vertex}
V=\{v\}\cup_{u\sim v} S(v,u),
\end{equation}
 a disjoint union. Equally obviously, for all pairs $v_1$ and $v_2$ of neighbors, 
 \begin{equation}\label{eq:V_as_disjoint_union_of_two_sectors}
V=S(V_1,v_2)\cup S(v_2,v_1).
\end{equation}
%
%
Then on each sector $S(v,u)\subset \complement K$ one has 
\[
0=\sum_{w\in S(v,u)} f(w) = f(v) + \sum_{\boldh^V\in \HorV(v,u)} \VRad f(\boldh^V),
\]
where $\HorV(v,u)$ denotes a disjoint sequence of horospheres whose union is $S(v,u)$, that exists by Lemma \ref{Lemma:complement_of_convex_splits_as_union_of_sectors}. But on each $\boldh^V\in \HorV(v,u)$ one has $\VRad f(\boldh^V)=0$ because $\boldh^V\subset S(v,u)\subset \complement K$. It follows that $f(v)=0$.
\end{proof}

\subsection{Intersections of horospheres and circles on homogeneous trees}
\label{SubS:Homogeneous_intersections}
The range and inversion of the vertex-horospherical Radon transform on homogeneous trees have been studied in~\ocites{Betori&Faraut&Pagliacci,Casadio_Tarabusi&Cohen&Colonna}; see also \ocite{Casadio_Tarabusi&Cohen&Picardello}. For this goal, one needs to compute the cardinality of the intersections of horospheres and circles. In this paper we extend the results to the edge-horospherical Radon transform and, later, to semi-homogeneous trees: therefore the intersection cardinalities that we must compute are quite a few.

\begin{remark}\label{rem:horospherical_notation_for_Radon_transforms_of_radial_functions}
Let $f$ be a radial function on $V$ with respect to the reference vertex $v_0$, and write $f_m:=f(v)$ if $\dist(v,v_0)=m$. Then $\phi=\VRad f$ has a similar radiality property: if we associate the vertex-horosphere $\boldh^V$ with its horospherical coordinates $(\omega,n)$ introduced in~Theorem~\ref{theo:VRad-range}, where $\omega$ is the tangent boundary point and $n$ is the horospherical index, then $\phi(\omega,n)$ is independent of $\omega\in\Omega$, and it depends only on $n$. Let us write $\phi_n=\phi(\omega,n)$.

Exactly in the same way, $\psi=\ERad f$ is a function on edge-horospheres, and if the edge-horospheres $\boldh^E$ are parameterized by their horospherical coordinates $(\omega,n)$ with respect to the reference edge $e_0$, then $\psi(\omega,n)$ depends only on $n$ if (and only if) $f$ is radial (around $e_0$): we write $\psi_n=\psi(\omega,n)$.

More generally, let $\mathfrakE_{v_0}$ be the radialization of a vertex-function $f$ around $v_0$ (that is, the radial function obtained by averaging over circles around $v_0$).  An analogous radialization operator, that, by abuse of notation, we denote again by $\mathfrakE_{v_0}$, acts on functions on $\HorV$: namely,
\begin{equation}\label{eq:radialization_on_HorV}
\mathfrakE_{v_0} \phi(\boldh)=\int_{K_{v_0}} \phi(\lambda \boldh) \,d\nu_{v_0}(\lambda),
\end{equation}
where the integral is taken over the stability subgroup of $v_0$ in $\Aut T$, and $\nu_{v_0}$ is its normalized invariant measure. Here the action of automorphisms over horospheres is the natural action introduced in Subsection \ref{SubS:Automorphisms_on_horospheres}. Then, by equivariance (Corollary \ref{cor:equivariance}), $\VRad$ and $\mathfrakE_{v_0}$ commute. The same holds for edge-horospheres and radialization around edges.
\end{remark}

The following fact \cite{Casadio_Tarabusi&Cohen&Picardello} is evident.

\begin{lemma}
\label{lemma:radial_VRad}
Let $f:V\mapsto\mathC$ be finitely supported and radial with respect to $v_0$ and $\boldh_n$ be a vertex-horosphere tangent to the boundary at some point $\omega$ and with horospherical index $n$ with respect to $v_0$, and for $m\geqslant 0$ let $C^V_m$ be the circle of radius $m$, that is, $C^V_m=\{v\in V:\dist(v,v_0)=m\}$. Let us denote by $k_V(n,m)$ the cardinality of the intersection $\abs{\boldh_n\cap C^V_m}$. Then 
\[
\phi_n=\sum_{m\geqslant n} k_V(n,m) \,f_m\,.
\]
\end{lemma}

The following result was first obtained in~\ocite{Betori&Pagliacci-2}*{Lemma~2.4} (see also \cite{Betori&Faraut&Pagliacci}*{Section 4}), that, unfortunately, has a misprint: $k_V(n,n)$ and $k_V(-n,n)$ are interchanged. 

\begin{lemma}
\label{lemma:homogeneous_intersection_cardinalities}
Let $n\in\mathZ$ and $m\in \mathN$. The intersection cardinalities $k_V(n,m)$ of~Lemma~\ref{lemma:radial_VRad} (that is, with horospherical index $n$ and radius $m$) are
\begin{equation*}
k_V(n,m)=\begin{cases}
 1                 &\text{if $m=     n      $,}\\
 q^m               &\text{if $m=    -n      $,}\\

(q-1)q^{(m-n-2)/2} &\text{if $m>     |n|    $ and $m-n$ is even,}\\
 0                 &\text{otherwise.}\end{cases}
\end{equation*}
\end{lemma}

\begin{proof}
Fix the reference vertex $v_0$ and $\omega\in\Omega$ and let $|v|=m$ and $v\in\boldh_n$. Let $j_\omega(v,v_0)$ be the join of the rays $[v_0,\omega)$ and $[v_0,v]$ introduced in Subsection \ref{SubS:Joins}, and let us write $x=|j_\omega(v,v_0)|$
and $y=\dist(v,j_\omega(v,v_0))$. Then $x, y$ are non-negative integers and $m=x+y$ (because $j_\omega(v,v_0)$ is in-between $v_0$ and $v$), and $n=x-y$ (by the Definition \ref{def:horospherical_index} of horospherical index). This is equivalent to $x=(m+n)/2$ and $y=(m-n)/2$ and $m\equiv n \; \mod 2$ (because $x$ and $y$ must be integers). Let us first consider the case $x=0$, that is, $v$ is either $v_0$ or lies at the opposite side of the boundary point $\omega$ with respect to $v_0$. Then $n=-m$ and $v$ varies over the part of the circle
$C^V_n$ at this side of $v_0$. There are $q^y=q^m$ such vertices, hence if $x=0$ then $k_V(n,m)=\abs{\boldh_{-m}\cap C^V_m}=q^m$ if
$n=-m$ and 0 otherwise.
\\
Now let us look at the case $y=0$. Then $|v|=x$, $n=0$ and $v=j_\omega(v,v_0)$. The vertex $v$ is the unique vertex in the horosphere $\boldh_0(\omega)$ that lies in the ray $\omega$, hence if $y=0$ then $k_V(n,m)=1$ if $n=0$ and 0 otherwise.
\\
In the remaining cases, $x\neq v \neq y$. The vertices $v$ that satisfy these conditions are in the sector $S_+(v_0,\omega)$ of vertices that lie at the positive side of $v_0$ in the orientation induced on $T$ by $\omega$, that is, 
 the connected component of $T\setminus \{v_0\}$ whose boundary contains $\omega_0$.
  The path from $v_0$ to $v$ lies in the ray $\omega$ for the first $x$ steps, then, at the next step, must move away from the ray $\omega$ (there are $q-1$ choices), henceafter must continue outwards (there are $q$ choices for any remaining step). Therefore there are $(q-1)q^{y-1}=(q-1)q^{(m-n-2)/2}$ such vertices, and $k_V(n,m)=(q-1)q^{(m-n-2)/2}$ if $x,y>0$. This and the fact that $m$ and $n$ have the same parity proves the statement.
 \end{proof}

An immediate verification yields:

\begin{corollary}[\cite{Casadio_Tarabusi&Cohen&Colonna}]
\label{cor:cardinalities_under_horospherical_reflection}
We have
\begin{equation*}
     \phi_{-n}
=q^n \phi_  n
\qquad\text{for every $n>0$.}
\end{equation*}
\end{corollary}


\subsection{Ranges of $\VRad$ and $\ERad$}
\label{SubS:VRad_and_ERad-ranges}
The range of the vertex-horospherical Radon transform $\VRad$ acting on the space of functions with finite support on $V$, where $T$ is the homogeneous tree $T_q$, was characterized in~\ocite{Casadio_Tarabusi&Cohen&Colonna} as the space of compactly supported functions on the space $\HorV$ of vertex-horospheres that satisfy the following Cavalieri conditions:

Fix a reference vertex $v_0$, and let $d\nu=d\nu_{v_0}$ be the equidistributed boundary measure on $\Omega=\Omega(T)$ induced by $v_0$ (see Subsection \ref{SubS:Boundary}). 
   
\begin{theorem}[
Cavalieri conditions \ocite{Casadio_Tarabusi&Cohen&Colonna}]
\label{theo:VRad-range}
Let  $\phi$ be a function on $\HorV$  and denote by $\phi(\omega,n)$ the value of $\phi$ on the vertex-horosphere $\boldh_n(\omega)$ tangent at $\omega$ of index $n$ (with respect to $v_0$).  
Then $\phi=\VRad f$ for some $f\in\ell^1(V)$ if and only if, for every $v_0\in V$ and $n>0$, the following \textit{Cavalieri conditions} hold
\begin{equation}
\label{eq:vertex_Cavalieri}
 q^n \int_\Omega \phi(\omega, n)\,d\nu_{v_0}(\omega)
=    \int_\Omega \phi(\omega,-n)\,d\nu_{v_0}(\omega)
\end{equation}
and the sum $\sum_{n=-\infty}^\infty \phi(\omega,n)$ does not depend on $\omega$. Indeed,
\begin{equation}
\label{eq:vertex_Cavalieri_trivial}
\sum_{n=-\infty}^\infty  \phi(\omega,n)=
\sum_{v\in V           }   f (   v    ).
\end{equation}
\end{theorem}

Here we have restricted attention to finitely supported functions to avoid clumsy conditions of decay at infinity necessary to insure convergence of the series that define the horospherical Radon transform. We only give an idea of the proof of this known fact. The result is a direct consequence of Corollary~\ref{cor:cardinalities_under_horospherical_reflection}: once a reference vertex $v_0$ and a boundary point $\omega$ are chosen for a homogeneous tree of homogeneity $q$, the number of vertices at a given distance $d$ from $v_0$ in the horosphere $\boldh_{v_0}(\omega,-n)$ for $n>0$ is $q^n$ times the number of vertices in the horosphere $\boldh_{v_0}(\omega,n)$ at the same distance from $v_0$. For instance, in the case $d=1$, $n=1$ this amounts to saying that there is only one \textit{forward} neighbor of $v_0$ (that is, closer to $\omega$), but $q$ \textit{backward} neighbors (one step back from $\omega$). Since at each generation the growth is $q$, this ratio remains the same for larger values of $d$ (that is, on circles of larger radii centered at $v_0$): see~Figures~\ref{Fig:horospheres_on_T3} and \ref{Fig:horospheres_on_T2}.
\\
Exactly the same geometry holds for edges, and we have the same ratio of forward and backward edges in a horosphere. Therefore the same result holds for the edge-horospherical Radon transform (now $n$ is the edge-horosphere index with respect to a reference edge $e_0$):

\begin{theorem}
\label{theo:ERad-range}
Let $\phi$ be a function on $\HorE$ and denote by $\phi(\omega,n)$ the value of $\phi$ on the edge-horosphere $\boldh_n(\omega)$ tangent at $\omega$ of index $n$ (with respect to $e_0$). 
Then $\phi=\ERad f$ for some finitely supported $f$ on $E$ if and only if, for every $e_0\in E$ and $n>0$,
\begin{equation}\label{eq:edge_Cavalieri}
 q^n \int_\Omega \phi(\omega, n)\,d\nu_{e_0}(\omega)
=    \int_\Omega \phi(\omega,-n)\,d\nu_{e_0}(\omega)
\end{equation}
and the sum $\sum_{n=-\infty}^\infty \phi(\omega,n)$ does not depend on $\omega$. Indeed,
\begin{equation}\label{eq:edgex_Cavalieri_trivial}
 \sum_{n=-\infty}^\infty \phi(\omega,n)
=\sum_{e\in E   }         f  (   e    ).
\end{equation}
\end{theorem}

We can rewrite the Cavalieri conditions for the range of the horospherical Radon transforms in terms of the parametrization of vertex and edge horospheres in terms of mixed vertex-edges horospherical numbers as in \eqref{eq:mixed_vertex_edge_horospherical_index}, as follows.

\begin{definition}[Edge horospherical Radon transforms in terms of vertex-edge distance]\label{Def:mixed_ERad_transform}
Choose a reference edge $e_0$ and let the reference vertex $v_0$ be one of the endpoints of $e_0$.
With notation as in \eqref{eq:mixed_vertex_edge_horospherical_index_as_related_to_vertex_horospherical_index}, for $f:E\to\mathC$ we define
\[
\ERad f(\omega,\,n+\frac12\,;\,v_0) = \sum_{e:\,{h(v_0,e,\omega)} = n+\frac12} f(e).
\]
\end{definition}

\begin{remark}
It follows immediately from \eqref{eq:mixed_vertex_edge_horospherical_index_as_related_to_vertex_horospherical_index} that
\[
\ERad f(\omega,\,n+\frac12\,;\,v_0) = \ERad f (\omega,n;e_0)
\]
if $e_0$ is at the same side of $\omega$ with respect to $v_0$, and
\[
\ERad f(\omega,\,n+\frac12\,;\,v_0) = \ERad f (\omega,n-1;e_0)
\]
if $e_0$ is at the opposite side of $\omega$ with respect to $v_0$.
\end{remark}


\begin{theorem}
\label{theo:Rad-range_with_mixed_parameters}
\begin{enumerate}
\item[$(i)$]
Let  $\phi$ be a function on $\HorV$  and denote by $\phi(\omega,n)$ the value of $\phi$ on the vertex-horosphere $\boldh_n(\omega)$ tangent at $\omega$ of index $n$ (with respect to $v_0$), and by $\phi\left(\omega,n+\frac12\right)$ the value of $\phi$ on the vertex-horosphere $\boldh_{n+\frac12}(\omega)$ tangent at $\omega$ of index $n+\frac12$ with respect to $e_0\ni v_0$.

Then $\phi=\VRad f$ for some finitely supported $f$ on $V$ if and only if, for every $v_0\in V, v_0\in e_0\in E$ and $n\geqslant 0$, the following \textit{Cavalieri conditions} hold
\begin{equation}
\label{eq:vertex_Cavalieri-mixed}
 q^n \int_\Omega \phi\left(\omega, n+\frac12\right)\,d\nu_{e_0}(\omega)
=    \int_\Omega \phi\left(\omega,-n-\frac12\right)\,d\nu_{e_0}(\omega)
\end{equation}
and the sum $\sum_{n=-\infty}^\infty \phi(\omega,n)$ does not depend on $\omega$.
\item[$(ii)$]
Now let $\psi$ be a function on $\HorE$. Denote by $\psi(\omega,n)$ the value of $\psi$ on the edge-horosphere $\boldh_n(\omega)$ tangent at $\omega$ of index $n$ (with respect to $e_0$) and by $\psi\left(\omega,n+\frac12\right)$ the value of $\psi$ on the edge-horosphere $\boldh_{n+\frac12}(\omega)$ tangent at $\omega$ of index $n+\frac12$ with respect to $v_0$. 
Then $\psi=\ERad f$ for some $f\in\ell^1(E)$ if and only if, for every  for every $v_0\in V, v_0\in e_0\in E$ and $n\geqslant 0$, 
\begin{equation}\label{eq:edge_Cavalieri-mixed}
 q^{n+1} \int_\Omega \psi\left(\omega, n+\frac12\right)\,d\nu_{v_0}(\omega)
=    \int_\Omega \psi\left(\omega,-n-\frac12\right)\,d\nu_{v_0}(\omega)
\end{equation}
and the sum $\sum_{n=-\infty}^\infty \psi(\omega,n)$ does not depend on $\omega$.
\end{enumerate}
\end{theorem}

\subsection{Inversion of $\VRad$}
\label{SubS:VRad-inversion}

Since $\VRad$ commutes with translations, it is enough to invert it at the reference vertex $v_0$. Therefore our inversion formula reconstructs the value of the function $f$ at $v_0$ from the values of its vertex-horospherical Radon transform $\phi=\VRad f$. 

We may as well write the formula for functions that are invariant under automorphisms that fix $v_0$, that is, for radial functions: if $f$ is not radial, then we obtain its value $f(v_0)$ by applying the radial inversion formula to the function $\widetilde{f}$ obtained from $f$ by taking radial averages around $v_0$ (since $\widetilde{f}(v_0)=f(v_0)$).
 Again by equivariance, this amounts to radializing $f$ before applying the inversion formula. By Remark \ref{rem:horospherical_notation_for_Radon_transforms_of_radial_functions}, this means to reconstruct the value of $f$ at $v_0$
starting with data on the horospheres 
that depend only on the horospherical indices but not on the specific boundary point.

\section{Direct computation of the inverse of $\VRad$}\label{Sec:direct_inversion_of_VRad}
Following~\ocite{Betori&Faraut&Pagliacci} (see also \ocite{Casadio_Tarabusi&Cohen&Colonna}) we now provide some  inversion formulas for the vertex-horospherical Radon transform on a homogeneous tree (as we shall see, there are more than one such formulas, indeed uncountably many that all agree on radial functions), but the proof below gives more emphasis to the geometrical setup. We also observe that similar results for the edge-horospherical Radon transform have been obtained in~\ocite{Casadio_Tarabusi&Cohen&Picardello}, in the more general setting of non-homogeneous trees, where the inversion is obtained recursively with respect to the distance from the reference vertex. By specializing attention to homogeneous or semi-homogeneous trees, one obtains some explicit inversion formulas for these settings: however, the recursive proof is at first sight less explicit.
\\
 We prove this radial inversion theorem for $\VRad$ now and for $\ERad$ later: the obvious extensions to non-radial functions are stated as corollaries following these theorems.

\begin{theorem}[Radial inversion formulas for $\VRad$ on a homogeneous tree]
\label{theo:radial_homogeneous_inversion_for_VRad}
Let $T$ be a homogeneous tree of homogeneity degree $q$. There are uncountably many  inversion formulas for the vertex-horospherical Radon transform acting on finitely supported radial functions $f$ on $V$, all equivalent in the sense that they agree on the image of $\VRad$ (characterized in Theorem \ref{theo:VRad-range}). These inversion formulas are of the following form: if $\phi=\VRad f$ and   
 $\phi_n$, defined in~Remark~\ref{rem:horospherical_notation_for_Radon_transforms_of_radial_functions}, is the sequence of values of $\phi$ (that does not depend on $\omega$ since $f$ is assumed radial),
then
\begin{equation*}
f(v_0)=\sum_{n=-\infty}^\infty d_n \phi_n\,.
\end{equation*}
Two different explicit choices for $d_n$ are
\begin{align*}
d_n &=\begin{cases}
1       &\qquad\text{for      $n  =  0$,}\\
1-q     &\qquad\text{for even $n  >  0$,}\\
0       &\qquad\text{otherwise,}\end{cases}\\[.2cm]
d_n &=\begin{cases}
1       &      \text{for even $n\leqslant 0$,}\\
1-q-q^n &      \text{for even $n  >  0$,}\\
0       &      \text{for  odd $n      $.}\end{cases}\\
\end{align*}
\end{theorem}

\begin{proof}
Before we begin the proof, it is worth recalling that, if $f=\chi_m$ is the characteristic function of the circle $C_m^V=\{v\in V
: \dist (v,v_0)=m\}$ and $\phi=\VRad f$, then, by Definition \ref{def:Radon}, $\phi_n$ turns out to be precisely $k_V(n,m)$.

Denote by $K^V$ the matrix with entries $K^V_{nm}=k_V(n,m)$, and let $f_m$ be the value of a radial function $f$ on the circle of radius $m\in\mathN$ around $v_0$ (in particular, $f_0=f(v_0)$). Let $\boldphi$ be the sequence whose entry at index $n$ is $\phi_n$, $n\in\mathZ$, and $\mb{f}$ the sequence with entries $f_m$, $m\geqslant 0$. Then the horospherical Radon transformation $\phi=\VRad f$ becomes the doubly infinite linear system $\phi_n=\sum_{m\geqslant 0} k_V(n,m) f_m$, that is $\boldphi=K^V \mb{f}$. Its inversion formula is $f_0=\sum_{n=-\infty}^\infty d_n \phi_n$: it is easy to see, as follows, that the coefficients $d_n$ are nothing else but the zero row of the inverse matrix $({K^V})^{-1}$. More generally, by applying the formula in the statement to the characteristic function of the circle $\{v\in V: \dist (v,v_0)=m\}$, we have
\begin{equation}\label{eq:inversion_row_by_column}
\sum_{n=-\infty}^\infty d_n\,K_V(n,m)= \delta_{m}\,.
\end{equation}

To make the argument more understandable, the rows and columns with indices $n,m$ near $0$ of the infinite matrix $K^V$ obtained in~Lemma~\ref{lemma:homogeneous_intersection_cardinalities} are listed in~Table~\ref{Table:vertex_matrix_homogeneous}.
\begin{table}
\caption{Cardinality $k_V(n;m)$ of the intersection of the vertex-circle of radius $m$ and the vertex-horosphere of index $n$.}
\label{Table:vertex_matrix_homogeneous}
\begin{tabular}{*{1}{>{$}r<{$}}||*{6}{>{$}c<{$}}}
& \multicolumn{6}{c}{$m$} \\
\cline{2-7}
n & 0 & 1 & 2 & 3 & 4 & 5 \\
\hline
 5 & 0 & 0 & 0   &  0     &  0       &  1       \\
 4 & 0 & 0 & 0   &  0     &  1       &  0       \\
 3 & 0 & 0 & 0   &  1     &  0       &  q-1     \\
 2 & 0 & 0 & 1   &  0     &  q-1     &  0       \\
 1 & 0 & 1 & 0   &  q-1   &  0       & (q-1)q   \\
 0 & 1 & 0 & q-1 &  0     & (q-1)q   &  0       \\
-1 & 0 & q & 0   & (q-1)q &  0       & (q-1)q^2 \\
-2 & 0 & 0 & q^2 &  0     & (q-1)q^2 &  0       \\
-3 & 0 & 0 & 0   &  q^3   &  0       & (q-1)q^3 \\
-4 & 0 & 0 & 0   &  0     &  q^4     &  0       \\
-5 & 0 & 0 & 0   &  0     &  0       &  q^5     \\
\end{tabular}
\end{table}
Consider now the complex sequences $\boldphi=\{\phi_n\}_{n\in\mathZ}$, $\mb{f}=\{f_m\}_{m\in\mathN}$. Then the horospherical Radon transformation $\phi=\VRad f$ becomes $\boldphi=K^V \mb{f}$. Denote by $\mb{k}_m$ the $m$-th column of $K^V$: that is, ${(\mb{k}_m)}_n=k_V(n,m)$. Then any (left) inverse matrix $J=({K^V})^{-1}$ has rows $\mb{j}_n$ ($-\infty < n <\infty$) that satisfy the rule
\begin{equation}
\label{eq:dual_system}
\langle \mb{j}_n,\mb{k}_m \rangle
=\delta_{      n        m}.
\end{equation}
We now find left inverses by recurrence, by using the fact that the entries of the matrix ${K^V}$ vanish for $|n|>m$ and
are non-zero for $|n|=m$.
Indeed, consider \eqref{eq:inversion_row_by_column}: if we denote by $\mb{d}$ the sequence $\{d_n\}$, then
\begin{equation}\label{eq:left_inverse}
\langle \mb{d},\mb{k}_m \rangle
=\delta_{      0        m}
\end{equation}
In particular, notice that $\mb{j}_0=\mb{d}$.\\
In order to obtain the inversion formula, that is to recover the zero component of $\mb{f}$, we do not need to find the whole inverse matrix $J$, but only its row $\mb{d}:=\mb{j}_0$ of index $0$. As we are about to see, this leaves us countably many degrees of freedom.

Write $\mb{d}=\{d_n\}_{n=-\infty}^\infty$. The coefficients $d_n$ can be obtained as follows. First we observe that the zero column of $K^V$ is the identity vector ($1$ at $0$ and $0$ elsewhere), so $d_0=1$. Column $1$ of $K^V$ does not vanish only at rows $n=0$, $\pm 1$. Since we have already chosen $d_0=1$, this yields a linear relation between $d_1$ and $d_{-1}$, that can be solved up to a degree of freedom. The linear relation is $d_1+qd_{-1}=0$ , and for simplicity we may choose, for instance, $d_1=d_{-1}=0$. An alternative simple choice is $d_1=q$, $d_{-1}=-1$.

Observe that we have been forced to start with $d_0=1$, but the entries $\{k_V(n;m)\}$, vanish  if $n$ and $m$ have different parity. Therefore, in order to find a solution for \eqref{eq:left_inverse}, we may as well assume that
 $d_n=0$ for all $n$ odd (although, as we have just seen, this is not necessary). Let us make this assumption in the rest of this proof.

From column $2$ we obtain the linear relation $d_2+(q-1)d_0+q^2 d_{-2}=0$, and the fact that $d_0=1$ yields the two new coefficients $d_{\pm 2}$ up to an additional degree of freedom: two distinct possible choices are $d_{-2}=0$ and $d_2=1-q$, or else $d_{-2}=1$ and $d_2=1-q-q^2$.

Henceforth we consider separately the special cases $d_{n}=0$ and $d_{n}=1$, for all even $n<0$.

The next even column yields the relation
\begin{equation*}
          d_  4
+(q-1)    d_  2
+(q-1)q   d_  0
+(q-1)q^2 d_{-2}
+     q^4 d_{-4}=0.
\end{equation*}
In the first case we have set $d_{-2}=0$, hence $d_2=1-q$, and we choose $d_{-4}=0$: so we obtain $d_4=1-q$. In the second case we have set $d_{-2}=1$, hence $d_2=1-q-q^2$, and we choose $d_{-4}=1$, so $d_4=1-q-q^4$. This confirms the formulas in the statement for $d_n$ up to $n=4$. By induction, it is easy to verify the two formulas for every $n$. We do so for the first and leave the second to the reader. Just notice that, for arbitrary $n=|v|\geqslant 0$ and $m=2j-n$ \;($1\leqslant j \leqslant n-1$), the general form of the previous identity becomes
\begin{equation}\label{eq:relazione_di_ricorrenza_per_i_pesi_della_formula_di inversione_di_VRad}
d_{n}+ q^n d_{-n} + (q-1) \sum_{j=1}^{n-1} q^{j-1} d_{n-2j}=0.
\end{equation}

By~Lemma~\ref{lemma:homogeneous_intersection_cardinalities} (see also Table~\ref{Table:vertex_matrix_homogeneous}) wthis amounts to show that, if we set $M=\frac m2 -1$, then
\[
(1-q)\biggl(1+(1-q)(q-1)\sum_{j=0}^{M-1} q^j \biggr)+ (q-1) q^M=0
\]
for every $M>0$. This identity is trivially verified.

Observe that the two explicit expressions for the weights $d_n$ in the statement are interchanged via the symmetry implicit in the Cavalieri condition \eqref{eq:vertex_Cavalieri}.

It is clear from the proof that for each doubly infinite sequence of complex numbers $\{b_n\}_{n=-1,-2,\dotsc}$ we have a radial inversion formula such that, say, $d_n=b_n$. We omit the complicated expression of this general formula, although the recursive method that yields it is the same as above.
\end{proof}

\begin{corollary}[The full homogeneous inversion formula for $\VRad$]
\label{cor:homogeneous_inversion_for_VRad}
Let $T$ be a homogeneous tree of homogeneity degree $q$, $f$ a finitely supported (or $\ell^1$) function on $V$ and $\phi=\VRad f$. Then
there are uncountably many equivalent inversion formulas for $\VRad$,  of the form
\begin{equation*}\label{eq:factorization}
f(v)= \sum_{n=-\infty}^\infty d_n \int_\Omega \VRad f(\boldh(n,\omega,v))\,d\nu_{v}(\omega).
\end{equation*}
Some interesting examples of appropriate coefficients $d_n$ are given in the previous Theorem \ref{theo:radial_homogeneous_inversion_for_VRad}.
\end{corollary}
\begin{proof}
As we have already remarked, the previous theorem can be regarded as  recovering the value  $f(v_0)$ after radializing $f$. By Remark \ref{rem:horospherical_notation_for_Radon_transforms_of_radial_functions} this amounts to radializing $\VRad f$ around some vertex, say $v_0$. Therefore 
there are uncountably many equivalent inversion formulas for $\VRad$,  of the following form:
\begin{equation*}
f(v_0)=\sum_{n=-\infty}^\infty d_n \int_\Omega \phi(\omega,n)\,d\nu_{v_0}(\omega)\,,
\end{equation*}
where $n$ is the level of the horosphere tangent at $\omega$ with respect to the section $\Sigma_{v_0}$, and the integrand depends only on $n$ and $\omega$ because we have radialized $f$, and so its Radon transform is constant on the parallel under $n\in A$ of the section $\Sigma_{v_0}$.
If $\lambda\in\Aut T$ and $v=\lambda v_0$, by applying this inversion formula to $\lambda^{-1}f$ and making use of the equivariance relations \eqref{eq:equivariance_of_horospheres_under_automorphisms} and \eqref{eq:covariance_of_measure_nu}, we obtain
\begin{align*}
f(v)&=\sum_{n=-\infty}^\infty d_n \int_\Omega \VRad (\lambda^{-1}f)(\boldh(n,\omega,v_0))\,d\nu_{v_0}(\omega)\\[.2cm]
&=   \sum_{n=-\infty}^\infty d_n \int_\Omega \VRad f(\boldh(n,\lambda\omega,v))\,d\nu_{v_0}(\omega)\\[.2cm]
&=   \sum_{n=-\infty}^\infty d_n \int_\Omega \VRad f(\boldh(n,\omega,v))\,d\nu_{v}(\omega).
\end{align*}
\end{proof}

\subsection{The Plancherel theorem for $\VRad$}
\label{SubS:VRad-Plancherel}

\begin{proposition}\label{prop:Plancherel_for_VRad}
Let $\{d_n\}$ be one of the sequences of inversion coefficients given in the previous Theorem \ref{theo:radial_homogeneous_inversion_for_VRad}. Consider the space $\spacef$ of finitely supported functions on $V$ and let $\spaceF=\VRad \spacef$. Let us equip $\spacef$ with the $\ell^2-$inner product
\[
(  f,g ) = \sum_{v\in V} f(v)\,\overline {g(v)}
\]
and $\spaceF$  with the following inner product: for $\phi,\psi\in\spaceF$,
\begin{equation*}
( \phi, \psi  )' = \int_{\HorV} \phi(n,\omega)\,\overline{(\psi(\cdot\,,\omega)}* d^\dagger)(n) \,d\xi(\omega,n),
\end{equation*}
where $\xi$ is the invariant measure on $\HorV$ of Definition \ref{def:invariant_measure_on_Hor}, the convolution on the variable $n$ is the convolution on the structure group considered in Lemma \ref{lemma:convolution-on_fibers_is_independent_of_chart},
the $d_n$'s are the coefficients of any of the Radon inversion formulas of Theorem \ref{theo:radial_homogeneous_inversion_for_VRad},
and $d^\dagger_{n}=d_{-n}$.
\\
Then the vertex-horospherical Radon transform $\VRad$ is an isometry from $\spacef$ to $\spaceF$ with respect to these inner products.

A similar statement holds for $\ERad$, whose inversion formula will be presented in Subsection \ref{SubS:ERad-inversion} below.
\end{proposition}
%
\begin{proof}
We give the proof for $\VRad$: the case of $\ERad$ is analogous. By linearity, it is enough to prove the statement when $f=\delta_v$ and $g=\delta_w$, where $v,w$ are any two vertices.
By Lemma \ref{lemma:convolution-on_fibers_is_independent_of_chart}, we know that the expression of the convolution $\phi(\cdot\,,\omega)* d^\dagger$ does not depend on the choice of the global chart, that is, of the reference vertex $v_0$. Therefore, without loss of generality, we can choose $w=v_0$. Then $\psi=\VRad \delta_{v_0}$ is equal to 1 on the horospheres through $v_0$ and vanishes everywhere else: that is, for every $\omega$, $\psi(\omega,n)=1$ if $n=0$ and zero otherwise. This of course means that $\psi*d^\dagger(\omega,n)=d_{-n}$ for every $\omega$. Similarly, by writing 
\begin{equation}\label{eq:new_notation_for_sets_of_horospheres}
\Omega(n,v)=\{\omega:h(v,v_0,\omega)=n\},
\end{equation}
we have
$\phi(\omega,n)=\VRad \delta_v(\omega,n)=\nu_{v_0}(\Omega(n,v))$.
\\
Hence it follows from the expression of the invariant measure $\xi$ given in Definition \ref{def:invariant_measure_on_Hor} that
\begin{align}\label{eq:explicit_expression_of_the_inner_product_for_Plancherel_for_VRad}
\langle\langle \phi, \psi \rangle\rangle &= \sum_{n=-\infty}^\infty q^n\int_\Omega \phi(\omega,n)\,\overline{\psi*d^\dagger}(\omega,n)\,d\nu_{v_0}(\omega) \notag  \\[.2cm]
&= \sum_n q^n d_{-n} \int_\Omega \phi(\omega,n)\,d\nu_{v_0}(\omega).
\end{align}
We deal first with the special case $v=v_0$. Since $\Omega(0,v_0)=\Omega$ and $\Omega(n,v_0)=\emptyset$ for $n\neq 0$, in this case
the only non-zero term in the series at the right hand side of \eqref{eq:explicit_expression_of_the_inner_product_for_Plancherel_for_VRad} is for $n=0$, so $\langle\langle \phi, \psi\rangle\rangle = 1 = \langle \delta_{v_0},\delta_{v_0} \rangle$, and the statement holds.
\\
Then let us take $v\neq v_0$, that is, $|v|=n>0$. We must show that $\langle\langle \phi, \psi \rangle\rangle=\langle \delta_v,\delta_{v_0}\rangle = 0$. By \eqref{eq:explicit_expression_of_the_inner_product_for_Plancherel_for_VRad}, this means
\begin{equation}\label{eq:explicit_expression_of_the_inner_product_for_Plancherel_for_VRad_for_2_different_deltas}
\sum_n q^n d_{-n} \int_\Omega \phi(\omega,n)\,d\nu_{v_0}(\omega)=0.
\end{equation}
Then let us consider a generic sequence $d_n$ and find out under what conditions on the coefficients $d_n$ this identity holds. We want to show that it holds if and only if the coefficients $d_n$ satisfy the recurrence relations for the weights that yield the inversion of $\VRad$, namely \eqref{eq:relazione_di_ricorrenza_per_i_pesi_della_formula_di inversione_di_VRad}.
\\
Let $[v_0, v_1, \dots, v_j,\dots,v_n=v]$ be the geodesic path from $v_0$ to $v$.  As in Proposition \ref{prop:computation_of_vertex-spherical_functions}, let  $D_n=\Omega(v_0,v)=\Omega(v)$ be the  boundary arc subtended by $v$, and let $D_j=\Omega(v_j)\setminus \Omega(v_{j+1})$ and $D_0=\Omega\setminus \Omega(v_0,v_1)=\Omega(v_1,v_0)$: we have seen there that $h(v,v_0,\omega)=n$ if $\omega\in \Omega(v)$, and $h(v,v_0,\omega)=2j-n$
if $\omega\in D_j$ (in particular, $h(v,v_0,\omega)=-n$ if $\omega\in D_0$). With notation as in \eqref{eq:new_notation_for_sets_of_horospheres}, this means $D_j=\Omega(2n-j,v)$. 
 We have also computed there the $\nu_{v_0}-$ masses of the sets $D_j$, that we recall here for the reader's convenience: $\nu_{v_0}(D_0)=q/(q+1)$, $\nu_{v_0}(D_n)=q^{-n}/(q+1)$ and $\nu_{v_0}(D_j)=q^{_j}(q-1)/(q+1)$ for $0<j<n$. Hence \eqref{eq:explicit_expression_of_the_inner_product_for_Plancherel_for_VRad_for_2_different_deltas} becomes
\begin{align*}
0&=\frac q{q+1}\;d_{-|v|} q^{-|v|}
+\frac {q-1}{q+1}\sum_{j=1}^{|v|-1} q^{-j} d_{2j-|v|} q^{2j-|v|} \;\;+\frac q{q+1}\;q^{-|v|}d_{|v|}q^{|v|}\\[.2cm]
&=
\frac{q^{1-|v|}}{q+1} \Bigl( d_{-|v|} + (q-1) \sum_{j=1}^{|v|-1} q^{j-1}d_{2j-|v|}\;\;+q^{|v|} d_{|v|}\Bigr).
\end{align*}
But this is exactly \eqref{eq:relazione_di_ricorrenza_per_i_pesi_della_formula_di inversione_di_VRad}, and the proof is finished.
\end{proof}

\subsection{Direct computation of inversion formulas for $\ERad$}
\label{SubS:ERad-inversion}
Here again, it is enough to give an inversion formula for the edge-horospherical Radon transform of functions $f$ on $E$ that are edge-radial, that is, radial around the reference edge $e_0$. We remark again that the inversion formula for edge-radial functions is not unique. Here we derive two of them, particularly nice. One of them is given also for historical reasons: the same result has been obtained, through a more general approach, in~\ocite{Casadio_Tarabusi&Cohen&Picardello}. 
The computation follows the guidelines of the radial horospherical Radon inversion formula established in~Theorem~\ref{theo:radial_homogeneous_inversion_for_VRad}, once we calculate the cardinalities $k_E(n;m)$ of the intersection of the edge-circle of radius $m$ and the edge-horosphere of index $n$. This is done via the same argument of~Lemma~\ref{lemma:homogeneous_intersection_cardinalities}. The result is as follows.

\begin{remark} \label{rem:homog_edge_intersection_volumes}
For all $m\in\mathN$ and $n\in\mathZ$,
\begin{align*}
k_E(n;m) =\begin{cases}
1 &\text{if $m=n$,}\\
(q-1){q}^{(m-n-1)/2}  &\text{if $\abs{n}<m$}\\[-.1cm]
&\quad\text{and } n+m \text{ is odd,}\\
     q^m &\text{if $n=-m$,}\\
     0                                               &\text{otherwise.}\end{cases}
\end{align*}
\end{remark}
\begin{table}
\caption{Cardinality $k_E(n;m)$ of the intersection of the edge-circle of radius $m$ and the edge-horosphere of index $n$.}
\label{Table:edge_matrix_homogeneous}
\begin{tabular}{*{1}{>{$}r<{$}}||*{6}{>{$}c<{$}}}
& \multicolumn{6}{c}{$m$} \\
\cline{2-7}
n & 0 & 1 & 2 & 3 & 4 & 5 \\
\hline
 5 & 0 & 0   &  0     &  0       &  0       &  1       \\
 4 & 0 & 0   &  0     &  0       &  1       &  q-1     \\
 3 & 0 & 0   &  0     &  1       &  q-1     &  0       \\
 2 & 0 & 0   &  1     &  q-1     &  0       & (q-1)q   \\
 1 & 0 & 1   &  q-1   &  0       & (q-1)q   &  0       \\
 0 & 1 & q-1 &  0     & (q-1)q   &  0       & (q-1)q^2 \\
-1 & 0 & q   & (q-1)q &  0       & (q-1)q^2 &  0       \\
-2 & 0 & 0   &  q^2   & (q-1)q^2 &  0       & (q-1)q^3 \\
-3 & 0 & 0   &  0     &  q^3     & (q-1)q^3 &  0       \\
-4 & 0 & 0   &  0     &  0       &  q^4     & (q-1)q^4 \\
-5 & 0 & 0   &  0     &  0       &  0       &  q^5     \\
\end{tabular}
\end{table}

As a result, the argument of~Theorem~\ref{theo:radial_homogeneous_inversion_for_VRad} now yields the following inversion formulas.

\begin{theorem}[Radial inversion  of $\ERad$ on homogeneous trees]
\label{theo:radial_homogeneous_inversion_for_ERad}
Let $T$ be a homogeneous tree of homogeneity degree $q$. For the edge-horospherical Radon transform $\psi=\ERad f$ of finitely supported radial functions $f$ on $E$, there are uncountably many choices of complex coefficients $l_n$, $n\in\mathN$, such that the following inversion formula holds. With $\psi_n$ as in~Remark~\ref{rem:horospherical_notation_for_Radon_transforms_of_radial_functions}, then
\begin{equation*}
f(e_0)=\sum_{n\geqslant 0} l_n \psi_n
\end{equation*}
Two distinct explicit choices for the coefficients $l_n$ are
\begin{align*}
l_n &=\begin{cases}
       0                             &\text{if $n$ is odd or $n  <  0$,}\\[.1cm]
       1                             &\text{if               $n  =  0$,}\\[.2cm]
\dfrac{1-q}{1+q}\bigl(1-(-q)^n\bigr) &\text{otherwise,}
\end{cases}\\[.2cm]
l_n &=\begin{cases}
          (-1)^n                     &\text{if               $n\leqslant 0$,}\\[.2cm]
2\dfrac{1-(-q)^n}{1+q}-1             &\text{if               $n  >  0$.}\end{cases}
\end{align*}
\end{theorem}

As for the vertex-horospherical Radon transform, the radial inversion formula for $\ERad$ is trivially equivalent to the following non-radial one:

\begin{corollary}[The full homogeneous  inversion formulas for $\ERad$]
\label{cor:homogeneous_inversion_for_ERad}
Let $T$ be a homogeneous tree of homogeneity degree $q$, $f$ a finitely supported (or $\ell^1$) function on $E$ and $\psi=\ERad f$. Then
there are uncountably many equivalent inversion formulas for $\ERad$, all of the following form:
\begin{equation*}
f(v_0)=\sum_{n=-\infty}^\infty l_n \int_\Omega \psi(h(n,\omega))\,d\nu(\omega)\,.
\end{equation*}
Some interesting examples of appropriate coefficients $l_n$ were given in the Theorem \ref{theo:radial_homogeneous_inversion_for_ERad}.
\end{corollary}

\section[Intertwining of the  Radon transforms for vertices and edges]{Intertwining the horospherical Radon transforms on vertices and edges}
\label{Sec:hom_Intertwining}
We introduced in~Section~\ref{Sec:hom_V-Radon_and_E-Radon} a horospherical Radon transform $\VRad$ for functions defined on vertices, and another, $\ERad$, for functions defined on edges in a homogeneous tree $T$, both equivariant under the group of automorphisms of $T$. Then, in Subsection~\ref{SubS:VRad-inversion}
and
\ref{SubS:ERad-inversion}, we proved an inversion formula for these transforms. 
The horospherical Radon transforms  map  $\ell^2(V)$ onto a subspace of the $L^2$ space of vertex-horospheres, $L^2(\HorV,\mu)\sim L^2(\Omega\times\mathZ,\mu)$ equipped with the invariant measure $\mu$ on $\HorV$ of Subsection \ref{SubS:invariant_measure_on_Hor}, and
analogously for the edge horospherical Radon transform,
 but are not isometries in the respective norms.   
 
 To show that the image of $\VRad$ is contained in $L^2(\HorV,\mu)\sim L^2(\Omega\times\mathZ,\mu)$  it is enough to show that $\|\VRad f\|_{L^2(\Omega\times\mathZ, \mu)}=\|f\|_{\ell^2(V)}$ when $f$ is the characteristic function of a vertex $v$ (a similar argument holds for edges). To prove this identity, observe that,
by definition, for every $v\in V$, $\VRad \chi_{v}$ is the characteristic function of the set of horospheres that contain $v$. In the global chart in $\HorV$ given by the choice of reference vertex, one has $\VRad \chi_{v}=\chi_{\Omega\times \{0\}}$. Since the measure $\mu$ is invariant under change of reference vertex, this shows that $\|\VRad \chi_{v}\|_{L^2(\Omega\times \mathZ,\mu)}=
\nu_{v}(\Omega)=1=\|\chi_{v}\|_{\ell^2(V)}$, by the definition \ref{def:invariant_measure_on_Hor} of $\mu$. Hence, on each $\delta_v$, $\VRad$ acts as an isometry.

However, the horospherical Radon transforms are not isometries everywhere with respect to the norms mentioned above, because they do not preserve the corresponding inner products. Indeed, given two vertices $v\neq w$, the functions $\delta_v$ and $\delta_w$ are orthogonal, because their support are disjoint, but their horospherical Radon transforms are not: $\VRad \delta_v$ is the characteristic functions of the set of horospheres through $v$, and this set as an pen intersection with the set of horospheres through $w$. A similar argument holds for edges.

Let us now consider an interesting related problem. It is easy to assign a natural bijection  between the spaces of edges and vertex horospheres (see Proposition \ref{prop:horosphere_correspondence} later). Any such bijection allows to define the composition $\ERad^{-1}\circ\VRad$, provided that the ranges of $\VRad$ and $\ERad$ overlap under this bijection. This might give an equivariant isometry from $\ell^2(V)$ to $\ell^2(E)$. But we shall see that it is not so, because of the the fact that the intersection of the images of $\ERad$ and $\VRad$ consists only of the null function. Let us first point out why the construction of such an isometry via horospherical Radon transforms would be surprising.

Typically, an equivariant map from $\ell^2(V)$ to $\ell^2(E)$ is obtained by starting with an equivariant map from $V$ to $E$, for instance by considering the automorphisms given by the left action of a free product on the tree labelled as its Cayley graph (with respect to a set of generators), as explained in Example \ref{example:trees_as_Cayley_graphs_of_free_products}. This family of automorphisms lifts to characteristic functions of vertices, with values on characteristic functions of edges. Hence it lifts to all functions on vertices, with image given by all functions on edges.
But an equivariant map of this type cannot arise from $\ERad^{-1}\circ\VRad$, in view of the last Remark. Indeed, in the notation adopted there, $\VRad \chi_{v_0}=\sum_{k=0}^n \chi_{W_k\times\{2k-n\}}$, where the left hand side is identified with  the right hand side via the choice of reference vertex $v_0$. If we choose and fix a reference edge $e_0$, the right hand side can be reinterpreted as the sum of characteristic functions in $\HorE$, but the result is not the characteristic function of an individual edge, because of the asymmetry between the geometric meaning of the vertex-horospherical index and the edge-horospherical index, explained in Remark \ref{rem:subtle_difference_between_horospherical_indices_for_vertices_and_for_edges} in terms of joins.

Analogously, let $\psi=\ERad \chi_{e_0}$, where $e_0$ is the fixed reference edge. Clearly, the function $\psi$ on the space $\HorE$ of edge-horospheres is given by $\psi(\boldh^E)=1$ if $e_0 \in\boldh^E$ and $0$ otherwise: that is, with the usual parametrization induced by $e_0$, $\psi(\omega,n)= 1$ for every $\omega$ if $n=0$, and $\psi\equiv 0$ if $n\neq 0$. Let now $f={\VRad}^{-1}\psi$. We can immediately compute $f$ by the inversion formula for $\VRad$ proved in~Subsection~\ref{SubS:VRad-inversion}. For every vertex $v$ denote by $n$ its distance from $e_0$: then
\begin{equation*}
f(v)=\begin{cases}
 \dfrac 1   { q+1    } &\text{for      $n=0$,}\\[.4cm]
 \dfrac{q-1}{(q+1)q^n} &\text{for  odd $n>0$,}\\[.4cm]
-\dfrac{q-1}{(q+1)q^n} &\text{for even $n>0$.}\end{cases}
\end{equation*}
Now, if $\VRad\phi=\psi=\ERad\chi_{e_0}$ for some $\phi$ on $\HorV$, one would expect that $\VRad f=\VRad\VRad^{-1}\psi=\psi=\ERad\chi_{e_0}$, but it is easily verified that this is false. So $\VRad^{-1}$ is a left inverse of $\VRad$ (by definition), but not a right inverse, and the function $\psi$ is likely not in the range of the operator $\VRad$. In the rest of this Section we shall prove this fact and more generally we shall show, as claimed above, that $\ERad^{-1}\circ\VRad$ does not provide a surjective isometry from $\ell^2(V)$ onto $\ell^2(E)$. Namely we shall show
that the images in $L^2(\Omega\times\mathZ)$ of $\VRad$ acting on $\ell^2(V)$ and of $\ERad$ on $\ell^2(E)$   do not coincide (indeed, the overlap is small), and so the composition $\ERad^{-1}\circ\VRad$ does not make sense on all of $\ell^2(V)$.

\subsection{The ranges of the vertex and edge horospherical Radon transforms do not overlap}
\label{SubS: different_horospherical_ranges}
Consider functions $f$ on $V$ and $g$ on $E$, and set $\phi=\VRad f$, $\psi=\ERad g$. We now prove that, for most functions $f$ on $V$, there is no function $g$ on $E$ such that
\begin{equation}
\label{eq:the_range_problem}
 \phi(  \boldh )
=\psi(\Xi {\boldh})
\qquad\text{for every }\quad \boldh\in\HorV
\end{equation}
where
$\Xi$ is the vertex-edge horospherical correspondence of
Proposition\ref{prop:horosphere_correspondence}. From now on, for every function  $\phi$ on $\HorV$,  we  write  $\Xi \phi(\boldh^E)= \phi(\Xi^{-1}\boldh^V)$: so the previous equality can be written as $ \phi =\Xi \psi$.
Let us apply to the present context the conditions for the range of the vertex and edge horospherical Radon transforms, stated in Theorems \ref{theo:VRad-range} and \ref{theo:ERad-range}.
\begin{definition}
Fix a reference edge $e_0$ and a reference vertex $v_0$ that belongs to $e_0$. Then both  spaces of vertex-horospheres $\HorV$ and of edge-horospheres $\HorE$ are parameterized by $\Omega\times\mathZ$, as in~Section~\ref{Sec:hom_V-Radon_and_E-Radon} (see Theorems \ref{theo:VRad-range} and \ref{theo:ERad-range}). The \emph{vertex-horospherical reflection around $v_0$} is the map $\sigma:
\HorV\to
\HorV$ given by
\begin{equation*}
 \sigma(\boldh_{n} (\omega, v_0))
=            \boldh_{-n}(\omega, v_0)
\end{equation*}
(here $n$ is the vertex-horospherical index induced by $\omega$ and the reference vertex $v_0$). Similarly, the \emph{edge-horospherical reflection around $e_0$} is the map of $\HorE$ into itself given by
\begin{equation*}
\tau(\boldh_{n} (\omega, e_0))
=            \boldh_{-n}(\omega, e_0).
\end{equation*}
\end{definition}

Denote by $\nu_{v_0}$ the equidistributed boundary measure with respect to the reference vertex $v_0$, and by $\nu_{e_0}$ the equidistributed boundary measure with respect to the reference edge $e_0$.

For functions $\phi$ on $\HorV$ and $\psi$ on $\HorE$ respectively, the Cavalieri conditions \eqref{eq:vertex_Cavalieri} and \eqref{eq:edge_Cavalieri} can be rewritten as follows. 
\begin{align}
\label{subeq:vertex_Cavalieri_revised}
  q^n \int_\Omega \phi(                   \boldh_n (\omega,v_0)) \,d\nu_{v_0}(\omega)
&=    \int_\Omega \phi(\sigma(        \boldh_n (\omega,v_0))) \,d\nu_{v_0}(\omega),\\
\notag
  q^n \int_\Omega \psi(\Xi             \boldh_n (\omega,v_0))  \,d\nu_{e_0}(\omega)
&=    \int_\Omega \psi(\tau(\Xi \boldh_n (\omega,v_0) ))\,d\nu_{e_0}(\omega).
\end{align}

The following example shows that one should not expect to have functions on vertices whose vertex-horospherical Radon transform coincides, via the canonical association, with the edge-horospherical Radon transform of a function on edges.
\begin{example}
\label{example:characteristic_function_of_reference_edge_not_in_joint_ranges}
Let us consider the simplest case: $\psi=\ERad \chi_{e_0}$. So, $\psi(\Xi{\boldh})=1$ if $e_0\in\Xi{\boldh}$, and $\psi(\Xi{\boldh})=0$ otherwise. It follows from~Lemma~\ref{prop:horosphere_correspondence} that the corresponding vertex-horospherical function $\phi$ is the characteristic function of the set of vertex-horospheres $\boldh$ that contain the vertex of $e_0$ opposite to $\omega$. That is, with $e_0=[v_0,v]$:
\begin{equation}\label{eq:VRad_of_characteristic_function}
\phi(\boldh)=\begin{cases}
1 &      \text   {if $\omega\in S(v_0,v)$  and $v_0\in\boldh$,}\\[-.1cm]
  &\qquad\text{or if $\omega\in S(v,v_0)$ and $v\in\boldh$,}\\
0 &      \text{otherwise.}\end{cases}
\end{equation}

Let us verify~condition~\eqref{subeq:vertex_Cavalieri_revised}, that is, \eqref{eq:vertex_Cavalieri}: 
one must verify that 
\[
q^n \int_\Omega \phi(\omega,n)\,d\nu_{v_0}(\omega) = \int_\Omega \phi(\omega,-n)\,d\nu_{v_0}(\omega),
\] 
for every $n$. In the parametrization $h\sim(\omega,n)$, \eqref{eq:VRad_of_characteristic_function} amounts to $\phi(\omega,n)=1$ if $\omega\in S(v_0,v)$ and $n=0$, or if $\omega\in S(v,v_0)$ and $n=-1$, and $\phi(\omega,n)=0$ otherwise. So we only need to consider the integers $n=0$ or $-1$. The case $n=0$ is trivially satisfied. For $n=-1$ one has $\int_\Omega \phi(\omega, -1)\,d\nu_{v_0}(\omega) = \int_{S_-(e_0)} \,d\nu_{v_0}(\omega) > 0$, but $\int_\Omega \phi(\omega, 1)\,d\nu_{v_0}(\omega) = 0$. Therefore the Cavalieri conditions are not satisfied, and $\phi$ is not in the range of the vertex-horospherical Radon transform.
\end{example}

\begin{remark}[Cavalieri conditions as radial sums]\label{rem:Cavalieri_conditions_as_radial_means}
The equivariance property \eqref{eq:Aut_acts_equivariantly_on_the_horospherical_index} means that the horospherical Radon transform $\VRad$ commutes with automorphisms, in particular with the automorphisms that fix a reference vertex $v_0$ (or a reference edge $e_0$). If $f$ is a function on $V$, 
let $\phi=\VRad f$ and $\psi=\Xi\phi\equiv \phi\circ \Xi^{-1}$, and let $\phi(\omega,n)$ and $\psi(\omega, n+\frac12)$ be the integrands in the Cavalieri conditions \eqref{eq:vertex_Cavalieri} and \eqref{eq:edge_Cavalieri-mixed} centered at $v_0$ (whose values of course depend on the choice of $v_0$). It
follows from  equivariance (Corollary \ref{cor:equivariance}) that, if  the Cavalieri conditions hold for  $\psi$ with respect to $v_0$, then they hold for $\Xi{{\VRad \mathfrakE_{v_0} f}}$, where $\mathfrakE_{v_0}$ is the radialization around $v_0$ introduced in Remark \ref{rem:horospherical_notation_for_Radon_transforms_of_radial_functions} (remember also that 
 $\VRad \mathfrakE_{v_0} = \mathfrakE_{v_0} \VRad$, as we noticed in the same Remark). 
The same holds if we start with functions $g$ on edges. 

In particular, both sides of the vertex- or edge-Cavalieri conditions (and also of the mixed conditions) can be expressed as series with appropriate coefficients of sums over circles centered at $v_0$ (respectively, $e_0$) of values of $f$ (respectively, $g$). Namely, let $\chi_n$ be the the linear functional on the space of functions on $V$ defined by $\chi_n (f) = \sum_{v:\dist(v,v_0)=n} f(v)$ (note that $\chi_n$ depends on the choice of $v_0$; for related more precise notation see Remark \ref{rem:chi_n} below). Then we have
\begin{align*}
\int_\Omega \phi(\omega,n)\,d\nu_{v_0}(\omega) &= \sum_{j\geqslant 0} a_j^{(n)} \chi_j(f)  ,  \\[.2cm]
\int_\Omega \psi\left(\omega,\,n+\frac12\right)\,d\nu_{v_0}(\omega) &= \sum_{j\geqslant 0} b_j^{(n+\frac12)} \chi_j(f) 
\end{align*}
for suitable coefficients $a_j^{(n)}$ and $b_j^{(n+\frac12)}$ that depend on $n$ (but, by equivariance, not on the choice of $v_0$). For functions $g$ on edges the same relation expresses the integrals of their horospherical Radon transforms with respect to $\nu_{e_0}$ in terms of
 their edge-circle sums.  

We can make further progress by observing that the value $\phi(\omega,n)$ of $\phi=\VRad f$ is the sum of the values of $f$  on the horosphere with parameters $(\omega,n)$ with respect to $v_0$, and the vertex closest to $v_0$ of this horosphere has distance $n$ from $v_0$. The same holds for the values $\psi(\omega,n+\frac12)$ of the  function $\psi=\Xi\phi$ given by the canonical association
of Proposition \ref{prop:horosphere_correspondence}. Therefore  the coefficients in the expansion above vanish for $j<|n|$, hence:
%
\begin{align}\label{eq:expansion_of_integrals_of_VRad}
\int_\Omega \phi(\omega,n)\,d\nu_{v_0}(\omega) &
= \sum_{j\geqslant |n|} a_j^{(n)} \chi_j(f) ,
\\[.2cm]
\int_\Omega \psi\left(\omega,\,n+\frac12\right)\,d\nu_{v_0}(\omega) &
= \sum_{j\geqslant |n|} b_j^{(n+\frac12)} \chi_j(f) .
\end{align}
Observe that the diagonal coefficients $a_n^{(n)}$ and $b_n^{(n+\frac12)}$ do not vanish: this is easily verified by choosing $f$ as the Dirac delta function supported on a vertex at distance $n$ from $v_0$ (see Proposition \ref{prop:the_images_of_ell^1_under_VRad_and_ERad_are_disjoint} below for an exact computation of these coefficients).
\end{remark}

\begin{theorem}
\label{theo:Cavalieri_conditions_as_upper-triangular_linear_operators_on_radial_means}
With notation as in Remark \ref{rem:Cavalieri_conditions_as_radial_means}, let $f$ be a function defined on vertices and set $\phi=\VRad f$ and $\psi=\Xi\phi$. 
Denote by $\chi^V_n(v_0,\cdot)$ the kernel given by the characteristic function of the vertex-circle  $\{w\in V: \dist(w,v_0)=n$, $n\geqslant 0\}$, and let
\begin{equation}\label{eq:circle-sum_as_convolution_operator}
                  \chi^V_n     f(v_0)
=\sum_{w\in V} \chi^V_n(v_0,w)f(w)
\end{equation}
(for the sake of simplicity, here we shall write $\chi_n(f)=\chi^V_n f(v_0)$).
\\
Then $\psi$  satisfies the mixed edge-Cavalieri conditions \eqref{eq:edge_Cavalieri-mixed} at $v_0\in V$ for $n\in\mathN$ if and only if the following upper-triangular linear system holds:
\begin{equation}\label{eq:mixed_Cavalieir_conditions_as_an_upper_triangular_linear_system}
q^{n+1} \sum_{j\geqslant n} b_j^{(n+\frac12)} \chi_j(f) - \sum_{j\geqslant n+1} b_j^{(-n-\frac12)} \chi_j(f) = 0.
\end{equation}
These mixed edge-Cavalieri conditions hold at $v_0$ for every $n$ if and only if the radialization of $f$ around $v_0$ is identically 0.

Hence, for every $f\in\ell^1(V)$, the edge-Cavalieri conditions for $\psi=\Xi{\VRad f}$ are satisfied at every $v_0$ and $n$ if and only if $f\equiv 0$: that is, there is no non-zero vertex-function whose vertex-horospherical Radon transform is canonically associated to a function 
that belongs to the image of $\ERad$. In particular, the images of the horospherical Radon transforms $\VRad$ and $\ERad$ (under the canonical correspondence) intersect only at the zero function.
\end{theorem}

\begin{proof} The result follows from the previous Remark \ref{rem:Cavalieri_conditions_as_radial_means} except for the last statement: a function $f$ such that $\psi=\Xi{\VRad f}$ is an edge-Radon transform must be identically zero. This follows from the fact that the matrix of the linear system \ref{eq:mixed_Cavalieir_conditions_as_an_upper_triangular_linear_system} is upper triangular with non-zero diagonal entries, hence injective. Therefore the only solution of the linear system is the zero sequence $\{\chi_n(f)\}$, that is, $f$ has zero radial averages around every vertex $v_0$ for every radius $n$. Then the fact that $f(v_0)=0$ follows trivially by choosing $n=0$.

Observe that the argument that proves $f\equiv 0$ does not really need the fact that the matrix of the linear system is upper-triangular, but only the weaker result that its first column is zero except at the first entry.
\end{proof}

For the sake of completeness, we  compute directly the coefficients of the expansions in Remark \ref{rem:Cavalieri_conditions_as_radial_means}:
\begin{proposition}\label{prop:the_images_of_ell^1_under_VRad_and_ERad_are_disjoint}
Let $f\in \ell^1(V)$.
With notation as in Theorem \ref{theo:Cavalieri_conditions_as_upper-triangular_linear_operators_on_radial_means}, the edge-horospherical function
$\psi$  canonically associated to $\phi=\VRad f$ satisfies the mixed edge-Cavalieri conditions \eqref{eq:edge_Cavalieri-mixed} at $v_0\in V$
for $n=0$ if and only if
\begin{equation}\label{eq:edge_Cavalieri-mixed_for_psi,row_n=0}
q(q+1) \chi_0(f) - q\chi_{1}(f) +(q-1)\sum_{j=1}^\infty\left(q^{1-j}\chi_{2j}(f)-q^{-j}\chi_{2j+1}(f)\right)=0,
\end{equation}
and it satisfies \eqref{eq:edge_Cavalieri-mixed} at $v_0$ for $n>0$ if and only if
\begin{equation}\label{eq:edge_Cavalieri-mixed_for_psi,row_n>0}
q^2 \chi_n(f) - q\chi_{n+1}(f) +(q-1)\sum_{j=1}^\infty\left(q^{1-j}\chi_{2j+n}(f)-q^{-j}\chi_{2j+n+1}(f)\right)=0.
\end{equation}
\end{proposition}
\begin{proof}
For $n=0$, the left hand side of the mixed edge-Cavalieri condition \eqref{eq:edge_Cavalieri-mixed} at $v_0$ is $q\int_\Omega \psi(\omega, \frac12)\,d\nu_{v_0}(\omega)$. For every $\omega$, let $e_0(\omega)=[v_0,v_1(\omega)]$ be the edge that touches $v_0$ and whose other end-vertex $v_1(\omega)$ belongs to the geodesic ray from $v_0$ to $\omega$ (that is, the edge starting at $v_0$ in the direction of $\omega$).
Then the edge-horosphere $\boldh^E_{1/2}(\omega,e_0)$ tangent at $\omega$ of mixed index $\frac12$ is the edge-horosphere through $e_0(\omega)$ all contained in the edge-sector $S(v_0,v_1(\omega))$, introduced in Subsection \ref{SubS:Joins}, consisting of all edges at the other side of $v_0$ with respect to $e_0(\omega)$. The number $\psi(\omega, \frac12)=\Xi{\VRad f}(\omega,\frac12)$ is the value of $\VRad f$ on the vertex-horosphere $\Xi{\boldh^E_{1/2}}(\omega,e_0)=\boldh^V_0(\omega, v_0)$ through $v_0$ associated to $\boldh^E_{1/2}(\omega,e_0)$ via the canonical association of Proposition \ref{prop:horosphere_correspondence}. This vertex-horosphere consists of the vertices that are the endpoints farther away from $\omega$ of the edges in $\boldh^E_{1/2}(\omega,e_0)$. In particular, the vertex $v_0$ belongs to $\boldh^V_0(\omega, v_0)$. Then the horosphere contains $q-1$ other vertices at distance 2 from $v_0$
in the sector $S(v_0,v_1(\omega))$ that belong to $\boldh^V_0(\omega, v_0)$: namely, the neighbors of $v_1(\omega)$ not belonging to the geodesic ray from $v_0$ to $\omega$. Next, the horosphere contains $q(q-1)$ vertices in the same sector at distance 4 from $v_0$, reached from the descendant $v_2(\omega)$ of $v_0$ two generations down along the geodesic ray $[v_0,\omega)$ by moving two  steps sideways with respect to this geodesic ray. The other vertices of $\boldh^V_0(\omega, v_0)$ are obtained similarly. Hence, the number $\psi(\omega, \frac12)$ is obtained by summing the values of $f$ over these vertices. The integral $\int_\Omega \psi(\omega, \frac12)\,d\nu_{v_0}(\omega)$ is given by integrating over $\Omega$ a sum of the values of $f$ over the vertices $v$ in the horosphere $\boldh^V_0(\omega, v_0)$, each weighted
by the $\nu_{v_0}$-mass of the boundary arc $U(v_0,v)$. These $\nu_{v_0}$-masses depend only on the distance $k=\dist(v,v_0)$, indeed $
\nu_{v_0}(U(v_0,v))=\frac 1{(q+1)q^{k-1}}$. Therefore, by equivariance (Corollary \ref{cor:equivariance}),  integration over $\Omega$ yields
\[
\int_\Omega \psi\left(\omega, \frac12\right)\,d\nu_{v_0}(\omega) =  \chi_0(f) +\frac{q-1}{q+1}\sum_{j=1}^\infty q^{-j} \chi_{2j}(f).
\]
Let us now look at the right hand side of \eqref{eq:edge_Cavalieri-mixed} at $v_0$ for $n=0$, that is, the integral $\int_\Omega \psi(\omega, -\frac12)\,d\nu_{v_0}(\omega)$. For every $\omega$, we now need to compute the value $\psi(\omega, -\frac12)$: this is a sum over the edges of the edge-horosphere $\boldh=\boldh^E_{-1/2}(\omega,e_{-1})$, where $e_{-1}$ is any edge touching $v_0$ different from $e_1(\omega)$. The horosphere $\boldh$ contains all these $q$ edges, but now, by Proposition \ref{prop:horosphere_correspondence},
 the value of $\psi$ on $\boldh$ is the sum for $e\in \boldh$ of the values of $f$ at the end-vertex of $e$ at the other side of $\omega$. Exactly the same argument as before now yields the right hand side of \eqref{eq:edge_Cavalieri-mixed} in the form
\[
\int_\Omega \psi\left(\omega,\, -\frac12\right)\,d\nu_{v_0}(\omega) = \frac q{q+1} \chi_1(f) + \frac{q-1}{q+1}\sum_{j=1}^\infty q^{-j}\chi_{2j+1}(f).
\]
By putting together the last two identities we see that the mixed Cavalieri condition at $v_0$ for $n=0$ is given by \eqref{eq:edge_Cavalieri-mixed_for_psi,row_n=0}.

The same idea can be used to prove \eqref{eq:edge_Cavalieri-mixed_for_psi,row_n>0}. Indeed, for $n>0$, at the left-hand side of \eqref{eq:edge_Cavalieri-mixed_for_psi,row_n>0} we need to consider the edge-horosphere tangent at $\omega$ that contains the edge $e_n(\omega)$ of the $n-$th generation in the geodesic $[v_0,\omega)$ as the closest edge to $v_0$. The value of $\psi$ on this horosphere is the sum of the values of $f$ on the end-vertices of all its edges that are at the opposite side of $\omega$. Instead, for the left hand side of  \eqref{eq:edge_Cavalieri-mixed_for_psi,row_n>0}, we need to consider the edge-horosphere tangent at $\omega$ that contains the edges that are ancestors of the $n-$th generation of $e_0(\omega)$ with respect to $\omega$, and the values of $f$ on the end-vertices opposite to $\omega$ of the edges of this horosphere. The same argument as before now expresses the Cavalieri condition for $n>0$ as identity \eqref{eq:edge_Cavalieri-mixed_for_psi,row_n>0}.
\end{proof}

\part{Zonal spherical functions on homogeneous trees via 
harmonic analysis}\label{Chap:spherical_functions}
This chapter deals with harmonic analysis and zonal spherical functions on a homogeneous tree from the viewpoint of 
convolution equations, and its application to the inversion of the horospherical Radon transforms via back-projection. For functions on vertices, all this goes back to \cites{Figa-Talamanca&Picardello-JFA, Figa-Talamanca&Picardello,Betori&Faraut&Pagliacci} and references therein. For functions on edges, spherical functions on the associated graph and the spectrum of its Laplace operator were introduced in \cite{Iozzi&Picardello-Springer}. A comprehensive study of spherical functions, spectra, Radon back-projection and Radon inversion for edges will appear in \cite{Casadio_Tarabusi&Picardello-spherical_functions_on_edges}.

\section[The theory of spherical functions]{The theory of spherical functions on vertices and on edges of homogeneous trees via difference equations}
\subsection{Convolutions} \label{SubS:convolution}
For homogeneous trees, the full automorphism group acts transitively on vertices and on edges, and therefore it defines a convolution. 

Let $\Gamma_{v_0}, \Gamma_{e_0}$ 
the stability subgroups of $\Gamma:=\Aut T$ at $v_0$, $e_0$ 
respectively, 
that were introduced in Remark \ref{rem:automorphisms} and Definition \ref{def:convolutions&Laplacians} (in order to conform with the established terminology, and with Definition \ref{def:convolutions&Laplacians}, we shall write $K$ instead of 
$\Gamma_{v_0}$, $\Gamma_{e_0}$ 
when no ambiguity arises). 
Clearly, the right coset space $\Gamma/\Gamma_{v_0}$ is in bijective correspondence with $V=V$, and similarly for $\Gamma/\Gamma_{e_0}\approx E$. 
Note that this identifies the right $K-$invariant functions on $\Gamma$ 
with the functions on $V$, $E$, 
respectively.

We have shown in Example \ref{example:trees_as_Cayley_graphs_of_free_products} how to choose a subgroup $G^V\subset \Gamma$  that is simply transitive on $V$, that is, with trivial stability subgroups. Similarly. we can choose a subgroup $G^E$ that acts simply transitively on $E$.

Then all vertices $v$ are uniquely labeled by elements $g_v\in G^V$ by the rule $v \mapsto g_v(v_0)$: this labeling is unique because $G^V$ is simply transitive on $V$. A similar labeling holds for edges in terms of $G^E$.
Then we can regard vertices and edges  as elements of the groups $G^V$, $G^E$. If $f,h$ are functions on $V$ (i.e., right $K-$invariant functions on $G^V$), we define 
\begin{equation}\label{eq:homog_convolution}
f*h(v)=\sum_{w\in V} f(w)\,h(w^{-1}v)\,,
\end{equation}
where, of course, if $w\in V$ is labeled by $g_w\in G^V$, then $w^{-1}$ is the vertex labeled by $g_w^{-1}$. The definitions of convolution for functions on $E$ 
are the same word by word. We stress the fact that the convolution of two (right $K-$invariant) functions depends on the labeling, that is on the choice of the subgroup of $\Aut T$ that defines such labeling.

If $G^V$, $G^E$ 
are doubly transitive, the action of the respective stabilizers $K$ fixes $v_0$, $e_0$
respectively and is transitive on all the circles, that is,
 the sets of elements at constant distance from these reference elements. Therefore the bi$-K-$invariant functions on $G^V$, $G^E$
 are precisely the radial functions on $V$, $E$ 
 : those which depend only on the distance from the respective reference elements. Then, for two-sided $K-$invariant functions, the convolution becomes
 \[
f*h(x)=\sum_{w} f(\dist(w,\,u_0))\,h(\dist(x,\,w))\,,
\] 
where $u_0=v_0$, $e_0$, 
respectively. But of course the distance between vertices, or edges, or flags does not depend on the particular choice of subgroup of the automorphism group. 
 Therefore, the convolution on bi-$K-$invariant functions does not depend on the chosen labeling.
 \\
 More generally, the convolution product in $\Aut T$ is well defined when the first function is  bi-$K$-invariant and the second is  right-$K$-invariant. Indeed,
let  $\mu$ be the Haar measure  on $G$ normalized on $K$. The  definition of convolution for functions on $G$ is $u_1*u_2(\tau)= \int_G u_1(\lambda^{-1}\tau)\,u_2(\lambda)\,d\mu(\lambda)$, since $G$ is unimodular. Let $\lambda\mapsto \widetilde\lambda$ be the canonical projection from $G$ to $G/K$, $\widetilde\mu$ the quotient measure of $\mu$ on $G/K$ and $u_0$ be the coset corresponding to the reference element ($u_0=v_0$ or  $e_0$, respectively). 
If $u_2$ is right $K$-invariant and $u_1$ bi-$K$-invariant, we have 
\begin{align*}
u_1*u_2(\widetilde\tau)&=\int_{G/K}\int_K  u_1(\kappa^{-1}\lambda^{-1}\tau)\,u_2(\lambda\kappa)\,d\mu(\kappa)\,d\widetilde\mu(\widetilde\lambda)\\[.2cm]
&= \int_{G/K} u_1\left({{\lambda^{-1}\tau}}\right)\,u_2(\lambda)\,d\widetilde\mu(\widetilde\lambda)
= \int_{G/K} u_1\left({\widetilde{\lambda^{-1}\tau}}\right)\,u_2(\widetilde\lambda)\,d\widetilde\mu(\widetilde\lambda)
.
\end{align*}

Convolution can be handled in a simpler way if we produce a suitable cross section in $G/K$, and precisely a subgroup $\Gamma$ of $G$ such that the quotient map $\Gamma\to G/K$ is 
one to one and onto. 
We would like to choose such subgroup $\Gamma$ based on the geometry of the tree. 
The way to do so has been known for a long time. Label every edge that touches a reference vertex $v_0\in V$ with a letter $a_0, \dots, a_q$.  
Then label with the same set of letters all edges stemming from a neighbor $v_1$ of $v_0$: 
 to the edge $[v_0, v_1]$ we assign the same letter that was assigned at the first step, and so on.  Now label each edge-path $\{v_0, v_1, v_2\}$ of length 2 by  the word $a_{j_1}a_{j_2}$, where $a_{j_1}$  is the letter associated to $[v_0,v_1]$ and $a_{j_2}$ is  associated to $[v_1,v_2]$.
 Iterate this procedure to label all finite paths of edges starting at $v_0$ (that is, all vertices) with words in the letters $a_0, \dots\, a_q$: the labeling is well defined because the tree has no loops.
\par
Consider the semigroup of all words in the letters $a_0, \dots,a_q$, with the product given by juxtaposition of words. Make this semigroup into a group by setting $a_j^{-1}=a_j$. The identity element of this group, that is the empty word, is associated to $v_0$. Limiting attention to reduced words, that is,  dropping all the words that contain two consecutive identical letters, we see that $V$ is in bijective correspondence with this group, that is isomorphic to the direct product $\Gamma=*_{j=0}^q 
\mathZ_2
$ because $T$ has no loops. Each reduced word $w=a_{j_1}\dots a_{j_n}$ can be regarded as a finite geodesic path in $V$, starting at $v_0$: each  vertex $w_{j}\neq v_0$ of this path is obtained from its predecessor $w_{j-1}$ by  multiplying $w_{j-1}$, regarded as an element of the group $\Gamma$, by a letter $a_{k_j}$ on the right.  Therefore the left translation action of $\Gamma$ onto itself becomes an action on this set of words, isometric in the  metric of $V$ (the natural distance in $T$), that is, preserving adjacency. Thus $\Gamma\subset \Aut T$. Moreover, $\Gamma\cap K_{v_0}$ is the identity element, that is the empty word.  $\Gamma$ is transitive, and clearly $T$ is the Cayley graph of $\Gamma$ with generators $a_0, \dots, a_q$.
If the homogeneity degree $q$ is odd, i.e.~there is an even number $q+1=2r$ of neighbors, then, by a similar argument, another simply transitive subgroup of $\Aut T$ is isomorphic to the free group $\mathF_r=*_{j=1}^r \mathZ$ \cites{Cartier-SeminaireBourbaki,Figa-Talamanca&Picardello}.  

We can build a similar labeling for edges, but there is a difference. 
 Let us regard $E$ as the set of vertices of some graph $\graphE$, whose edges correspond to pairs of adjacent edges in $E$. Then $\graphE$ has loops, and this leads to a simply transitive group of isometries that
is not a free product of $q+1$ copies of $\mathZ_2$ or a free group, but the free product $\mathZ_{q+1}*\mathZ_{q+1}$, where    $\mathZ_{q+1}$ is the cyclic group with $q+1$ elements. 

Indeed, the Cayley graph of this group (with respect to the generators of its factors) has $E$ as its set of vertices. More precisely, $E$ can be regarded as a graph 
via the dual graph construction, that identifies each edge 
with a vertex of the associated graph. 
Two vertices 
 of this graph are contiguous if the corresponding edges in $E$ join at a vertex $v\in V$. Then the graph
$\graphE$ is the symmetric graph of complete polygons of $q+1$ sides, defined in  \cite{Iozzi&Picardello}. The vertices of each polygon in this graph 
correspond to the $q+1$ edges in $E$ that share a given endpoint; each vertex in the graph 
belongs to exactly two polygons, namely those  corresponding to the two endpoints in $V$ of the edge in $E$ associated to that vertex of the graph.
For each pair of vertices in the graph,
the shortest path connecting them lies in only one chain of consecutively adjacent polygons.
This ``tree of polygons'' is the Cayley graph of $\mathZ_{q+1}*\mathZ_{q+1}$ if we choose  all the $2q$ non-zero elements of each factor $\mathZ_{q+1}$ as generators. Clearly, this graph is hyperbolic in the sense of Gromov.

The \emph{block distance} between two vertices 
of the graph
is the number of ``polygons'' visited by the shortest path 
joining them
plus 1: so two vertices belonging to the same polygon have distance 1, two vertices in adjoining polygons have distance 2 and so on.

Nevertheless, every group $\Gamma$ of automorphisms acting transitively and simply transitive on either $V$ or $E$ induces a convolution product on functions thereon, although the convolution depends on the choice of $\Gamma$. On the other hand,the full automorphism group $G$ does not act simply transitive. We now show in which sense $G$ defines a convolution property. We shall consider the action of $G$ on $E$: the statements are exactly the same for $V$.

Take $f$  bi-$K$-invariant  and $g$ right-$K$-invariant functions on $G$.   
Note that $f$ and $g$ can be regarded as functions on $E$, with $f$ radial around $e_0$. 
Their convolution on  $E$ becomes

\begin{equation}\label{eq:homogeneous_convolution}
f*g(e)=\sum_{e'\in E} f(\dist(e,e'))\,g(e').
\end{equation}
If $f$ and $g$ are both bi-$K$-invariant, this is
\begin{equation}\label{eq:radial_homogeneous_convolution}
f*g(e)=f*g(\dist(e,e_0))=\sum_{y\in E} f(\dist(e,e'))\, g(\dist(e',e_0)).
\end{equation}
So, the convolution of bi-$K$-invariant functions is bi-$K$-invariant, hence 
$ L^1_\#$ is a convolution algebra. This algebra is the closure in the $ L^1$ norm of the algebra $\mathfrak{R}_\#$ of radial finitely supported functions. Of course, radial functions are constant on the circles in $E$ with center $e_0$. 
From now on, we shall use the term \emph{radial} instead of bi-$K$-invariant functions on $G$. 

For every functions $f,g$ on $E$ with $f$ radial, we observe that
$\tau (f*g)=f*\tau g$, where $\tau g$ is defined by $\tau g(e) = g(\tau^{-1} e)$. Indeed, 
\begin{equation}\label{eq:automorphisms_commute_with_convolutions_by_radial_functions}
\begin{split}
\tau(f*g)(e) &= \sum_{e'} g(e')\,f(\dist(\tau^{-1}e,\,e')) = \sum_{e''} g(\tau^{-1} e'')\,f(\dist(\tau^{-1}e,\,\tau^{-1} e'') )
\\[.1cm]
&= \sum_{e''} g(\tau^{-1}e'')\,f(\dist(e,\,e'') )= f*\tau g(e).
\end{split}
\end{equation}

\subsection{Generating formulas for the the radial convolution algebras}
\begin{definition}[Circles]\label{def:circles} The \emph{vertex-circle} in $V$ with center $v$ in  is the subset $C(v,n)\subset V$ of all vertices at distance $n$ from the vertex $v$. The set of all vertex-circles is denoted by $C_V$. Recall that the characteristic function of the circle with center $v$ and radius $n$ is denoted by $w\mapsto \chi_n(v,w)$. The \textit{edge-circle} $C(e,n)$ is defined analogously as the set of edges at distance $n$ from the edge $e$. The set of all edge-circles is denoted by $C_E$, and
 the set of all circles by $C(T)$.

The \textit{edge-centered vertex-circle} $C_V\left(e,n+\frac12\right)$ is the set of vertices $v$ such that $\dist(v,e)=n+\frac12$ (the distance was introduced in Definition \ref{Def:Distance}). The \textit{vertex-centered edge-circles} $C_E\left(v,n+\frac12\right)$ are defined analogously.
\end{definition}

\subsubsection{The algebra of radial functions on the vertices of a homogeneous tree}
\begin{remark}
\label{rem:chi_n}
We introduced in \eqref{eq:circle-sum_as_convolution_operator} the operator $\chi^V_n$ on functions on $V$ given by summation over 
the circle $C(v,n)$. 
Then 
$\chi^V_n f(v) = \CircV^{(n)}f(v)$. With notation as in~Definition~\ref{def:convolutions&Laplacians}, this is the same as the convolution operator by the function $v\mapsto \chi^V_n(o,v)$, where $o$ is any fixed vertex, $\chi^V_n$ is now the characteristic function of $C(o,n)$
 and the convolution is induced by the action $\Aut(T)/K_{o}$. 
\end{remark}

The following lemma appears in~\ocite{Figa-Talamanca&Picardello} and in several references quoted therein. For the benefit of the reader, we sketch its simple proof here.

\begin{lemma}
\label{lemma:radial_convolution_vertex-recurrence_relations}
Consider the convolution operators $\chi^V_n
$ introduced in~Remark~\ref{rem:chi_n}, and the composition of these operators, that we denote by 
$\chi^V_n*\chi^V_m$ when we regard them as convolutors. Then $\chi^V_0=\delta_0$ is the identity and
\begin{equation*}
\chi^V_1 *
\chi^V_n=\begin{cases}
      \chi^V_ 1    &\text{if $n=0$,}\\[.2cm]
 (q+1)\chi^V_ 0
+     \chi^V_ 2    &\text{if $n=1$,}\\[.2cm]
  q   \chi^V_{n-1}
+     \chi^V_{n+1} &\text{if $n>1$.}\end{cases}
\end{equation*}
\end{lemma}

\begin{proof}
The convolution by $\chi^V_1$ is the sum over neighbors. Consider this convolution at a vertex $v$ at distance $n$ from a reference vertex $o$. Since every vertex has $q+1$ neighbors, in this sum $q$ neighbors of $v$ have length $n+1$ (that is, are at distance $n+1$ from $o$), and one has length $n-1$, unless $n=0$ (in this case $v=o$ and all neighbors have length $1$). Therefore, if $v$ has length $n+1$, only one term of the sum in $\chi^V_1*\chi^V_n(v)$ is non-zero, but if $v$ has length $n-1$ there are $q$ non-zero terms (one for each neighbor of $v$ in the support of $\chi^V_n$, that is for each forward neighbor of $v$).
\end{proof}

By normalization one immediately sees the following:

\begin{corollary}\label{cor:radial_convolution_vertex-recurrence_relations}
Consider the following normalized operators: the Laplace operator $\mu_1=\frac 1{q+1} \chi^V_1$, and more generally, for $n>1$, $\mu_n=\frac 1{(q+1)q^{n-1}}\chi^V_n$. Then, if $n>0$,
\begin{equation*}
\mu_1 *
\mu_n
=  \frac1{q+1}\;\mu_{n-1}
 +\frac{q}{q+1}\;\mu_{n+1}
\end{equation*}
and of course $\mu_1 \mu_0=\mu_1$.
\end{corollary}

\subsubsection{The algebra of radial functions on the edges of a homogeneous tree}
Every edge different from the reference edge $e_0$ has $q$ forward neighbors (farther from $e_0$), one backward neighbor (closer to $e_0$) and $q-1$ neighbors at the same distance from $e_0$ as $e$.
Hence the following convolution recurrence relations are clear. (For a slightly different relation for
  vertex convolution, see \cite{Figa-Talamanca&Picardello} and its references).
  By abuse of notation, a convolution operator with a radial convolution kernel $f$ will be denoted again by $f$.

Observe that, if
$\CE(n)=\{e\in E:|e|=n>0\}$, then
\begin{equation}\label{eq:number_of_edges_of_length_n}
\left| C(n) \right|= 2 q^n.
\end{equation}

\begin{lemma}
\label{lemma:radial_convolution_recurrence_relations_for_edges}
Let $\xi _n$ be the characteristic function of the set of edges at distance $n$ from $e_0$. 
Then
\begin{equation}\label{eq:recurrence_relation_for_edges_homogeneous}
\xi _1 *
\xi _n
= 
\begin{cases}
\xi _1                                                   &\qquad\text{if } n=0;\\[.2cm]
2q\xi _0       +(q-1)\xi _1  +     \xi _2    &\qquad\text{if } n=1;\\[.2cm]
q   \xi _{n-1} +(q-1)\xi _n  +     \xi _{n+1}&\qquad\text{if } n>1.
\end{cases}
\end{equation}
\end{lemma}

\begin{proof}
The case $n=0$ is trivial, so let $n>0$.
Let $f$ be a function on $E$ and call $v_-,v_+$ the endpoint vertices of $e_0$. For every edge $e\sim e_0$,
\begin{equation}\label{eq:chi-square,homogeneous+edges}
                   \xi_1^2 f(e )
=\sum_{e''\sim e'  }   \sum_{e' \sim e}
     f(e'').
\end{equation}
The edges $e''$ in this double sum are
\begin{enumerate}
\item
the edge $e_0$ itself, counted as many times as there are neighbors of $e_0$ (namely $2q$),
\item the edges adjacent to $e_0$, each counted as many times as the neighbors of $e_0$ on the same side of $e'$, except $e'$ itself (namely $q-1$),
\item the edges at distance $2$ from $e_0$, each counted once.
\end{enumerate}
Therefore  
\eqref{eq:recurrence_relation_for_edges_homogeneous} is proved for $n=1$.
\par
Now, for $n>1$,
 \[
 \xi _1 * \xi _n f(e)
=
 \sum_{e''\sim e'  }   \sum_{\dist(e', e)=n }  f (e'').
 \]
The only difference with respect to the previous argument is in~case 
(1). Indeed, in this case  an edge $e''\sim e'$ is such that $\dist(e'',e)=n-1$ if and only if
 $\dist(e',e)=n$. Then $e'$ belongs to the path from $e$ to $e''$. But for each $e''$ there are exactly $q$ such edges $e'$. This yields~\eqref{eq:recurrence_relation_for_edges_homogeneous}. 
\end{proof}

By normalization we obtain the following result. 

\begin{corollary}\label {cor:recurrence_relations_for_homogeneous_edges}

Normalize $\xi_n$ as $\eta_0=\xi_0$, 
and for $n>1$, $\eta_n=\frac 1{2q^{n}}\xi _n$. Then, for $n\geqslant 1$,
\begin{equation*}
\eta_1 *
\eta_n
=  
\frac1{2q}\eta_{n-1} 
+\frac{q-1}{2q}\eta_{n}
 +\frac12 \eta_{n+1}.
 \end{equation*}
In particular the algebra $\mathfrak R_\#$ of radial finitely supported functions on $E$ (with identity) is generated by $\eta_1$, hence it is commutative.
\end{corollary}

\begin{remark}\label{rem:distribution_of_length_of_adjoining_edges}
The middle summand in this recurrence relation is a consequence of the fact that the lengths of neighbors of $e\neq e_0$ have different parities: there are $q$  neighbors  of length $|e|+1$, one neighbor of length $|e|-1$, and $q-1$ other neighbors of length $|e|$. This does not happen in the case of vertices.
\end{remark}

\begin{definition}\label{{def:edge-radialization}def:radialization}
Let $g$ be a function on $E$. 
Define the
radialization operator around $e_0$ as
\begin{equation*}
\calE g (e)= \sum_{e':\,|e'|=|e|} \frac{g(e')}{ |C(|e|)| }\;.
\end{equation*}
\end{definition}

If $f, g$ are functions on $E$, set $\langle f,g\rangle=\sum_{e} f(e)\,g(e)$ whenever the series is absolutely convergent. Note that $(f,g)=\langle f,\overline{g}\rangle$ is the $ \ell^2$-inner product.
There is an obvious but useful property satisfied by the radialization operator, 
and we mention it here for later use: if  $f, g$ are  functions on $E$ with $f$ radial such that $fg$ is summable, then  
\begin{equation}\label{eq:radialization}
\langle f,g\rangle=\langle f, \calE g \rangle.
\end{equation}
%

\subsection{Spherical Fourier transforms of radial functions}

\begin{remark}
[The spherical Fourier transform is not multiplicative on convolution products]\label{rem:the_spherical_Fourier_transform_is_not_multiplicative_on_convolutions_of_non-radial_functions}

It is interesting to observe that the properties of convolutions under the spherical Fourier transform are not the same as in the case of, say, Euclidean spaces. Indeed, the spherical Fourier transform of the convolution $f*h$ of two functions on vertices or on edges is not equal to the pointwise product $\calF f \cdot \calF h$. For instance, take two vertices $v$ and $w$: then $\delta_v *\delta_w=\delta_{vw}$, but, by Proposition \ref{prop:computation_of_vertex-spherical_functions}, $\calF \delta_{vw}=\phi^V_z(vw)\neq \phi^V_z(v) \phi^V_z(w)= \calF \delta_{v} \calF \delta_{w}$. 
Exactly the same argument applies to functions on edges.

However, we now show that the spherical Fourier transform maps convolutions of \emph{radial} functions to pointwise products. Therefore, for the purpose of harmonic analysis, the useful and natural convolution algebras on trees are the radial algebras. On the other hand, these algebras are a useful environment in analysis only if they form a commutative algebra (otherwise, the spherical Fourier transform would kill their non-commutative algebraic structure and cannot be an isomorphism). 
\end{remark}

\begin{theorem}\label{theo:the_spherical_Fourier_transform_is_multiplicative_on_convolutions_of_radial_functions}
[The spherical Fourier transform is a multiplicative homomorphism on  convolutions of radial functions]
If $f, g$ are radial functions on vertices, and $h, k$ are radial on edges, then
\begin{align*}
\radspherFour^V_{v_0}(f*g)&=\radspherFour^V_{v_0}f \;\radspherFour^V_{v_0}g,\\
\radspherFour^E_{e_0}(h*k)&=\radspherFour^E_{e_0}h \;\radspherFour^V_{v_0}k.
\end{align*}
\end{theorem}
\begin{proof}
Let $\chi^V_n$ and $\chi^E_n$, 
($n\geqslant 0$) be the characteristic functions of the circles of vertices (respectively, edges) 
at distance $n$ from the reference vertex or edge, 
introduced in Lemma \ref{lemma:radial_convolution_vertex-recurrence_relations}.
and Lemma \ref{lemma:radial_convolution_recurrence_relations_for_edges}, where $\chi^E_n$ was denoted by $\xi_n$
. 
As done there, denote by $\mu_n$ 
its $\ell^1$ normalization. In particular, $\mu_0=\chi^V_0=\delta_{v_0}$. 
We have proved in 
that Lemma that the spaces of radial functions with identity 
is generated under convolution by $\chi^V_1$ 
(or equivalently by $\mu_1$) 
hence 
it is a commutative convolution algebra. By linearity, the statement is equivalent to the identity 
\begin{align*}
\radspherFour^V_{v_0}(\mu_1*\mu_1)&=(\radspherFour^V_{v_0}\mu_1)^2,\\
\end{align*}

By Corollary \ref{cor:radial_convolution_vertex-recurrence_relations}, $\mu_1*\mu_1 = \frac 1{q+1}\;\mu_0 + \frac q{q+1}\;\mu_2$.
For every neighbor $v$ of $v_0$, consider the vertex-horospheres that contain $v$. By Definition \ref{def:horospheres}, those horospheres with tangency point $\omega$ in the boundary arc $\Omega(v)$ subtended by $v$ have horospherical index $h(v,v_0,\omega)=1$, and those with tangency point in $\complement \Omega(v)$ have index $-1$. On the other hand, as we saw in Subsection \ref{SubS:Boundary}, $\nu_{v_0}(\Omega(v))=\frac 1{q+1}$, and
$\nu_{v_0}(\Omega\setminus \Omega(v))=\frac q{q+1}$.
Notice that $\int_{\Omega} q^{zh(v,v_0,\omega)} \,d\nu_{v_0}(\omega)$ does not depend on $|v|=1$, by the isotropy of the measure $\nu_{v_0}$. Therefore
\begin{align}\label{eq:vertex-gamma_function_as_spherical_transform_of_mu_1}
\radspherFour^V_{v_0}(\mu_1) &= \frac 1{q+1} \radspherFour^V_{v_0}(\chi^V_1) = \frac1{q+1} \sum_{|v|=1} \int_{\Omega(v)\cup\complement( \Omega(v))} q^{zh(v,v_0,\omega)} \,d\nu_{v_0}(\omega) \notag \\[.2cm]
&=
\frac 1{q+1}\;q^z + \frac q{q+1}\;q^{-z} = \frac{q^z +q^{1-z}}{q+1}:=\gamma^V(z).
\end{align}
Let us now compute $\radspherFour^V_{v_0}(\mu_2)$. Let $|v|=2$ and, as in the proof of Proposition \ref{prop:computation_of_vertex-spherical_functions}, let
 $D_1=\Omega(v_-)\setminus \Omega(v)$ and $D_0 =\Omega \setminus \Omega(v_-)$.
Then, by Definition \ref{def:horospherical_index} of horospherical index, $h(v,v_0,\omega)=2$ if $\omega\in \Omega(v)$,  $h(v,v_0,\omega)=0$
if $\omega\in D_1$ , and $h(v,v_0,\omega)=-1$ if $\omega\in D_0$.

 Moreover, as seen in the same proof,
   $\nu_{v_0}(\Omega(v))=\frac 1{q(q+1)}$  is the reciprocal of the number of vertices whose distance from $v_0$  is the same as the  distance of $v$. Thus, as $|v|=n$, 
   \begin{align*}
   \nu_{v_0}(D_1)&=\frac 1{q+1}-\frac 1{q(q+1)}= \frac {q-1}{q(q+1)}\;,\\[.2cm]
   \nu_{v_0}(D_0)&= \frac q{q+1}\;.
   \end{align*}
   
Again by invariance of the boundary measure under the automorphisms that fix $v_0$, we now have
\begin{equation}\label{eq:spherical_transform_of_mu_2}
\radspherFour^V_{v_0}(\mu_2)=\frac 1{q(q+1)}   \; q^{2z} + \frac {q-1}{q(q+1)} \; + \frac q{q+1}\; q^{-2z}.
\end{equation}
We have seen in Corollary \ref{cor:radial_convolution_vertex-recurrence_relations} that $\mu_1*\mu_1=(\mu_0+q\mu_2)/(q+1)$. Since $\calF(\mu_0)\equiv 1$, the first identity follows from this and an easy verification based upon \eqref{cor:radial_convolution_vertex-recurrence_relations} and \eqref{eq:vertex-gamma_function_as_spherical_transform_of_mu_1}.

Similarly, for every edge with $|e|=1$, let $\Omega(e)$ be the boundary arc subtended by its vertex opposite to $e_0$. Since there are $2q$ such edges, again from Subsection \ref{SubS:Boundary} we know that $\nu_{e_0}(\Omega(e))=1/(2q)$. Moreover, again by the formula of Definition \ref{def:horospherical_index},
the horospherical indices of the edge-horospheres $\bs{h}(e, \,\omega;\,e_0)$ that contain $e$ are as follows.
If the boundary
 boundary tangency point $\omega$ belongs to $\Omega(e)$ the index is 1; if it belongs to $\Omega(e')$ for an edge $e'\neq e,\,|e'|=1$ (remember that there are $q-1$ such edges) then the index is 0; if $\omega$ is at the opposite boundary arc of $e_0$, then the index is $-1$ (and there are $q$ such edges). The index of the horospheres that contain $e$ but are tangent in $\complement \Omega(e)$ is $-1$. Therefore, 
 \begin{align}\label{eq:edge-gamma_function_as_spherical_transform_of_eta_1}
\radspherFour^E_{e_0}(\eta_1) &= \frac 1{2q} \radspherFour^E_{e_0}(\chi^E_1) = \frac1{2q} \sum_{|e|=1} \int_{\Omega} q^{zh(e,e_0,\omega)} \,d\nu_{e_0}(\omega) \notag \\[.2cm]
&=
\frac 1{2q}\;q^z + \frac {q-1}{2q}\;+ \frac12 \; q^{-z}  = \frac  {q^z +(q-1) + q^{1-z}}{2q}:=\gamma^E(z).
\end{align}
%
%
%
Let $|e|=2$, write $e=[v_2,\, v_3]$, $e_-=[v_1,\,v_]$, $e_0=[v_0,\,v_1]$ and let $\Omega_+(e_0)$ be the boundary arc subtended by $e_0$ on the side of $e$. Moreover, let
 $D_2=\Omega(e_-)\setminus \Omega(e)$, $D_1= \Omega_+(e_0) \setminus \Omega(e_-)$ and
 and $D_0 =\Omega \setminus \Omega_+(e_0)$.
Then, again by definition of horospherical index, $h(e,e_0,\omega)=2$ if $\omega\in \Omega(e)$. Moreover,  $h(e,e_0,\omega)=1$ if $\omega\in D_2$ (these horospheres, besides $e$, contain also $e_-$ that has length 1). Furthermore,
$h(e,e_0,\omega)=-1$ if $\omega\in D_1$ , and $h(e,e_0,\omega)=-2$ if $\omega\in D_0$.
 Moreover, as in the same proof,
 \begin{align*}
 \nu_{e_0}(\Omega(e))   &=\frac 1{2q^2}\;,\\[.2cm]
 \nu_{e_0}(\Omega(e_-)) &=\frac 1{2q}\;,\quad \text{hence }\quad \nu_{e_0}(D_2)=\frac1{2q}-\frac 1{2q^2}=\frac{q-1}{2q^2}\;,\\[.2cm]
 \nu_{e_0}(D_1)    &= \frac12 - \frac 1{2q}= \frac{q-1}{2q}\;,\\[.2cm]
 \nu_{e_0}(D_0)    &= \frac 12\;.
 \end{align*}
Therefore
\begin{align*}
\radspherFour^E_{e_0}(\eta_2)&=\frac 1{2q^2} \calF (\chi^E_2) = \frac 1{2q^2} \sum_{|e|=2} \int_{\Omega} q^{zh(e,e_0,\omega)} \,d\nu_{e_0}(\omega)\\[.2cm]
&= \frac 1{2q^2}  \; q^{2z} + \frac {q-1}{2q^2} \; q^z + \frac {q-1}{2q}\; q^{-z} +\frac12\; q^{-2z}.
\end{align*}
Now the second identity of the statement follows easily from this and \eqref{eq:edge-gamma_function_as_spherical_transform_of_eta_1} and the identities $\radspherFour^E_{e_0}(\eta_0)\equiv 1$ and
$\eta_1 * \eta_1= \frac1{2q}\;   \eta_{0}
+\frac{q-1}{2q} \;\eta_ 1
+ \frac12\;    \eta_{2}$
(that is a consequence of  Corollary \ref{cor:recurrence_relations_for_homogeneous_edges}.
\end{proof}

\begin{definition}\label{def:gamma}
In \eqref{eq:vertex-gamma_function_as_spherical_transform_of_mu_1}
and \eqref{eq:edge-gamma_function_as_spherical_transform_of_eta_1}
we introduced a 
special function that will appear often in the sequel:
\begin{subequations}\label{eq:gamma}
\begin{align}
\label{eq:gamma_V}
\gamma^V(z) &= \frac {q^z + q^{1-z}}{q+1}\,. \\[.15cm]
\label{eq:gamma_E}
\gamma^E(z) &= \frac {q^z + q-1 + q^{1-z}}{2q}\,.
\end{align}
\end{subequations}
\end{definition}

\subsection{Characterizations of spherical functions}\label{Characterizations_of_spherical_functions}
In the next statements we adopt the following notation: for every radial function $f$ on $V$ and $g$ on $E$, 
we write $f_n=f(v)$ when $|v|=n$, and $g_n=g(e)$ when $|e|=n$. 
\begin{theorem}\label{theo:characterization_of_spherical_functions}
The following hold:
\begin{enumerate}
\item [$(i)$] The vertex-
spherical function $\phi^V_z$ is the only function $\phi$ on $V$ that satisfies the following properties: $\phi$ is radial, $\phi(v_0)=1$,  $\phi_1=\gamma^V(z)$ and, for $n>0$, 
\begin{equation} \label{eq:recurrence_relation_of_vertex_spherical_function}
\phi_1\phi_n=  \frac1{q+1}\;\phi_{n-1} +\frac{q}{q+1}\;\phi_{n+1}
\end{equation}
\item [$(ii)$] The spherical function $\phi^V_z$ is the only radial function $\phi$ on $V$ such that $\phi(v_0)=1$ and 
\begin{equation} \label{eq:vertex_spherical_functions_as_eigenfunctions}
\mu_1 \phi=\gamma^V(z)\;\phi.
\end{equation}
\item [$(iii)$] The spherical functions $\phi^V_z$ are the only radial functions $\phi$ on $V$ such that the map $L:h\mapsto \langle \phi,\,h\rangle_V$ is a non-zero multiplicative functional on the (commutative) convolution algebra $\mathfrak{R}_V$ of radial functions on $V$ with finite support. Specifically, $\phi^V_z$ gives rise to the functional $L$ such that $L(\mu_1)=\gamma^V(z)$.
%
\item [$(iv)$]  A function $\phi$ on $V$ is spherical if and only, for every $w\in V$, its translate $\delta_w *\phi$ satisfies the multiplicative rule $\calE (\delta_w*\phi) (v)= \phi(w)\phi(v)$ (where $\calE$ is the radialization operator on functions on vertices, introduced in Definition \ref{def:vertex-radialization}.
\item [$(v)$] The edge-spherical function $\phi^E_z$ is the only function $\psi$ on $E$ that satisfies the following properties: $\psi$ is radial, $\psi(v_0)=1$,  $\psi_1=\gamma^E(z)$ and, for $n>0$, 
\begin{equation} \label{eq:recurrence_relation_of_edge_spherical_function}
\psi_1\psi_n=  \frac1{2q}\psi_{n-1} 
+\frac{q-1}{2q}\psi_n
 +\frac12 \psi_{n+1}.
\end{equation}
\item [$(vi)$] The spherical function $\phi^E_z$ is the only radial function $\psi$ on $E$ such that $\psi(e_0)=1$ and 
\begin{equation} \label{eq:edge_spherical_functions_as_eigenfunctions} 
\eta_1  \psi=\gamma^E(z)\;\psi.
\end{equation}
\item [$(vii)$] The spherical functions $\phi^E_z$ are the only radial functions $\psi$ on $E$ such that the map $L(h)=\langle \psi,\,h\rangle_E$ is a non-zero multiplicative functional on the (commutative) convolution algebra $\mathfrak{R}_E$ of radial functions on $E$ with finite support. Specifically, $\phi^E_z$ gives rise to the functional $L$ such that $L(\eta_1)=\gamma^E(z)$. 
\item [$(viii)$]  A function $\psi$ on $E$ is spherical if and only, for every $a\in E$, its translate $\delta_a *\psi$ satisfies the multiplicative rule $\calE (\delta_a *\psi) (e)= \psi(a)\psi(e)$.
\item [$(ix)$] 
Part $(iv)$ and $(vii)$ 
lead to the following generalization of Theorem \ref{theo:the_spherical_Fourier_transform_is_multiplicative_on_convolutions_of_radial_functions}: the spherical Fourier transform of the convolution of a radial and a non-radial function (either on vertices or on edges)
 is the product of the respective spherical Fourier transforms.
\end{enumerate}
\end{theorem}
\begin{proof}  It follows by Remark \ref{rem:obvious_properties_of_spherical_functions} that  $\phi^V_z$ is radial and $(\phi^V_z)_n=\radspherFour^V_{v_0}\mu_n (z)$.
The multiplicativity property of Theorem \ref{theo:the_spherical_Fourier_transform_is_multiplicative_on_convolutions_of_radial_functions} yields $\radspherFour^V_{v_0}(\mu_1\,\mu_n)=
\radspherFour^V_{v_0} \mu_1\;\radspherFour^V_{v_0}\mu_n=(\phi^V_z)_1\,(\phi^V_z)_n$. Then it is obvious (and it was already observed) that $\phi^V_z(v_0)=\radspherFour^V_{v_0}\mu_0(z)=1$. Moreover, by \eqref{eq:vertex-gamma_function_as_spherical_transform_of_mu_1}, 
\begin{equation}\label{eq:vertex-gamma_function_as_spherical_transform_of_mu_1-bis}
(\phi^V_z)_1=\radspherFour^V_{v_0}\mu_1(z)=\gamma^V(z).
\end{equation}
The identity \eqref{eq:recurrence_relation_of_vertex_spherical_function} follows from this by applying the spherical Fourier transform to both sides of the recurrence relation of Corollary \ref{cor:radial_convolution_vertex-recurrence_relations}. This identity is a second order difference equation that has a unique solution that satisfies the two assigned initial values $\phi_0=1$ and $\phi_1=\gamma^V(z)$. This proves part $(i)$, and part $(iv)$ follows in the same way by the edge-recurrence relation \eqref{eq:recurrence_relation_of_edge_spherical_function}.

We claim that
the spherical function $\phi^V_z$ satisfies the identity \ref{eq:vertex_spherical_functions_as_eigenfunctions}. Indeed, every vertex $v$ with $|v|=n>0$ has $q$ forward neighbors $v_+$ (with $|v_+|=n+1$) and one predecessor $v_-$ (with $|v_-|=n-1$). Therefore,  by \eqref{eq:recurrence_relation_of_vertex_spherical_function} and \eqref{eq:vertex-gamma_function_as_spherical_transform_of_mu_1-bis},
\[
\mu_1*\phi^V_z(v)=\frac 1{q+1} (\phi^V_z)_{n-1} + \frac q{q+1} (\phi^V_z)_{n+1} = (\phi^V_z)_{1}\,(\phi^V_z)_{n} =\gamma^V(z)\,(\phi^V_z)_{n}
\]
(notice here that the eigenfunction identity is equivalent to the recurrence relation \eqref{eq:recurrence_relation_of_vertex_spherical_function}, and in particular it determines uniquely its solution $\phi$ once the initial values $\phi_0=(1)$ and $\phi_1(=\gamma^V(z))$ are assigned).
On the other hand, since the spherical function is radial, it is clear that $\mu_1*\phi^V_z)(v_0)=\langle \mu_1,\, phi^V_z\rangle_V= (phi^V_z)_1 = \gamma^V(z)$. This proves the claim, and we have already noticed that uniqueness follows. So part $(ii)$ is proved, and the same argument proves part $(v)$.
Finally, notice that a multiplicative functional $L$ on $\mathfrak{R}_V$ satisfies $L(\delta_{v_0})^2=L(\delta_{v_0}*\delta_{v_0})=L(\delta_{v_0})$. Therefore $L(\delta_{v_0})=1$ or 0, and in the latter case $L=0$. Notice also that, if $L$ is multiplicative under convolution, then, by Corollary \ref{cor:radial_convolution_vertex-recurrence_relations},
\begin{equation*}
L(\mu_1) L(\mu_n)
=(L\mu_1 * \mu_n)
=  \frac1{q+1}\;L(\mu_{n-1})
 +\frac{q}{q+1}\;L(\mu_{n+1})\,.
\end{equation*}
But this is exactly the recurrence relation \eqref{eq:recurrence_relation_of_vertex_spherical_function} of the vertex-spherical function, that, as we know from pert $(i)$, has the unique solution $L(\mu_n)=(\phi^V_z)_n$ provided that the value of $L(\mu_1)$ is $\gamma^V(z)$. Thus all multiplicative functionals on $\mathfrak{R}_V$ arise in this way. Conversely, the spherical function gives rise to a multiplicative functional. This proves $(iii)$, and, by making use of \eqref{eq:recurrence_relation_of_edge_spherical_function}, the same argument yields $(vi)$.

Let us prove $(iv)$. Let $|v|=n$. Observe that $\mu_n*\phi^V_z (w):=\sum_{u\in V} \mu_n(u^{-1})\,\phi^V_z(wu)$. Let us write
$\Phi_w=\calE(\delta_{w}*\phi^V_z)$. Then, by Remark \ref{rem:radialization}, the previous identity becomes
\[
\mu_n*\phi^V_z (w) = \langle \mu_n,\,\delta_{w}*\phi^V_z \rangle_V = \langle \mu_n,\, \Phi_w\rangle_V = (\Phi_w)_n=\Phi_w(v)\,.
\]
On the other hand, $\mu_n*\phi^V_z$ is a multiple of $\phi^V_z$, by part $(ii)$ and the recurrence relation of Corollary \ref{cor:radial_convolution_vertex-recurrence_relations}. Therefore $\Phi_w(v)=\mu_n*\phi^V_z (w) = c\,\phi^V_z(w)$ for all $w$ and some $c\in\mathC$  (that depends on $v$). But $\Phi_{v_0}(v)=\calE(\delta_{v_0}*\phi^V_z)(v)=\phi^V_z(v)=\phi^V_z(v)\,\phi^V_z(v_0)$, hence $c=\Phi_{v_0}(v)=\phi^V_z(v)$ and $\calE(\delta_{w}*\phi^V_z)(v)=\Phi_w(v)=\phi^V_z(w)\,\phi^V_z(v)$. This shows that the spherical functions satisfy the translation property of part $(iv)$.
\\
Conversely, il $\phi$ is a non-zero function that satisfies the identity $\calE (\delta_w*\phi) (v)= \phi(w)\phi(v)$ for every $v,\,w\in V$, then, by choosing any $w$ such that $\phi(w)\neq 0$, we see that $\phi$ must be radial. Let us consider the functional on $\mathfrak{R}_V$ defined by $L(f)=\langle \phi,\,f\rangle_V$. If $h\in\mathfrak{R}_V$, one has, again by Remark \ref{rem:radialization},
\begin{align*}
L(f*h)&=\sum_{v,w\in V} f(v)\,h(w)\,\phi(vw) = \sum_{v\in V} f(v) \langle \delta_v * \phi,\,h\rangle_V \\
&= 
\sum_{v\in V} f(v) \langle \calE(\delta_v * \phi),\,h\rangle_V = \sum_{v,w\in V} f(v)\,h(w)\,\phi(v)\,\phi(w) = L(f)\, L(h).
\end{align*}
This proves $(iv)$, and the proof of $(viii)$ is similar. 
Finally, part $(ix)$ 
 is clear.
\end{proof}

\begin{remark}
The previous Theorem \ref{theo:characterization_of_spherical_functions} is a crucial tool in the theory of spherical functions. All of its results follow from the recurrence relations \eqref{eq:recurrence_relation_of_vertex_spherical_function} of part $(i)$ and \eqref{eq:recurrence_relation_of_edge_spherical_function} of part $(v)$ of its proof, that have been obtained by means of the recurrence relations of the algebra of radial functions, given in Corollary \ref{cor:radial_convolution_vertex-recurrence_relations}. In line with our approach to harmonic analysis via integral geometry, we observe that the whole proof can be obtained directly from the explicit formulas \eqref{eq:vertex-spherical_function_via_direct_computation}. for the vertex-spherical functions and \eqref{eq:edge-spherical_function_via_direct_computation} for the edge-spherical functions. 
These 
This explicit formula were derived via integral geometry, namely from 
 the horospherical definition of spherical functions (see \eqref{eq:def_of_vertex-spherical} in Definition \ref{def:zonal_spherical_functions}), instead than from algebraic properties. We omit this straightforward but tedious direct verification, and limit ourselves to observe that, when we express  spherical functions as integrals over sections in the fiber bundle of powers of Poisson kernels (with respect to the respective boundary measures, as in \eqref{eq:vertex-spherical_function_via_direct_computation},  and \eqref{eq:edge-spherical_function_via_direct_computation}), 
 their recurrence relations follow easily from the convolution identities of the Poisson kernel. Indeed, it is immediate to see that these kernels are eigenfunctions of the Laplace operators $\mu_1$ and $\eta_1$, respectively, 
 exactly as in Proposition \ref{prop:image_of_boundary_functions_under_Poisson_transform}. We leave the details of this more elegant approach to the reader.
\end{remark}

\section[Spherical functions and the spectrum of the Laplacian]{Spherical functions and spectral theory of the Laplacian on vertices and edges of a homogeneous tree}\label{Sect:spectral_theory}
The results of this Section for vertices are in \cite{Figa-Talamanca&Picardello}. The results for edges have been recently proved in \ocite{Casadio_Tarabusi&Picardello-spherical_functions_on_edges}.
\begin{proposition}\label{prop:spectral_theory_of_spherical_functions}
Let $T$ be a homogeneous tree with homogeneity degree $q>1$ and $z\in\mathC$.
\begin{enumerate}
\item[$(i)$] For $|v|=n$, $\phi^V_z(v) \sim q^{-n\left(\min\{\Real z, \, 1-\Real z\}\right)}$. Therefore the spherical functions $\phi^V_z$  belong to $\ell^\infty(V)$ if and only if $0\leqslant \Real z \leqslant 1$; they belong to $\ell^p$ for $p>2$ if and only if $1/p \leqslant \Real z \leqslant 1/p'$, where $1/p' = 1 - 1/p$; they belong to $\bigcap_{p>2}\ell^p$ if and only if $\Real z= 1/2$. The same result holds for $\phi^E_z$.
\item[$(ii)$] The spectrum of $\mu_1$ on $\ell^1(V)$ is the ellipse $S_V^{(1)}=\left\{\gamma^V(z): 0\leqslant \Real z \leqslant 1\right\}$. The same result holds for the $\ell^1$-spectrum of $\eta_1$, that is the ellipse $S_E^{(1)}=\left\{\gamma^E(z): 0\leqslant \Real z \leqslant 1\right\}$. $S_E^{(1)}$ is the translate of $S_V^{(1)}$ given by $S_E^{(1)}=\frac{q-1}{2q} + S_V^{(1)}$.
\end{enumerate}
\end{proposition}
\begin{proof}
To prove $(i)$, consider first the $\ell^1$ completions $\ell^1_\#(V)$ of the finitely supported radial algebra $\mathfrak R_V$. $\mathfrak R_V$ is a commutative Banach algebra with unit: therefore their Gelfand spectrum is the spectrum on $\ell^1_\#$ of the  generator $\mu_1$. But the Gelfand spectrum consist of those multiplicative linear functionals that are bounded on $\ell^1_\#$, that are given by the bounded spherical functions. Hence, the spectrum of $\mu_1$ on $\ell^1_\#(V)$ is $\{\gamma^V(z): 0\leqslant \Real z \leqslant 1\}$.
\\
Now, a continuous multiplicative functional on $\ell^1(V)$ is also continuous and multiplicative on $\ell^1_\#(V)$, and so the $\ell^1$ spectrum is contained in the above set. On the other hand, every number $\gamma^V(z)$ with $0\leqslant \Real z\leqslant 1$ satisfies the equation $\mu_1\phi^V_z=\gamma^V(z)\,\phi^V_z$ (here and in what follows, $\mu_1$ is regarded as the average operator on neighbors: it may as well be regarded as the convolution operator by the normalized chracteristic function of the set of neighbors of $v_0$, that, by abuse of notation, we also denote by $\mu_1$).
For simplicity, write $\lambda=\gamma^V(z)$. If $g$ is a non-radial function in $\ell^1(V)$ that satisfies the same equation $\mu_1 g=\lambda\,g$, we claim that its radialization around $v_0$, introduced in Definition \ref{def:vertex-radialization},
 satisfies the same equation. Indeed, the forward neighbors of the vertices in the circle $C_V(n,v_0)$ yield $C_V(n+1,v_0)$
without repetitions, and the backward neighbors of the vertices in $C_V(n,v_0)$ yield $C_V(n-1,v_0)$ with $q$ repetitions.
Therefore, if $|v|=n$,
\[
\chi^V_1 \calE g (v) = \frac q{|C_V(n,v_0)|} \sum_{|v'|=|v|-1|} g(v') + \frac 1{|C_V(n,v_0)|}  \sum_{|v'|=|v|+1|} g(v').
\]
But $|C_V(n,v_0)|=q|C_V(n-1,v_0)|$ (if $v\neq v_0$, of course), and $|C_V(n,v_0)|=|C_V(n+1,v_0)|/q$. Hence the last identity becomes
\begin{align*}
\chi^V_1 \calE g (v) &= \frac 1{|C_V(n-1,v_0)|} \sum_{|v'|=n-1} g(v') + q\;\frac 1{|C_V(n+1,v_0)|}  \sum_{|v'|=n+1} g(v')\\[.2cm]
&= \calE g (v^-) +\sum_{v^+\sim v,\,v+>v} \calE g(v^+)\,,
\end{align*}
where we denoted by $v^-$ the predecessor of $v$ (its neighbor closer to $v_0$), and by $v^+$ its neighbors at the opposite side.
This means that $\mu_1 \calE g=\calE(\mu_1 g)$. But then $\mu_1 \calE g = \lambda \calE g$, and the claim is proved.
\\
Therefore  all eigenvalues of $\mu_1$ on $\ell^1_\#(V)$ are also eigenvalues on $\ell^1(V)$, and the $\ell^1$ spectrum is the same. 

The identity $S_E^{(1)}=\frac{q-1}{2q} + S_V^{(1)}$ is immediately verified: for more details, see Theorem \ref{theo:ell^2-spectra} below.
\end{proof}

\begin{corollary} \label{cor:exponential_are_eigenfunctions_of_Laplace_operators}
The computation in the proof of Proposition \ref{prop:spectral_theory_of_spherical_functions} shows
 that, away from the respective reference elements, the exponentials $y_z(v)=q^{-z|v|}$ 
 and $y_z(e)=y_z^{-z|e|}$ 
 are eigenfunctions of $\mu_1$. 
 and $\eta_1$, respectively. 
Indeed,
 for $v\neq v_0$, 
 or $e\neq e_0$, one has
\begin{align*}
\mu_1 y_z&=\gamma^V(z)\, y_z.
\\
\eta_1 y_z&=\gamma^E(z)\, y_z\,.
\end{align*} 
\end{corollary}

\subsection{
The $\ell^2$-spectrum of radial functions on vertices of homogeneous trees}
The next theorem is the celebrated estimate of Haagerup \cite{Haagerup} for convolution operators on functions on $V$.
\begin{lemma}
Let $T$ be a homogeneous tree of homogeneity degree $q>1$ and $f:V\to \mathC$ be a function supported on the circle $C_V(n,v_0)$. Then its norm in the $C^*$-algebra $C^*_\lambda$ of bounded left convolution operators on $\ell^2(V)$ (that in the next Sections will be also denoted by $LCv_2$) satisfies the inequality
\[
\|f\|_{C^*_\lambda} \leqslant (n+1) \|f\|_2\,.
\]
\end{lemma}
\begin{proof}
In this proof, we need to regard the convolution as induced by a simply transitive subgroup of $\Gamma\subset \Aut T$, for instance isomorphic to a free group (for $q$ even) or more generally a free product of $q$ copies of $\mathZ_2$, as explained in Example \ref{example:trees_as_Cayley_graphs_of_free_products}. Therefore the vertices will be regarded as elements of this discrete group $\Gamma$. Then the convolution depends on the choice of the subgroup $\Gamma$, and, if $K$ is the stabilizer of a vertex $v_0$, it allows to convolve any two right-$K$-invariant functions (that is, functions on $V$). Instead, the convolution in $\Aut T$ allows only to convolve bi-$K$-invariant functions (i.e., functions on $V$ radial around $v_0$) with
right-$K$-invariant functions (Subsection \ref{SubS:convolution}).  It is easy to see that, if the first function is radial, then the convolution from $\Aut T$ coincides with that from $\Gamma$. So, if the function $f$ in the statement is radial, the result holds independently on the choice of $\Gamma$, and this will be the scope of our applications. 

Denote by $\chi_k$ the characteristic function of the circle $C_V(k,v_0)$, and, for $g\in\ell^2(V)$, let $g_k=g\chi_k$.
Then $\|g\|_2^2=\sum_{k=0}^\infty \|g_k\|_2^2$. Now look at $h=f*g$. Its truncations $h_m=h\chi_m$ satisfy the inequality
\begin{equation}\label{eq:expression_of_convolution}
\|h_m\|^2_2\leqslant\sum_{k\geqslant 0} \|(f*g_k)\chi_m\|^2_2 \,.
\end{equation}

\emph{We claim that } $ |\ (f*g_k)\chi_m\|_2  \leqslant \|f\|_2\,\|g\|_2 $.

To prove this claim, observe that 
\[
|f*g_k(v)|=\left|\sum_{x\in V} f(vu^{-1})\,g_k(u)\right| = \sum_{xu=v,\, |x|=n,\,|u|=k} f(x)\,g_k(u).
\]
%
On the other hand, as $f$ is supported on vertices of length $n$, $(f*g_k)\chi_m=0$ unless $|n-k|\leqslant m \leqslant n+k$. Moreover, $m=n+k-2p$ where $p$ is the number of cancellations 
in the word $xv^{-1}$ in \eqref{eq:expression_of_convolution}. Therefore the word $(xv^{-1}$ has length $|x|+|v|-2m$, hence, $n+k-m$ is even, that is $(f*g_k)\chi_m = 0 $ if $n+k-m$ is odd. Then let us compute $(f*g_k)\chi_m$ when $n+k-m$ is even and $|n-k|\leqslant m \leqslant n+k$. Let us first look at the case $m=0$. Then there are no cancellations and $v=xu$ for only one word $x$ of length $n$ and $u$ of length $k$, namely the subwords $x$ and $u$ of $xu$ consisting ot the first $n$ letters and the last $k$ letters, respectively. Then, for $v=xu$, $|v|=n+k$, $|x|=n$, $|u|=k$,
\begin{equation}\label{eq:convolution_without_cancellation}
(f*g_k) (v) f(x)\, g(u)
\end{equation}
and  so, for $m=0$,
\begin{align*}
\|(f*g_k)\chi_m\|^2_2
&= \sum_{xu=v,\, |x|=n,\,|u|=k} |f(x)|^2 |g_k(u)|^2
\leqslant \sum_{ |x|=n,\,|u|=k} |f(x)|^2 |g_k(u)|^2\\[.2cm]
& = \|f\|^2_2 \; \|g_k\|^2_2\,.
\end{align*}
Now let us take $m>0$ (of the same parity of $n+k$). Then, if as before $|x|=n$ and $|u|=k$, then the reduced expression of  word $v=xu$  consists of the first $n-p$ letters of $x$ followed by the last $k-p$ letters of $u$.
Then we define new functions $f'$ on the words of length $n-p$ and $k-p$, respectively, by
\begin{align*}
f'(t)&=\left(. \sum_{|v|=p} |f(tv)|^2 \right)^{\frac12} &\text{ if  $|t|=n-p$ and 0 otherwise,} \\[.2cm]
g'(t)&=\left(. \sum_{|v|=p} |g(v^{-1}u))|^2 \right)^{\frac12}  &\text{ if  $|u|=k-p$ and 0 otherwise,} 
\end{align*}.
But every $y\in V$ with $|y|=n$ can be uniquely written as $y=tv$ with $|t|=n-p$ and $|v|=p$. Therefore 
\[
\| f' \|^2_2 = \sum_{|t|=n-p} \left( \sum_{|v|=p} |f(tv)|^2 \right) =\|f\|^2_2\,.
\]
The same argument yields $ \| g' \|^2_2 =  \| g'\|^2_2 $.
Now let $|s|=m$. Then $s=t'u'$ where $t'$ consists of the fisrst $n-p$ letters of $s$ and $u'$ of the last $k-p$ letters. If $s=tu$ with $|t|=n$ and $|u|=k$, then $t=t'v$ and $u=v^{-1}u'$ for some $v$ with $|v|=p$. Hence, by Schwarz inequality nad \eqref{eq:convolution_without_cancellation},
\begin{align*}
|(f*g)(s)| &= \left| \sum_{|v|=p,\, |t'v|=n,\, v^{-1}u'|=k} f(t'v)\,g(v^{-1}u') \right| \\[.2cm]
&= \left| \sum_{|v|=p} f(t'v)\,g(v^{-1}u') \right| \leqslant  \left( \sum_{|v|=p} |f(t'v)|^2 \right)^{\frac12}   \left( \sum_{|v|=p} |g(v^{-1}u'|^2 \right)^{\frac12} 
\\[.2cm]
& =f'(t')\,g'(u') = f'*g'(s) \,.
\end{align*}
Therefore $|(f*g_k)\chi_m| \leqslant (f' *g')\chi_m$ for all functions supported on circles, and 
 \[
 \|(f*g_k)\chi_m\|_2  \leqslant |\ f' *g' \|_2 \leqslant \|f'\|_2\|g'\|_2 = \|f\|_2\,\|g\|_2 \,.
 \]
This proves the claim.
It follows that
\begin{align*}
\sum_{k\geqslant 0} \|(f*g_k)\chi_m\|_2 &\leqslant \|f\|_2 \sum_{k=|n-m|}^{n+m} \|g_k\|_2
\leqslant \|f\|_2 \left(\sum_{k=|n-m|}^{n+m} \|g_k\|_2^2\right)^\frac12 \left( \sum_{k=|n-m|}^{n+m} 1 \right)^\frac12
\end{align*}
Observe that the distance between the integers $n+m$ and $|n-m$ is $\min\{n,m\}$. Therefore $\left( \sum_{k=|n-m|}^{n+m} 1 \right)^\frac12 \leqslant \sqrt{n+1}$. Thus, for all $h=f*g$ with $f$ supported on $C_E(n,e_0)$, one has
\begin{align*}
\|h_m\|_2 \leqslant \sqrt{n+1}\; \|f\|_2 \left( \sum_{k\geqslant 0} \|g_k\|_2^2\right)^\frac12 = \sqrt{n+1}\; \|f\|_2 \, \|g\|_2
\end{align*}
Hence, by interchanging the order of summation, 
\begin{align*}
\|h\|^2_2 &= \sum_{m=0}^\infty \|h_m\|^2_2 \leqslant (n+1) \|f\|^2_2 \sum_{m=0}^\infty  \sum_{j=0}^{\min\{m,n\} } \|g_{m+n-j}\|^2_2\\[.1cm]
&= (n+1) \|f\|^2_2 \sum_{j=0}^n \sum_{m=j}^\infty  \|g_{m+n-j}\|^2_2    \leqslant  (n+1) \|f\|^2_2 \sum_{i=0}^n  \|g_{i}\|^2_2   
=   (n+1) \|f\|^2_2 \, \|g\|^2_2\,.                 
\end{align*}
This completes the proof.
\end{proof}

\begin{theorem}[Haagerup's theorem] 
\label{theo:Haagerup} 
Let $f\in\ell^2(V)$ be such that $\sum_{n=0}^\infty (n+1)\|f\chi_n\|_2 <\infty$. Then $f\in C^*_\lambda$ and
\[
\| f \|_{C^*_\lambda} \leqslant \sum_{n=0}^\infty (n+1)\|f\chi_n\|_2\,.
\]

\end{theorem}
\begin{proof}
Denote by $\widetilde{f}_m$ the truncation of $f$ on tha vertex-ball of radius $n$ around $e_0$, that is,  $\widetilde{f}_m=\sum_{n=0}^m f\chi_n$. For $k>0$, by the triangular inequality and the previous lemma, 
 $\widetilde{f}_{m+k} - \widetilde{f}_m\|_{C^*_\lambda}\leqslant \sum_{m+1}^{m+k} (n+1)\|f\chi_m\|_2 \to 0$ as $m\to\infty$. Therefore the sequence $\widetilde{f}_m$ is a Cauchy sequence in the $C^*_\lambda$-norm, and it converges to $f$. Thus 
 \[
 \|f\|_{C^*_\lambda}=\lim_{m\to\infty} \| \widetilde{f}_m \|_{C^*_\lambda} \leqslant \sum_{n=0}^\infty (n+1)\|f\chi_n\|_2\,.
 \]
 \end{proof}

An estimate for the norm of convolution operators on $\Gamma$ (or equivalently, on $ \graphE$) was given in \cite{Iozzi&Picardello}, hence it holds for convolution operators on $ \ell^2(E)$. Its statement is more general than we need here. The part relevant here is the following inequality:\\
\emph{
If  $f:\Gamma\to \mathC$ is supported on the circle of words of block distance $n>0$ from the identity, then for every $g\in \ell^2(\Gamma)$ one has }
\begin{equation}\label{eq:Haagerup_for_edges}
\|f*g\|_2 \leqslant (q-1)(n+1) \|f\|_2\,\|g\|_2\,.
\end{equation}
Here the constant  is larger than the one originally claimed in \cite{Iozzi&Picardello}, although asymptotically equivalent, but in reality the argument of  \cite{Iozzi&Picardello}*{Theorem 1\,$(ii)$} leads to the constant of \eqref{eq:Haagerup_for_edges}.

\subsection{$\ell^p$-spectra of the Laplace operators on vertices and edges}\label{SubS:spectral_theory}

\begin{theorem}\label{theo:ell^2-spectra}
The following hold:
\begin{enumerate}
\item[$(i)$]
The spherical functions $\phi^V_z$ are positive definite if and only if $\gamma^V(z)$ 
belongs to the real interval $[-1,1]$.
\item[$(ii)$]
The spectrum of $\mu_1$ as a convolution operator on $\ell^2(V)$ is the real segment 
\[
S_V=\left\{\gamma^V(z):  \Real z =1/2 \right\}= \left[\gamma^V\left(\frac12 +i\frac\pi{\ln q}\right), \gamma^V\left(\frac12\right)\right]=\left[-\frac{2\sqrt{q}}{q+1}\,,\; \frac{2\sqrt{q}}{q+1}\right].
\]
Therefore the spectral radius of $\mu_1$ on $\ell^2(V)$ is $\rho(\mu_1)=2\sqrt{q}/(q+1)$, and its resolvent at the eigenvalue $\gamma^V(z)$ is  the convolution operator by the function 
\begin{equation}\label{eq:resolvent_of_mu1}
s_z(v)=\displaystyle  \frac{q+1}{q^{-z}-q^z}\;q^{-z|v|} =\frac 1{c(1-z)}\;\frac 1{q^{-z}-q^{z-1}}\;q^{-z|v|}
\end{equation}
with $\Real z>\frac12$: i.e., $\mu_1 s_z - \gamma^V(z) s_z = \delta_{v_0}$.
Here the $c$-function is as in Proposition \ref{prop:computation_of_vertex-spherical_functions}. (Since $\gamma^V(z)=\gamma^V(1-z)$, the same eigenvalue can be obtained for $\Real z>\frac12$ and $\Real z < \frac12$,
but for $\Real z<\frac12$ the function $s_z$ grows as $|v|$ increases, hence it does not act as a convolution operator on
 $\ell^2(V)$).
\item[$(iii)$] The spherical functions $\phi^E_z$ 
are positive definite if and only if  $\gamma^E(z)$ 
belongs to the real interval $[-1,1]$.

The spectrum of $\eta_1$ on $\ell^2(E)$ is the real segment 
\begin{align*}
S_E=\left\{\gamma^E(z):  \Real z =1/2 \right\}&=\left[\gamma^E\left(\frac12 +i\frac\pi{\ln q}\right), \gamma^E\left(\frac12\right)\right]\\[.2cm]
&=\left[\frac{q-1}{2q}-\frac 1{\sqrt{q}}\,,\;  \frac{q-1}{2q}+\frac 1{\sqrt{q}}\right],
\end{align*}
 its spectral radius $\rho(\eta_1)$ is $\displaystyle\frac{q-1}{2q}+\frac 1{\sqrt{q}}\;$, and its resolvent $\rho_{\gamma^E(z)}$ at the eigenvalue $\gamma^E(z)$ is the convolution operator on $\ell^2(E)$ given by the function 
 \begin{equation}\label{eq:explicit_expression_of_the_edge-resolvent}
 r_z(e)=\displaystyle\frac{2q}{q^{1-z}-q^z-(q-1)}\;q^{-z|e|} = \frac 1{d(1-z)}\;\frac 1{q^{-z}-q^{z-1}}\;q^{-z|e|}
 \end{equation} 
 with $\Real z>\frac12$: i.e., $\eta_1 r_z - \gamma^E(z) r_z = \delta_{e_0}$.  (The $d$-function was introduced in \eqref{eq:d(z)};
 for $\Real z<\frac12$, $r_z$ grows for increasing $|e|$, hence it does not act as a convolution operator on
 $\ell^2(E)$).
 \item[$(iv)$] 
 \begin{align*}
  r_z(e)
 =\frac 1{d(1-z)}\;\frac 1{q^{-z}-q^{z-1}}\;q^{-z|e|}
  \end{align*}
  (the $d$-function was introduced in \eqref{eq:d(z)}).
 \end{enumerate}
 \end{theorem}
\begin{proof} 
The proof is taken from \cite{Figa-Talamanca&Picardello}*{Chapter 3, Lemma 3.2 and Theorem 3.3}.
Parts $(iii)$ and $(iv)$ have been proved recently in \ocite{Iozzi&Picardello-Springer}.

We deal first with functions on $V$ (parts $(i)$ and $(ii)$).
Let $G^V$ be 
the subgroup of $G=\Aut T$ introduced in Subsection \ref{SubS:convolution} in order to label the edges of $T$ with the elements of a group. For simplicity, let us just identify $G^V$ with $V$: this gives a meaning to expressions as $v^{-1}$ where $v$ is a vertex. Now it is clear that, if a spherical functions $\phi^V_z$ is positive definite, then $-1\leqslant \gamma^V(z) \leqslant 1$: this is so because every radial positive definite functions must be real valued and bounded. Indeed, for every $v\in V$, the edges $v$ and $v^{-1}$ have the same length, by construction of $G^V$ (see the Remark quoted above), because the word of the inverse is composed of the same letters but in opposite order. So, since $\phi=\phi^V_z$ is radial, $\phi(v^{-1})=\phi(v)$, but since it is positive definite, $\phi(v^{-1})=\overline{\phi(v)}$. Hence such $\phi$ is real valued. Being positive definite, it is also bounded (by the value at the origin $v_0$). But the  spherical functions $\phi^V_z$ are bounded if and only if $0\leqslant \Real(z) \leqslant 1$, and they are real only if $\gamma^V(z)$ is real (because $\gamma^V(z)=\phi^V_z(v)$ for $|v|=1$; actually, this is an if and only if, thanks to the recurrence relation in  \eqref{eq:recurrence_relation_of_vertex_spherical_function}).

It remains to show that, if  $-1\leqslant \gamma^V(z) \leqslant 1$, then $\phi^V_z$ is positive definite, that is, it induces a  functional $f\mapsto \langle f,\, \phi^V_z \rangle$ that takes positive values on the positive elements $f$ of the involutive algebra $\ell^1(V)$. These positive elements form the positive cone generated by functions of the type $f=h^* *h$, where $h^*(v)=\overline {h(v^{-1})}$: these are positive definite functions. The radialization maps positive definite functions (that is, positive elements of the involutive algebra $\ell^1(V)$), to radial positive definite functions (that is, positive elements of the involutive algebra $\ell_\#^1(V)$) \cite{Figa-Talamanca&Picardello}*{Chapter 3, Lemma 1.3}, so $\calE f$ is a linear combination with positive coefficients of functions of the type $g^* *g$ for some radial $g$. Without loss of generality, we can restrict attention to one such function.
But then, by Remark \ref{rem:radialization},
 the multiplicative property of spherical functions as functionals on the radial algebra (Theorem \ref{theo:characterization_of_spherical_functions}\,$(iv)$) and the fact that $\phi^V_z$ is real valued for $-1\leqslant \gamma^V(z) \leqslant 1$, we have
\begin{align*}
\langle f,\phi^V_z\rangle = \langle \calE f,\phi^V_z\rangle = \langle g^* * g, \phi^V_z\rangle 
=
\overline{\langle g,\phi^V_z\rangle}\,\langle g,\phi^V_z\rangle \geqslant 0.
\end{align*}

This shows that spherical functions give rise to positive multiplicative functionals on on the positive elements $f$ of the involutive algebra $\ell^1$ if and only if their eigenvalues belong to $[-1,1$]: hence it completes the proof of $(i)$. 
Now let us handle $(ii)$. For the relevant background, the reader is referred to \cite{Eymard}.
\\
We have seen that $\phi^V_z$ is positive definite if $\gamma^V(z)\in [-1,1]$, so, in particular, if $z=z_t=\frac12 + it$ with $t\in\mathR$ and if $z=\sigma$ with $\frac12 \leqslant \sigma \leqslant 1$. Note that $\phi_{\frac12 +it}\in\ell^{2+\epsilon}$ (Proposition \ref{prop:spectral_theory_of_spherical_functions}~$(ii)$). By the asymptotic rate of decay of the spherical functions (Propositions \ref{prop:computation_of_vertex-spherical_functions} 
we know that the product $\phi^V_\sigma \phi^V_{\frac12 + it}$ is positive definite and belongs to $\ell^2$, hence to the Fourier algebra $A(H)=\ell^2(V)*\ell^2(V)$. In particular, $\phi^V_\sigma \phi^V_{\frac12 + it}$ has norm 1 in the Banach algebra $B(H)$ of positive definite functions.
On the other hand, $\phi^V_\sigma$ converges pointwise to 1 as $\sigma\to 1$ (actually, $\phi^V_1$ is the constant function 1, and the pointwise behaviour is continuous with respect to $\sigma$). Therefore $\phi^{\frac12 +it}$ is the pointwise limit of functions 
$\phi^V_\sigma \phi^V_{\frac12 + it}$ of norm 1 in $B(H)$, hence it defines a  multiplicative linear functional continuous on $C^*_\lambda$. We have proved that $\gamma^V(\frac12 +it/\ln q)$ belongs to the spectrum of $\mu_1$ in $C^*_\lambda$. Note that $\rho_V:=\gamma^V(\frac12)=\frac {2\sqrt{q}} {q+1}$, $\rho'_V:=\gamma^V(\frac12+i\pi/\ln q)=-\frac {2\sqrt{q}} {q+1}$ and $\{ \gamma^V(z):\, \Real z=
\frac12\}=[\rho'_V,\rho_V]$. 
\\
Now suppose that $\gamma^V(z)\notin [-\rho_V,\rho_V]$. Then $\Real z\neq \frac12$: since $\gamma^V(z)=\gamma^V(1-z)$ we may assume $\Real z>\frac12$. As before, write $y_z(v)=q^{-z|v|}$. We have seen in  Corollary \ref{cor:exponential_are_eigenfunctions_of_Laplace_operators}
that $\mu_1 y_z(v)=\gamma^V(z)\,y_z(v)$ for every $v\neq v_0$. Instead, $\mu_1 y_z(v_0)$ is the value of $y_z$ on vertices of length 1, that is $q^{-z}$; moreover, $y_z(v_0)=1$. Hence 
\[
\left((\mu_1-\gamma^V(z)\,\mathbb I)\, y_z\right)(v_0)=q^{-z}-\gamma^V(z)=\frac{q^{1-z}-q^z}{q+1}\,.
\]
Therefore
\[
(\mu_1-\gamma^V(z)\,\mathbb I)\, y_z = \frac{q^{1-z}-q^z)}{q+1}\,\delta_{v_0}\,.
\]
Thus, for $\Real z> \frac12$, the function $s_z(e)=\displaystyle\frac{q+1}{q^{1-z}-q^z}\;y_z^{-z|e|}$ is the resolvent of $\mu_1$ provided that it is a bounded convolution operator on $\ell^2(V)$. But this follows from Haagerup's inequality for edges, 
$ \|r_z\|_{C^*_\lambda} \leqslant (n+1)  \sum_{n\geqslant 0} (n+1) \|\chi_n\,r_z\|_2$ (
Theorem \ref{theo:Haagerup}). Indeed, if $|e|=n$, then $r_z(n):=r_z(e)\sim y_z^{-nz}$, and so, by writing $\chi_n=\chi_{C^V(n,v_0)}$, we obtain
\[
\sum_{n\geqslant 0} (n+1) \|\chi_n\,r_z\|_2 \sim \sum_{n\geqslant 0} (n+1) \left| q^{-nz}\right| \| \chi_n \|_2= \sum_{n\geqslant 0} (n+1) q^{-(\frac12 -\Real z)n} < \infty,
\]
by \eqref {eq:number_of_edges_of_length_n}.
%
This proves part $(ii)$.
Now let us consider functions on $E$
and prove part $(iii)$.
Let us identify $E$ with the simply transitive subgroup $\Gamma$ of $G$ introduced in Subsection \ref{SubS:convolution} in order to label  edges as group elements: this gives a meaning to expressions such as $e^{-1}$ where $e$ is an edge. 
Noticing that, for every $e\in E$, the edges $e$ and $e^{-1}$ have the same length by construction of $\Gamma$, we see that part $(iii)$ follows from the same argument of parts $(i)$ and $(ii)$.
To prove $(iv)$, suppose that $\gamma(z)\notin [-\rho,\rho]$. Then $\Real z\neq \frac12$: since $\gamma(z)=\gamma(1-z)$ we may assume $\Real z>\frac12$. As before, write $y_z(e)=q^{-z|e|}$. We know from Corollary \ref{cor:exponential_are_eigenfunctions_of_Laplace_operators}
that $\eta_1 y_z(e)=\gamma(z)\,y_z(e)$ for every $e\neq e_0$. Instead, $\eta_1 y_z(e_0)$ is the value of $y_z$ on edges of length 1, that is $q^{-z}$; moreover, $y_z(e_0)=1$. Hence 
\[
\left((\eta_1-\gamma(z)\,\mathbb I)\, y_z\right)(e_0)=q^{-z}-\gamma(z)=\frac{q^{1-z}-q^z-(q-1)}{2q}\,.
\]
Therefore
\begin{equation}\label{eq:resolvent_of_eta1}
(\eta_1-\gamma(z)\,\mathbb I)\, y_z = \frac{q^{1-z}-q^z-(q-1)}{2q}\,\delta_{e_0}\,.
\end{equation}
Thus, for $\Real z> \frac12$, the function $r_z(e)=\displaystyle\frac{2q}{q^{1-z}-q^z-(q-1)}\;y_z^{-z|e|}$ is the resolvent of $\eta_1$ provided that it is a bounded convolution operator on $\ell^2$. But this follows from Haagerup's inequality for edges \eqref{eq:Haagerup_for_edges}, 
$ \|r_z\|_{C^*_\lambda} \leqslant (q-1)  \sum_{n\geqslant 0} (n+1) \|\chi_n\,r_z\|_2$. So it is enough to prove that the series converges. But, if $|e|=n$, then $
r_z(e)\sim q^{-nz}$, and so
\[
\sum_{n\geqslant 0} (n+1) \|\chi_n\,r_z\|_2 \sim \sum_{n\geqslant 0} (n+1) \left| q^{-nz}\right| \| \chi_n \|_2= \sum_{n\geqslant 0} (n+1) q^{(\frac12 -\Real z)n} < \infty,
\]
by \eqref{eq:number_of_edges_of_length_n}.
 The last identity in \eqref{eq:explicit_expression_of_the_edge-resolvent} follows from the definition of $c(z)$ in 
Proposition \ref{prop:computation_of_vertex-spherical_functions}.
This completes the proof.
\end{proof}

In order to compute the spectrum of $\mu_1$ and $\eta_1$ as convoution operators on $\ell^p$, we need a preliminary well-known result (see, for instance, the proof of \cite{Figa-Talamanca&Picardello}*{Theorem 3.3}).
\begin{proposition}\label{prop:cvp=cvq}
Let $G$ be a locally compact group with Haar measure $m$ and let $1\leqslant p,q<\leqslant\infty$ with $\frac1p + \frac1q=1$. Let us write $L^p$ for $L^p(G,m)$.  
\begin{enumerate}
  \item[$(i)$] If $\frac1p + \frac1q = 1$, the algebra of left convolution operators bounded on $L^p$ is isometrically isomorphic to the algebra of right convolution operators on $L^q$.
    \item[$(ii)$] For any function $h$ on $G$, write $\check g(x)=g(x^{-1})$, and let $\Radial$ be a commutative convolution algebra of functions $g$ on $G$ such that $g=\check g$ (for instance, the algebra of radial functions on the vertices or on the edges of a homogeneous tree, with the convolution product induced by a free group or free product that acts simply transitively on the tree: see Example \ref{example:trees_as_Cayley_graphs_of_free_products}).
If $G$ is unimodular, then there is an isometric isomorphism between the algebras $Cv^p_\#$ and $Cv^q_\#$ of bounded convolution operators on $L^p$ (respectively, $L^q$) generated by $\Radial$ (hence commutative, so at the same time left and right convolution operators).
\end{enumerate}
\end{proposition}
\begin{proof}
Let $f\in LCv^p$ and $h\in L^q$. Then, if $e$ denotes the identity element of $G$, for every $h\in L^q$ we have 
$
\|h*f\|_q = \sup \{ |\langle g,h*f\rangle| :\, \|g\|_p\leqslant 1\}$. But $|\langle g,h*f\rangle| =|g*h*f(v)| = |f*g*h(v)|
\leqslant \|f*g\|_p \|h\|_q \leqslant \|f\|_{LCv^p} \|g\|_p \|h\|_q$. Therefore $ \|f\|_{LCv^p} \leqslant  \|f\|_{RCv^q}$.
A symmetric argument shows that $ \|f\|_{LCv^q} \leqslant  \|f\|_{RCv^p}$. This proves $(i)$.

To prove $(ii)$, take $f\in\Radial$. Since $G$ is unimodular, its Haar measure $m$ is invariant under inversion, in the sense that $\int_G h(y^{-1})\,dm(y) = \int_G h(y)\,dm(y)$ for every $h$ (and actually, the map $g\mapsto \check g$ is an isometry of $L^p$). Therefore $f*g(x) = \int_G f(xy) g(y^{-1}) \,dm(y) = \int_G \check f(xy) g(y^{-1})  \,dm(y) =
\int_G f(yx^{-1}) g(y)  \,dm(y)$  Therefore
$f*g(x) = \int_G f(y) g(yx) \,dy = \int_G f(y) \check g(x^{-1}y^{-1}) \,dy = \check g * f(x^{-1}) = (\check g * f)\,\check{ } \,(x)$. But since $g\mapsto \check g$ is an isometry of $L^p$, this shows that $\| f*g\|_p = \| g*f\|_p$ for every $g\in L^p$, and so $LCv_\#^p)$ and $RCv_\#^p$ are isometrically isomorphic. By part $(i)$, $LCv^p$ is isometrically isomorphic to $RCv^q$. Therefore $Cv^p_\#$ and $Cv^q_\#$ are isometrically isomorphic.
\end{proof}

We also need the following definition and estimate:
\begin{definition} [Spherical polynomials]\label{def:spherical_polynomials}
Let $T=T_q$. The -spherical functions $\phi^V_z$ 
and $\phi^E_z$ 
satisfy the second order recurrence relations 
 \eqref{eq:recurrence_relation_of_vertex_spherical_function} and \ref{eq:recurrence_relation_of_edge_spherical_function}, respectively.
It follows that 
 there are polynomials $P_n$ 
 and $Q_n$ 
 of degree $n$ such that $\mu_n=P_n(\mu_1)$
 and $\eta_n = Q_n(\eta_1)$. 
 On the other hand, the functionals $f\mapsto \langle \phi^V_z, f\rangle$ 
 and  $g\mapsto \langle \phi^E_z, g\rangle$ 
 are multiplicative on the radial algebras (parts $(iii)$ and $(vii)$ of Theorem \ref{theo:characterization_of_spherical_functions}), hence 
\begin{equation}\label{eq:spherical_polynomials}
\begin{split}
\phi^V_z(n)=\langle \phi_z^V,\,\mu_n\rangle =  \langle \phi_z^V,\,P_n(\mu_1)\rangle = P_n(\langle \phi_z^V,\,\mu_1\rangle) = P_n(\gamma^V_z),
\\[.2cm]
\phi^E_z(n)=\langle \phi_z^E,\,\eta_n\rangle =  \langle \phi_z^E,\,Q_n(\eta_1)\rangle = Q_n(\langle \phi_z^E,\,\eta_1\rangle) = Q_n(\gamma^E_z).
\end{split}
\end{equation}
The functions  $P_n$ are called \emph{spherical polynomials}. 
By abuse of notation, sometimes also the functions $\widetilde{P}_n(z)=P_n(\gamma^V(z))$ are called
 spherical polynomials (although they are not polynomials in the variable $z$).
\end{definition}

\begin{corollary} \label{cor:majorization_principle} Let $T=T_q$. For every $z\in\mathC$,
$|\widetilde{P_n}(z)|\leqslant |\widetilde{P_n}(\Real z)|$. 
and $|\widetilde{Q_n}(z)|\leqslant |\widetilde{Q_n}(\Real z)|$. 
So, by \eqref{eq:spherical_polynomials}, 
one has the \emph{majorization principle} $|\phi^V_z(n)|\leqslant |\phi^V_{\Real z}(n)|$. 
and $|\phi^V_z(n)|\leqslant |\phi^V_{\Real z}(n)|$.
\end{corollary}
\begin{proof}
This follows from the following estimate coming directly from  \eqref{eq:spherical_polynomials} and Remark \ref{rem:spherical_functions_and_Poisson_kernels}: 
if $|x|=n$,
\begin{align*}
|\widetilde{P_n}(z)|&=\left|\phi^V_z(n)\right|=\left| \int_\Omega \Poiss^z(x,\omega)\,d\nu_{v_0}(\omega)\right| \\[.2cm]
&\leqslant 
\int_\Omega \Poiss^{\Real z}(x,\omega)\,d\nu_{v_0}(\omega)
 \leqslant \left|\phi^V_{\Real z}(n)\right|=\left|\widetilde{P_n}(\Real z)\right|.
\end{align*}
and similarly for $\phi^E_z$ and $Q_n$.
\end{proof}

\begin{theorem}\label{theo:ell^p-spectra}
Let $s_z$, $r_z$ be the resolvent functions introduced in Theorem \ref{theo:ell^2-spectra}.
\begin{enumerate}
  \item[$(i)$]
  The spectrum of $\mu_1$ as a convolution operator on $\ell^p(V)$ for $1 < p < 2$ and on $\ell^q(V)$ for $2 < q < \infty$, $\frac1p + \frac1q = 1$, i.e. the Gelfand spectrum of the commutative Banach algebra $Cv_\#^p(V)$ of radial convolution operators on $\ell^p(V)$, is the ellipse $S_V^p=\{\gamma^V(z):\, \frac1q \leqslant \Real z \leqslant \frac1p$ (that, for $p=\frac12$, becomes a segment). If $\gamma^V(z)\notin S_V^p$ and $\Real z> \frac1p$, then $s_z$ defines a bounded operator on $\ell^p(V)$ and on $\ell^q(V)$ and satisfies the resolvent equation $(\mu_1 - \gamma^V(z) \mathbb I) s_z = \delta_{v_0}$.
  
   \item[$(ii)$]
  The spectrum of $\eta_1$ as a convolution operator on $\ell^p(E)$ for $1 < p < 2$ and on $\ell^q(E)$ for $2 < q < \infty$, $\frac1p + \frac1q = 1$, i.e. the Gelfand spectrum of the commutative Banach algebra $Cv_\#^p(E)$ of radial convolution operators on $\ell^p(E)$, is the ellipse $S_E^{(p)}=\{\gamma^V(z):\, \frac1q \leqslant \Real z \leqslant 1-\frac1q=\frac1p$, that is, the translate $\frac{q-1}{2q} + S_V^p$ of the spectrum of $\mu_1$ on $\ell^p(V)$. If $\gamma^V(z)\notin S_E^{(p)}$ and $\Real z> \frac1p$, then $r_z$ defines a bounded operator on $\ell^p(E)$ and on $\ell^q(E)$ and satisfies the resolvent equation $(\mu_1 - \gamma^V(z) \mathbb I) r_z = \delta_{v_0}$. 
 \end{enumerate}
\end{theorem}
\begin{proof}
Part $(i)$ has been known for a long time \cite{Figa-Talamanca&Picardello}*{Chapter 3, Theorem 3.3}.
Part $(ii)$ has been proved recently in \ocite{Casadio_Tarabusi&Picardello-spherical_functions_on_edges}. We prove  part $(i)$ only, since the arguments are the same.

By Proposition \eqref{eq:vertex_spherical_functions_as_eigenfunctions}, $\mu_1  \phi^V_z = \gamma^V(z) \phi^V_z$. If $\frac1q < \Real z < 1-\frac1q$, then $\phi^V_z \in \ell^q(V)$ by part $(i)$ of Proposition \ref{prop:spectral_theory_of_spherical_functions}.
Therefore $\mu_1 -\gamma(z) \mu_0$ is not invertible as an operator on $\ell^q(V)$. By part $(ii)$ of Proposition \ref{prop:cvp=cvq}, $Cv^p_\#=Cv^q_\#$, and so $\mu_1 -\gamma(z) \mu_0$ is not invertible on $\ell^p(V)$: therefore $\{
\gamma^V(z):\, \frac1q < \Real z < \frac1p =1-\frac1q\}$ is contained in the spectrum of $\mu_1$ on $\ell^p(V)$ and on $\ell^q(V)$, and so is its closure $S^p_V$. 
\\
In order to prove the converse inclusion it is enough to show that, for $\Real z>\frac1p$, the resolvent function $s_z$
of \eqref{eq:resolvent_of_mu1} (that satisfies the identity $)\mu_1 -\gamma^V(z) \mu_0) s_z = \mu_0$) gives rise to a bounded left convolution operator on $\ell^p(V)$. To make notation simpler, let us set $y_z(v)=q^{-z|v|}$. 
We observed in \eqref{eq:number_of_edges_of_length_n} that the circle of edges of length $n$ has cardinality $(q+1)q^{n-1}$. Then, by \eqref{eq:edge_spherical_functions_as linear_combinations_of_exponentials}, 
there exists a constant $C$ such that
\begin{equation}\label{eq:norm_of_s_z_as_convolutor_on_ell_p}
\begin{split}
\|s_z\|_{Cv_\#^p(V)} &\leqslant C \|y_z\|_{Cv_\#^p(V)} = C \sum_{n=0}^\infty q^{-n\Real z} q^n \| \mu_n\|_{Cv_\#^p(V)} 
\\
&
= C \sum_{n=0}^\infty q^{(1-\Real z)n} \| \mu_n\|_{Cv_\#^p(V)}\,.
\end{split}
\end{equation}
Since $\ell^1(V)$ is a Banach algebra with identity, the norm of a function $f\in \ell^1$ as a convolution operator on $\ell^1$ is its $\ell^1$-norm. Therefore, observing that $Cv_\#^2(V) = C^*_\lambda(V)$ and applying the Riesz convexity theorem, we see that
\begin{equation}\label{eq:norm_of_mu_n_as_convolutor_on_ell_p}
\| \mu_n\|_{Cv_\#^p(V)} \leqslant \| \mu_n\|_1^{\frac1p - \frac 1q} \|\mu_n\|_{C^*_\lambda(V)}^{\frac2q} = \|\mu_n\|_{C^*_\lambda(V)}^{\frac2q},
\end{equation}
since $\| \mu_n \|_1 = 1$. 
For future use, let us quote the corresponding inequality for $\| \eta_n\|_{Cv_\#^p(E)}$:

\begin{align}\label{eq:norm_of_r_z_as_convolutor_on_ell_p}
\|r_z\|_{\Conv_\#^p} &\leqslant C \|y_z\|_{\Conv_\#^p} = 2C \sum_{n=0}^\infty q^{(1-\Real z)n} \| \eta_n\|_{\Conv_\#^p}\,.
\end{align}

As in Theorem \ref{theo:ell^2-spectra}, denote by $S_V$ the spectrum of $\mu_1$ in $C^*_\lambda(V)$.

Let $Q_n$ be the spherical polynomial of Definition \ref{def:spherical_polynomials}.
Clearly, $Q_n$ maps surjectively the eigenvalues $\gamma^V(z)$ of $\mu_1$ to the eigenvalues of $Q_n(\mu_1)=\mu_n$. Therefore the eigenvalues of $\mu_n$ are $\{Q_n(\gamma^V(z)=\phi^V_z(n): \,\gamma^V(z)\in S_V\}$.
 Hence
$\|\mu_n\|_{C^*_\lambda(V)} = \sup\{ |\phi^V_z(n)|:\, \gamma^V(z) \in S_V\}$. Observe that $\gamma^V(z) \in S_V$ if and only if $\Real z = \frac12$. 
Then, by the majorization principle of Corollary \ref{cor:majorization_principle},
\begin{equation}\label{eq:majorization_for_spherical_polynomials}
\sup\{ |\phi^V_z(n)|:\, \gamma^V(z) \in S_V\} = \sup_t |\phi^V_{\frac12 + it} (n)| = \phi^V_{\frac12} (n).
\end{equation}
Therefore 
\[
\|\mu_n\|_{C^*_\lambda(V)}^{\frac2q} = \phi^V_{\frac12}(n)^{\frac2q} =\left(1+\frac{q-1}{q+1}\;n\right)^{\frac2q} q^{-\frac nq}.
\]
Hence, by \eqref{eq:norm_of_mu_n_as_convolutor_on_ell_p},
%
\[
\| V\|_{Cv_\#^p(V)} < C \sum_{n=0}^\infty q^{(1-\Real z)n} \left(1+\frac{q+1}{q-1}\;n\right)^{\frac2q} q^{-\frac nq} < \infty
\]
if $\Real z> \frac1p$. This shows that $V$ is a bounded operator on $\ell^p$ when $\Real z> \frac1p$, and completes  the proof.
\end{proof}

\begin{remark}
\label{rem:change_of_eigenvalues_under_parity} 
The alternating function $\epsilon(v)=(-1)^{|v|}$, i.e. the parity of vertices, is an eigenfunction of $\mu_1$ with eigenvalue $-1$, and multiplication by $\epsilon(v)$ maps eigenfunctions of $\mu_1$ to other eigenfunctions of $\mu_1$, since $\epsilon(u)=-\epsilon(v)$ for all neighbors $u$ of each vertex $v$. Indeed, if $h$ is an eigenfunction with eigenvalue $\gamma^V $, then $\epsilon h$ has eigenvalue $-\gamma^V $, and the spectrum of $\mu_1$ is invariant under reflection around the origin, as already observed. Note that $-\gamma^V (z)=\gamma^V (z+i\pi/\ln q)$, and so the resolvent at the eigenvalue $-\gamma^V (z)$ is $\epsilon s_z = s_{z+i\pi/\ln q}$. Moreover, $\epsilon$ does not change if we move the reference vertex $v_0$ to another vertex at even distance, and becomes $-\epsilon$ if we move the reference vertex by an odd distance.

The same is not  true for edges. 
Indeed, unless $e=e_0$, there are edges $e'\sim e$ such that $|e'|=|e|$. If we set $\epsilon(e)=(-1)^{|e|}$, now $\epsilon$ depends on the reference edge in a non-trivial way, and $\gamma^E(z+i\pi/\ln q)\neq -\gamma^E(z)$. We show that 
$\epsilon r_z=r_{\widetilde{z}}$ for some $\widetilde{z}\in\mathC$.
Indeed, we claim that $\epsilon r_z$ is a multiple of the resolvent at an eigenvalue $\gamma^E(\widetilde{z})$ given by
$\gamma^E(\widetilde{z}) 
= \sigma(\gamma^E(z))$,
where $\sigma$ is the central reflection of  $\mathC$ around the center $\frac{q-1}{2q}$ of the $\ell^2$ spectrum $S$, i.e., $\sigma(w)=\frac{q-1}q - w$. That is,
\begin{equation}\label{eq:change_of_eigenvalue_for_parity_multiplication_of_resolvents_for_edges}
\gamma^E(\widetilde{z})
= \frac{q-1}q-\gamma^E(z)\,.
\end{equation}
So $\widetilde{z}\in S$ if (and only if) $z\in S$. 

Let us prove the claim. Again by Remark \ref{rem:distribution_of_length_of_adjoining_edges}, 
for every $e\neq e_0$,
\[
\eta_1 (\epsilon r_z) (e) =\frac{-q^{1-z}-q^z+(q-1)}{2q}\;\epsilon r_z  (e)= \left(\frac{q-1}q - \gamma^E(z)\right) \epsilon r_z (e)\,.
\]
 This proves the claim: a constant multiple $u_z$ of $\epsilon r_z$ satisfies the resolvent equation
 \eqref{eq:resolvent_of_eta1}
 at the eigenvalue $\displaystyle \frac{q-1}q - \gamma^E(z)$. Moreover, since $|\epsilon r_z|=|r_z|$, by \eqref{eq:norm_of_r_z_as_convolutor_on_ell_p}
 $u_z$ satisfies Haagerup's estimate \eqref{eq:Haagerup_for_edges},  hence it is the $\ell^2-$resolvent of $\eta_1$ at this eigenvalue. 

By the expression \ref{eq:gamma} of $\gamma^E$, \eqref{eq:change_of_eigenvalue_for_parity_multiplication_of_resolvents_for_edges} reduces to 
$q^{\widetilde{z}}+q^{1-\widetilde{z}}=-\left( q^{z}+q^{1-z} \right)$. Equivalently, by writing $z$, $\widetilde{z}\in S$ 
as $z=\frac12 + it$, $\widetilde{z}=\frac12 +is$ with $t,s\in [0,\,\pi/\ln q)$, one has $\cos(t\ln q)=-\cos(s\ln q)$. The solutions are $t\ln q=\pi\pm s\ln q$, that is,  restricting to $[0,\,\frac \pi{\ln q}]$,
\begin{equation}\label{eq:z-tilde}
\widetilde{z} = z + \frac{ik\pi}{\ln q} \quad \text{ or }\quad  \widetilde{z} = \overline{z} + \frac{ik\pi}{\ln q}
\end{equation}
with $k=0,1$. %
\end{remark}

\section[The  Plancherel formula]{
The Plancherel formula for the spherical Fourier transform
}\label{Sect:hom_Plancherel}
\begin{corollary}\label{cor: convolutions_with_vertex-spherical_functions}
The following hold:
\begin{enumerate}
\item[$(i)$] $\widehat{\mu_1}(z)=\gamma^V(z)$. 
\item[$(ii)$] For every $h\in\ell^1_\#(V)$ one has $\phi^V_z*h=h*\phi^V_z=\widehat{h}(z) \,\phi^V_z$. 
\item[$(iii)$] If $\phi^V_z$ is real (that is, if $\gamma(z)\in\mathR$, or equivalently, $\Real z=1/2$ or $\Imag z = in\pi/\ln q$, $n\in\mathZ$), then, for every $h\in\ell^1_\#(V)$, the $\ell^2$ inner product $( \phi^V_z*h,h )_2 = \sum_{v\in V} \phi^V_z*h(v)\,\overline{h(v)}=
\langle \phi^V_z*h,\overline h\rangle$ is equal to $|\widehat h(z)|^2$. 
\end{enumerate}
\end{corollary}
\begin{proof}
The spherical functions $\phi^V_z$ 
are radial and, by Proposition \ref{prop:initial_value_of_spherical_functions},
 take the value $\gamma^V(z)$ on vertices 
 of length 1: this proves part $(i)$.
By \eqref{eq:vertex_spherical_functions_as_eigenfunctions}, 
$\phi^V_z*\mu_1=\mu_1*\phi^V_z=\gamma^V(z) \,\phi^V_z$. Therefore, by \eqref{eq:spherical_polynomials}, every finitely supported radial function $h$ on $V$ 
is a polynomial in $\mu_1$, 
say $h=P(\mu_1)$. 
Hence $\phi^V_z*h=h*\phi^V_z=P(\gamma^V(z))\,\phi^V_z=\widehat{h}(z)\,\phi^V_z$. 
This proves $(ii)$ for finitely supported functions, and the extension to the $\ell^1$ completion is obvious. 
In particular,   if $h$ is radial, $\langle \phi _z,  h\rangle = \phi _z*h (e_0)  = \widehat  { h}(z)$. Then, if $\phi _z$ is real,  $\langle \phi _z, \overline h\rangle = 
\overline{\widehat{h}(z)}$.
Hence $\langle \phi _z*h, \overline h\rangle = \widehat{h}(z) \langle \phi _z,\overline h\rangle = 
\widehat{h}(z)\,\overline{\widehat{h}(z)}$ whenever $\phi _z$ is real.
\end{proof}

Exactly in the same way one has the analoghous results for functions on edges:

The spherical Fourier transform $\widehat{h}(z)$ of a function $h\in\ell^1$ at $z\in\mathC$ is defined as 
\[
\widehat{h}(z) = \langle h,\phi _z\rangle = \sum_{e\in E} h(e)\,\phi _z(e).
\]
\begin{corollary}\label{cor:  convolutions_with_edge-spherical_functions}
\begin{enumerate}
\item[$(i)$] 
$\widehat{\eta_1}(z)=\gamma(z)$.
\item[$(ii)$] For every $h\in\ell^1_\#$ and $\phi_z\in\ell^\infty$ one has $\phi _z*h=h*\phi _z=\widehat{h}(z) \,\phi _z$. \item[$(iii)$] If $\phi _z$ is bounded and real (that is, if $\gamma(z)\in\mathR$, or, equivalently, if $\Real z=1/2$ or if $\Imag z = in\pi/\ln q$, $n\in\mathZ$ and $0\leqslant \Real z\leqslant 1$), then, for every $h\in\ell^1_\#$, the $\ell^2$ inner product $( \phi _z*h,h ) = 
\langle \phi _z*h,\overline h\rangle$ is equal to $|\widehat h(z)|^2$.
\end{enumerate}
\end{corollary}

A celebrated Harish-Chandra theorem  gives the Plancherel measure for the spherical Fourier transform on semi-simple Lie groups in terms of the $c$-function. An analogue for free groups acting on trees
 has been proved in \cite{Figa-Talamanca&Picardello}*{Chapter 3, Section 4} by direct computation: it yields a similar inversion formula on $\ell^1(V)$. The inversion formula is the following.
 Let $J$ be an interval in the imaginary line such that $\gamma^V(J)=\spectrum(\mu_1)=[-\rho_V,\rho_V]$ and $\gamma^V$ is a bijection of $J$ on $[-\rho_V,\rho_V]$ (for instance, $J=\{ \frac12 + it:\, 0\leqslant t \leqslant \frac \pi {\ln q} \}$. Then
\[
h(v_0)= C_V \int_J \frac{\widehat {h}(z)} {|c(z)|^2} \, dz = C_V \int_0^{\pi/\ln q} \frac{\widehat{h}\left(\frac 12 + it\right)}{\left|c(\frac 12 +it)\right|^2}
\;dt\,,
\]
where $C_V$ is a constant.  We could have proved the same formula before  
by methods of integral geometry (the Fourier slice theorem). 
However, we prefer to give a new proof of both inversion formulas via the more elegant approach followed in \cite{Faraut&Picardello} for free products of cyclic groups. This approach is based upon the well-known Carleman formula \ocite{Dunford&Schwartz}*{pg. 192}: for any $h\in\ell^2_\#(V)$ and any continuous compactly supported $u$ on $\mathR$, the spectral resolution $E(d\lambda)$ of the operator $\eta_1$ on $\ell^2$ is given by
\begin{equation}\label{eq:Carleman}
\int^{\infty}_{-\infty} u(\lambda)\,( E(d\lambda)\,h,h)_2 =- \frac 1{2i}\lim_{\epsilon\to 0^+} \int^{\infty}_{-\infty} u(\lambda)\,\bigl((\xi_{\lambda+i\epsilon}-\xi_{\lambda-i\epsilon})*h,h \bigr)_2 \,d\lambda,
\end{equation}
where, as in the last statement of Theorem \ref{theo:ell^2-spectra}\,$(iii)$, the function $\rho_\lambda$ is the resolvent of $\mu_1$ on $\ell^2(V)$ at the eigenvalue $\lambda$. Note that the resolvent is holomorphic outside the spectrum of $\mu_1$, hence the limit at the right hand side vanishes outside $\spectrum(\mu_1)$ and the integration domain can be limited to the interval $\spectrum(\mu_1)$.
\begin{theorem}[The Plancherel theorem for the spherical Fourier transform ]\label{theo:Plancherel}
Let us denote by $c$ be the coefficient of the splitting of the vertex-sherical function as linear combination of exponentials as in \eqref{eq:vertex_spherical_functions_as linear_combinations_of_exponentials} and $d$ ithe corresponding coefficient for edge-spherical functions in \eqref{eq:edge_spherical_functions_as linear_combinations_of_exponentials}.
Then, for every $h\in\ell^1_\#(V)$, the following Plancherel formula holds:
\[
\|h\|_2^2=\frac {q\,\ln q}{2(q+1)} \int_0^{\pi/\ln q} \left|\widehat{h}\left( \frac12+it \right)\right|^2\left| c\left(\frac12+it\right)\right|^{-2}dt\,,
\]
and, for every $h\in\ell^1(V)$ and $v\in V$, the following inversion formula holds:
\[
h(v_0)=\frac {q\,\ln q}{2(q+1)} \int_0^{\pi/\ln q} \widehat{h}\left( \frac12+it \right)\,\left| c\left(\frac12+it\right)\right|^{-2}dt.
\]
Similar formulas hold for every $h\in\ell^1_\#(E)$:
For every $h\in\ell^1_\#(E)$, 
\[
\|h\|_2^2=\frac{\ln q}4 \int_0^{\pi/\ln q} \left|\widehat{h}\left( \frac12+it \right)\right|^2\left| d\left(\frac12+it\right)\right|^{-2}dt\,,
\]
and for every $e\in E$,
\[
h(e)=\frac{\ln q}4 \int_0^{\pi/\ln q} \widehat{h}\left( \frac12+it \right)\phi _{\frac12+it}(e)\left|d\left(\frac12+it\right)\right|^{-2}dt.
\]
\end{theorem}

\begin{proof}
These results for functions on $V$ have been known for a long time. The approach followed here is inspired by \cite{Faraut&Picardello}. This paper deals with free products instead of free groups, hence it can be used equivalently for $E$ (see Subsection \ref {SubS:convolution}), with an appropriate reformulation. Since this reformulation
is quite recent \cite{Casadio&Picardello-semihomogeneous_spherical_functions}, and the arguments are nearly the same, we only give the proof for $E$. For simplicity, in this proof we write $\gamma$ instead of $\gamma^E$.

The resolvent at the eigenvalue $\gamma(z)$ is the exponential function $r_z$ computed in Theorem \ref{theo:ell^2-spectra}\,$(iii)$. Recall that $\xi_\gamma=r_z$. 
 If $z=\frac12+it$, it follows from \eqref{eq:gamma} that $\gamma(z)=\frac{q-1}{2q} +\frac{\sqrt{q}\cos(t\ln q)}q$, and 
 \begin{equation}\label{eq:small_increments_of_gamma}
 \begin{split}
 \gamma(z+\delta)&=\gamma(z)+\frac {(q^\delta-1)q^z+(q^{-\delta}-1)q^{1-z}}{2q} \\[.2cm]
 &= \gamma(z)+\frac {(\cosh(\delta\ln q)-1) \cos(t\ln q)  + i  \sinh(\delta\ln q) \sin(t\ln q)}{   \sqrt{q}     }\;.
 \end{split}
 \end{equation}

 Let us write $i\varepsilon=
 \gamma(z+\delta)- \gamma(z)$, hence 
 \begin{equation}\label{eq:epsilon_as_a_function_of_delta}
 \gamma(z)+i\varepsilon= \gamma(z+\delta).
 \end{equation}
 Then, by  \eqref{eq:small_increments_of_gamma},
 \begin{equation}\label{eq:epsilon_come_funzione_di_delta}
 \begin{split}
 \varepsilon
 &=  \frac {\sinh(\delta\ln q) \sin(t\ln q)} {   \sqrt{q}     } - i \frac {(\cosh(\delta\ln q)-1)\cos(t\ln q)} {   \sqrt{q}     }
 \\[.2cm]
 &=  \frac {\sin(t\ln q)} {   \sqrt{q}     }\;\delta\ln q + O(\delta^2),
 \end{split}
 \end{equation}
 and
 $\delta(\varepsilon)=
 \frac {   \sqrt{q}     }{\ln q \;\sin(t\ln q)} \;\varepsilon + O(\varepsilon^2)$.
 Therefore 
 $\delta$ is asymptotically proportional to $\varepsilon$ for $\varepsilon\to 0$, with a non-negative constant of proportionality. Hence, as 
 $\mathR\ni \varepsilon \to 0^+$, $\delta(\varepsilon)$ tends to 0 and is asymptotically tangent to the real axis, unless the coefficient of proportionality vanishes. The coefficient vanishes only if $\sin (t \ln q)=0$, that is at the extreme points of the spectrum $S$, and this does not affect the spectral measure in the interior of the spectral interval (it might only produce atoms at the endpoints).   Note that $\gamma\left(\frac 12 + w\right) = \gamma\left(\frac 12 - w\right)$.
 By  
\eqref{eq:epsilon_come_funzione_di_delta},
 $\gamma\left(\frac12 + it\right)-i\varepsilon=\gamma\left(\frac12 + it-\delta +O(\delta^2)\right)=\gamma\left(\frac12 - it+\delta +O(\delta^2)\right)$. Therefore 
we can rewrite the right-hand side of Carleman's formula \eqref{eq:Carleman} as
\begin{multline*}
- \frac 1{2i}\lim_{\varepsilon\to 0^+} \int_{S} u(\gamma)\,\bigl((\xi_{\gamma+i\varepsilon}-\xi_{\gamma-i\varepsilon})*h,h \bigr)_2 \,d\gamma 
\\[.2cm]
\begin{split}
= - \frac 1{2i} \lim_{\delta\to 0^+} &\int_{0}^{\pi/\ln q} u\left(\frac12 + it\right)\,
 \\[.2cm]
 &\quad \cdot \bigl((\xi_{\gamma\left(\frac12+it+\delta\right)}-\xi_{\gamma\left(\frac12+it-\delta +O(\delta^2)\right)})*h,h \bigr)_2 \;\gamma'\left(\frac12+it\right)\,dt.
 \end{split}
\end{multline*}
 It follows 
 from the expression of $r_z$ given in Theorem \ref{theo:ell^2-spectra}\,$(ii)$ that, for $z=\frac12 +it$,
 \begin{align*}
\xi_{\gamma(z+\delta)} = r_{z+\delta}=\displaystyle 
\frac 1{d(\frac12 -it-\delta)}\;\frac 1{q^{-\frac12 -it-\delta}-q^{-\frac12 +it+\delta}}\;q^{-\left(\frac12 +it+\delta\right)|e|}\,.
\end{align*} 
Note that, by the expression of the $c$-function in \eqref{eq:c(z)}, the denominator in the last identity has a non-zero limit as $\delta\to 0$, hence $\lim_{\delta\to 0} r_{\frac12 +it\pm\delta}=\lim_{\epsilon\to 0} \xi_{\gamma\left(\frac12 +it\right) \pm i\epsilon}$ exists and is non-zero.

 Because of the rate of decay of $\xi_{\gamma(z)}$, in the equality $\xi_{\gamma(z+\delta)}=r_{z+\delta}$, if $\delta\in\mathR$ we must limit attention to $\delta>0$ because the $\ell^2-$resolvent $r_{z+\delta}$ must belong to $\ell^2$ (and in general we should limit attention to $\Real\delta>0$). This means that, when approaching $\gamma(\frac 12 +it)\in S$ with $\gamma(\frac 12 +it)\pm i\varepsilon=\gamma\left(\frac 12 + it\pm\delta +O(\delta^2)\right)=\gamma\left(\frac 12 \pm it + \delta+O(\delta^2)\right)$, we are approaching the point $z=\frac12 \pm it$ in the $z-$plane always from the right half space. In terms of the variable $\gamma(z)$, although $\delta>0$ implies that $\epsilon\to 0^+$, the subscript $\gamma\left(\frac12+it+\delta\right)$ in the integrand, for $0 < t < \pi/\ln q$,  approaches $\gamma\left(\frac12 + it\right)$ from above, and the subscript $\gamma\left(\frac12+it-\delta +O(\delta^2)\right)$ approaches the same eigenvalue, now regarded as $\gamma\left(\frac12 + i(-t)\right)$, 
from below as requested. 
\par
Note also that, by \eqref{eq:gamma}, 
\[
d\gamma(z)=\frac{\left(q^z-q^{1-z}\right)\ln q}{2q}\,dz=\frac{\ln q}2\;(q^{z-1}-q^{-z})\,dz\,.
\]
Therefore, for $0<t<\pi/\ln q$,
\begin{multline*}
\lim_{\varepsilon\to 0^+} \left(\xi_{\gamma(\frac12 +it)+i\varepsilon}-\xi_{\gamma(\frac12 +it)-i\varepsilon} \right)d\gamma \\[.2cm]
\begin{aligned}&= 
\frac{\ln q}2\left( \frac 1{d\left(\frac12 -it\right)}\;
\frac {q^{-\left(\frac12 +it\right)|e|}}{q^{-\frac12 -it}-q^{-\frac12 +it}} -  \frac 1{d\left(\frac12 +it\right)}\;
\frac {q^{-\left(\frac12 -it\right)|e|}}{q^{-\frac12 +it}-q^{-\frac12 -it}}\right)
\\[.2cm]
&\qquad\qquad \cdot i
\left(q^{-\frac12+it}-q^{-\frac12-it}\right)dt\\[.2cm]
&= -i\,
\frac{\ln q}2\left( \frac {q^{-\left(\frac12 +it\right)|e|}}{d(\frac12 -it)} +  \frac {q^{-\left(\frac12 -it\right)|e|}}{d(\frac12 +it)}\right)dt
= -i\,
\frac{\ln q}2 \; \frac 1  {\left| d\left(\frac12+it\right)\right|^{2}}\;\phi _{\frac12+it}(e)\;dt
\end{aligned}
\end{multline*}
(the last identity follows from identities \eqref{eq:formula_simmetrica_per_phi/|d|^2}
and  \eqref{eq:conjugation_property_of_c_and_d}: $\overline{d\left(\frac12 +it\right)}=d\left(\frac12 -it\right)$).
Then, by Corollary \ref{cor:  convolutions_with_edge-spherical_functions}\,$(iii)$,
\begin{multline*}
\lim_{\varepsilon\to 0^+}\left((\xi_{\gamma(\frac12 +it)+i\varepsilon}-\xi_{\gamma(\frac12 +it)-i\varepsilon})*h,h \right)_2 \,d\gamma 
\\[.2cm]
\begin{split}
&=-i\,
\frac{\ln q}2 \; \frac 1  {\left| d\left(\frac12+it\right)\right|^{2}}\;\left(\phi _{\frac12+it}*h,h\right)_2 \,dt
 \\[.2cm]
&= -i\,
\frac{\ln q}2 \; \frac 1  {\left| d\left(\frac12+it\right)\right|^{2}}\;\left|\widehat{h}\left(\frac12 +it\right)\right|^2\,dt.
\end{split}
\end{multline*}

Since $\gamma'$ vanishes at the endpoints of the spectrum, and $\lim_{\varepsilon\to 0} \xi_{\gamma\left(\frac12 +it\right) \pm i\varepsilon}$ is finite for every $t\in\mathR$, the integrand in Carleman's formula is bounded, and the dominated convergence theorem allows us to take the limit for $\delta\to 0$ inside the integral. This yields the Plancherel formula of the statement, because, as $u$ is a continuous function on $\mathR$ that is constantly 1 on $S$, by definition of spectral measure, 
\[
\int^{\infty}_{-\infty} u(\gamma)\,( E(d\gamma)h,h)_2 = 
\int_S ( E(d\gamma)h,h)_2 = \|h\|_2^2\,.
\]
\par
In particular, the Plancherel measure has no poles at the extreme points of the interval $\left[\frac12,\,\frac12 +i\,\frac\pi{\ln q}\right]$. Let us give a direct proof of this fact. By   \eqref{eq:epsilon_as_a_function_of_delta} and \eqref{eq:epsilon_come_funzione_di_delta},   $\gamma(1/2+\delta)-\gamma(1/2)= i\varepsilon= \frac {\cosh(\delta\ln q)-1} {   \sqrt{q}     }$ and so now, if $\delta>0$, then $\delta\mapsto \gamma(1/2 + \delta)$ approaches $\gamma(1/2)$ from the right, that is from the positive real semi-axis. Instead, again for $\delta>0$,
$\gamma(1/2+i\pi/\ln q+\delta)-\gamma(1/2+i\pi/\ln q)=- \frac {\cosh(\delta\ln q)-1} {   \sqrt{q} }$, and  the other endpoint $\gamma(1/2+i\pi/\ln q)$ is approached from the negative real semi-axis. But since we need $\varepsilon\in\mathR$, we cannot require any longer that $\delta>0$, but, as mentioned above, we need at least $\Real\delta>0$.
We have shown that the expressions $\gamma(1/2+\delta)-\gamma(1/2)$ and  $\gamma(1/2+i\pi/\ln q+\delta)-\gamma(1/2+i\pi/\ln q)$ are asymptotically proportional to $\delta^2$. 
Therefore the  displacements
 $\gamma\left(1/2\right)\pm i\varepsilon$ with $\varepsilon>0$ are obtained by the two curves in the $z-$plane  $\varepsilon\mapsto 1/2 +\delta(\pm i \varepsilon) \approx 1/2 + C \sqrt{\pm i\varepsilon}$, with $C=\frac{(\ln q)^2}{2\sqrt{q}}$ (here $\delta(\epsilon)$ is as after \eqref{eq:epsilon_come_funzione_di_delta}, and the determination of the complex square root is 
 $\sqrt{\pm i\varepsilon}=e^{\pm i\frac\pi 4} \sqrt{\varepsilon}$ so that $\Real\delta=C\Real \sqrt{\pm i\varepsilon}>0$). 
 These curves are asymptotically tangent to the half-lines at slope $\pm 1$ respectively given by $s\mapsto 1/2 + (1\pm i)s$ with $s>0$.
Similarly, $\gamma(1/2 +i\pi/\ln q)$ is approached by the curves $\varepsilon\mapsto1/2 +i\pi/\ln q +\delta(\pm i\varepsilon)= 1/2 + i\pi/\ln q- C \sqrt{\pm i\varepsilon}$, asymptotically tangent to the half-lines with slopes $\pm1$ given by $s\mapsto 1/2 + i\pi/\ln q +(1\pm i)s$ with $s<0$. Here the determination of the square root is the same as before.
Now  it is easy to see from the expression of $r_z$ that, along these curves, the difference $(r_{\frac12 + (1+i)\delta}-r_{\frac12 + (1-i)\delta})(e)$ tends to zero as $\delta\to 0$. The same fact happens at $z=1/2 + i\pi/\ln q$. Hence there are no poles even at the extreme points, and so no atoms for the Plancherel measure.
\end{proof}

\begin{corollary}\label{cor:inversion_formula}
For every $h\in\ell^1$ and for every $e\in E$,
\[
h(e)=\frac{\ln q}4 \int_0^{\pi/\ln q} \widehat{h}\left( \frac12+it \right)\phi _{\frac12+it}(e)\left|d\left(\frac12+it\right)\right|^{-2}dt.
\]
\end{corollary}

\section[The Schwartz class and the Paley--Wiener theorem]{The Schwartz class on vertices of homogeneous trees and the Paley--Wiener theorem}
The Schwartz class in a tree is essentially $\bigcap_{p<2} \ell^{p}$, and the space of distribution is $\bigcap_{p>2}\ell^{p}$. These spaces on vertices of homogeneous trees, as well as the corresponding Paley--Wiener theorem, were introduced in \cite{Betori&Faraut&Pagliacci}. 
\begin{definition}[Schwartz class on homogeneous  trees] \label{def:Schwartz}
Choose a reference element $v_0\in V$ and let $|v|=\dist(v,v
_o)$. If $T=T_q$ is homogeneous, $f:V\to \mathC$ and $r\in\mathR$, consider the seminorms
\begin{equation}\label{eq:seminorms_in_S(V)}
\left| f \vphantom{F_r} \right|_r=\sup_{v\in V} (1+|v|)^r\,|f(v)|\,q^{\frac{|v|}2}.
\end{equation}
The Schwartz class $\calS=\calS(V)$ consists of all functions $f$ such that $|f|_r <\infty$ for every $r$. Since $|f|_r\leqslant |f|_s$ for $r\leqslant s$, $\calS$ is equipped with the topology is induced by a countable family of seminorms, hence it is a Fr\'{e}ch{e}t space.

A similar definition holds for the Schwartz class of functions on edges, $\calS(E)$.
\end{definition}

The factor $q^{\frac{|v|}2}$ in the definition of the Schwartz class is an $\ell^2$ norm compensation for growth of circles. Indeed:

\begin{corollary} \label{cor:the_dual_of_the_Schwartz_classes_and_inclusions}%
\begin{enumerate}
\item[$(i)$] The dual space of $\calS(V)$ is the space $\calS'(V)$ of all complex functions $g$ such that $|g|_r<\infty$ for some $r$ (this exponent $r$ will be called the \emph{order} of $g$).
\item[$(ii)$] For every $p<2<s$, $\ell^p(V)\subset \calS(V) \subset \ell^2(V)\subset \calS'(V)\subset \ell^s(V)$.
\end{enumerate}
The same results hold for the Schwartz class of functions on edges, $\calS(E)$.
\end{corollary}
\begin{proof}
If $f\in\calS$, then  
\begin{equation}\label{eq:Schwartz-decay}
|f(v)|\leqslant |f|_s\,(1+|v|)^{-s} (q)^{-|v|/2}
\end{equation}
 for every $s$. On the other hand, $g$ satisfies this inequality for $s=r$. Hence
\[
|\langle f,g\rangle| \leqslant |f|_s|g|_r \sum_{v\in V} (q)^{-|v|} (1+|v|)^{-(r+s)},
\] 
and the right hand side is finite for $s > 1-r$.
Therefore $g\mapsto \langle f,g \rangle$ is a continuous functional on $\calS$. 
Part $(ii)$ follows immediately from the asymptotic decay given by \eqref{eq:Schwartz-decay}.
\end{proof}

Now we can proceed as in \cite{Betori&Faraut&Pagliacci}. We first observe the following useful fact:
\begin{remark}\label{rem:seminorms_in_terms_of_the_spherical_function_at_the_spectral_radius}
The Schwartz class and its dual can be equivalently defined in terms of the seminorms
\[
\left| f \vphantom{F_r} \right|'_r=\sup_{v\in V} \frac{(1+|v|)^r\,|f(v)|}{\phi_{\frac12}(v)},
\]
where $\phi_{\frac12}$ is the vertex-spherical function at $z=\frac12$ as in Proposition \ref{prop:computation_of_vertex-spherical_functions}. 
Indeed, as a consequence of
\eqref{eq:the_vertex_spherical_function_at_the_spectral_radius},
$c_q \left| f \vphantom{F_r} \right|_{r-1} < \left| f \vphantom{F_r} \right|'_r < C_q \left| f \vphantom{F_r} \right|_{r-1}$,
where $c_q$ and $C_q$ are positive constants that depend only on $q$.
\end{remark}
\begin{proposition}\label{prop:convolutions_between_Schwartz_classes_and_distributions}
For a homogeneous tree $T=T_q$, denote by $\calS_\#$, $\calS'_\#$ the subspaces of $\calS$, respectively $\calS'$, consisting of radial functions on $V$ or on $E$ (with respect to the respective reference elements), and denote by $*$ the convolution product. Then
\begin{enumerate}
\item[$(i)$] $\calS_\# * \calS \subset \calS$,
\item[$(ii)$] $\calS_\# * \calS' \subset \calS'$,
\item[$(iii)$]  $\calS'_\# * \calS \subset \calS'$.
\end{enumerate}
\end{proposition}
\begin{proof}
If we use the seminorms introduced in Remark \ref{rem:seminorms_in_terms_of_the_spherical_function_at_the_spectral_radius}, by \eqref{eq:homog_convolution} and \eqref{eq:Schwartz-decay} we have, for all $f\in\calS$, $g\in\calS_\#$, $r\in\mathR$ and $s>0$,
\[
|g*f(u)|\leqslant |f|_r|g|_s \sum_{y\in V} (1+\dist(u,y))^{-s} \phi_{\frac12}(\dist(u,y))\, (1+|y|)^{-r}\phi_{\frac12}(y),\]
where, as before, we denote by $\phi_{\frac12}(n)$ the value of $\phi_{\frac12}$ at a vertex of length $n$.\\

We shall show that, for all $t\in\mathR$, $|g*f(v)|\leqslant C |f|'_r |g|'_s (1+|v|)^{-t} \,\phi^V_{\frac12} (v)$ for some $C>0$ that depends on $r$ and $s$ but not on $f$ and $g$.
 Since the seminorm $|\cdot|'_t$ is monotonically increasing with $t$, it is enough to restrict attention to $t>0$. 
 Since $g$ is radial, we have $g*f(v)= \sum_{v'\in V} g(\dist(v,v')) f(v')$ and
\[
|g*f(v)|\leqslant |f|'_r |g|'_s \sum_{v'\in V} (1+\dist(v,v'))^{-s}(1+|v'|)^{-r}\varphi(v')\,\varphi(\dist(v,v')).
\]
\\
 Let us prove $(i)$. Note that $|g|'_s<\infty$ for every $s$ because $g\in\calS$. 
As $|v|=\dist(v,v_0)$, it follows from the triangular inequality that
$1+|v| \leqslant (1+\dist(v,v')) (1+|v'|)$. Then, if we choose $s=t>0$,
\begin{equation*}
(1+\dist(v,v'))^{-s} \leqslant (1+|v'|)^{s}  (1+|v|)^{-s},
\end{equation*}
and the previous inequality gives
\begin{align*}
|g*f(v)|&\leqslant |f|'_r |g|'_s (1+|v|)^{-s} \sum_{v'\in V} (1+|v'|)^{s-r} \varphi (v')\varphi (\dist(v,v')) \\[.2cm]
&= 
|f|'_r|g|'_s (1+|v|)^{-s} \sum_{n\geqslant 0} (1+n)^{s-r} \varphi (n) \sum_{|v'|=n} \varphi (\dist(v,v')).
\end{align*}
By the multiplicative property of
Theorem \ref{theo:characterization_of_spherical_functions}\,$(iv)$, 
\begin{equation}\label{eq:equazione_funzionale_della_funzione_sferica}
\sum_{|v'|=n} \varphi (\dist(v,v')) = \sum_{|v'|=n} (\delta_v*\varphi) (v') =  \left|C^V(v_0,n)\right| \varphi(n) \,\varphi(v).
 \end{equation}
Remember that, by \eqref{eq:number_of_edges_of_length_n},
for $n>0$,
 $\left|C^V(v_0,n)\right|  =(q+1)q^{n-1}$.
Therefore
\begin{align}\label{eq:seminorm_of_the_convolution}
|g*f(v)| & \leqslant C
|f|'_r|g|'_s (1+|v|)^{-s} \varphi (v) \sum_{n\geqslant 0}(1+n)^{s-r} \varphi (n)^2   q^n 
\notag
\\[.2cm]
&< 
C' |f|'_r|g|'_s (1+|v|)^{-s} \varphi (v) \sum_{n\geqslant 0} (1+n)^{s-r} n^2  
\end{align}
by 
the asymptotics of the spherical function given in  Proposition \ref{prop:spectral_theory_of_spherical_functions}\,$(i)$.
For $s<r-3$ the series converges. Hence, for $0<s<r-3$, $|g*f(v)| < C |f|'_r|g|'_s (1+|v|)^{-s} \varphi (v) $ (remember that $s=t$).

 We have just proved that, if $|f|'_r<\infty$ and $|g|'_s<\infty$, and $s>0$, then $|g*f|_t <\infty$ for $r>s+3$. Part $(i)$ follows because $r$ is arbitrary.
 
 In part $(ii)$, $|f|'_r<\infty$ for some $r$ and $|g|'_s<\infty$ for all $s$. Let us choose $s>0$, and use the triangular inequality in the form
$1+|v'| \leqslant  (1+\dist(v,v')|)( 1+|v|)$, hence
$ (1+\dist(v,v')^{-s}  \leqslant (1+ |v|)^{s} (1+|v'|)^{-s}$.
We make use again of \eqref{eq:equazione_funzionale_della_funzione_sferica}, and
 instead of  \eqref{eq:seminorm_of_the_convolution} we now get
$|g*f(v)| \leqslant C'''|f|'_r |g|'_s (1+|v|)^{s} \varphi (v) \sum_{n\geqslant 0} n^2 (1+n)^{-(s+r)} $. The series converges for $s>3-r$. Since $s$ is arbitrary, this inequality is satisfied for suitable $s>0$. This proves $(ii)$.

 Let us prove $(iii)$.  Now $|f|'_r<\infty$ for all $r$ but $|g|'_s<\infty$ only for some $s$. Choose $r>0$, and use the triangular inequality in the form $1+\dist(v,v') \leqslant  (1+|v|)( 1+|v'|)$. Since $r>0$,
$
 (1+|v'|)^{-r}  \leqslant (1+\dist(v,v'))^{-r} (1+|v|)^{r}$.
Let us write the multiplicative property of
Theorem \ref{theo:characterization_of_spherical_functions}\,$(iv)$ in the form 
\[
\sum_{v': \dist(v,v')=n} \varphi (v') =
\sum_{|v''|=n} (\delta_v * \varphi) (v'') = \left|C^V(v_0,n)\right| \varphi(n) \,  \varphi(v).
\]
Then the same steps lead to
$|g*f(v)| \leqslant C |f|'_r \,|g|'_s (1+|v|)^{r} \varphi (v) \sum_{n\geqslant 0} n^2 (1+n)^{-(s+r)} $, and the series converges for $r>3-s$. Since $s$ is fixed but $r$ is arbitrary, this inequality is satisfied for suitable $r>0$, and $(iii)$ is proved.
\end{proof}

By Theorem \ref{theo:ell^2-spectra}, that the spectrum of $\mu_1$ as convolution operator on $\ell^2(V)$ is $S_V:=\gamma^V\left\{ \frac12 +it :  \, 0\leqslant t \leqslant \pi/\ln q \right\}=[-\rho_V,\rho_V]$ with $\rho_V=2\sqrt{q}/(q+1)$. Note that $\gamma^V(0)=\rho_V$.

We return to the functions $\widetilde{P}_n\left(\phi^V_z(1)\right)=\widetilde{P}_n\left(\gamma(z)\right) = P_n(z) =\phi^V_z(n)$ and $\widetilde{P}_n\left(\phi^E_z(1)\right) =\phi^E_z(n)$
introduced in Definition \ref{def:spherical_polynomials}. We consider these functions on the spectra 
of $\mu_1$ (respectively,  $\eta_1$) on $\ell^2$.

\begin{lemma} [Estimates of derivatives of spherical polynomials] For every $k\in\mathN $,
\begin{align*}
|D^k\widetilde{P}_n(x)| & \leqslant \rho_V^{-k} q^{-\frac n2} (1+n)  \big(n(n-1)\dots(n-k+1)\big)^2 \quad \text{ for } x\in S_V\,.
\\[.2cm]
|D^k\widetilde{Q}_n(x)|  &\leqslant q^{-\frac {n-k}2} \left(1+n\;\frac{q-1}{2\sqrt{q}}\right)  \big(n(n-1)\dots(n-k+1)\big)^2 \quad \text{ for } x\in S. 
\end{align*}
\end{lemma}
\begin{proof} 
We use the  argument of \cite{Betori&Faraut&Pagliacci}*{Lemma 3.4}.
By  Definition \ref{def:spherical_polynomials} and \eqref{eq:majorization_for_spherical_polynomials},  one has $|\widetilde{P}_n\left(\gamma^V(\frac12 +it)\right)|=|\phi^V_{\frac12 +it}(n)|  \leqslant \phi^V_\frac12(n)$, that is, 
\begin{equation*}
|\widetilde{P}_n(x)| \leqslant \left( \phi_0^V\right)(n) < (1+n)\,q^{-\frac n2} \quad \text{ for all }  x\in S_V\,.%
\end{equation*}
Similarly, 
\begin{equation}\label{eq:majorization_for_edge-spherical_polynomials}
|Q_n(z)|\leqslant |Q_n(\Real z)| \text { for every $z$.}
\end{equation}
It is known \cite{Lorentz}*{p. 40} that, if a polynomial $U_n$ of degree $n$ is bounded by a constant $M$ on $[-1,1]$, then
$|U'_n|\leqslant Mn^2$ in $[-1,1]$. Then let us put $U_n(t)=\widetilde{P}_n\left(\rho_V t\right)=\widetilde{P}_n\left( \frac { 2\sqrt{q}} {q+1}\; t\right)$. Now $U_n$ is defined on $[-1,1]$ and takes the same values attained by $\widetilde{P}_n$ on  $[-\rho_V,\rho_V]$, so $U_n(t)\leqslant M_n=\left( 1+n\;\frac{q-1}{q+1} \right) q^{-\frac n2} $. Therefore
\[
|U'_n(t)|=\rho_V \left|\left(\widetilde{P}_n\left( \rho_V t \right)\right)\right| \leqslant n^2 M_n \,,
\] 
hence 
$|\widetilde{P}'_n(x)| \leqslant n^2 \left(1+n\;\frac{q-1}{q+1}\right) q^{-\frac {n}2} \rho_V^{-1}$  on 
$S_V$.  

Note that $U'_n$ is a polynomial of degree $n-1$  and has values in $[-1,1]$ bounded by $n^2M_n$. Therefore $U''_n= \rho_V^2\,\widetilde{P}_n'' $
has values bounded by $n^2(n-1)^2 M_n$, hence 
\begin{equation*}
\begin{split}
|\widetilde{P}_n''(x| &\leqslant \rho_V^{-2}  (n-1)^2 n^2 M_n =
\rho_V^{-2}  q^{-\frac{n}2} n^2 (n-1)^2 \left( 1 + n\;\frac{q-1}{q+1}\right) 
\\[.0cm]
& < 
\rho_V^{-2} q^{-\frac{n}2} n^2 (n-1)^2 \left( 1 + n\;\frac{q-1}{q+1}\right)(n+1).
\end{split}
\end{equation*}
By iteration, we obtain the inequality of the statement for $D^k\widetilde{P}_n$.

We now consider the edge-spherical polynomials $\widetilde{Q}_n$.

By Theorem \ref{theo:ell^2-spectra}, 
the spectrum $S_E$ of $\eta_1$ as a convolution operator on $\ell^2$ is $S_E:=\gamma\left\{ \frac12 +i \mathR\right\} =[a_-,\, a_+]$ with $a_\pm=\frac{q-1}{2q}\pm \frac 1{\sqrt{q}}$.
Note that  
$\gamma(\frac12)= a_+$.

We limit attention to the values of  $\widetilde{Q}_n$ on the spectrum $S_E$.
By 
\eqref{eq:majorization_for_edge-spherical_polynomials}, one has $|\widetilde{Q}_n\left(\gamma(\frac12 +it)\right)|=|\phi _{\frac12 +it}(n)|  \leqslant \phi _\frac12(n)$, that is, 
$|\widetilde{Q}_n(x)| \leqslant  \phi_\frac12(n) = M_n$  for all $x\in S$,   with $M_n=\left(1+n\;\frac{q-1}{2\sqrt{q}}\right)\,q^{-\frac n2}
$.
The same argument used for $\widetilde{P}_k$ shows that, since $|\widetilde{Q}_k|\leqslant M$ on $[a_-,a_+]$, then its derivative satisfies 
$|D\widetilde{Q}_k|\leqslant 2Mk^2/(a_+-a_-)$ in $[a_-,a_+]$. Here $a_+-a_-=2/\sqrt{q}$. Hence $|D \widetilde Q_n | \leqslant M_n n^2 \sqrt{q} = n^2 \left(1+n\;\frac{q-1}{2\sqrt{q}}\right) q^{-\frac {n-1}2}$ on 
$S_E$.  
Since $D \widetilde Q_n$ is a polynomial of degree $n-1$, the same argument shows that  $D^2 \widetilde Q_n$
has values bounded by $q n^2(n-1)^2 M_n=q^{-\frac{n-2}2} n^2 (n-1)^2 \left( 1 + n\;\frac{q-1}{2\sqrt{q}}\right)$. By iteration, we obtain the 
statement.  
\end{proof}

We can now give the proof of the Paley--Wiener theorem, taken from \cite[Theorem 3.3]{Betori&Faraut&Pagliacci}.
\begin{theorem}[The Paley--Wiener theorem for the spherical Fourier transform]\label{theo:Paley--Wiener}
The vertex-spherical Fourier transform is an isomorphism of $\calS_\#(V)$ onto $C^\infty(S_V)$, and of $\calS'_\#(V)$ onto the space of distributions with support in the interval $S_V$. The same statement holds for the edge-spherical Fourier transform.
\end{theorem}
\begin{proof} We prove this statement for $V$: the arguent for $E$ is the same. It is enough to show the first part of the statement, as the second follows by duality.

Let $h\in\calS_\#(V)$ and let us write its value on vertices of length $n$ as $h(n)$. The number of vertices of length $n$ is $\frac{q+1}q\; q^n$, hence  the spherical Fourier transform  of $h$ is $\widehat{h}(z)= \frac{q+1}q \sum_{n\geqslant 0} q^n h(n) P_n(z)$. We first show that $\widehat{h}\in C^\infty(S_V)$.
By factorizing this expression via the eigenvalue map $\gamma^V$, it is enough to prove that the series
$\sum_{n\geqslant 0} q^n h(n) \widetilde{P_n}(x)$ converges in $C^\infty(S_V)$, that is, that for each $k\geqslant 0$ the series
\begin{equation}\label{eq:derivatives_of_spherical_transform}
D^k \widehat{h}(x) =  2\;\frac{q+1}q \sum_{n\geqslant 0} q^n h(n) D^k\widetilde{P_n}(x)
\end{equation}
converges uniformly. The assumption $h\in\calS_\#(V)$ implies $\left| h(n) \right| \leqslant 
c_{k,r} (1+n)^{r} q^{-n/2}$ for every $r\in\mathR$: we choose here $r-=2k+3)$. Then, by the previous lemma, the terms of the series in \eqref{eq:derivatives_of_spherical_transform}  decay as $n^{-2}$ and the uniform convergence follows.

We now prove the converse, that is that every function in $C^\infty(S_V)$ is a spherical Fourier transform of a function in $\calS_\#(V)$. Take any $g\in C^\infty(S_V)$ and write ${g}^\dagger\left(\frac12+it\right)=g(x)$ if $x=\gamma^V\left(\frac12+it\right)$.
By the inversion formula of Corollary \ref{cor:inversion_formula} we must show that the radial function
\[
h(n)  =\frac {q\,\ln q}{2(q+1)} \int_0^{\pi/\ln q} \phi_{\frac12 +it}(n)\,{g}^\dagger\left(\frac12+it\right) \, \left|c\left(\frac12+it\right)\right|^{-2}dt
\]
belongs to $\calS_\#(V)$. We recall \eqref{eq:formula_simmetrica_per_phi/|d|^2},
\[
\frac { \phi_{\frac12 +it}(n)}{\left|c\left(\frac12+it\right)\right|^2}   = \frac{q^{-\left(\frac12 + it\right)n}}{
c\left(\frac12-it\right)} + \frac{q^{-\left(\frac12 - it\right)n}}{
c\left(\frac12+it\right)}
\]
and note that the two terms at the right hand side are equal because $c\left(\frac12-it\right)=\overline{c\left(\frac12+it\right)}$ by \eqref{eq:conjugation_property_of_c_and_d}. This yields 
\[
h(n)  =\frac {q\,\ln q}{2(q+1)}\; q^{-\frac n2} \int_0^{\pi/\ln q} q^{int} \,\frac{{g}^\dagger\left(\frac12+it\right)} {c\left(\frac12+it\right)}\;dt.
\]
Since $c\left(\frac12+it\right)\neq 0$ for $0\leqslant t \leqslant \pi/\ln q$, the sequence $h(n) \,q^{n/2}$ is the $n$-th Fourier coefficient of a function in $C^\infty[0,\pi/\ln q]$, and so it is rapidly decreasing. This means that $h\in\calS_\#(V)$.
\end{proof}

\section[Horospherical Radon transforms on $\calS$]{The horospherical Radon transforms as a bounded operator on the Schwartz class, and their dual transforms}

We have computed in Lemma \ref{lemma:homogeneous_intersection_cardinalities}
the cardinality $k_V(n,m)=\abs{\boldh_n\cap \CV(m,v_0)}$  of the intersection of vertex-circles and horospheres, and
in Remark \ref{rem:homog_edge_intersection_volumes} the cardinality $k_E(n,m)=\abs{\boldh_n\cap \CV(m,v_0)}$  of the intersection of vertex-circles and horospheres
 (by radial symmetry, here it is inessential to specify the tangency point $\omega$). 
Here we shall need to make use only of their asymptotics, that are as follows.

\begin{lemma}\label{lemma:horosphere-circle_intersections}
\begin{enumerate}
\item[$(i)$]
$k_V(n,m)$ 
vanishes if $m<|n|$ and otherwise is bounded by a constant (that depends only on $q$) times $q^{\frac {m-|n|}2}$.
\item[$(ii)$]
$k_E(n,m)$ vanishes if $m<|n|$ and is $\displaystyle O\left(q^{\frac {m-n}2}\right)$ if $n$ is fixed and $m \to \infty$.
\end{enumerate}
\end{lemma}

Note that, if $|v|$ 
has value $n$, then the union of all horospheres in $\HorV$ 
or $\HorE$ 
of index $n$ that pass through $v$ 
or $e$ 
is the sector subtended by $v$
or $e$, 
hence it has measure proportional to $q^{-n}$. As we did for $\calS(V)$, 
and $\calS(E)$, 
we now use this factor for an $\ell^2$ normalization in defining Schwartz classes on spaces of horospheres. This definition for $\HorV$ was given in \cite{Betori&Faraut&Pagliacci}, where, however, the horospherical index is the opposite as here (their index $n$ of $\boldh(\omega,n)$ tends to $-\infty$ as the horosphere $\boldh$ shrinks to $\omega$). 
\begin{definition}[Schwartz classes on spaces of horospheres]\label{def:Schwartz_on_horospheres}
The Schwartz class $\calS(\HorV)$ 
(respectively, $\calS(\HorE)$) 
is the space of all continuous functions $F$ on $\HorV$ 
(respectively, $\HorE$) 
such that, for all $r\in\mathR$,
\begin{equation}\label{eq:seminorms_of_S(HorV)}
|F|_r:=\sup_{n\in\mathZ, \,\omega\in\Omega}  (1+|n|)^r q^{\frac {n}2} \left|F(\omega,n)\right|<\infty\,.
\end{equation}
The right hand side is a seminorm on the spaces $\calS(\HorV)$.
or  $\calS(\HorE)$. 
Equipped with these seminorms, these space is a Fr\'{e}ch{e}t space.

Its dual space is the space of distributions $\calS'(\HorV)$. 
(respectively, $\calS'(\HorE)$). 
As in the proof of Corollary \ref{cor:the_dual_of_the_Schwartz_classes_and_inclusions},  this distribution space coincides with the space of all continuous functions $F$ on $\HorV$ 
(respectively, $\HorE$) 
such that, for some $r\in\mathR$,
$|F|_r<\infty$.
\end{definition}

\begin{theorem}\label{theo:VRad_e'_continua}
The horospherical Radon transforms are continuous linear operator on $\calS(V)$ to $\calS(\HorV)$ and on $\calS(E)$ to $\calS(\HorE)$, respectively.
\end{theorem}
\begin{proof}
The argument is the same for $\VRad$ and $\ERad$: we write it for edges.

By \eqref{eq:Schwartz-decay}, if $f\in\calS$, $r\in\mathR$ and $\boldh=\boldh(\omega,n)\in\HorE$,
\[
|\ERad f(\boldh)| \leqslant |f|_r \sum_{e\in\boldh} (1+|e|)^{-r}q^{-\frac{|e|}2} = |f|_r\,\ERad w_r(\boldh),
\]
where $ w_r(e)=(1+|e|)^{-r}q^{-\frac{|e|}2}$. Therefore it is enough to show that $|\ERad w_r|_{r-1}$ is finite.
One has
\[
\ERad w_r (\omega,n)= \sum_{e\in\boldh(\omega,n)} (1+|e|)^{-r}q^{-\frac{|e|}2} = \sum_{m=|n|}^\infty k(n,m) (1+m)^{-r} q^{-\frac m2}.
\]
The series on the right hand side does not depend on $\omega$, and by Lemma \ref{lemma:horosphere-circle_intersections} its terms are bounded by $C (1+m)^{-r}q^{-n/2}$ for some $C>0$. Hence, for each $r>1$, one has $\sum_{m\geqslant 0} (1+|n|+m)^{-r}< C (1+|n|)^{1-r}$. Therefore $|\ERad w_r(\omega,n)| < C q^{-\frac {n}2} (1+|n|)^{1-r}$ and so 
$|\ERad w_r|_{r-1}$ is finite for $r>1$.

\end{proof}

\begin{definition}[Dual horospherical Radon transform] 
 For any $v\in V$, recall that the special section $\Sigma_v\subset\HorV$ consists of the set of horospheres containing $v$. Remember that, on the basis of  Definition \ref{def:invariant_measure_on_Hor},  $\HorV$ is equipped with the invariant measure $\xi$ such that $d\xi(\boldh(\omega,n))=q^n d\nu_{v}(\omega)$, where $n$ is the horospherical index with respect to the reference vertex 
 $v$,
$\nu_{v}$ is the probability measure on the boundary invariant under the subgroup $K_v\subset \Aut T$ that fixes $v\in V$, and $\xi$ does not depend on the choice of the reference vertex. If $F$ is a measurable function on $\HorV$, its dual horospherical Radon transform is the average value of $F$ on horospheres that pass through a given vertex. That is, if we fix a reference vertex 
$v_0$,
\begin{align*}
\VRad^* F (v) &= \notag
\int_{\Sigma_{v}}F(\boldh)\,d\xi(\boldh)\\[.2cm]
&=\int_{\Sigma_{v}} F(\boldh)\,d\nu_{v}(\pi (\boldh))\notag \\[.2cm]
&=\int_\Omega F(\boldh{(\omega,0)})\,d\nu_v(\omega)\notag \\[.2cm]
&= \int_\Omega F(\boldh(\omega,n))\,q^n d\nu_{v_0}(\omega)\,.
\end{align*}
%
%
Let us set $B({v,n})=\left\{ \omega:\,v\in \boldh(\omega,n)\right\}$. Then the last expression becomes
\begin{align}\label{eq:dual_Radon_transform}
\VRad^* F ((v) &= \sum_{n=-\infty}^\infty \int_{B(v,n)} F(\omega, n)\, d\nu_(v)(\omega) \notag  \\[.2cm]
&=
 \sum_{n=-\infty}^\infty q^{n}  \int_{B(v,n)} F(\omega, n)\, d\nu_{(v)_0}(\omega)\,,
\end{align}
 because, by \eqref{eq:homogeneous_invariant_vertex-measure}, the factor $q^{n}$ is nothing but $d\nu_{v}/d\nu_{v_0}(S_e)$.
 \\
 From now on, for the sake of simplicity, we shall write $\nu$ instead of $\nu_{v_0}$.
 
 The dual transform $\ERad^*$ of $\ERad$ is defined in an analogous way on measurable functions on $\HorE$.
\end{definition}

\begin{theorem}\label{theo:VRad^*_e'_continua}
The dual Radon transforms are continuous linear operator from $\calS(\HorV)$ to $\calS'(V)$, and from $\calS(\HorE)$ to $\calS'(E)$. 
\end{theorem}
\begin{proof}
This result is new for $\ERad^*$, but was proved for $\VRad^*$ in \cite{Betori&Faraut&Pagliacci}*{Theorem 4.2}. The arguments in the two setups are the same. We write the proof for $\VRad^*$, because of some misprints in that reference where the analogue of Lemma \ref{lemma:horosphere-circle_intersections}.
%
Let $F\in\calS(\HorV)$. For every $r\in\mathR$, by \eqref{eq:dual_Radon_transform} and \eqref{eq:seminorms_of_S(HorV)},
\begin{align*}
|\VRad^* F (v)| &\leqslant \sum_{n=-\infty}^\infty q^{n}  \int_{B(v,n)} |F(\omega, n)|\, d\nu(\omega)
\\[.2cm]
&\leqslant |F|_r \sum_{n=-\infty}^\infty q^{\frac n2} (1+|n|)^{-r} \, \nu (B(v,n)).
\end{align*}
For each $v$ of fixed length, $\nu(B(v,n))$ is independent of $v$, by rotational invariance. Therefore $\nu ( B(v,n))$ for $ |v|=m$ is proportional to the intersection cardinality $k_V(n,m)$
, namely, 
  $\nu (B(v,n)) =k_V(n,m)/|C_V(m,v_0)|$ if $|n|\leqslant m$, and 0 otherwise.
But $|C_E(m,v_0)|\approx q^{m}$, hence, by Lemma \ref{lemma:horosphere-circle_intersections}, $\nu 
(B(v,n))
\approx q^{-(n+m)/2}$ if $|n|\leqslant m=|ev$, and 0 otherwise.
Therefore
\[ 
\displaystyle |\VRad^* F (v)|  \leqslant C |F|_r \, q^{-\frac{|e|}2} \sum_{n=-|v|}^{|e|} (1+|n|)^{-r} 
.\]
If $r\geqslant 0$, the terms of the sum are non-decreasing, hence $\sum_{n=1}^{m} (1+|n|)^{-r} \approx \int_0^m (1+x)^{-r} dx=1-(1+m)^{1-r}$. The right-hand side is bounded over $m$ if and only if $r\geqslant 1$.
Thus, for all $r\geqslant 1$ and some $C>0$,
\[
\displaystyle |\VRad^* F (v)|  \leqslant C |F|_r \, q^{-\frac{|v|}2}.
\]
Hence the  $ |\VRad^* F|_0 < C |F|_r$  when $r> 1$. This shows that $\VRad$ is continuous from $\calS(\HorV)$ to $\calS'(V)$.
\end{proof}

\section{$\Rad^*\Rad$ as a bounded operator on $C^\infty[\spectrum(\ell^1_\#)]$, and its invertibility}

By their equivariant definition, the horospherical Radon transform and its dual transform commute with the action of $\Aut T$. Therefore the same is true for the operators $\VRad^*\VRad$ and $\ERad^*\ERad$, and, by Theorems \ref{theo:VRad_e'_continua} and \ref{theo:VRad^*_e'_continua}, we have:

\begin{corollary}\label{cor:R*R_is_a_bounded_convolutor_fron_S_to_S'}
 $\VRad^*\VRad$ and $\ERad^*\ERad$ are bounded convolution operator on $\calS$ to $\calS'$. 
 \end{corollary}
We shall call $\VRad^*\VRad$ and $\ERad^*\ERad$ the \emph{Radon convolution operators,} or \emph{Radon blurs}. Let $\Psi^V=\VRad^*\VRad  \delta_{v_0}$ and $\Psi^E=\ERad^*\ERad  \delta_{e_0}$ be their convolution kernel.s Note that $\Psi^V\in \calS'(V)$ and $\Psi^E\in \calS'(E)$.

\begin{remark}\label{rem:goal_of_Radon_inversion}
Our goal is to invert $\VRad$ and $\ERad$, that is, to find bounded left convolution inverses of ${\VRad}^*\VRad$, that is,  functions $\Phi^V$ and $\Phi^E$ such that $\Phi^V*\Psi^V =\delta_{v_0}$ and $\Phi^E*\Psi^E =\delta_{e_0}$ and such that the convolution operator with kernel $\Phi^V$, respectively $\Phi^E$, maps the image $\VRad^*\VRad \calS(V)$
to $\calS(V)$ and is continuous in the topology of $\calS'(V)$, and similarly for $\ERad^*\ERad \calS(E)\to calS(E)$. We do this in the rest of this Section.
\end{remark}

In order to find the inverses of the convolutors $\Psi^V$ and $\Psi^E$
 we first need to compute $\Psi^V$ and $\Psi^E$ explicitly. Because of invariance under $\Aut T$, $\Psi^V$ and $\Psi^E$ are radial functions: let us write its value $\left(\Psi^V\right)_{n}$ on elements of length $n$ as $\psi^V_n$, and $\left(\Psi^E\right)_{n}=\psi^E_n$.

By the definition of $\VRad^*$ one knows that $\Psi^V(v)=\VRad^* \VRad \delta_{v_0} (v)= \VRad^* \chi_{{\HorV}_{v_0}}=\xi({\HorV}_v\cap{\HorV}_{v_0})$, and similarly for $\Psi^E(e)$. It is not difficult to compute their values $\psi^V_n$ for $|v|=n$ and $\psi^E_n$ for $|e|=n$. 

 The explicit formula of the convolution kernel $\Psi^V$ (that is, the kernel of the vertex horospherical Radon convolution operator) appears in \cite{Betori&Pagliacci-2}*{Lemma~2.4} and \cite{Betori&Faraut&Pagliacci}*{Section 4}), with some misprints. 
 The result is 

\begin{equation}\label{eq:Psi^V}
\Psi^V_n=
\left\{ 
\begin{array}{lll}
1 &\qquad\textrm{if }\quad &n=0,\\[.2cm]
\frac{q-1}{q+1}\;q^{-\frac n2}&\qquad\textrm{if }\quad &n>0 \textrm{ is even,}\\[.2cm]
0&\qquad\textrm{if }\quad &n>0 \textrm{ is odd}. 
\end{array}
\right.
\end{equation}
The inverse of the vertex-horospherical Radon convolutor $\Psi^V$ has been computed in \cite{Betori&Faraut&Pagliacci}*{Theorem~5.2}.

For $|e|=n$,  it is easy but tedious to verify that the value $\psi_n$ of $\Psi(e)=\ERad^* \ERad \delta_{e_0} (e)= \ERad^* \chi_{(\HorE)_{e_0}}=\mu((\HorE)_e \cap {(\HorE)_{e_0}})$ 
is   
\begin{equation}\label{eq:Psi^E}
\psi^E_n=
\left\{ 
\begin{array}{lll}
1&\qquad\textrm{if }\quad &n=0,\\[.2cm]
0&\qquad\textrm{if }\quad &n \textrm { even, } n>0,\\[.2cm]
\frac{q-1}2\;q^{-\frac{n+1}2} &\qquad\textrm{if }\quad &n \text{ odd}. 
\end{array}
\right.
\end{equation}

\begin{remark}\label{Rem:R^*_maps_S_to_S'_but_not_to_S}
From \eqref{eq:Psi^V} and \eqref{eq:Psi^E} it is immediate to verify that the function
$\Psi^V$ (respectively, $\Psi^E$) belongs to $\calS'(V)$ (respectively,  $\calS'(E)$) but not to $\calS(V)$ (respectively,  $\calS(E)$), because the seminorms $|\Psi^V|_r$ and $|\Psi^E|_r$ in \eqref{eq:seminorms_of_S(HorV)} are finite if and only if $r\leqslant 0$.

On the other hand, the function $\VRad \delta_{v_0}$, being the characteristic function of the set of horospheres through $v_0$, is bounded and compactly supported on $\HorV$, so it belongs to $\calS(V)$. Therefore $\VRad^*\VRad$ does not map $\calS(\HorV)$ into itself. The same is true for $\ERad^*\ERad$. We have already observed that
$\VRad^*\VRad$ and $\ERad^*\ERad$, being convolution operators, commute with the automorphisms of the tree: hence \eqref{eq:Psi^V} and \eqref{eq:Psi^E} show that $\VRad^* \VRad \delta_{V}$ belongs to $\calS'(\HorV)$ for every $v$ and similarly for $\ERad^*\ERad$. An alternative direct proof of Corollary \ref{cor:R*R_is_a_bounded_convolutor_fron_S_to_S'} follows easily from this remark.
\end{remark}

Now we need to compute the edge-spherical Fourier transform of the radial functions $\Psi^V$ and $\Psi^E$. The arguments in the two cases are quite similar. We give the statements in both cases, but we write the details only for $\Psi^E$, since for $\Psi^V$ they have been known for a long time \cite{Betori&Faraut&Pagliacci}

By Corollary \ref{cor:R*R_is_a_bounded_convolutor_fron_S_to_S'} (or Remark \ref{Rem:R^*_maps_S_to_S'_but_not_to_S}),
 $\Psi^E\in\calS'_\sharp(E)$, so its spherical Fourier transform is a distribution on $S_E=\spectrum_{\ell^2}(\eta_1)$, by the Paley--Wiener Theorem \ref{theo:Paley--Wiener}. We now show that the spherical functions $\phi_z^E$ does not belong to $S(E)$ for any $z\in\mathC$, and belongs to $S'(E)$ if and only if $x=1/2$.

\begin{remark} [$\phi_z\in\calS'_\sharp$ only if $\Real z=1/2$]\label{rem:spherical_functions_are_not_in_Schwartz_class}
The spherical functions $\phi_z^E$ belong to $\calS'_\sharp(E)$ if and only if $\Real z=\frac12$, by the form of the seminorms \eqref{eq:seminorms_in_S(V)} and Proposition \ref{prop:computation_of_edge-spherical_functions}, but it never belongs to $\calS_\sharp(E)$. Indeed, if $|e|=n$ then $\phi_z^E(e)=d_z q^{-zn}+d_{1-z}q^{(z-1)n}$
 if $z\neq 1/2 + 2i
 \pi/\ln q$ (see \eqref{eq:d(z)}). Hence $|\phi_z^E(n)|\sim q^{(\Real z-1)n}$ if $\Real z>1/2$ and $|\phi_z^E(n)|\sim q^{-n\Real z}$ if $\Real z<1/2$. If $\Real z=1/2$ then $\phi_z^E$ is bounded, except at $z=1/2$ and $z=1/2+i\pi/\ln q$, where, by \eqref{eq:the_edge_spherical_function_at_the_spectral_radius}, $|\phi_z^E(n)|\sim \sqrt n$. Hence, for every $r\in\mathR$ and $x=\Real z> 1/2$, by \eqref{eq:seminorms_in_S(V)}, one has
\begin{equation*}
|\phi_z^E|_r \sim \sup_n(1+|n|)^r q^{\left(x-\frac12\right)n}=\infty,
\end{equation*}
and the same if true for $x<1/2$. For $x=1/2$ the same formula shows that $|\phi_{\frac12+it}^E|_r<\infty$ if and only if $r\leqslant 0$ provided that $0<t<\pi/\ln q$, and $|\phi_{\frac12+it}^E(n)|_r<\infty$ if and only if $r\leqslant -1/2$ if $t=0$ or $t=\pi/\ln q$.
\\ 
This proves that $\phi_z^E$ never belongs to $\calS(E)$, and belongs to $\calS'(E)$ if and only if $x=1/2$. The same is true for $\phi_z^V$, by a similar computation
\end{remark}

Therefore we cannot write $\widehat{\Psi^E}(z)=\langle \Psi^E,\,\phi_z^E\rangle_E$ as in \eqref{eq:edge-horospherical_Fourier_transform}, because this inner product does not converge in the ordinary sense:
we shall compute this expression by symbolic calculus, based on functional analysis and the resolvents. 

 \begin{lemma}\label{lemma:Fourier_transform_of_the_resolvent}
 The following hold:
\begin{enumerate}
\item [$(i)$] \cite{Betori&Faraut&Pagliacci}*{Lemma 5.1}
For $v\in V$, let $w_z(v)=q^{-z|v|}$, $\widetilde{w}_z(v)=(-1)^{|v|}w_z(v)$ and $L_z=\frac12 (w_z + \widetilde{w}_z)$. 
Let $s_z= \frac{q+1}{q^{-z}-q^z}\;q^{-z|v|} =\frac 1{c(1-z)}\;\frac 1{q^{-z}-q^{z-1}}\;q^{-z|v|}$ be the resolvent of $\mu_1$ at the eigenvalue $\gamma^V(z)$ as in Theorem \ref{theo:ell^2-spectra}~$(ii)$.
Then
\begin{equation*}
L_z(v) =
\begin{cases}
q^{-z|v|}=\frac{q^{-z}-q^z}{q+1}\;s_z(v)&\qquad\text{if $|v|$ is even, }\\
0&\qquad\text{if $|v|$ is odd.}
\end{cases}
\end{equation*} 
Moreover, 
\begin{equation*}
\Psi^V = \frac 2{q+1}\,\delta_{v_0}+ \frac{q-1}{q+1}\, L_{\frac12}\,.
\end{equation*}
%
The spherical Fourier transform of $L_z$ is
 \[
 \widehat{L}_z(w)= \frac{q^z-q^{-z}}{q+1}\;\frac{2\gamma^V(z)}{\gamma^V(z)^2-\gamma^V(w)^2}\,.
 \]
 \item[$(ii)$]
 For $e\in E$, let $y_z(e)=q^{-z|e|}$, $\widetilde{y}_z(e)=(-1)^{|e|}y_z(e)=  y_{\widetilde{z}}(e)$      (notation as in Remark \ref{rem:change_of_eigenvalues_under_parity}) and $M_z=\frac12 (y_z - \widetilde{y}_z)$.
As in Theorem \ref{theo:ell^2-spectra}~$(iv)$, let $r_z=\displaystyle\frac{2q}{q^{1-z}-q^z-(q-1)}\;q^{-z|e|}
$ be the resolvent of $\eta_1$ at the eigenvalue $\gamma^E(z)$. Then
\begin{equation}\label{eq:M_z}
M_z(e)=
\begin{cases} q^{-z|e|} = \frac{q^{1-z}-q^{z}-(q-1)}{2q}\;r_z(e) &\qquad\text{if $|e|$ is odd, }\\
0&\qquad\text{if $|e|$ is even.}
\end{cases}
\end{equation} 
In particular, 
\begin{equation}\label{eq:M_1/2_and_Psi^E_1/2}
\Psi^E = \delta_{e_0}+ \frac{(q-1)\sqrt{q}}2 M_{\frac12}\,,
\end{equation}
and
\begin{align}\label{eq:Fourier_transform_of_M_z}
\widehat{M_z}(w)&=\frac12\;\left(\widehat{y}_z-\widehat{\widetilde{y}}_z\right)(w)\notag\\[.2cm]
&=
\frac{q^{1-z}-q^{z}}{2q}\;
\frac{\gamma^E(w)-\frac{q-1}{2q}}
{\gamma^E(w)^2 - \gamma^E(z)^2-
\frac{q-1}q\; (\gamma^E(w)-\gamma^E(z))}\notag\\[.2cm]
&\qquad -
\frac{q-1}{2q}\;
\frac{\gamma^E(z)-\frac{q-1}{2q}}
{\gamma^E(w)^2 - \gamma^E(z)^2-
\frac{q-1}q\; (\gamma^E(w)-\gamma^E(z))}
\,.
\end{align}
 \end{enumerate}
 \end{lemma}
\begin{proof} 
The first identity of part $(i)$ follows from
by \eqref{eq:Psi^V}.

Since $\widehat{\mu}_1(w)=\sum_{v\in V} \mu_1(v) \phi^V_w(v) =\mu_1 * \phi^V_w(v_0)= \gamma^V(w)$ by \eqref{eq:vertex_spherical_functions_as_eigenfunctions}, it follows from Corollary \ref{cor: convolutions_with_vertex-spherical_functions}~$(i)$ and
 the resolvent equation
$
(\mu_1-\gamma^V(z)\,\delta_{v_0})*s_z = \delta_{v_0}
$ 
that $\widehat{s}_z(w)=1/\left(\gamma^V(w)-\gamma^V(z)\right)$. 
Note that   $y_z= \left((q^{-z}-q^z)/(q+1)\right)\,s_z$ by  Theorem \ref{theo:ell^2-spectra}~$(iii)$. Therefore 
%
\begin{equation*} 
\widehat{y}_z(w)=\frac{q^{-z}-q^z}{q+1}\;\frac 1{\gamma^V(w)-\gamma^V(z)}\;.
\end{equation*}
On the other hand, 
\[
\widehat{\widetilde{y}}_z(w) =
\widehat{y}_{z+i\pi/\ln q}(w)= \frac{q^z-q^{-z}}{q+1}\;\frac 1{\gamma^V(w)+\gamma^V(z)}\;.
\]
Hence

\begin{align*}
\widehat L_z(w)&=\frac12 \left(\widehat{y}_z + \widehat{\widetilde{y}}_z\right)(w)
\\[.2cm]
&=
\frac{q^{-z}-q^{z}}{2(q+1)}\;\left( \frac 1{\gamma^V(w)-\gamma^V(z)} - \frac 1{\gamma^V(w)+\gamma^V(z)} \right) 
\\[.2cm]
&= \frac{q^{-z}-q^{z}}{2(q+1)} \frac{2\gamma^V(z)}
{\gamma^V(w)^2 - \gamma^V(z)^2}\;.
\end{align*}
This completes the proof of $(i)$. 
To prove $(ii)$, note that $\widetilde{y}_z(e)= (-1)^{|e|} y_z(e)=y_{\widetilde{z}}$.
Then $M_z(e)= q^{-(2k+1)z}$ if $|e|=2k+1$ and  $M_z(e)=0$ if $|e|=2k$. 
Then, by \eqref{eq:Psi^E}, $M_{\frac12}(e)=\frac2{(q-1)\sqrt{q}}\,\Psi^E(e)$ for $e\neq e_0$ and $M_{\frac12}(e_0)=0$. This proves \eqref{eq:M_1/2_and_Psi^E_1/2}. 

Since $\widehat{\eta}_1(w)=\sum_{e\in E} \eta_1(e) \phi^E_w(e) =\eta_1 * \phi^E_w(e_0)= \gamma^E(w)$ by \eqref{eq:edge_spherical_functions_as_eigenfunctions}, it follows from Corollary \ref{cor:  convolutions_with_edge-spherical_functions}~$(i)$ and
 the resolvent equation
$
(\eta_1-\gamma^E(z)\,\delta_{e_0})*r_z = \delta_{e_0}
$ 
that $\widehat{r}_z(w)=1/\left(\gamma^E(w)-\gamma^E(z)\right)$. 
Note that by  Theorem \ref{theo:ell^2-spectra}~$(iii)$ we have $y_z= \left((q^{1-z}-q^z-(q-1))/2q\right)\,r_z$. Therefore 
%
\begin{equation*} 
\widehat{y}_z(w)=\frac{q^{1-z}-q^z-(q-1)}{2q}\;\frac 1{\gamma^E(w)-\gamma^E(z)}\,.
\end{equation*}
On the other hand, by 
\eqref{eq:z-tilde} and \eqref{eq:change_of_eigenvalue_for_parity_multiplication_of_resolvents_for_edges},
\[
\widehat{\widetilde{y}}_z(w) =
\widehat{y}_{z+i\pi/\ln q}(w)= \frac{q^z-q^{1-z}-(q-1)}{2q}\;\frac 1{\gamma^E(w)+\left(\gamma^E(z)-\frac{q-1}q\right)}\;.
\]
But
\begin{multline*}
\frac 1{\gamma^E(w)-\gamma^E(z)} - \frac 1{\gamma^E(w)+\gamma^E(z)-\frac{q-1}q} \\[.2cm]= \frac{2\gamma^E(z)-\frac{q-1}q}
{\gamma^E(w)^2 - \gamma^E(z)^2-
\frac{q-1}q\; (\gamma^E(w)-\gamma^E(z))}
\end{multline*}
and
\begin{multline*}
\frac 1{\gamma^E(w)-\gamma^E(z)} + \frac 1{\gamma^E(w)+\gamma^E(z)-\frac{q-1}q} \\[.2cm]= \frac{2\gamma^E(w)-\frac{q-1}q}
{\gamma^E(w)^2 - \gamma^E(z)^2-
\frac{q-1}q\; (\gamma^E(w)-\gamma^E(z))}\;.
\end{multline*}
Hence

\begin{align*}
\widehat{M_z}(w)&=\frac12\;\left(\widehat{y}_z-\widehat{\widetilde{y}}_z\right)(w)\\[.2cm]
&=
\frac{q^{1-z}-q^{z}}{4q}\;
\frac{2\gamma^E(w)-\frac{q-1}q}
{\gamma^E(w)^2 - \gamma^E(z)^2-
\frac{q-1}q\; (\gamma^E(w)-\gamma^E(z))}\\[.2cm]
&\qquad -
\frac{q-1}{4q}\;
\frac{2\gamma^E(z)-\frac{q-1}q}
{\gamma^E(w)^2 - \gamma^E(z)^2-
\frac{q-1}q\; (\gamma^E(w)-\gamma^E(z))}
\,.
\end{align*}

This proves $(ii)$. 
Note that the formula for $\widehat{L}_z$ is simpler than for $\widehat{M}_z$, because the expression of $\gamma^V(z)$ does not have a constant term at the numerator, and as a consequence the denominator of $\widehat{L_z}(w)$ does not have a linear term in $\gamma^V(w)$: instead, the denominator of $\widehat{M}_z$ has a linear term in $\gamma^E(w)$.
\end{proof}

\begin{theorem}

\[\widehat{\Psi^E}(w)
= 1 - \frac{(q-1)^2}{4q^2 \left( \gamma(w) - \frac{q-1}{2q}\right)^2 - 4q}\; .
\]
\end{theorem}

\begin{proof}
Recalling that $\gamma^E\left(\frac12\right)= \frac{q-1}{2q}+\frac 1{\sqrt{q}}$ (Theorem \ref{theo:ell^2-spectra}~$(iii)$), we see that the numerator $\gamma^E(1/2)-\frac{q-1}{2q}$ of the expression \eqref{eq:Fourier_transform_of_M_z}
of $M_{\frac12}$ equals $\frac 1{\sqrt{q}}$, and, by \eqref{eq:M_1/2_and_Psi^E_1/2},
\begin{equation}\label{eq:Fourier_transform_of_Psi^E_1/2}
\begin{split}
\widehat{\Psi^E}(w) 
&= 
1 - \frac{(q-1)^2}{4q^2}\;
\frac{1}
{ \left( \gamma(w) - \frac{q-1}{2q} +\frac 1{\sqrt{q}} \right) \left( \gamma(w) - \frac{q-1}{2q}-\frac 1{\sqrt{q}} \right)} 
\\[.2cm]
&
=
1 - \frac{(q-1)^2}{4q^2 \left( \gamma(w) - \frac{q-1}{2q}\right)^2 - 4q}\; .
\end{split}
\end{equation}
\end{proof}
\begin{theorem}
$1/\widehat{\Psi^E}$ is 
the symbol of a left inverse of $\ERad^*\ERad:\calS(E)\to\calS'(E)$, an operator continuous in the topology of $\calS'(E)$.
\end{theorem}
\begin{proof}
By Corollary \ref{cor:R*R_is_a_bounded_convolutor_fron_S_to_S'} and the Paley--Wiener theorem \ref{theo:Paley--Wiener}, the multiplication operator by 
$\widehat{\Psi^E}$ is a distribution on $S_E$ and is bounded on $C^\infty(S_E)$ to the distribution space $\calS'(S_E)$.
Since $\widehat{\Psi^E}$ is smooth in the interior of $S$ and has simple poles at the endpoints, this multiplication maps $C^\infty(S)$ to the space $C$ of smooth functions with simple poles at the endpoints (hence to the distribution space $\calS'(S)$). Hence $C$ is the spherical Fourier transform of the image of $\ERad^*\ERad$ acting on $\calS(E)$.

If there exists an inverse convolution operator, then its symbol must be $1/\widehat{\Psi^E}$, and multiplication by this function should give rise a bounded operator from $\calS'(S_E)$ to $C^\infty(S_E)$, that of course is impossible. 
\\
However, we claim that  multiplication by $1/\widehat{\Psi^E}$ is a continuous operator on $\calS'(S_E)$ to  $\calS'(S_E)$. 
Indeed, we first show that $1/\widehat{\Psi^E} \in C^\infty(S)$. Since $\gamma$ 
is holomorphic, by \eqref{eq:Fourier_transform_of_Psi^E_1/2} it is enough to show that $\widehat{\Psi^E}$ never vanishes in $S$.
This amounts to
\begin{equation*}
4q^2{ \left(\gamma^E(w) - \frac{q-1}{2q}\right)^2 -4q}\neq (q-1)^2
\end{equation*}
for $w=\frac12 +it$, where we can restrict attention to $0\leqslant t \leqslant \pi/\ln q$. As $\gamma^E\left(\frac12+it\right)=\frac{\cos (t \ln q)}{\sqrt{q}}\;+\;\frac{q-1}{2q}$ by  \eqref{eq:gamma}, the non-vanishing condition becomes
\begin{equation*}
4q \cos^2(t \ln q)  \neq  4q+(q-1)^2.
\end{equation*}
This is true for every $t$ when $q\neq 1$, that is, on non-trivial trees.
\par
A similar argument shows that $1/\widehat{\Psi^E}\neq 0$ in the interior of $S$, but has a zero of order one at the endpoints $\pm\rho$. So the multiplication by $1/\widehat{\Psi^E}$ maps $C^\infty(S)$ to the proper subspace of $C^\infty(S)$ of functions vanishing at the endpoints with zeroes of order one, and maps the functions that are smooth in the interior of $S$ and have simple 
poles at the endpoints, that is the image $C$ of $(\ERad^*\ERad)^{\widehat{}}$,
surjectively onto $C^\infty(S)$. 
\end{proof}

\begin{remark}
An easy computation shows that 
\[
\widehat{\Psi^E}\left(\frac12 + it\right)= 1+ \frac {(q-1)^2}{4q \sin^2(t\ln q)}
= \left( \frac{4 q\sin^2(t\ln q)}{(q-1)^2 + 4 \sin^2(t\ln q)}\right) ^{-1}
\]
and
\begin{align*}
c\left(\frac12 + it\right)&= \frac12 \left( 1+i\;\frac {q-1}{2\sqrt{q}\sin(t\ln q)}\right),\\[.2cm]
\left|c\left(\frac12 + it\right)\right|^{-2}&= \frac {16\, q \sin^2(t\ln q)}{4q\sin^2(t\ln q) +(q-1)^2}\; 
\end{align*}
and, for $ t\neq ki/\ln q$,
\[
\phi _{\frac12+it}(e)= q^{-|e|/2}\left(\cos(t\ln q \,|e|) + \frac{q-1}{2\sqrt{q}\sin(t\ln q)}\;\sin (t\ln q \,|e|)   \right). 
\]
 
 Therefore, 
by the inversion formula \eqref{cor:inversion_formula} applied to $1/\widehat{\Psi^E}$,
 \begin{multline*}
\Phi (e)
={\ln q} \;q^{-\frac{|e|}2} \int_0^{\pi/\ln q} \frac{4q \sin^2(t\ln q)}{ 4 \sin^2(t\ln q) +(q-1)^2}\quad
 \frac {16 q \sin^2(t\ln q)}{4q\sin^2(t\ln q) +(q-1)^2} \\[.2cm]
\boldsymbol\cdot \left(\cos(t\ln q \,|e|) + \frac{q-1}{2\sqrt{q}\sin(t\ln q)}\;\sin (t\ln q \,|e|)   \right) dt.
\end{multline*}
The denominators in the first two factors of the integrand never vanish and are bounded away from 0, so these factors are positive and bounded (moreover, each tends to zero quadratically at the endpoints of the interval).
The third factor is bounded in absolute value by $1+\frac{q-1}{2\sqrt{q}}\;|e|$, attained when $t$ tends to 0 or $\pi/\ln q$, 
and
therefore $|\Phi (e)|< C  q^{-\frac{|e|}2} $ for some constant $C$. So the seminorm of order $0$ of $\phi$ is finite, hence $\phi$ belongs to $\calS'_\#$ and to $\ell^p$ for $p>2$. If this were the exact rate of decay, then H\"{o}lder's inequality would not imply that the series that defines $\Phi *\Psi$ is absolutely convergent, because, by \eqref{eq:Psi^E},
  $\Psi\in \ell^r$ only for $r>2$. 
\par
But the actual decay of $\Phi$ is faster, thanks to the third factor in the integral. Indeed, its first summand, $\cos(t\ln q \,|e|)$, multiplied by
the product of the first two factors,
yields its cosine Fourier coefficient of order $|e|$ of this product, that is a $C^\infty$ periodic function; hence this term of the integral vanishes faster than polynomially as $|e|\to\infty$. The second summand of the last factor yields a sine Fourier coefficient that behaves in the same way, because the sine function in its denominator is canceled by the numerators of the other factors. Therefore $\Phi(e) = \tau(|e|)\, q^{-\frac{|e|}2}$ with $\tau$ vanishing at infinity faster than polynomially, hence $\Phi\in\calS_\#$ and the series that defines $\Phi *\Psi$ is convergent because $\Psi\in\calS'_\#$. Hence not only does $\Phi$ define a bounded operator $\calS'(E)\to\calS'(E)$ (Proposition  \ref{prop:convolutions_between_Schwartz_classes_and_distributions}\,$(ii)$) (mapping in particular  $\Psi$ to $\Phi * \Psi \in \calS'_\#$), but it defines also a bounded convolution operator on $\calS(E)$ (Proposition \ref{prop:convolutions_between_Schwartz_classes_and_distributions}\,$(i)$). Note, however, that $\Phi\notin\ell^2$ and does not define a bounded convolution operator on $\ell^2$ to $\ell^\infty$, and the convolution with $\ell^2$ functions, and even with distributions, as for instance $\Psi$, is convergent but may not be absolutely convergent. 
\end{remark}

\begin{theorem}
$1/\widehat{\Psi^E}$ is 
the symbol of a left inverse of $\ERad^*\ERad:\calS(E)\to\calS'(E)$, an operator continuous in the topology of $\calS'(E)$.
\end{theorem}
\begin{proof}
It follows from part $(i)$ 
of Lemma \ref{lemma:Fourier_transform_of_the_resolvent}
 that

\begin{align*}
\widehat{\Psi^V}(w) &= \frac 2{q+1}+ \frac{q-1}{q+1}\,\frac{\sqrt{q}-\frac1{\sqrt{q}}}{q+1}\;\frac{\frac{4\sqrt{q}}{q+1}}{\frac{4q}{(q+1)^2}-\gamma^V(w)^2}
\\[.2cm]
&=\frac 2{q+1}+ \frac{(q+1)\left(\sqrt{q}-\frac1{\sqrt{q}}\right)\;\frac{4\sqrt{q}}{q+1}}{4q-(q+1)^2\gamma^V(w)^2}
\;.
\end{align*}
So $\widehat{\Psi^V}$ is bounded away from zero on the spectrum if and only if, for $\Real\, w=1/2$, there are no real solutions of the equation
\[
\frac{(q-1)^2}{(q+1)^2}\;\frac{2}{\gamma^V(w)^2-\frac{4q}{(q+1)^2}}=1
\]
that is,
\[
\gamma^V(w)^2=\frac{4q+2(q-1)^2}{(q+1)^2}=1+\frac{(q-1)^2}{(q+1)^2}\,.
\]
Here $\gamma^V(w)$ ranges over the spectrum $S_V$, hence $\gamma^V(w)^2$, hence it is bounded by the $\ell^2(V)-$spectral radius, that is less than 1. Or more explicitly, covers the interval $[0, 4q/(q+1)^2]$. But $4q<(q-1)^2+(q+1)^2=2(q^2+1)$ for $q>1$, hence the right hand side is larger than $4q/(q+1)^2$ and the equality  is never satisfied, and the (symbol of) the inverse operator always exists. Of course $\widehat{\Psi^V}(w)$ diverges at the edges of the spectrum, where its reciprocal vanishes. 
In particular, the reciprocal is a $C^\infty$ function
: therefore it defines a continuous multiplication operator on $\calS'(S_V)$ that maps $C^\infty(S_V)$ to $C^\infty(S_V)$. Hence $\widehat{\VRad^*\VRad}$ has a continuous inverse from $\calS'(S_V)$ to $\calS'(S_V)$ that maps  $C^\infty(S_V)$ to $C^\infty(S_V)$ (but not continuous from $\calS'(S_V)$ to  $C^\infty(S_V)$): it is the multiplication operator by $1/\widehat{\Psi^V}$. By the Paley--Wiener theorem \ref{theo:Paley--Wiener}, $\VRad^*\VRad:\calS(V)\to\calS'(V)$ has a left inverse convolution operator: the inverse is the convolution operator
whose spherical Fourier transform is $1/\widehat{\Psi^V}$, that exists in the sense of distributions by the Paley--Wiener theorem since $1/\widehat{\Psi^V}$ is a distribution on $S_V$, and  is a bounded multiplier on $\calS'(V)$
 to $\calS'(V)$ (but not to $\calS(V)$).
 \end{proof}

\part{Trees as flag manifolds}
\section{Flags}\label{Sect:Flags}
In this Section we shall introduce flags on homogeneous and semi-homogeneous trees, define horosphere of flags and the corresponding horospherical Radon transform $\FRad$. Flags are edges with the choice of one of their two vertices, and therefore implicitly correspond to oriented edges. The fact that flags are built by pairing edges with one of the joining vertices gives a factorization of $\FRad$ via $\ERad$ and $\VRad$. This hierarchy, in turn, will allow us to derive an inversion formula for $\FRad$ from those already proved for $\ERad$ and $\VRad$. The hierarchy has its geometric foundation on the fact that the tree can be regarded as a simplicial complex formed by the join of one-dimensional elements (edges) and zero-dimensional elements (vertices). 
\\ We could derive inversion formulas directly, by computing, as in the previous Sections, the volumes of the intersections of flag-horospheres and flag-circles. But, besides being more complicated combinatorially, this procedure would not exploit the simplicial complex geometry. Instead our approach does, and therefore is promising for the future goal of extending the horospherical Radon transform theory to higher rank Bruhat--Tits buildings (see \cite{Brown,Ronan}.

Let us consider a semi-homogeneous tree with homogeneity degrees $p$ and ${q_-}$. Let us observe again that the subsets $V_{q_+}$ and $V_{q_-}$ of vertices of the same parity are the two orbits of $V$ under the automorphism group if ${q_+}\neq {q_-}$, but when ${q_+}={q_-}$ the homogeneous case) the full automorphism group has only one orbit. To encompass the homogeneous case within the semi-homogeneous one without having to make frequent exceptions, we abuse of notation and denote by $\Aut T$, in the homogeneous case, the subgroup of the full group of automorphisms that preserves the parity of vertices.

Edges are non-oriented pairs of vertices. We now introduce oriented edges, that is, oriented pairs of vertices. For the sake of future extensions to affine buildings, where instead of edges and their two endpoints we have higher dimensional simplices, we prefer to regard oriented edges as flags:

\begin{definition}
\label{def:flags}
As anticipated in Subsection \ref{SubS:Vertices_edges_flags},
a flag is a pair $f=(e,v)$ consisting of an edge $e\in E$ and one of its endpoints $v\in V$. We occasionally refer to the vertex $v$ as the \textit{ending vertex}, or \textit{final vertex}, of the flag. In other words, the choice of final vertex induces an orientation on the edge, and in this sense a flag is an oriented edge together with the vertex pointed by the edge. Intuitively, two flags are (or better, have edges) oriented in the same way if one can be moved to the other by a translation along a geodesic that contains both. Instead, they have opposite orientation if, in order to move one to the other, we need a translation and a flip of the endpoint vertex. More precisely, two flags have the same orientation if the distance between their edges is the same as the distance between their vertices. On the other hand, if two flags $f_1=(e_1,\,v_1)$ and $f_2=(e_2,\,v_2)$ have opposite orientation, then $\dist(e_1,\,e_2)=\dist(v_1,\,v_2)\pm 1$, and the sign is $+$ if $v_1$ and $v_2$ belong to the shortest path connecting $e_1$ and $e_2$, whereas is $-$ if the two vertices are external to this shortest path.

We denote by $\pi_E$ and $\pi_V$ the canonical projections from flags $(e,v)$ to their edge component $e$ and vertex component $v$.
\end{definition}

\begin{equation*}
\begin{tikzcd}
{} & F \arrow{ld}[swap]{\pi_E} \arrow{rd}{\pi_V} \\
E & & V
\end{tikzcd}
\end{equation*}

\begin{remark}[Covering maps and orbits under automorphisms]\label{rem:orbits_under_automorphisms}
The embedding $F\subset E\times V$ allows to restrict to $F$ the canonical projections $\pi_E$, $\pi_V$ of $E\times V$ onto the factors: we call these restriction again $\pi_E$, $\pi_V$. Then $F$ splits as the disjoint union $F=F_{q_+}\cup F_{q_-}$, where $F_{q_+}=\pi_V^{-1}V_{q_+}$ and $F_{q_-}=\pi_V^{-1}V_{q_-}$ are the subsets of flags that end with a vertex of homogeneity degree ${q_+}$ or ${q_-}$, respectively. Since a vertex of homogeneity degree ${q_+}$ belongs to ${q_+}+1$ edges, the projection $\pi_V:F_{q_+}\to V_{q_+}$ is $({q_+}+1)$-to-one, and similarly, $\pi_V:F_{q_-}\to V_{q_-}$ is $({q_-}+1)$-to-one. On the other hand, every edge contains two vertices, and so the projection $\pi_E:F\to E$ is two-to-one (it induces a two-fold covering of $E$).

The automorphism group $\Aut T$ acts transitively on $E$ but has the two orbits $V_{q_+}$, $V_{q_-}$ on $V$; moreover, it is clear that an automorphism $\lambda$ preserves the join of edges and vertices, in the sense that it maps an edge $e$ with final vertex $v$ to the edge $\lambda(e)$ with final vertex $\lambda(v)$. Therefore the automorphism group, acting on the space of flags, has exactly two orbits, $F_{q_+}$ and $F_{q_-}$.

The isotropy subgroup of a flag in $\Aut T$ is never transitive on the boundary $\Omega$, not even for homogeneous trees, because there are no automorphisms that reverse the orientation of the flag (that is, the choice of the vertex in its edge). For this reason, the isotropy subgroup of a flag has two orbits on the boundary.
\\
By the way, notice that the same happens for the isotropy group of an edge in a semi-homogeneous non-homogeneous tree. Instead, on a homogeneous trees, there are automorphisms that reverse any edge, that is, swap its vertices, and so the isotropy group of an edge is transitive on $\Omega$. Finally, on any tree, homogeneous or semi-homogeneous, the isotropy subgroup of a vertex is transitive on the boundary, because there are automorphisms that move any of its joining edges to any other joining edge.
\end{remark}

\begin{definition} [Unified notion of distance between vertices, edges and flags: the tree as a simplicial complex metric space]
\label{def:unified_distance} 
Regard  the geodesics of $T$ as lines whose edges have length 1. Then the tree becomes a metric spaces where edges are segments of length 1 and points in these segments have a natural distance. Choose any number $0<\xi<1/4$. The definition of distance between vertices is the natural distance. In each edge $e$, choose the middle point $p_e$: then the distance between edges is the same as the natural distance of their middle points.
\\
For each flag $f=(e,v)$, we introduce a \emph{barycenter}, the point $p_f$ in $e$ at distance $\xi$ from $v$. On the space of flags we introduce the following distance: 
\begin{equation}\label{eq:distance_on_flags}
\dist(f_1,f_2)=\dist(p_{f_1},p_{f_2}).
\end{equation}
%
For instance, the circle of radius $1-2\xi$ around $f_0=(e_0,\,v_0)$ consists of the flag with the same edge $e_0$ but opposite orientation (that is, whose flag-vertex is the vertex $v_1$ that touches $e_0$ but is opposite to $v_0$). The circle of radius 1 consists of ${q}_{v_0}+{q}_{v_1}$ flags, whose edges are all the edges that touch $e_0$ in $v_0$ or in $v_1$ and have the same orientation of $f_0$ (that is, their flag-vertices are adjacent to  $v_0$). If $u$ denotes any vertex adjacent to $v_0$ except $v_{1}$ (there are $q_{v_0}$ of them), then the circle of radius $2\xi$ consists of the flags whose edges are of type $(u,v_0)$ and whose flag-vertex is $v_0$, hence of opposite orientation with respect to $f_0$. On the other hand, for every neighbor $v'_{1}$ of $v_{1}$ different from $v_0$, the flags whose edge is $(v'_{1},v_{1})$ and whose flag-vertex is $v_{1}$ form the circle of radius $2-2\xi$ (again, the orientation of these flags is opposite to $f_0$, and their number is $q_{v_{1}}$). By iteration, we see
that there are three families of radii, $\mathN$, $\mathN + 2\xi$ and $\mathN + 1-2\xi$; the spheres of integral radius $r$ consist of the flags oriented as $f_0$ with edges $r$ steps away from $e_0$.

Summarizing, for $f_1=(e_1,\,v_1)$ and $f_2=(e_2,\,v_2)$, we have
\[
\dist(f_1,\,f_2)=
\begin{cases}
\dist(v_1,\,v_2)      &\quad \text{if $\dist(v_1,\,v_2)=\dist(e_1,\,e_2)$};\\[.2cm]    
\dist(v_1,\,v_2)+2\xi &\quad \text{if $\dist(v_1,\,v_2)<\dist(e_1,\,e_2)$};\\[.2cm]
\dist(v_1,\,v_2)-2\xi &\quad \text{if $\dist(v_1,\,v_2)>\dist(e_1,\,e_2)$}.
\end{cases}
\]
More compactly, in all cases above we can write
\[
\dist(f_1,\,f_2)=(1-2\xi)\dist(v_1,\,v_2) + 2\xi\dist(e_1,\,e_2).
\]
Note that this formula shows that, if $v_0,\,v_1$ and $v_2$ are three consecutive vertices and $e_0=(v_0,\,v_1)$, $e_1=(v_1,\,v_2)$, then the distance between the flag $f=(e_0,\,v_0)$ and its flip $f'=(e_0,\,v_1)$ is $1-2\xi$, while the distance of $f'$ from the flag $f''=(e_1,\,v_2)$ is $2\xi$. But if $T$ is homogeneous, we want to make $F$ two-point homogeneous under isometries, hence we must assign different distances to these pairs of flags in different relative geometric positions. So we must request $1-2\xi\neq 2\xi$, that is, the distance between flag-vertex and  flag-barycenter must be $\xi\neq1/4$. We choose $0<\xi<1/4$, but equivalently we might  have chosen  $0<\xi<1/2$ and $\xi\neq 1/4$.

More generally, we extend the notion of distance to $V\cup E\cup F$ as follows.  If $v\in V$, $e\in E$ and $f\in F$,
\begin{align*}
\dist (v,e)=\dist(v,p_e);\\
\dist (v,f)=\dist(v,p_f);\\
\dist (e,f)=\dist(p_e,p_f).
\end{align*}
Observe that $\dist(e,f)=\dist(e,e_f)+\dist(e_f,f)$, and $\dist(v,e)=\dist(v,v')+1/2$, where $v'$ is the vertex of $e$ closer to $v$ (in particular, $\dist(v,e)=1/2$ if $v\in e$, and the distance between vertices and edges takes values in $\mathN+1/2$. Finally, $\dist(v,f)=d(v,v_f)+(1-\xi)$ if $v$ is on the positive side of $f$ (that is, in the sector subtended by $e_f$ on the side of $v_f$), and $\dist(v,f)=d(v,v_f)+(1-\xi)$ otherwise.
\end{definition}

\begin{definition}[Flag-horospheres]
Let us choose and fix a reference flag $f_0=(e_0,v_0)$ and, for every boundary point $\omega\in\Omega$, define the \textit{flag-horospherical index} of a flag $f=(e,v)$ as the
pair of its edge and vertex horospherical indices (with respect to $e_0$ and $v_0$, respectively) introduced in~Definition~\ref{def:horospherical_index}:
\begin{equation*}
       h(f,f_0,\omega)
=\left(h(e,e_0,\omega),h(v,v_0,\omega)\right).
\end{equation*}
Then the definition of horospheres as sets where the horospherical index is constant extends naturally to flags: once a reference flag $f_0=(e_0,v_0)$ is chosen, a flag-horosphere $\boldh_{n}(\omega,f_0)$ tangent at $\omega\in\Omega$ is the set of all flags with the same flag-horospherical index $n$ with respect to $\omega$.
\\
It is clear that this definition of flag-horospherical indices characterizes flag-horospheres tangent at $\omega$ as the sets of flags $f$ with a preassigned constant value (the flag-horospherical index) of the difference between the number of flags whose vertex points toward $\omega$ (say, positively oriented) and those whose vertex points in the opposite direction (negatively oriented) in the chain of flags from $f_0$ to $f$.,This is the precise analogue of the characterization of vertex-horospheres and edge-horospheres given in Remark \ref{rem:horospherical_indices_and_side_steps}.
\end{definition}

\begin{lemma}\label{lemma:coordinates_of_flag-horospheres}
Let $e_0$ be a reference edge and $v_0$ one of its vertices, and choose $v_0$ as the reference edge and $f_0=(e_0,v_0)$ as the reference flag. A flag $f=(e,v)$ belongs to the flag-horosphere $\boldh_{(n,k)}(\omega,f_0)$ tangent at $\omega\in\Omega$ if and only if $k=n$ or $k=n-1$. For each vertex $v$ in the vertex-horosphere $\boldh_{n}(\omega,v_0)$, there is only one flag in the flag-horosphere $\boldh_{(n,n)}(\omega,f_0)$ that contains $v$ (flag in the same side of $\omega$ with respect to $v$), but there are ${q_-}$ flags in the flag-horosphere $\boldh_{(n,n-1)}(\omega,f_0)$ that contain $v$ (flags in the opposite side of $\omega$ with respect to $v$).
\end{lemma}

\begin{proof}
In order to have $f=(e,v)\in\boldh_{(n,k)}(\omega,f_0)$, it must happen that $v\in\boldh_{n}(\omega,v_0)$, that $e\in\boldh_{e_0}(\omega,k)$ and that $(e,v)$ is a flag, that is, $v$ is the terminal vertex of $e$.

Consider a flag $f_1=(e_1,v_1)$ that lies in $\boldh_{(n,k)}(\omega, f_0 )$. Then $e_1\in\boldh_{k}(\omega,e_0)$ and $v_1\in\boldh_{n}(\omega,v_0)$. By the definition of vertex and edge horospherical indices given in~Subsection~\ref{SubS:Horospheres}, and in particular by~Remark~\ref{rem:horospherical_indices_and_side_steps}, $v_1$ belongs to $e_1$ if and only if either $k=n$ (in which case $v_1$ is the vertex of $e_1$ opposite to the direction of the boundary point $\omega$, that is, farther from $\omega$), see~Figure~\ref{Fig:Flag-horospheres_with_equal_indices},
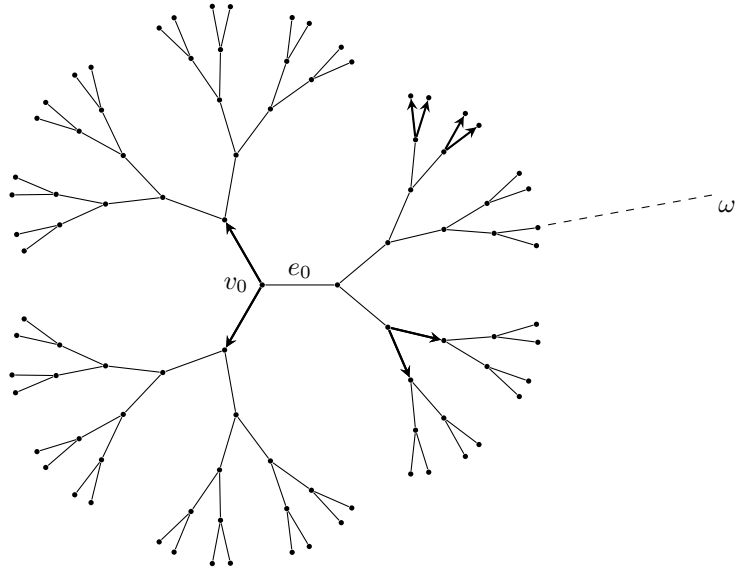
\begin{figure}
\silenceable{\begin{tikzpicture}[grow cyclic,level/.style={
 sibling  angle=120 /(1+1/2)^(#1-1),
 level distance=10mm/(1+1/7)^(#1-1)},
 >=stealth]
\begin{scope}
\tikzstyle{every node}=[circle,fill,inner sep=.7pt]
\path[rotate=-120]node(startA){}
 child foreach\x in{,,}{node{}
   child foreach\x in{,} {node{}
     child foreach\x in{,} {node{}
       child foreach\x in{,} {node{}
         child foreach\x in{,} {node{}
         }
       }
     }
   }
 };
\end{scope}
%
\foreach\x in{1,2}{
\foreach\y in{1,2} {
 \draw[->,thick](startA)         --(startA-\x);
 \draw[->,thick](startA-3-1)     --(startA-3-1-\x);
 \draw[->,thick](startA-3-2-2-\x)--(startA-3-2-2-\x-\y);
};
};
\node at(-.1,0)[label=-180:$v_0   $]{};
\node at(.5,0)[label=90:$e_0$]{};
\draw[dashed](3.8,.8)--(6,1.2)node[label=- 60:$\omega$]{};
\end{tikzpicture}}
\caption{The flag-horosphere tangent at $\omega$ of index $(-1,\,-1)$ (that is, $-\frac34$ if the flag-barycenter is chosen at $\xi=\frac14$) with respect to $v_0$. The flag edges are marked by arrows, with the arrowtips pointing to the respective flag vertices. In this case all arrows point away from $\omega$.}

\label{Fig:Flag-horospheres_with_equal_indices}
\end{figure}
or else $k=n-1$ (in which case $v_1$ is the vertex of $e_1$ closer to $\omega$), see~Figure~\ref{Fig:Flag-horospheres_with_different_indices}.
\begin{figure}[h!]
\silenceable{\begin{tikzpicture}[grow cyclic,level/.style={
 sibling  angle=120 /(1+1/2)^(#1-1),
 level distance=10mm/(1+1/7)^(#1-1)},
 >=stealth]
\begin{scope}
\tikzstyle{every node}=[circle,fill,inner sep=.7pt]
\path[rotate=-120]node(startA){}
 child foreach\x in{,,}{node{}
   child foreach\x in{,} {node{}
     child foreach\x in{,} {node{}
       child foreach\x in{,} {node{}
         child foreach\x in{,} {node{}
         }
       }
     }
   }
 };
\end{scope}
%
\foreach\x in{1,2}{
\foreach\y in{1,2} {
 \draw[<-,thick](startA)         --(startA-\x);
 \draw[<-,thick](startA-3-1)     --(startA-3-1-\x);
 \draw[<-,thick](startA-3-2-2-\x)--(startA-3-2-2-\x-\y);
};
};
\node at(-.1,0)[label=-180:$v_0   $]{};
\node at(.5,0)[label=90:$e_0$]{};
\draw[dashed](3.8,.8)--(6,1.2)node[label=- 60:$\omega$]{};
\end{tikzpicture}}
\caption{The flag-horosphere tangent at $\omega$ of index $(0,\,-1)$ (that is, $-\frac14$ if the flag-barycenter is chosen as before at $\xi=\frac14$) with respect to $v_0$. Again, the flag edges are marked by arrows, with the arrowtips pointing to the respective flag vertices. In this case each such vertex belongs to $q$ flags in the horosphere, and all arrows point towards $\omega$.}
\label{Fig:Flag-horospheres_with_different_indices}
\end{figure}
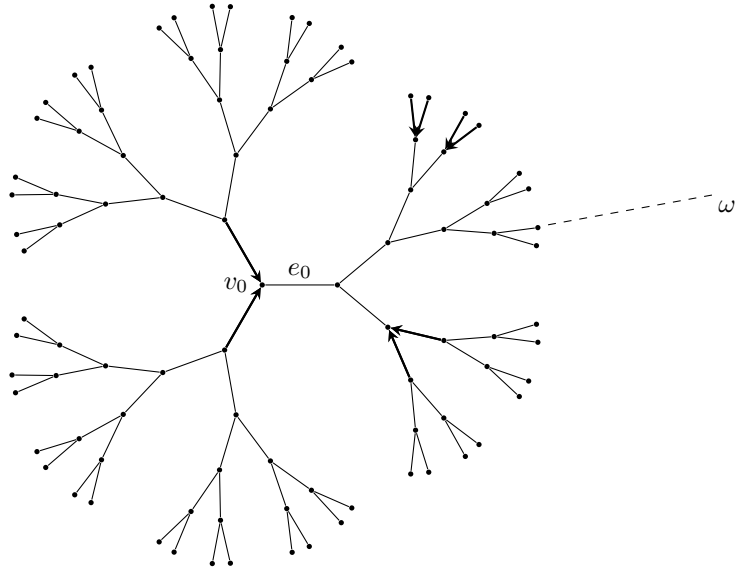
In the first case, each vertex that belongs to a flag of index $(n,n)$ is the final vertex of only one edge in that flag. In the second case, there are ${q_-}$ edges in flags of horospherical index $(n,n-1)$ that share the same final vertex $v$, that is, the ${q_-}$ edges with final vertex $v$ that point towards the endpoint $\omega$.
\end{proof}

\begin{remark} [Projections of flag-horospheres onto vertex and edge-horospheres]\label{rem:two-fold_coverings}
By the previous Lemma \ref{lemma:coordinates_of_flag-horospheres}, if $\boldh(\omega, f_0)$ is a flag-horosphere tangent at $\omega$, the vertices of its flags form a vertex-horosphere. However, each vertex-horosphere tangent at $\omega$ is the set of vertices of two flag-horospheres (according to the cases $k=n$ or $k=n-1$ in Lemma \ref{lemma:coordinates_of_flag-horospheres}): these are the flags whose edges are closer to (respectively, farther from) $\omega$ than these vertices. 
\\
Similarly, the set of edges of a flag-horosphere gives rise to an edge-horosphere, but each edge-horosphere is the set of edges of two flag-horospheres, according to the choice of the flag-vertex among the two edge-vertices.

So, the canonical projections of flag-horospheres are well defined onto vertex and edge-horospheres and are two-fold coverings.

\begin{equation}\label{diag:flag-horospheres}
\begin{tikzcd}
{} & \HorF
\arrow{ld}[swap]{\pi_E}{2:1} \arrow{rd}[swap]{2:1}{\pi_V} \\
\HorE & & \HorV
\end{tikzcd}
\end{equation}
However, $\pi_E \times \pi_V$ is a one-to-one projection onto $\HorE \times \HorV$, because, if the flags in two flag-horospheres have both the same edges and the same vertices, then obviously they must coincide.
\end{remark}

\begin{remark}[Orbits under automorphisms in the flag manifold]\label{rem:orbits_over_F-horospheres}
We have seen that the automorphism group has two orbits on the flag manifold, namely the sets of flags whose final vertices have homogeneity degrees ${q_+}$ or ${q_-}$, respectively. Since automorphisms map vertex-horopheres to vertex-horospheres and edge-horospheres to edge-horospheres, they map flag-horospheres to flag-horospheres. However, consider a flag $f(e,v)$ in a flag-horosphere tangent at a boundary point $\omega$. The edge $e$ splits the tree into two connected components: its final vertex $v$ may be either in the component whose boundary contains $\omega$ or in the opposite component. The automorphism group cannot interchange the side of the final vertex, because it cannot swap opposite vertices of the same edge, since it must preserve the homogeneity degrees (for a homogeneous tree this is of course false for the full automorphism group, but we restricted attention to its subgroup that preserves edge orientation, that is, that does not swap these vertices). Therefore there are four orbits of the action of the automorphism group over $\HorF$: they are given by all flag-horospheres whose flags $(e,v)$ have final vertex on the side of the tangency point or, respectively, on the opposite side, and  homogeneity degree $p$ or ${q_-}$ respectively.
\end{remark}

\section{The flag-horospherical Radon transform and its inversion}
The reader can check that the contents of this Section hold, more generally, for all trees, not only homogeneous or semi-homogeneous.

\begin{definition}
\label{def:FRad}
The \textit{flag-horospherical Radon transform} $\FRad$ is the operator from functions on $F$ to functions on the flag-horospherical fiber bundle, defined by
\begin{equation*}
 \FRad                 u(\omega,n)
=\sum_{f\in\boldh_{n} (\omega,f_0)} u(f).
\end{equation*}
\end{definition}

\subsection{Canonical projections for functions on flags}
Functions on $E$ and on $V$ embed into functions on $F$, in the following sense: for any $h:E\to\mathC$ define $\widetilde{\pi}_E^{-1} h(f)=h(e)$ if $f\in F$ and $e\in E$ with $\pi_E(f)=e$. Similarly, for $h:V\to\mathC$ define
$\widetilde{\pi}_V^{-1} h(f)=h(v)$ if $f\in F$ and $v\in V$ with $\pi_V(f)=v$.

Conversely, functions on $F$ project to functions on $E$, $V$ via a composition with the respective canonical projections $\pi_E$, $\pi_V$ in the following sense. Notice that
 $\pi_E$ is a twofold covering (because the same edge belongs to two flags according to the choice of the vertex) and $\pi_V$ is a covering of order ${q_-}_v +1$ at the vertex $v$, because this vertex belongs to a different flag for each incoming edge. Then, for any $h:F\to\mathC$, $v\in V$ and $e\in E$ we set
\begin{definition}\label{def:projections_of_flag-functions} [Symmetric projections for functions on flags]\label{def:projections_on_flags}
 \begin{align}
 \begin{split}\label{eq:projections_on_flags}
 \widetilde{\pi}_V h(v)&=
 \sum_{f:\,\pi_V f=v} h(f)\,,\\[.2cm]
 \widetilde{\pi}_E h(e)&=
 \sum_{f:\,\pi_E f=e} h(f)\,.
 \end{split}
 \end{align}

 It is obvious that these projections map $L^1$ functions on flags to $L^1$ functions on vertices and ends, respectively.
 \begin{equation}\label{diag:L^1}
\begin{tikzcd}
\ell^1(E) & & \ell^1(V) \\
{} & \ell^1(F) \arrow{lu}[swap]{\widetilde{\pi}_E} \arrow{ru}{\widetilde{\pi}_V}
\end{tikzcd}
\end{equation}
The inverse maps of these projections are set-valued functions from $\ell^1(E)$ and $\ell^1(V)$ to (sets of functions in) $\ell^1(F)$. However, we select a particular representative in each of these sets in $\ell^1(F)$, the most symmetric one, by defining
\begin{align}
 \begin{split}\label{eq:inverse_projections_on_flags}
 \widetilde{\pi}_V^{-1} g(f)&=\frac 1{{q_-}_v+1} \, g(v) \qquad\text {if }\quad
 \pi_E f=v\,,\\[.1cm]
 \widetilde{\pi}_E^{-1} g(f)&=\frac 1{2} \,g(e) \qquad\qquad\quad\text {if }\quad
 \pi_E f=e\,.
 \end{split}
 \end{align}

 It is immediately verified that $\widetilde{\pi}_V\widetilde{\pi}_V^{-1}$ and $\widetilde{\pi}_E\widetilde{\pi}_E^{-1}$ are the identity operators on functions on vertices or edges, respectively; instead, $\widetilde{\pi}_V^{-1}\widetilde{\pi}_V$ is the operator on functions on flags that averages the values over the set of flags with the same terminal vertices, and
 $\widetilde{\pi}_V^{-1}\widetilde{\pi}_V$ is the operator that averages on the two flags with the same edge.
  \end{definition}

 \begin{remark}[Projecting functions of flag-horospheres to functions of vertex or edge-horospheres]\label{rem:projections_of_functions_of_flag-horospheres} 
 Now a comment is in order, concerning the way to project functions of flag-horospheres to functions of vertex or edge-horospheres. In view of the two-fold coverings considered in diagram \ref{diag:flag-horospheres}, each vertex-horosphere and edge-horosphere is the image under the respective canonical projections of two flag-horospheres. We want to consider a linear projection operator over functions, so when we project a function $h$ on $\HorF$ to functions $g_E$ on $\HorE$ and $g_V$ on $\HorV$ we have one degree of freedom, since we have two values available for $h$. Often, the simplest choice is the most symmetric one, that consists in taking sums with equal coefficients, as in the case of the projections 
 $\widetilde{\pi}_E$ and  $\widetilde{\pi}_E$ of Definition \ref{def:projections_of_flag-functions}. However, a one parameter set of different choices is available. Every different choice made here might give rise to a different expression of the inversion formula for $\FRad$ that we will prove later (of course, as before, all these expression would coincide when restricting attention to functions in the image of $\FRad$). Nevertheless, we shall restrict attention to the present choice, the most symmetric and easier to write. For future referencee, let us summarize our choice and notation in a diagram,
where $\mathfrakF(\HorV\to\mathC):=\mathC^{\HorF}$ denotes the space of functions defined on vertex-horospheres, and similarly for edge-horospheres. Here of course the action of $\widetilde{\pi}_V$ and $\widetilde{\pi}_E$ on functions on flag-horospheres is again the sum as in
\eqref{eq:projections_on_flags}:
 \begin{align}
 \begin{split}\label{eq:projections_for_functions_on_flag-horospheres}
 \widetilde{\pi}_V u(\boldh_{n}(\omega,v_0))&=
 \sum_{\boldh_{(n,k)}(\omega,f_0):\,\pi_V \boldh_{(n,k)}(\omega,f_0)=\boldh_{n}(\omega,v_0)} u(\boldh_{(n,k)}(\omega,f_0))\,,\\[.2cm]
\widetilde{\pi}_E u(\boldh_{n}(\omega,e_0))&=
 \sum_{\boldh_{(n,k)}(\omega,f_0):\,\pi_E \boldh_{(n,k)}(\omega,f_0)=\boldh_{k}(\omega,e_0)} u(\boldh_{(n,k)}(\omega,f_0))\,.
 \end{split}
 \end{align}

\begin{equation}\label{diag:functions_on_flag-horospheres}
\begin{tikzcd}
{} & \mathfrakF(\HorF\to\mathC)
\arrow{ld}[swap]{\widetilde{\pi}_E} \arrow{rd}{\widetilde{\pi}_V} \\
\mathfrakF(\HorE\to\mathC) & & \mathfrakF(\HorV\to\mathC)
\end{tikzcd}
\end{equation}
One fact is relevant here. The projections $\widetilde{\pi}_E$ and $\widetilde{\pi}_V$ are two-to-one maps onto $\mathfrakF(\HorV\to\mathC)$ and $\mathfrakF(\HorV\to\mathC)$, but $\widetilde{\pi}_E \times \widetilde{\pi}_V$ is a one-to-one projection onto $\mathfrakF(\HorE\to\mathC) \times \mathfrakF(\HorV\to\mathC)$, because $\pi_E \times {\pi}_V$ is a one-to-one projection onto $\HorE \times \HorV$, as observed at the end of Remark \ref{rem:two-fold_coverings}.
 \end{remark}
 
 By the triangular inequality, the following is clear:
\begin{lemma} $\widetilde{\pi}_V$ and $\widetilde{\pi}_E$ are injections: $ \|\widetilde{\pi}_Vh\|_{\ell^1(V)}\leqslant \|h\|_{\ell^1(F)}$, $\quad \|\widetilde{\pi}_Eh\|_{\ell^1(E)}\leqslant \|h\|_{\ell^1(F)}$.
\end{lemma}

\subsection{Lifting to $F$ and image of $\VRad$ and $\ERad$}
\begin{remark}[Lifting procedure] \label{rem:reconstruction_from_projections}
Given functions $g_E$ on $E$ and $g_V$ on $V$, we want to find a lifting to $F$, that is, a function $h$ on $F$ that satisfies the projection identities $\widetilde{\pi}_V h = g_V$ and $\widetilde{\pi}_E h = g_E$. Because of the fact that the coverings from $F$ to $V$ and $E$ are not one to one, this reconstruction is not unique.
\\
For this goal, let us start with any vertex $v$ and, for all edges $e_k=(v_k,v)$ that contain $v$, choose any values $h(f_k)$ on the flags $f_k=(e_k,v)$, with the only constraint that $\sum_{f=(e_k,v)\,k=1,\dots,{q_-}_v+1} h(f)= ({q_-}_v+1)g_V(v)$. Then assign to the opposite flags $f'_k=(e_k,v_k)$ the value $h(f'_k)=2g_E(e)-h(f_k)$. Then the function $h$ satisfies the required projection identities for all flags whose vertices and edges are at distance at most 1 from $v$. Now we move one step out, to the vertices $v_k$, and iterate the procedure. This shows that the function $h$ satisfies the required identities everywhere.
\end{remark}

\begin{lemma}\label{lemma:image_of_FRad}
If $h$ is a function on flags and $g_V=\widetilde{\pi}_V h$ and $g_E=\widetilde{\pi}_E h$, then
\begin{equation}\label{eq:image_of_FRad}
\sum_{v\in V}g_V(v)=\sum_{e\in E}g_E(e)\,.
\end{equation}
\end{lemma}
\begin{proof} $\sum_{f\in F}h(f)=\sum_{v\in V}g_V(v)=\sum_{e\in E}g_E(e)$, because 
 every flag $f$ gives only one contribution to each of the three sums: namely, in the term corresponding to $f$ 
 itself in the first sum, and to its vertex or edge in the other two.
\end{proof}

\begin{lemma}\label{lemma:injectivity_on_L^1_of_flags} 
If $g_V\in \ell^1(V)$ and $g_E\in \ell^1(E)$, then the lifting of the previous Remark is unique, and belongs to $\ell^1(F)$.
\end{lemma}
\begin{proof}
The statement is equivalent to the following uniqueness property: the only $h\in \ell^1(F)$ such that $\pi_Vh=\equiv 0 \equiv \pi_Eh$ is $h=0$. To prove this, let $g_V=g_E=0$. By contradiction, suppose that, in the above lifting procedure, we choose a nonzero value at a given flag $f=(e,v): h(f)=\alpha\neq 0$. Let $e=(v',v)$ and $v_k, k=1,\dots,{q_-}_v$ be the neighbors of $v$; write $e_k=(v_k,v)$ and $f^{(+)}_k=(e_k,v_k)$.
Then, at the first iteration, the procedure sets $\sum_{k=1}^{{q_-}_v} h(f_k)=-h((v,v'),v')=-\alpha$, and so $\sum_{k=1}^{{q_-}_v} |h(f_k)|\geqslant|\alpha|$. Denote by $S^-_v$ the \emph{incoming star} of $v$, consisting of all the flags $f^{(-)}_k=(e_k,v)$ with vertex $v$: what we have just shown is that $\|h\|_{S_v}\geqslant 2|\alpha|$. Since $g_E=0$, the values of $h$ at these flags are the opposites of the values at the flags $f^{(+)}_k$. Now consider the flags $f'_{jk}$ in the incoming stars at the vertices $v_k\neq v'$: for each such star, the same procedure shows that
$\sum_{k=1}^{{q}_{v_k}} |h(f'_{jk})|\geqslant|h(f^{(+)}_k)|$, and so the sum of the modules over all flags that belong to some incoming star at one of the vertices $v_k\neq v'$ is not less than $|\alpha|$. Proceeding iteratively in this way, we see that $h\notin \ell^1(F)$, a contradiction.
\end{proof}

\begin{definition}
[Forward and backward sectors subtended by a flag]
For convenience, let us define backward and forward sectors induced by a flag $f$ as the subsets of $E \times V$ denoted by $S_-(f)$, $S_+(f)$ respectively, given by
\begin{align*}
S_-(f)&=\{(e,v):\,d(e,e_f)<d(e,v_f)\quad\text{and }\quad d(v,e_f)<d(v,v_f)\}\\[.1cm]
S_+(f)&=\{(e,v):\,d(e,e_f)>d(e,v_f)\quad\text{and }\quad d(v,e_f)>d(v,v_f)\}
\end{align*}
Notice that the fact that the distance between edges and vertices never coincides with the distance between edges or between vertices has as consequence the fact that $S_-(f)\cup S_+(f)=E\times V$.
\end{definition}
Let us denote by $\widetilde{\pi}$ the joint projection of $\ell^1(F)$ to $\ell^1(E) \times \ell^1(V)$ whose existence is proved in Lemma \ref{lemma:injectivity_on_L^1_of_flags}: $\widetilde{\pi}(f)=(\widetilde{\pi}_E(f), \widetilde{\pi}_V(f))$. Moreover, let 
$v_f\in V$ and $e_f\in E$  the vertex, respectively edge, of the flag $f\in F$. 

\begin{theorem}[Image of $\FRad$ and factorization of its inverse  through the inverses of $\VRad$ and $\ERad$]\label{theo:image_of_FRad}
The pair $(g_V,g_E)$ satisfies condition \eqref{eq:image_of_FRad} if and only if it belongs to the image of $\widetilde{\pi}$.
If the pair $(g_V,g_E)$ belongs to the image of $\widetilde{\pi}$, then the inverse flag-horospherical Radon transform $\FRad^{-1}$ is the following:
for every flag $f\in F$, for any $\lambda\in \mathC$,
\begin{equation*}
\begin{split}
\pi^{-1}(g_E,g_V)(f)
&=         \lambda \biggl(\sum_{e:\,d(e,e_f)<d(e,v_f)}g_E(e)
                         -\sum_{v:\,d(v,e_f)<d(v,v_f)}g_V(v)\biggr)\\
&\qquad-(1-\lambda)\biggl(\sum_{e:\,d(e,e_f)>d(e,v_f)}g_E(e)
                         -\sum_{v:\,d(v,e_f)>d(v,v_f)}g_V(v)\biggr)
\end{split}
\end{equation*}
\end{theorem}
\begin{proof}
Let $C=\sum_{v\in V}g_V(v)=\sum_{e\in E}g_E(e)$ as in Lemma \ref{lemma:image_of_FRad}. Then $\sum_{e\in S_-} g_E(e)=C-\sum_{e\in S_+} g_E(e)$, and similarly for the sum over vertices. Then it is clear that, for every $\lambda$, 
\begin{multline*} 
\lambda\left(\sum_{e\in S_-} g_E(e)-\sum_{ev\in S_-} g_V(v)\right)-(1-\lambda)\left(\sum_{e\in S_+} g_E(e)-\sum_{v\in S_-+} g_V(v)\right)\\
=
\sum_{e\in S_-} g_E(e)-\sum_{ev\in S_-} g_V(v)=
\sum_{e\in S_+} g_E(e)-\sum_{v\in S_-+} g_V(v)
\end{multline*}
independently of the choice of $\lambda$.
\\
Now let us prove the identity of the statement: it is enough to limit attention to $\lambda=1$, hence the right hand side becomes 
\[
\sum_{e:\,d(e,e_f)<d(e,v_f)}g_E(e)
                         -\sum_{v:\,d(v,e_f)<d(v,v_f)}g_V(v)\,
\] 
the difference of the two sums over the background sector.
This difference vanishes at each pair $(e,v)\in S_-(f)$ for the same argument of the proof of Lemma \ref{lemma:image_of_FRad}. The only contributions to the right hand side of the identity of the statement arise when the vertex is $v_f$ and the edge $e_f$. Indeed, the set $\{ (e,v): d(e,e_f)<d(e,v_f), d(v,e_f)<d(v,v_f)\}$ consists of all pairs $(e,v)$ in $S_-$, including the pair $(e_f,v')$ where $v'$ is the other vertex of $e_f$, but excluding $(e_f,v_f)$.
A straightforward inspection of this case shows that the identity holds. Indeed, the non-zero contribution to the right hand side is $u(f):=g_E(e_f)-g_V(v_f)$, and by applying the projections of Definition \ref{def:projections_on_flags} one easily sees that $\widetilde{\pi}_V(u)=g_V$ and 
$\widetilde{\pi}_E(u)=g_E$.

Indeed, let us now take $\lambda=1/2$ for simplicity. Then it is immediately seen that, for each edge $e=(v_{e,1},v_{e,2})$, one has $\widetilde{\pi}^{-1}(g_V,g_E)(e,v_{e,1})+\widetilde{\pi}^{-1}(g_V,g_E)(e,v_{e,2})=g_E(e)$; and, for every vertex $v$, $\lambda=1/({q_-}_v+1)$: then $\sum_{e:\,v\in e}\widetilde{\pi}^{-1}(g_V,g_E)(e,v)=g_V(v)$. Therefore the necessary condition \eqref{eq:image_of_FRad} is also sufficient.
\end{proof}
\subsection{Inversion of $\FRad$}
We are now ready to invert $\FRad$ by factoring the inverse through the inverses of $\ERad$ and $\VRad$, as follows.  Consider the diagram obtained by composing \eqref{diag:L^1}
 and \eqref{diag:functions_on_flag-horospheres}:

\begin{equation*}
\begin{tikzcd}
{} & \ell^1(F) \arrow{ld}[swap]{\widetilde{\pi}_E} \arrow{d}{\FRad}\arrow{rd}{\widetilde{\pi}_V} \\
\ell^1(E)\arrow{d}{\ERad} & \mathfrakF(\HorF\to\mathC)
\arrow{ld}[swap]{\widetilde{\pi}_E} \arrow{rd}{\widetilde{\pi}_V} & \ell^1(V) \arrow{d}{\VRad} \\
\mathfrakF(\HorE\to\mathC) & & \mathfrakF(\HorV\to\mathC)
\end{tikzcd}
\end{equation*}

As observed at the end or Remark \ref{rem:projections_of_functions_of_flag-horospheres}, the projections $\widetilde{\pi}_E$ and $\widetilde{\pi}_V$ are two-to-one maps onto $\mathfrakF(\HorV\to\mathC)$ and $\mathfrakF(\HorV\to\mathC)$, but $\widetilde{\pi}_E \times \widetilde{\pi}_V$ is a one-to-one projection onto $\mathfrakF(\HorE\to\mathC) \times \mathfrakF(\HorV\to\mathC)$. Also, $\widetilde{\pi}_E\times \widetilde{\pi}_V$ is injective on $\ell^1(F)$, by Lemma \ref{lemma:injectivity_on_L^1_of_flags}.
Therefore all the morphisms in the following variant of the previous diagrams are injective:

\begin{equation}\label{diag:FRad_to_product_of_factors}
\begin{tikzcd}
\ell^1(F) \arrow{r}{\widetilde{\pi}_E\times\widetilde{\pi}_V} \arrow{d}{\FRad} & \ell^1(E)\times \ell^1(V) \arrow{d}{\ERad\times\VRad} \\
\mathfrakF(\HorF\to\mathC) \arrow{r}{\widetilde{\pi}_E\times\widetilde{\pi}_V} & \mathfrakF(\HorE\to\mathC)\times\mathfrakF(\HorV\to\mathC)
\end{tikzcd}
\end{equation}
If we can invert the flow of this diagram, then we can express the inverse of $\FRad$ as
\begin{equation}\label{eq:FRad-inverse}
\FRad^{-1}= (\widetilde{\pi}_E\times \widetilde{\pi}_V) \circ (\ERad^{-1} \times \VRad ^{-1}) \circ (\widetilde{\pi}_E^{-1}\times \widetilde{\pi}_V^{-1})\,.
\end{equation}
Since all the maps are injective, this is possible if and only if the image of $\mathfrakF(\HorF\to\mathC)$ under $\widetilde{\pi}_E\times \widetilde{\pi}_V$ is contained in the image of $\ell^1(F)$ under $\widetilde{\pi}_E\times \widetilde{\pi}_V$. This, in turn, is equivalent to the commutativity of the diagram. So, we now prove that the diagram is commutative. Here, if course, is where the particular choice of inverse projections made in \eqref{eq:inverse_projections_on_flags}
(see Remark \ref{rem:projections_of_functions_of_flag-horospheres} 
) matches the choice made in defining the projections \eqref{eq:projections_on_flags} to produce the correct inverse; other compatible choices would yield different expressions. However, we shall not write down the explicit formula for the composition of operators
in \eqref{eq:FRad-inverse}: the interested reader can easily derive it from the inversion formula of Theorem \ref{theo:image_of_FRad}.

\begin{theorem} The diagram \eqref{diag:FRad_to_product_of_factors} is commutative, hence the inverse flag-horospherical Radon transform is given by \eqref{eq:FRad-inverse}.
\end{theorem}
\begin{proof}
By linearity, in order to show that the diagram is commutative, it is enough to restrict attention to the function $\delta_f$ with value 1 on an arbitrary flag $f$ and 0 elsewhere. Then $\FRad(\delta_f)=\chi_{\HorF^{(f)}}$ is the characteristic function of the flag-horospheres containing $f$. Hence, by \eqref{eq:projections_for_functions_on_flag-horospheres}, the action of $\widetilde{\pi}_V$ on $\FRad(\delta_f$  yields, on each vertex-horosphere, the sum of the values of $\chi_{\HorF^{(f)}}$ on the two flag-horospheres that cover this edge-horosphere in the sense of Remark \ref{rem:two-fold_coverings}. However, only one of these two flag-horospheres contains $f$ (the other contains, instead, flags with vertex $v_f$ but whose edge is on the opposite side of $f$ with respect to $v_f$: see again Remark \ref{rem:two-fold_coverings}). So, on the vertex-horospheres that contain the vertex $v_f$, $\widetilde{\pi}_V\FRad(\delta_f)$ has value 1 and elsewhere 0: let us denote this function by $\beta_f$. A completely analogous estimate holds for $\widetilde{\pi}_E\FRad(\delta_f)$.

Now let us follow the diagram on the other side. The projection $\widetilde{\pi}_V$ maps $\delta_f$ to $\delta_{v_f}$, by \eqref{eq:projections_on_flags}. Then $\VRad$ maps $\delta_{v_f}$ to the characteristic function of the vertex-horospheres that contain $v_f$, that is, exactly to $\beta_f$. A completely analogous result holds for 
$   \ERad\widetilde{\pi}_E \delta_f$. This proves that the diagram is commutative.
\end{proof}

\part*{Bibliography}
%



\end{document}